\documentclass[11pt, reqno]{article}

\usepackage{stmaryrd}
\usepackage{amsfonts}
\usepackage{amssymb}
\usepackage{amsthm}
\usepackage{amsmath}
\usepackage{amscd}
\usepackage{hyperref}
\usepackage{enumitem}
\usepackage[all]{xy}
\usepackage{tikz}
\usetikzlibrary{arrows}
\usepackage{tikz-cd}
\usepackage{graphicx}
\usepackage[top=3.5cm, bottom=3.5cm, left=2.5cm, right=2.5cm]{geometry}
\usepackage[capitalise,compress]{cleveref}
\usepackage{makecell}

\title{\bfseries Balls in groups: volume, structure and growth}
\date{}
\author{\bfseries Romain Tessera and Matthew Tointon}

\crefname{prop}{Proposition}{Propositions}

\newtheorem{prop}{Proposition}[section]
\newtheorem{theorem}[prop]{Theorem}
\newtheorem{lemma}[prop]{Lemma}

\newtheorem{corollary}[prop]{Corollary}

\newtheorem{statement}[prop]{Statement}

\theoremstyle{definition}
\newtheorem{definition}[prop]{Definition}
\newtheorem*{definition*}{Definition}

\newtheorem*{remark*}{Remark}
\newtheorem*{remarks*}{Remarks}
\newtheorem{remark}[prop]{Remark}
\newtheorem{remarks}[prop]{Remarks}

\theoremstyle{remark}

\theoremstyle{plain}
\newcommand{\G}{\mathcal{G}}
\newcommand{\Cay}{\mathrm{Cay}}
\newcommand{\hdim}{\operatorname{hdim}}
\newcommand{\inj}{\operatorname{inj}}
\newcommand{\injz}{\inj^{\mathrm{Z}}}
\newcommand{\diam}{\operatorname{diam}}
\newcommand*{\rad}{\mathop{\textup{rad}}\nolimits}

\newcommand{\R}{\mathbb{R}}

\newcommand{\C}{\mathbb{C}}

\newcommand{\N}{\mathbb{N}}
\newcommand{\Z}{\mathbb{Z}}
\newcommand{\Q}{\mathbb{Q}}

\newcommand{\n}{\mathfrak{n}}
\newcommand{\g}{\mathfrak{g}}

\newcommand{\ph}{\varphi}

\newcommand{\eps}{\varepsilon}
\newcommand{\ord}{\text{\textup{ord}}}

\newcommand{\Span}{\text{\textup{Span}}}

\newcommand{\Aut}{\mathrm{Aut}\,}

\newcommand{\GL}{\mathrm{GL}}

\newcommand*{\vol}{\mathop{\textup{vol}}\nolimits}
\newcommand{\normal}{\trianglelefteqslant}
\newcommand{\la}{\langle}
\newcommand{\ra}{\rangle}

\renewcommand{\ge}{\geqslant}
\renewcommand{\le}{\leqslant}
\renewcommand{\geq}{\geqslant}
\renewcommand{\leq}{\leqslant}
\renewcommand{\hat}{\widehat}
\renewcommand{\tilde}{\widetilde}
\renewcommand{\L}{\mathcal{L}}

\numberwithin{equation}{section}

\newcommand{\change}[1]{{\color{purple} #1} }

\begin{document}
\newcounter{enumres}
\maketitle

\begin{center}
\footnotesize{Institut de Math\'ematiques de Jussieu-Paris Rive Gauche\\ Email: \href{mailto:romain.tessera@imj-prg.fr}{romain.tessera@imj-prg.fr}}

\medskip

\footnotesize{School of Mathematics, University of Bristol\\ Email: \href{mailto:m.tointon@bristol.ac.uk}{m.tointon@bristol.ac.uk}}

\bigskip
\bigskip

\textit{In memory of Avinoam Mann}
\end{center}

\bigskip

\begin{abstract}
We develop a fine-scale structural theory for balls of polynomial volume in finitely generated groups. Our main result shows that a polynomial-volume hypothesis on a single ball implies that every larger ball can, after quotienting by a certain canonical finite normal subgroup, be embedded in a lattice of a Lie group whose connected component is nilpotent with sharply bounded index and dimension. These Lie groups and lattices form chains of successive quotients that are essentially unique, giving a quantitative fine-scale analogue of Losert's strengthening of the Mal'cev completion. This describes the fine structure of the group with sharp algebraic and quantitative precision; for example, it leads to sharp bounds in Breuillard, Green and Tao's finitary version of Gromov's polynomial-growth theorem. These results have a wide range of applications in various aspects of the theory of vertex-transitive graphs, including for instance random walks, percolation, and scaling limits. We obtain some of these applications in the present paper, and treat others in companion papers. Some are due to or joint with other authors.
\end{abstract}

\tableofcontents

\section{Introduction}
Gromov's celebrated theorem asserts that a finitely generated group has polynomial growth if and only if it is virtually nilpotent. A theorem of Mann states that, on quotienting by a finite normal subgroup, the index of the nilpotent subgroup can be bounded explicitly in terms of the degree of polynomial growth, and a recent theorem of Losert shows that the resulting quotient can be embedded ``without redundancy'' as a discrete cocompact subgroup in a Lie group with nilpotent connected component. This embedding refines the Mal'cev completion of a finite-index torsion-free nilpotent subgroup; we call it the \emph{Mal'cev--Losert completion}.

It turns out that a polynomial bound on the volume of even a single ball in a finitely generated group imposes significant algebraic structure. In particular, results of Shalom and Tao and of Breuillard, Green and Tao show that such a hypothesis is already sufficient to imply the conclusion of Gromov's theorem. The purpose of this paper is to develop a more precise understanding of this phenomenon, by describing in full detail the interplay between volume and algebraic structure for all larger balls. More precisely, our main result shows that, under a suitable polynomial-volume hypothesis on the ball of radius $n$ in the group $G$, there is a sequence $r_0\le n<r_1<\cdots<r_k$ of radii such that if $H_i$ is the maximal normal subgroup contained in the ball of radius $r_i$, then $G/H_i$ admits a ``local'' embedding into a ``local Mal'cev--Losert completion'', where it generates a lattice that is ``locally isomorphic'' to $G/H_i$ on scales between $r_i$ and $r_{i+1}$. In each case, the dimension, homogeneous dimension and index of the connected component are sharply controlled in terms of the growth of $G$ on scales between $r_i$ and $r_{i+1}$. This sequence of local Mal'cev--Losert completions, including the sequence of radii $r_i$ with respect to which they are defined, is essentially unique.

As a first application, we obtain sharp bounds in the finitary polynomial-volume theorem of Breuillard, Green and Tao: under the appropriate polynomial-volume hypothesis of degree $d$, we prove that the relevant quotient can be taken to have a $d$-nilpotent subgroup with index at most $g(d)$, the maximum order of a finite subgroup of $\mathrm{GL}_d(\mathbb{Z})$. These bounds on the nilpotence and the index are both best possible.

Our multi-scale result also leads to a wide range of other applications; for example, in a companion paper \cite{ttresist} we present a number of applications to random walks on vertex-transitive graphs; here the multi-scale nature of our result and the sharpness of the bounds are both crucial. The results of this companion paper have in turn already been applied by a number of other authors \cite{berestycki2023covertime,Easo_2025,hartarsky2026percolation,hermon2021interchange,hutch-toint}, with further applications forthcoming. In other forthcoming work \cite{ttGH,ttSpheres} we will present applications to graph scaling limits. We briefly describe these works, and give full details of several other corollaries, in \cref{sec:apps} of the present paper.

The results of the present paper are mostly phrased as theorems about groups of polynomial growth, which is a rather restrictive class. However, in reality much of the utility of our theorems lies in the fact that when phrased in the contrapositive they can be viewed as statements about groups that do \emph{not} exhibit polynomial growth at a given scale. It is these extremely general statements that often turn out to be useful in the probabilistic applications, for example.

\subsection{A brief history of quantitative versions of Gromov's theorem}\label{sec:history}
Gromov's theorem is a foundational result in geometric group theory that has itself had numerous applications, for example in probability \cite{var,DGRSY} and differential geometry \cite[\S $F_+$]{gromov.metric.structures}. The theorem was originally proved by Gromov \cite{gromov} in the 1980s, and has since been given shorter proofs by both Kleiner \cite{kleiner} and Ozawa \cite{ozawa}.
\begin{theorem}[Gromov \cite{gromov}]\label{thm:gromov}
Let $d$ be a non-negative integer, and let $G$ be a group with finite symmetric generating set $S$ containing the identity. Suppose
\begin{equation}\label{eq:o(n^(d+1)}
|S^n|=o(n^{d+1})
\end{equation}
as $n\to\infty$. Then $G$ is virtually nilpotent with growth degree at most $d$.
\end{theorem}
Given a finitely generated group $G$, and a finite symmetric generating set $S$ containing the identity, we define the \emph{growth degree} $\deg(G)$ of $G$ by
\[
\deg(G)=\inf\{d\in\R:\text{there exists $C>0$ such that $|S^n|\le Cn^d$ for all $n\in\N$}\}.
\]
It is a widely known and straightforward exercise to check that this quantity is independent of the choice of generating set. Note that $\deg(G)=\infty$ is possible, and indeed by Gromov's theorem it is in some sense the norm. If $\deg(G)<\infty$ then $G$ is said to have \emph{polynomial growth}.

The conclusion that $\deg(G)\le d$ is not normally included in the statement of Gromov's theorem, not least because it follows trivially from the well-known fact that if $G$ is virtually nilpotent then $\deg(G)\in\Z$ and for every finite symmetric generating set $S$ of $G$ containing the identity there exist $c,C>0$ such that
\begin{equation}\label{eq:BG}
cn^{\deg(G)}\le|S^n|\le Cn^{\deg(G)}
\end{equation}
for every $n\in\N$ \cite{bass,guiv}. We mention it here for comparison with the results below, which have weaker variants of the polynomial-growth hypothesis under which the implications for $\deg(G)$ are rather subtler.

There is no hope of bounding the index of a nilpotent subgroup of $G$ in \cref{thm:gromov} as stated. For example, if $G$ is a direct product of a finite simple group and $\Z^d$ then $|S^n|=o(n^{d+1})$ as $n\to\infty$ for any finite symmetric generating set $S$, but there is no uniform bound on the index of a nilpotent subgroup. In a sense, however, this example captures the only obstruction to having such a bound, thanks to the following theorem of our late collaborator Avinoam Mann, to whose memory we dedicate this paper.
\begin{theorem}[Mann {\cite[Theorem 9.8]{mann.book}}]\label{thm:mann}
Let $d\in\N$, and suppose $G$ is a finitely generated virtually nilpotent group with growth degree at most $d$. Then there exist normal subgroups $H,\Gamma\normal G$ with $H\le\Gamma$ such that $H$ is finite, $\Gamma/H$ is torsion-free nilpotent, and $[G:\Gamma]\le g(d)$, where $g(d)$ is the maximum order of a finite subgroup of $\GL_d(\Z)$.
\end{theorem}
It is known that $g(d)$ is finite for all $d\in\N$, and in fact an upper bound was given by Minkowski \cite{mink} back in 1887. It is known that $g(d)\le(2d)!$ for all $d$ \cite[p.~175, eq.~(16)]{newman}, and that $g(d)\le2^dd!$ with equality achieved only by the orthogonal group $\mathrm{O}_d(\Z)$ for sufficiently large $d$ \cite{friedland}; see Section \ref{sec:g(d)} for further information on $g(d)$. It is also well known that a finite-by-nilpotent group is virtually nilpotent, so \cref{thm:mann} refines the conclusion of Gromov's theorem.
\begin{remark}\label{rem:g(d).optimal}
The upper bound $g(d)$ is the best possible bound on $[G:\Gamma]$ in terms of $d$ in \cref{thm:mann}, as can be seen by considering the example of a semidirect product of $\Z^d$ with a subgroup of $\GL_d(\Z)$ of order $g(d)$; see \cref{prop:g(d)}.
\end{remark}

We write $g(d)$ for the maximum order of a finite subgroup of $\GL_d(\Z)$ throughout the rest of this paper.

A classical theorem of Mal'cev \cite{malcev,raghu} states that a finitely generated torsion-free nilpotent group embeds discretely and cocompactly in an essentially unique simply connected nilpotent Lie group, called its \emph{Mal'cev completion}; thus, the conclusion of \cref{thm:mann} implicitly says, moreover, that after modding out by a finite subgroup and passing to bounded index, $G$ is isomorphic to a discrete cocompact subgroup of a Lie group. A remarkable recent result of Losert shows that this conclusion remains true even without the need to pass to the finite-index subgroup, refining \cref{thm:mann} still further as follows.
\begin{theorem}[Losert \cite{Losert}]\label{thm:losert.intro}
Let $d\in\N$, and suppose $G$ is a finitely generated virtually nilpotent group with growth degree at most $d$. Then $G$ has a maximal normal finite subgroup $H$, and $G/H$ is isomorphic to a discrete cocompact subgroup $G'$ in a Lie group of the form $\mathcal{L}=F\ltimes N$, where $N$ is a simply connected nilpotent Lie group of homogeneous dimension at most $d$, $F$ is a group of order at most $g(\dim N)$ acting faithfully on $N$, and $G'N=\mathcal{L}$.
\end{theorem}
The bound $g(\dim N)$ on the index of $N$ is not stated in Losert's paper, but follows straightforwardly from combining his result with the proof of Mann's \cref{thm:mann}; see the end of \cref{sec:ML}. Again, both the bounds of $d$ on the homogeneous dimension and $g(\dim N)$ on the index are sharp. Unsurprisingly, the simply connected nilpotent Lie group $N$ appearing in \cref{thm:losert.intro} is isomorphic to the Mal'cev completion of the quotient $\Gamma/H$ appearing in \cref{thm:mann}; see \cref{lem:Losert}.

As with the Mal'cev completion, the group $\mathcal{L}$ appearing in \cref{thm:losert.intro} is unique up to (unique) isomorphism; see \cref{thm:Losert}. We will call $\mathcal{L}$ the \emph{Mal'cev--Losert completion} of $G$; see \cref{sec:ML} for the full definition. We will call the homomorphism $j:G\to\L$ with kernel $H$ and image $G'$ realising the isomorphism $G/H\cong G'$ the corresponding \emph{Mal'cev--Losert factorisation}; again, see \cref{sec:ML} for the full definition. We define the \emph{dimension of $G$}, written $\dim G$, to be the dimension of the Lie group $N$. We will sometimes abuse notation and identify $G/H$ with $j(G)$, and hence view $G/H$ as a subgroup of $\L$.

\begin{remark*}
Losert's full result actually applies more generally to locally compact groups of polynomial growth; see \cref{thm:Losert}.
\end{remark*}

\subsubsection*{Gromov's theorem under finitary assumptions}
In his original paper, Gromov used a compactness argument to give a finitary version of his theorem \cite[p.~71]{gromov}. Precisely, he showed that for all $C,d\in\N$ there exist $n_0,k,c\in\N$ such that if $|S^n|\le Cn^d$ for all $n=1,2,\ldots,n_0$ then $G$ has a nilpotent subgroup of index at most $k$ and class at most $c$. The quantities $n_0,k,c$ are all ineffective.

Shalom and Tao \cite{st} adapted Kleiner's proof of Gromov's theorem to obtain a stronger version of this result, with effective bounds in the conclusion and with a weaker hypothesis under which a polynomial upper bound on $|S^n|$ need only hold for a single value of $n$. As we observed in recent joint work with Lyons and Mann \cite{lmtt}, the results of our first paper \cite{proper.progs} refine the conclusion of the Shalom--Tao theorem to obtain the following optimal bound on the growth degree.
\begin{theorem}[{\cite[Corollary 1.3]{lmtt}}]\label{thm:st}
For every non-negative integer $d$ there exists $\eps=\eps(d)>0$ such that if $G$ is a group with finite symmetric generating set $S$ containing the identity and
\begin{equation}\label{eq:abs.poly.growth}
|S^n|\le\eps n^{d+1}
\end{equation}
for some $n\in\N$ then $G$ has a nilpotent subgroup of index $O_{n,d}(1)$ and growth degree at most $d$.
\end{theorem}
See \cite[Theorem 1.6]{lmtt} for an explicit value of $\eps(d)$ with which \cref{thm:st} holds.

We think of a condition such as \eqref{eq:abs.poly.growth}, in which we have a bound on the volume of only a single ball in the group, as \emph{polynomial-volume} condition, in contrast to the \emph{polynomial-growth} condition \eqref{eq:o(n^(d+1)}. In particular, whilst Gromov's theorem is often referred to as his \emph{polynomial-growth theorem}, we will refer to theorems such as \cref{thm:st} as \emph{polynomial-volume theorems}. One notable consequence of \cref{thm:st} is that a polynomial-volume hypothesis implies polynomial growth.

Just as the index in \cref{thm:gromov} cannot be bounded, so the dependence on $n$ of the index of the nilpotent subgroup in \cref{thm:st} is unavoidable. Indeed, given $n\in\N$, if $G$ is a non-abelian finite simple group of order at most $\eps_dn^{d+1}$ and $S$ is a symmetric generating set for $G$ then $S$ certainly satisfies \eqref{eq:abs.poly.growth}, but $G$ does not have a nilpotent subgroup with index bounded independently of $n$. However, again this is essentially the only obstruction to bounding the index independently of $n$: building on work of Hrushovski \cite{hrushovski}, Breuillard, Green and Tao \cite{bgt} gave a qualitative refinement of \cref{thm:st}, as follows.
\begin{theorem}[Breuillard--Green--Tao {\cite[Corollary 11.5]{bgt}}\footnote{The fact that $\Gamma\normal G$ is not mentioned in the statement of \cite[Corollary 11.5]{bgt}, but it is mentioned explicitly in the proof (specifically in the proof of \cite[Corollary 11.2]{bgt}).}]\label{thm:bgt.gromov}
For every $d\in\N$ there exists $n_0^*=n_0^*(d)\in\N$ such that if $G$ is a group with finite symmetric generating set $S$ containing the identity and
\begin{equation}\label{eq:rel.poly.growth}
|S^n|\le n^d|S|
\end{equation}
for some $n\ge n_0^*$ then there exist normal subgroups $H,\Gamma\normal G$ with $H\le\Gamma$ such that $H\subseteq S^n$, such that $\Gamma/H$ is $O(d)$-nilpotent, and such that $[G:\Gamma]\ll^*_d1$. 
\end{theorem}
Here and from now on in this paper we use asterisks to indicate bounds that are ineffective. Thus, for example, the quantity $n_0^*(d)$ appearing in \cref{thm:bgt.gromov} is ineffective, as is the constant implied by the $\ll^*_d$ notation. Constants or bounds that are not adorned with asterisks could in principle be computed explicitly from our arguments, even if we do not always do so. See Chapter \ref{sec:background} for more details on this notation. We also follow Breuillard, Green and Tao \cite[Remark 1.9]{bgt} in defining a group to be \emph{$\ell$-nilpotent} if it admits a generating set $u_1,\ldots,u_\ell$ such that $[u_i,u_j]\in\la u_{j+1},\ldots,u_\ell\ra$ whenever $i<j$. Note, in particular, that an $\ell$-nilpotent group is nilpotent with rank, class and Hirsch length at most $\ell$. Conversely, a nilpotent group of rank $r$ and class $c$ is $O_{r,c}(1)$-nilpotent.

The proof of \cref{thm:bgt.gromov} rests on the theory of \emph{approximate groups}. Given $K\ge1$, a subset $A$ of a group $G$ is said to be a \emph{$K$-approximate group} if it is symmetric and contains the identity and there exists $X\subseteq G$ of size at most $K$ such that $A^2\subseteq XA$. For many purposes this is essentially equivalent to the a priori weaker property of having \emph{doubling} at most $K$ in the sense that $|A^2|\le K|A|$. The relevance of approximate groups to polynomial growth comes from the standard fact that if $S$ satisfies \eqref{eq:rel.poly.growth} then there exists $m$ with $n^{1/2}\le m\le n$ such that $S^m$ is an $O_d(1)$-approximate group (this follows from \cref{lem:poly.pigeon,lem:tripling->AG} below, for example). For more background on approximate groups and some of their applications see for example the second author's book \cite{tointon.book}, Breuillard's lecture notes \cite{MR3220883}, 
or any of the surveys \cite{bgt.survey,app.grps,tointon.survey}.

\begin{remarks*}\label{rem:bgt.gromov}\mbox{}
\begin{enumerate}[label=(\arabic*)]
\item It follows from the Bass--Guivarc'h formula, as presented in Section \ref{sec:nilp}, that if $G$ is a virtually $\ell$-nilpotent group then $\deg(G)\le1+\frac12\ell(\ell-1)$. The growth degree of the quotient $G/H$ given by \cref{thm:bgt.gromov} is therefore at most $O(d^2)$. Note, however, that the nilpotence of a nilpotent group cannot be bounded in terms of its growth degree since e.g. $\Z_2^m$ has growth degree $0$. The Hirsch length is always at most the growth degree, however.

\item\label{item:bgt.gromov.vnilp} It is shown in the proof of \cite[Corollary 11.7]{bgt} that if $H\normal G$ is finite and $G/H$ is nilpotent then $G$ has a nilpotent subgroup of index $|H|!$, so the group $G$ appearing in \cref{thm:bgt.gromov} has a nilpotent subgroup of index $O_d^*(|S^n|!)$ and growth degree at most $O(d^2)$.

\item The preceding remark explains why we described \cref{thm:bgt.gromov} as a \emph{qualitative} refinement of \cref{thm:st}: in contrast to \cref{thm:st}, the dependence on $d$ of the index of a nilpotent subgroup of $G$ it gives is ineffective.
\end{enumerate}
\end{remarks*}

As well as the conclusion, another important sense in which \cref{thm:bgt.gromov} refines \cref{thm:st} is that the hypothesis \eqref{eq:rel.poly.growth} of \cref{thm:bgt.gromov} is in a weaker form than that of the hypothesis \eqref{eq:abs.poly.growth} of \cref{thm:st}. In general, we call a bound of the form $|S^n|\le f(n)$ such as \eqref{eq:abs.poly.growth} an \emph{absolute} upper bound on $|S^n|$, and a bound of the form $|S^n|/|S|\le f(n)$ such as \eqref{eq:rel.poly.growth} a \emph{relative} upper bound.

\subsubsection*{Lie progressions and fully finitary statements}

\cref{thm:bgt.gromov} is in fact a simplified version of an even stronger statement, in which the conclusion as well as the hypothesis is finitary. Roughly speaking, this conclusion states that the ball $S^n$ is contained in boundedly many translates of a certain `finite-by-nilpotent' structure called a \emph{progression}. This is a clear analogue of the conclusion of Mann's result that a group of polynomial growth is contained in boundedly many cosets of a finite-by-nilpotent \emph{group}.

We will present this stronger result in terms of objects that we call \emph{Lie progressions}, which we implicitly introduced in our first paper \cite{proper.progs}. Very roughly, a Lie progression $P$ is a subset of a group that, modulo some finite subgroup, is a homomorphic image of a `box' in a lattice of a simply connected nilpotent Lie group. Canonical examples of such boxes include boxes of the form
\begin{equation}\label{eq:ab.prog}
\left([-n_1,n_1]\times\cdots\times[-n_d,n_d]\right)\cap\Z^d,
\end{equation}
where the lattice is $\Z^d$ and the Lie group is $\R^d$; and
\begin{equation}\label{eq:heis.prog}
\left\{
\left(
\begin{array}{ccc}
	1 & \ell_1 & \ell_3 \\
	0 &   1      & \ell_2 \\
	0 &   0      & 1
\end{array}
\right)
:
\begin{array}{l}
	\ell_1\in[-m,m]\cap\Z,\\
	\ell_2\in[-n,n]\cap\Z,\\
	\ell_3\in[-mn,mn]\cap\Z
\end{array}
\right\},
\end{equation}
where the lattice is the integer Heisenberg group
\[
H(\Z)=\left(
\begin{array}{ccc}
	1 & \Z & \Z \\
	0 &   1 & \Z \\
	0 &   0 & 1
\end{array}
\right)
\]
and the Lie group is the real Heisenberg group
\[
H(\R)=\left(
\begin{array}{ccc}
	1 & \R & \R \\
	0 &   1 & \R\\
	0 &   0 & 1
\end{array}
\right).
\]
The \emph{dimension} of $P$ is then the dimension of the Lie group from which it is projected (e.g. $d$ in the case of \eqref{eq:ab.prog}, or $3$ in the case of \eqref{eq:heis.prog}). There are also various other parameters that control the complexity of a Lie progression, which we suppress for now for simplicity. For a detailed definition of Lie progressions and full statements of the ensuing results, see \cref{sec:statements}.

In the finitary theory of polynomial volume, Lie progressions and the Lie groups they are projected from play a role analogous to that of torsion-free nilpotent groups and their Mal'cev completions in the infinitary theory of polynomial growth. To emphasise the analogy, given a group $G$ with finite symmetric generating set $S$ containing the identity, let us say that a finite subset $A\subseteq G$ is a \emph{virtual Lie progression with index $k$} if there exists a subset $X\subseteq S^k$ of cardinality $k$ and a Lie progression $P$ such that
\begin{equation}\label{eq:v.Lie.prog}
XP^m\subseteq A^m\subseteq XP^{\eta m}
\end{equation}
for all $m\in\N$ and some $\eta\in\N$. This condition is slightly stronger than the one appearing in the full conclusion of \cref{thm:bgt.gromov} in the original Breuillard--Green--Tao paper, but it arises from their conclusion without too much difficulty, as was observed implicitly by Breuillard and the second author~\cite{bt} and, independently, by Tao~\cite[(1.4)]{tao} (although in each case using a slightly less refined precursor to Lie progressions).

Of course, how meaningful condition \eqref{eq:v.Lie.prog} is depends rather importantly on the parameter $\eta$. For the purposes of this introduction, let us say simply and somewhat loosely that $A$ is a \emph{virtual Lie progression of bounded complexity} if $\eta$ and all of the parameters governing the complexity of $P$ are bounded.

The fully finitary statement of \cref{thm:bgt.gromov} is then as follows; see \cref{thm:finitary.poly-vol.detailed} for a detailed statement.
\begin{theorem}[fully finitary Gromov-type theorem -- simple form]\label{thm:finitary.poly-vol}
For every $d\in\N_0$ there exists $n_0^*=n_0^*(d)\in\N$ and $\eps=\eps(d)>0$ such that if $G$ is a group with finite symmetric generating set $S$ containing the identity and
\[
|S^n|\le\eps n^{d+1}|S|
\]
for some $n\ge n_0^*$ then $S^n$ is a virtual Lie progression with dimension at most $O_d(1)$, index at most $O^*_d(1)$, and complexity bounded effectively in terms of $d$ only.
\end{theorem}
This morally follows from \cref{thm:bgt.gromov} and the arguments of \cite{bt,tao} mentioned above, but to get the statement in terms of Lie progressions requires input from our first paper. At the expense of making these arguments rather longer, one ought to be able to improve the bound on the dimension from $O_d(1)$ to $O(d)$ as in the bound on the nilpotence in \cref{thm:bgt.gromov}, 
but as the reader will see, such an improvement would in any case be superseded by results we shall present shortly.

\subsection{Outline of new results}\label{sec:outline}
The overriding purpose of this paper is to develop a complete, quantitatively sharp and fully finitary theory of virtual Lie progressions and hence, by \cref{thm:finitary.poly-vol}, balls with polynomial volume. At the centre of this theory is a single theorem that provides a highly detailed desription of the fine-scale structure of powers virtual Lie progressions. This in particular leads to the optimal bounds on the dimension in \cref{thm:finitary.poly-vol}, as well as finitary analogues of Mann's \cref{thm:mann} and Losert's \cref{thm:losert.intro} to complement \cref{thm:finitary.poly-vol} just as \cref{thm:mann,thm:losert.intro} complement Gromov's theorem in the infinitary setting. In particular, the analogue of Mann's theorem allows us to obtain the opimal bound on the index in \cref{thm:finitary.poly-vol}, in stark contrast to the ineffective previous best bound.

As a warm-up for our main theorem, let us first present our analogues of Mann's and Losert's theorems. Our finitary analogue of Mann's theorem for virtual Lie progressions is as follows; see \cref{prop:fine-scale.optimise.index} for a precise (and stronger) statement.
\begin{theorem}[fully finitary Mann-type theorem; rough statement]\label{thm:finite.mann}
Suppose $G$ is a group with finite symmetric generating set $S$ containing the identity, and suppose $S^n$ is a virtual Lie progression of dimension at most $d$. Then there exists $q\in\N$ such that $S^{qn}$ is a virtual Lie progression of dimension at most $d$ and index at most $g(d)$; both $q$ and the complexity of $S^{qn}$ are bounded in terms of $d$ and the complexity of $S^n$.
\end{theorem}

To state our analogue of Losert's theorem we need a finitary notion to fill the role played by Mal'cev--Losert completions in the infinitary setting. As with Lie progressions, we defer the precise definitions until \cref{sec:statements}; for now, we will formulate our result in terms of a slightly fuzzy notion of `local Mal'cev--Losert completion'.

To motivate this notion, consider the group $G=\Z_m\times\Z_n$ with $m\le n$. This group is finite, so its Mal'cev--Losert completion is trivial. However, at scales below $m$ this group is locally indistinguishable from $\Z^2$; in particular, the balls of radius $1,2,\ldots,\lfloor m/2\rfloor-1$ in the two groups are isometric and grow quadratically. To reflect this, we would like to define the `local Mal'cev--Losert completion' of $G$ at scales less than $m/2$ to be $\R^2$. On scales between $m$ and $n/2$, the group $G$ locally looks like $\Z_m\times\Z$, so its `local Mal'cev--Losert factorisation' should have kernel $\Z_m$ and its `local Mal'cev--Losert completion' should be $\R$. It is only once we reach scale roughly $n$ that the `local Mal'cev--Losert' completion should finally agree with the (`global') Mal'cev--Losert completion and be trivial.

This description recalls an argument of Tao \cite{tao}, who showed that a group satisfying the hypotheses of \cref{thm:finitary.poly-vol} can, in a certain ultralimit, be ``modelled'' by different simply connected nilpotent Lie groups at different scales. He then exploited this to estimate the growth of powers of $S^n$, a result we quote as \cref{thm:tao} below. Part of the point of introducing the local Mal'cev--Losert completion will be to give a more explicit, finitary and quantitatively sharp description of this phenomenon.

Now let us introduce some precise notation: given a group $G$ with a finite symmetric generating set $S$ containing the identity, for each $r\in\N$, and write $R_{G,S}(r)$ for the set of relations of length at most $r$ in the elements of $S$ that are satisfied in $G$, which is to say the set of words of length at most $r$ in the elements of $S$ that evaluate to the identity in $G$. We then call $\la S\mid R_{G,S}(n)\ra$ the presentation of $G$ \emph{truncated at scale $n$}. The fact that the group $G=\Z_m\times\Z_n$ above has `local Mal'cev--Losert completion' $\R^2$ at scales less than $m/2$ then corresponds to the fact that the presentation of $G$ truncated at scales less than $m/2$ defines $\Z^2$.

Combining these ideas, let us say roughly that a group $G$ has a \emph{local Mal'cev--Losert completion at scale $n$} if
\begin{itemize}
\item the group $\hat G=\la S\mid R_{G,S}(\alpha n)\ra$ has polynomial growth for some integer $\alpha\ge2$;
\end{itemize}
and, writing $\hat S$ for the natural copy of $S$ in $\hat G$,
\begin{itemize}
\item the maximal normal finite subgroup $H\normal\hat G$ is contained in $\hat S^n$;
\item writing $\L=F\ltimes N$ for the Mal'cev--Losert completion of $\hat G$ and $j:\hat G\to\L$ for the corresponding Mal'cev--Losert factorisation, there exists a `box' $B$ in $N$, like \eqref{eq:ab.prog} in $\R^d$ or \eqref{eq:heis.prog} in the real Heisenberg group, and a set $X\subseteq\hat G$ of size $|F|$ such that writing $P=j^{-1}(B)$ we have
	\begin{itemize}
	\item $X\la P\ra=\hat G$
	\item $XP^m\subseteq\hat S^{mn}\subseteq XP^{\eta m}$ for all $m\in\N$ for some $\eta\in\N$.
	\end{itemize}
\end{itemize}
Similarly to Lie progressions, we define the \emph{dimension} of this local Mal'cev--Losert completion to be the dimension of $N$, and its \emph{index} to be the size of $X$. Recall from \cref{thm:losert.intro} (on which we based the definition of the Mal'cev--Losert completion) that the index is at most $g(\dim N)$.

Of course, the existence or otherwise of a local Mal'cev--Losert completion at scale $n$ depends on the choice of generating set $S$ as well as on the group $G$, but for the purposes of this introduction we will take it that a group $G$ always comes with some fixed generating set $S$.

As with Lie progressions, there are various parameters controlling the complexity of a local Mal'cev--Losert completion, most obviously $\alpha$ and $\eta$ but also others that we suppress in this rough description. Note that the natural projection $\pi:\hat G\to G$ is injective on $\hat S^n$ (because $\alpha\ge2$), meaning that $S^n$ is isomorphic to the subset $\hat S^n\subseteq\L$, justifying the terminology `local Mal'cev--Losert completion'.

We can now give a simplified version of our main result; see \cref{prop:chain.of.progs-index} for a detailed statement.
\begin{theorem}[local Mal'cev--Losert completions -- simple form]\label{thm:main.intro}
Suppose $G$ is a group with finite symmetric generating set $S$ containing the identity, and suppose $S^n$ is a virtual Lie progression of dimension at most $d$. Then there exist non-negative integers $d'\le d$ and $r_0<r_1<\cdots<r_{d'}$ and $d\ge\ell_0>\ell_1>\cdots>\ell_{d'}$ with $n\le r_0\ll n$ such that
\begin{itemize}
\item for each $i$, the group $G$ has a local Mal'cev--Losert completion $\L_i$ at scale $r_i$ with dimension $\ell_i$;
\item $\la S\mid R_{G,S}(O(r_{d'}))\ra=G$, and $\la S\mid R_{G,S}(O(r_i))\ra=\la S\mid R_{G,S}(\Omega(r_{i+1}))\ra$ for each $i=0,\ldots,d'-1$.
\end{itemize}
All implied constants, and the complexity of the $\L_i$, depend only on $d$ and the complexity of $S^n$.
\end{theorem}
Note that the index of the local Mal'cev--Losert completion at scale $r_i$ is at most $g(\ell_i)$. In particular, thanks to the bound on the index implicit in the definition of a local Mal'cev--Losert completion, \cref{thm:main.intro} subsumes \cref{thm:finite.mann}, although the structure of the proof is such that we prove \cref{thm:finite.mann} first and then use that to obtain \cref{thm:main.intro}; see the proof of \cref{prop:fine-scale.optimise.index} for full details.

\cref{thm:main.intro} should be interpreted as saying that between scales $r_i$ and $r_{i+1}$ the group $G$ is `locally isomorphic' to a group with Mal'cev--Losert completion $\L_i$. The second point in particular says that the group admits a presentation with generators $S$ in which all relations longer than $n$ must have length comparable to one of the $r_i$.

\subsubsection*{Application to balls of polynomial size}

As we described above, the motivation for \cref{thm:main.intro} is to combine it with \cref{thm:finitary.poly-vol} to obtain a multi-scale description of any group containing a ball of polynomial size. Conceptually this is a straightforward matter, but the resulting corollary would inherit the suboptimal dimension bound of \cref{thm:finitary.poly-vol}.

It turns out that in order to overcome this we actually need to apply \cref{thm:main.intro} \emph{twice}, in a kind of bootstrapping argument. A crucial observation is that by reducing $\eps$ in \cref{thm:finitary.poly-vol} one can obtain the optimal bound of $d$ on the dimension (see \cref{thm:multiscale}). However, this alone is not enough, because the necessary value of $\eps$ depends on the bound on the index in the conclusion \cref{thm:finitary.poly-vol}; since that bound is ineffective, the required value of $\eps$ would also be rendered ineffective. Thankfully, it turns out that by applying \cref{thm:main.intro} to the conclusion of \cref{thm:finitary.poly-vol} as it stands, we can obtain a version of \cref{thm:finitary.poly-vol} with an \emph{effective} bound on the index (see \cref{thm:rel.hom.dim.k(d).effective}). We are then free to reduce $\eps$ by an \emph{effective} amount and obtain the optimal dimension bound of $d$ as just described. The resulting statement is given in detail in \cref{thm:finitary.poly-vol.optimal}. Applying \cref{thm:main.intro} once again, to the conclusion of \emph{this} statement, we can then at last obtain the following quantitatively sharp multi-scale description of a group containing a ball of polynomial size; again, we simplify the statement for this introduction, and refer the reader to \cref{cor:main.poly.vol} for a detailed statement.

\begin{corollary}[sharp fine-scale polynomial-volume theorem -- simple form]\label{cor:finitary.poly-vol}
For all $d\in\N_0$ there exists $n_0^*=n_0^*(d)\in\N$ and $\eps=\eps(d)>0$ such that if $G$ is a group with finite symmetric generating set $S$ containing the identity and
\[
|S^n|\le\eps n^{d+1}|S|
\]
for some integer $n\ge n_0^*$ then there exist non-negative integers $d'\le d$ and $r_0<r_1<\cdots<r_{d'}$ and $d\ge\ell_0>\ell_1>\cdots>\ell_{d'}$ with $n^{1/2}\ll_dr_0\le n<r_1$ such that
\begin{itemize}
\item for each $i$, the group $G$ has a local Mal'cev--Losert completion at scale $r_i$ with dimension $\ell_i$ and complexity bounded in terms of $d$ only;
\item $\la S\mid R_{G,S}(O_d(r_{d'}))\ra=G$, and $\la S\mid R_{G,S}(O_d(r_i))\ra=\la S\mid R_{G,S}(\Omega_d(r_{i+1}))\ra$ for each $i=0,\ldots,d'-1$.
\end{itemize}
\end{corollary}

\subsubsection*{Local Lie-group completions at \emph{all} scales}
As general as Mal'cev--Losert completions are, the need for them to be virtually simply connected does slightly reduce their flexibility. Consider, for example, the group $\Z_n=\Z/n\Z$. At scales far below $n$ it looks like $\Z$ and so its local Mal'cev--Losert completion is $\R$. At scales approaching $n$, on the other hand, we see that the group is finite and so has trivial Mal'cev--Losert completion. However, there is a non-trivial connected Lie group that `completes' $\Z_n$ at scales close to $n$ rather more meaningfully than the trivial group, namely $\R/n\Z$. We would like to introduce some greater flexibility into our local Lie-group completions in order to capture this `better completion' at scale~$n$.

Let us now consider a slightly more complicated example that illustrates this phenomenon in higher resolution: the finite Heisenberg group
\[
H(\Z_n)=\left(
\begin{array}{ccc}
	1 & \Z_n & \Z_n \\
	0 &   1 & \Z_n \\
	0 &   0 & 1
\end{array}
\right)
\]
with generators
\[
\left(
\begin{array}{ccc}
	1 & 1 & 0 \\
	0 &   1 & 0 \\
	0 &   0 & 1
\end{array}
\right),
\qquad
\left(
\begin{array}{ccc}
	1 & 0 & 0 \\
	0 &   1 & 1 \\
	0 &   0 & 1
\end{array}
\right).
\]
This group has diameter roughly $n$, and it has a subgroup
\[
Z_n=\left(
\begin{array}{ccc}
	1 & 0 & \Z_n \\
	0 &   1 & 0 \\
	0 &   0 & 1
\end{array}
\right)
\]
with diameter roughly $n^{1/2}$. At scales well below $n^{1/2}$, we do not yet observe the finiteness of $Z_n$ and so $H(\Z_n)$ looks locally like the integer Heisenberg group $H(\Z)$. At such scales, the local Mal'cev--Losert completion is thus the real Heisenberg group $H(\R)$. At scales well above $n^{1/2}$, the subgroup $Z_n$ looks negligibly small. At such scales that are also well below $n$, the group then looks locally like $H(\Z)/Z_n\cong\Z^2$, and the local Mal'cev--Losert completion is thus $\R^2$. At scales comparable to $n^{1/2}$, however, the subgroup $Z_n$ is both evidently finite and not negligible. At such scales, $H(\Z_n)$ looks locally like
\[
\left(
\begin{array}{ccc}
	1 & \Z & \Z_n \\
	0 &   1 & \Z \\
	0 &   0 & 1
\end{array}
\right)
=
\left(
\begin{array}{ccc}
	1 & \Z & \Z \\
	0 &   1 & \Z \\
	0 &   0 & 1
\end{array}
\right)
\left/
\left(
\begin{array}{ccc}
	1 &  0 & n\Z \\
	0 &   1 & 0 \\
	0 &   0 & 1
\end{array}
\right)
\right.
.
\]
This embeds discretely and cocompactly in the connected Lie group
\begin{equation}\label{eq:Heis/Z}
\left(
\begin{array}{ccc}
	1 & \R & \R/n\Z \\
	0 &   1 & \R \\
	0 &   0 & 1
\end{array}
\right)
=
\left(
\begin{array}{ccc}
	1 & \R & \R \\
	0 &   1 & \R \\
	0 &   0 & 1
\end{array}
\right)
\left/
\left(
\begin{array}{ccc}
	1 &  0 & n\Z \\
	0 &   1 & 0 \\
	0 &   0 & 1
\end{array}
\right)
\right.
,
\end{equation}
and so we would like to take this as the Lie-group completion of $H(\Z_n)$ on scales comparable to~$n^{1/2}$.

Underlying this example is the fact that the subgroup
$\left(\begin{smallmatrix}
1 &  0 & n\Z \\
	0 &   1 & 0 \\
	0 &   0 & 1
\end{smallmatrix}\right)$
is central in $H(\Z)$. The crucial consequence of this is that this subgroup is also central in $H(\R)$, and hence normal in $H(\R)$, allowing us to define the quotient \eqref{eq:Heis/Z}.

This is in fact a completely general phenomenon. Recall that in \cref{thm:main.intro,cor:finitary.poly-vol} the group has no new relations on scales of order strictly between $r_i$ and $r_{i+1}$. It turns out that, whilst new relations may start to emerge on scales close to $r_{i+1}$, all such relations lift to elements lying in the centre of the simply connected part of $\L_i$. This in turn allows us to refine \cref{thm:main.intro,cor:finitary.poly-vol} by defining a virtually connected-nilpotent quotient $\L_i/K_i$ that serves as a `local Lie-group completion' of $G$ at scales comparable to $r_{i+1}$. See \cref{thm:inj/Z,rem:GH} for precise details.

\subsection{First corollaries}\label{sec:1st.cors}

An immediate consequence of \cref{cor:finitary.poly-vol} is the following sharp version of \cref{thm:bgt.gromov}.
\begin{corollary}\label{thm:rel.hom.dim}
For every non-negative integer $d$ there exist $\eps=\eps(d)>0$ and $n_0^*=n_0^*(d)\in\N$ such that if $G$ is a group with finite symmetric generating set $S$ containing the identity and
\[
|S^n|\le\eps n^{d+1}|S|
\]
for some integer $n\ge n^*_0$ then there exist normal subgroups $H,\Gamma\normal G$ with $H\le\Gamma$ such that $H\subseteq S^n$, such that $\Gamma/H$ is $d$-nilpotent, and such that $[G:\Gamma]\le g(d)$.
\end{corollary}
Although not obvious from the rough version of \cref{cor:finitary.poly-vol} we have stated above, another straightforward consequence of the more detailed \cref{cor:main.poly.vol,thm:dimbounds} is the following sharp version of \cref{thm:st}.
\begin{corollary}\label{thm:abs.hom.dim}
For every non-negative integer $d$ there exist $\eps=\eps(d)>0$ and $n_0^*=n_0^*(d)\in\N$ such that if $G$ is a group with finite symmetric generating set $S$ containing the identity and
\[
|S^n|\le\eps n^{d+1}
\]
for some $n\ge n_0^*$ then there exist normal subgroups $H,\Gamma\normal G$ with $H\le\Gamma$ such that $H\subseteq S^n$, such that $\Gamma/H$ is $d$-nilpotent with growth degree at most $d$, and such that $[G:\Gamma]\le g(d)$.
\end{corollary}
The bound on the index is optimal in each case for the same reason as given in \cref{rem:g(d).optimal}; this is in stark contrast to the previous, ineffective bound given by \cref{thm:bgt.gromov}, or even the bound depending on $n$ given by \cref{thm:st}. The example of $\Z^d$ with a generating set of the form $[-L,L]^d\cap\Z^d$ shows that the bounds on the nilpotence in both results, and the growth degree in \cref{thm:abs.hom.dim}, are also optimal.

One might wonder whether it is possible to bound the growth degree by $d$ in \cref{thm:rel.hom.dim} as in \cref{thm:abs.hom.dim}. However, this is too much to ask for, hence the need for \cref{thm:abs.hom.dim} as a separate result: in the Heisenberg group, which is $3$-nilpotent and has growth degree $4$, Tao gave an example of to show that for each $n\in\N$ there exists a generating set $S_n$ satisfying $|S_n^n|\ll n^3|S_n|$ \cite[Example 1.11]{tao}. More generally, it is not hard to adapt this example to show that in a $d$-nilpotent fililform group $G$, which has growth degree $1+\frac12d(d-1)$, there exists, for each $n\in\N$, a generating set $S_n$ satisfying $|S_n^n|\ll_Gn^d|S_n|$. Thus, it is not just the bound on the nilpotence that is optimal, but the choice of nilpotence as the quantity to bound.

\begin{remark}We could make \cref{thm:rel.hom.dim,thm:abs.hom.dim} valid for all $n$, rather than just $n\ge n_0^*$, at the expense of making $\eps$ ineffective: choosing $\eps<1/(n_0^*)^{d+1}$ would in each case ensure that the hypothesis was not satisfied for any $n<n_0^*$.
\end{remark}

\cref{cor:finitary.poly-vol} also subsumes and makes much more quantitative two other results on the structure of balls with polynomial size: a result of Tao \cite{tao} from around ten years ago on the further growth of such balls, and a recent result of Easo and Hutchcroft \cite{unif.finite.pres} giving a uniform version of the fact that groups of polynomial growth are finitely presented. In each case, our work also gives far deeper qualitative insight into the structure underpinning these prior results.

To motivate Tao's result, note that whilst the bounds of \eqref{eq:BG} are clearly sharp as $n\to\infty$, in the sense that they agree up to the constants $c$ and $C$, this masks a certain amount of `fine-scale' behaviour. For example, the group $\Z \times \Z_m$ has growth degree $1$, but for the standard generating set $S$ the volume of $S^n$ grows quadratically in the range $n\le m$, and it is only for $n>m$ that the linear growth becomes visible. Indeed, even in the group $\Z$ the volume at small scales can be polynomial of arbitrarily large degree with an appropriate choice of generating set; for instance, if we let $S=\{1,m,m^2,\ldots,m^{d-1}\}$ for $m$ very large then we have $|nS|\gg_dn^d$ for $n<m$.

Tao captured this phenomenon with a result showing that the `local' growth rate of a group of polynomial growth can change boundedly many times. To make his result precise, define a function $f:[1,\infty)\to\R$ to be \emph{piecewise monomial} if there exist $1=x_0<x_1<\ldots<x_k=\infty$ and $C_1,\ldots,C_k$ and $d_1,\ldots,d_k\ge0$ such that $f(x)=C_ix^{d_i}$ whenever $x\in[x_{i-1},x_i)$. Call the $x_i$ the \emph{boundaries} of $f$, the restrictions $f|_{[x_{i-1},x_i)}$ the \emph{pieces} of $f$, and each $d_i$ the \emph{degree} of the piece $f|_{[x_{i-1},x_i)}$.
\begin{theorem}[Tao {\cite[Theorem 1.9]{tao}}]\label{thm:tao}
For every $d>0$ there exists $n_0^*=n_0^*(d)$ such that if $G$ is a group with finite generating set $S$ containing the identity such that $|S^n|\le n^d|S|$ for some $n\ge n_0^*$ then there exists a non-decreasing continuous piecewise-monomial function $f:[1,\infty)\to[1,\infty)$ with $f(1)=1$ and at most $O_d(1)$ distinct pieces, each of degree a non-negative integer at most $O_d(1)$, such that $|S^{mn}|\asymp_df(m)|S^n|$ for all $m\in\N$.
\end{theorem}

By studying the growth of the Mal'cev--Losert completions arising from \cref{cor:finitary.poly-vol,thm:main.intro}, we recover \cref{thm:tao} with the \emph{optimal} bound on the degrees of the pieces in the piecewise-monomial function $f$, as well as some further information on the changes in degree. Combining this with previous work of ours \cite{ttTrof} means it applies more generally to vertex-transitive graphs, as follows. Given a vertex-transitive graph $\G$ and a non-negative integer $n$, write $\beta_\G(n)$ for the number of vertices in some (hence any) ball of radius $n$ in $\G$.
\begin{corollary}[growth of balls with polynomial volume]\label{cor:tao}
For every $d\in\N_0$ there exist $\eps=\eps(d)>0$ and $n_0^*=n_0^*(d)\in\N$ such that if $\G$ is a connected, locally finite vertex-transitive graph such that $\beta_\G(n)\le \eps n^{d+1}\beta_\G(1)$ for some $n\ge n_0$, then there exists a non-decreasing continuous piecewise-monomial function $f:[1,\infty)\to[1,\infty)$ with integer boundaries and $f(1)=1$ such that
$\beta_\G(m)\asymp_df(m/n)\beta_\G(n)$
for every $m\ge n$, and such that
\begin{enumerate}[label=(\roman*)]
\item the pieces of $f$ have degree at most $\frac12d(d-1)+1$;\label{item:degrees-rel}
\item the number of boundaries of $f$ across which the degree decreases is at most $d$;\label{item:drops}
\item the number of boundaries of $f$ across which the degree increases is at most $\frac16d^3-\frac12d^2+\frac13d$.\label{item:increases}
\end{enumerate}
If $\beta_\G(n)\le \eps n^{d+1}$, then in fact
\begin{enumerate}
\item[(i\,$'$)]the pieces of $f$ have degree at most $d$.\label{item:degrees-abs}
\end{enumerate}
\end{corollary}
\begin{remark*}
The bounds in \ref{item:degrees-rel}, (i$'$) and \ref{item:drops} are all sharp. Indeed, equality in (i\,$'$) and \ref{item:drops} is achieved in $\Z_{m_1}\times\cdots\times\Z_{m_d}$ with respect to its standard generating set, and it is not hard to adapt \cite[Example 1.11]{tao} to achieve equality in \ref{item:degrees-rel} in a lattice in a filiform Lie group of dimension $d$ for arbitrary $d\ge3$. This last example also shows that the bound in \ref{item:increases} has to be at least quadratic in $d$.
\end{remark*}
An immediate consequence of \cref{cor:tao} that is of particular interest for applications is the following result, in which the bounds $\frac12d(d-1)+1$ and $d$ on the degree of polynomial growth are both optimal by the same examples as in the previous remark.
\begin{corollary}[polynomial volume implies polynomial growth]\label{cor:Benj}
For every $d\in\N$ there exist $\eps=\eps(d)>0$ and $n_0^*=n_0^*(d)\in\N$ such that the following holds for any connected, locally finite vertex-transitive graph $\G$.
\begin{enumerate}[label=(\roman*)]
\item If $\beta_\G(n)\le\eps n^{d+1}\beta_\G(1)$ for some $n\ge n_0$ then $\beta_\G(m)\ll_d(m/n)^{\frac12d(d-1)+1}\beta_\G(n)$ for every $m\ge n$.\label{item:Benj.rel}
\item If $\beta_\G(n)\le\eps n^{d+1}$ for some $n\ge n_0$ then $\beta_\G(m)\ll_d(m/n)^d\beta_\G(n)$ for every $m\ge n$.
\end{enumerate} 
\end{corollary}
The fact that we obtain the optimal $\frac12d(d-1)+1$ bound on the growth degree in \ref{item:Benj.rel} is crucial for applications to random walks we describe in \cref{sec:apps} below.

\medskip

Easo and Hutchcroft's result is a uniform version of the classical fact that a group of polynomial growth is finitely presented. As well as being of independent interest, their result is a crucial ingredient in their resolution of Schramm's famous \emph{locality conjecture} in percolation \cite{Eas-Hut}.

Given a group $G$ with finite symmetric generating set $S$, say that $(G,S)$ has a \emph{new relation on scale $n$} if $\la S\mid R_{G,S}(2^{n+1})\ra\ne\la S\mid R_{G,S}(2^n)\ra$.
\begin{theorem}[Easo--Hutchcroft \cite{unif.finite.pres}]\label{thm:fin.pres}
For each $K>0$ there exist $n_0^*=n_0^*(K)\in\N$ such that if $G$ is a group and $S$ is a finite symmetric generating set for $G$ containing the identity and satisfying $|S^{2n}|\le K|S^n|$ for some integer $n\ge n_0^*$ then
\[
\#\Bigl\{m\in \mathbb{N} : m\ge \log_2n \text{ and $(G,S)$ has a new relation on scale $m$} \Bigr\}\ll^*_{K,|S|}1.
\]
\end{theorem}
The hypothesis used in \cite{unif.finite.pres} is actually the stronger $|S^{3n}|\le K|S^n|$, but the weaker hypothesis quoted here is sufficient by a recent result of ours stated below as \cref{thm:doubling=>tripling}.

\cref{cor:finitary.poly-vol} immediately gives the following more quantitatively precise version of \cref{thm:fin.pres} under a polynomial-volume assumption.
\begin{corollary}[sharp uniform finite presenation of groups with polynomial growth]\label{cor:fin.pres.poly}
Given $d\in\N_0$ there exist $\eps=\eps(d)>0$ and $n_0^*=n_0^*(d)\in\N$ such that if $G$ is a group with finite symmetric generating set $S$ containing the identity and
\[
|S^n|\le\eps n^{d+1}|S|
\]
for some $n\ge n_0^*$ then there exist $d'\le d$ and intervals $U_1,\ldots,U_{d'}$ of diameter at most $O_d(1)$ in $\R$ such that every scale greater than $\log_2n$ on which $(G,S)$ has a new relation lies in $\bigcup_{i=1}^{d'}U_i$.
\end{corollary}
We can also make the bound in \cref{thm:fin.pres} independent of $|S|$ by a variant of this argument; see \cref{cor:fin.pres}.

\begin{remark*}
The bound of $d$ on the number of intervals $U_i$ in \cref{cor:fin.pres.poly} is sharp, as can be seen by considering groups of the form $G=\prod_{i=1}^d(\Z/m_i\Z)$ with respect to their standard generating sets.
\end{remark*}

\subsection{Applications and extensions}\label{sec:apps}
One of the most important aspects of the theory we develop in this paper is the range of applications it provides. In this section we give a brief summary of some highlights. Most of the proofs are presented in \cref{chap:apps} and various recent and forthcoming papers.

\paragraph{Probability.} Some of the most striking applications to date of this work are those to probability on groups and graphs. We cover this aspect of our theory in much more detail in a companion paper \cite{ttresist}, but to add some context to the present paper let us briefly describe a particular application to the \emph{recurrence} or \emph{transience} of the simple random walk on a vertex-transitive graph.

Given a graph $\G$, recall that the \emph{return probability} of the simple random walk starting at some vertex $x$ is the probability that the random walk returns to $x$ at some point after its initial step. The walk is called \emph{recurrent} if this probability is $1$, and \emph{transient} otherwise. A famous theorem of Varopoulos \cite{varoRec} (see also \cite[\S 6]{grigo}) states that the simple random walk on a vertex-transitive graph is recurrent if and only if the graph has polynomial growth of degree at most $2$, the main content being the fact that if the growth is superquadratic then the random walk is transient.

In our companion paper \cite{ttresist} we prove a finitary refinement of this result, which roughly shows that if a ball in a vertex-transitive graphs has volume slightly superquadratic in its radius then the probability that a simple random walk starting at the centre of that ball returns to the centre before escaping the ball is bounded uniformly away from $1$. As well as implying Varopoulos's theorem, a minor variation of this also verifies and strengthens a certain analogue of Varopoulos's result for \emph{finite} graphs conjectured by Benjamini and Kozma \cite{bk}. We refer the reader to the companion paper \cite{ttresist} for full statements, but we take the opportunity now to advertise the following rather striking corollary; see \cite{ttresist} for a proof.
\begin{corollary}[gap at $1$ for return probabilities of random walks on vertex-transitive graphs]\label{cor:gap}
There exists a universal constant $c>0$ such that the simple random walk on a connected, locally finite vertex-transitive graph is either recurrent or has return probability at most $1-c$.
\end{corollary}

Since we first circulated a preprint of \cite{ttresist}, a number of other probabilistic applications of the results in it and the present paper have emerged, including universality theorems for cover-time fluctuations \cite{berestycki2023covertime}; a comparison between the mixing time of the interchange process and random walks on transitive graphs \cite{hermon2021interchange}; non-triviality of the supercritical phase for percolation on finite transitive graphs \cite{hutch-toint}; a gap at $1$ for critical percolation parameters $p_c$ on vertex-transitive graphs \cite{Easo_2025}, analogous to the gap for return probabilities given by \cref{cor:gap}; and a bound on the critical percolation parameter $p_c$ of high-degree vertex-transitive graphs \cite{hartarsky2026percolation}.

\paragraph{Finite groups.} In the case that $G$ is finite, a natural value of $n$ at which to seek to apply \cref{thm:rel.hom.dim} is the \emph{diameter} of $G$ with respect to $S$, denoted and defined by $\diam_S(G)=\min\{n\in\N:S^n=G\}$. For $n=\diam_S(G)$, the hypothesis of \cref{thm:rel.hom.dim} translates to a condition of the form
\begin{equation}\label{eq:large.diam}
\diam_S(G)\ge A\left(\frac{|G|}{|S|}\right)^\frac1{d+1}.
\end{equation}
Breuillard and the second author \cite{bt} have studied groups satisfying \eqref{eq:large.diam}, calling them \emph{almost flat}. One of their main results used \cref{thm:bgt.gromov} to show that if $\delta>0$ and $G$ is a group with diameter large enough in terms of $d$ and $\delta$ and satisfying \eqref{eq:large.diam}, then $G$ has a normal subgroup $H\subseteq S^{\lfloor\diam_S(G)^{1/2+\delta}\rfloor}$ such that $G/H$ has an abelian subgroup of index and rank at most $O_{d,\delta}(1)$ \cite[Theorem 4.1 (2)]{bt}.

Using our results, we can remove the need for $\delta$ in this result, as well as giving the optimal bound on the rank and an explicit bound on the index, as follows (we parametrise the exponent as $2/(d+2)$ rather than $1/(d+1)$ purely to have a cleaner relationship between it and the bound on the rank).
\begin{corollary}\label{cor:bt}
For every non-negative integer $d$ there exist $A=A(d)$ and $D^*=D^*(d)$ such that if $G$ is a finite group with symmetric generating set $S$ containing the identity such that $\diam_S(G)\ge D^*$ and
\begin{equation}\label{eq:bt}
\diam_S(G)\ge A\left(\frac{|G|}{|S|}\right)^\frac2{d+2},
\end{equation}
then there is a normal subgroup $H\normal G$ contained in $S^{O_d(\diam_S(G)^{1/2})}$ such that $G/H$ has an abelian subgroup of rank at most $d$ and index at most $g(d)$.
\end{corollary}
\begin{remark*}
To see that the bound on the rank is sharp, consider the group $G=\Z_{n^2}\times\Z_{Cn}^{d-1}$ with respect to its standard generating set.
\end{remark*}

\paragraph{Vertex-transitive graphs.} Trofimov \cite{trof} has famously given an extension of Gromov's theorem to vertex-transitive graphs of polynomial growth. Recall that we write $\beta_\G(n)$ for the number of vertices in a ball of radius $n$ in a vertex-transitive graph $\G$. Then such a graph $\G$ is said to have \emph{polynomial growth} if, writing , we have $\beta_\G(n)\le Cn^d$ for some $C,d\in\N$ and all $n\in\N$. Trofimov's result can be stated in various ways (see e.g. \cite[Theorem 2.1]{ttTrof} or \cite{woess}), but a particularly simple formulation states that a vertex-transitive graph of polynomial growth is \emph{quasi-isometric} to some locally finite Cayley graph, where, as usual, given $C\ge1$ and $K\ge0$ and metric spaces $X,Y$, a map $f:X\to Y$ is said to be a \emph{$(C,K)$-quasi-isometry} if
\[
C^{-1}d(x,y)-K\leq d(f(x),f(y))\leq Cd(x,y)+K
\]
for every $x,y\in X$, and if every $y\in Y$ lies at distance at most $K$ from $f(X)$. Since polynomial growth is preserved by quasi-isometries, this Cayley graph is in turn virtually nilpotent by Gromov's theorem.

Like Gromov, Trofimov also proved a finitary version of his result stating that for all $C,d\in\N$ there exists $n_0^*\in\N$ such that if $\beta_\G(n)\le Cn^d$ for all $n=1,2,\ldots,n_0^*$ then $\G$ has polynomial growth \cite{trof.finitary}. In a previous paper of ours \cite{ttTrof}, we proved a stronger finitary refinement of Trofimov's theorem, analogous to the refinement of Gromov's theorem provided by \cref{thm:bgt.gromov}. By inserting the results of the present paper into the proof of that result, we can immediately refine its conclusions further. The resulting statement requires some additional notation; we present this notation in \cref{sec:VTGs}, where we state the result as \cref{cor:trof}.

A particular corollary of our work on Trofimov's theorem was a certain extension to vertex-transitive graphs of Breuillard and the second author's original version of \cref{cor:bt} \cite[Corollary 2.5]{ttTrof}. Amongst other things, this extension implies that for every $\eps\in(0,1]$ and $\delta>0$ there exists $n_0^*=n_0^*(\eps,\delta)\in\N$ such that if $\G$ is a connected, finite vertex-transitive graph of diameter at least $n_0^*$ satisfying
\[
\diam(\G)\ge \left(\frac{|\G|}{\beta_\G(1)}\right)^\eps
\]
then $\G$ is $(1,\diam(\G)^{\frac12+\delta})$-quasi-isometric to a Cayley graph of a group containing an abelian subgroup with rank and index at most $O_{\eps,\delta}(1)$.

Using \cref{cor:bt}, we can remove the need for $\delta$ in this result, and give the optimal bound on the rank of the abelian subgroup and an explicit bound on its index, as follows.
\begin{corollary}\label{thm:trof.finite}
For every $d\in\N_0$ there exist $A=A(d)>0$ and $n_0^*=n_0^*(d)\in\N$ such that if $\G$ is a connected, finite vertex-transitive graph such that $\diam(\G)\ge n_0$ and
\[
\diam(\G)\ge A\left(\frac{|\G|}{\beta_\G(1)}\right)^\frac2{d+2}
\]
then $\G$ is $(1,O_d(\diam(\G)^{\frac{1}{2}}))$-quasi-isometric to a Cayley graph of a group containing an abelian subgroup with rank at most $d$ and index at most $g(d)$.
\end{corollary}
See \cref{cor:trof.finite} for a more detailed version of \cref{thm:trof.finite} featuring a number of additional conclusions that we omitted from \cref{thm:trof.finite} for simplicity.

\paragraph{Volume doubling.} In his proof of \cref{thm:gromov}, Gromov used the fact that if $G$ has polynomial growth of degree $d$ then there exists $K=K(d)$ such that, for any finite symmetric generating set $S$ of $G$ containing the identity, there exists an infinite set $N$ of natural numbers such that
\begin{equation}\label{eq:doubling}
|S^{2n}|\le K|S^n|
\end{equation}
for all $n\in N$. A condition of the type \eqref{eq:doubling} is called a \emph{volume-doubling condition} at scale $n$, or for brevity a \emph{doubling condition} at scale $n$.

In fact, as is well known and proved in \cref{lem:poly.pigeon} below, if $|S^n|\le n^d|S|$ for a given $n\in\N$ then there exists $m\in\N$ with $\sqrt{n}\le m\le n$ such that $|S^{2m}|\le K|S^m|$. As we discussed above, the proof of \cref{thm:bgt.gromov} also uses this observation, and in fact relies on the following similar result with a doubling hypothesis in place of the polynomial-volume hypothesis.
\begin{theorem}[Breuillard--Green--Tao {\cite[Corollary 11.2]{bgt}}\footnote{The bound $H\subseteq S^{O^*_K(n)}$ and the explicit bound of $6\log_2K$ on the class of $\Gamma/H$ are not stated in \cite[Corollary 11.2]{bgt}, but they follow from using \cite[Theorem 2.12]{bgt} instead of \cite[Corollary 1.7]{bgt} in the proof. The fact that $\Gamma\normal G$ is also not stated in \cite[Corollary 11.2]{bgt}, but is mentioned explicitly in the proof.}]
\label{thm:bgt.doubling}
For every $K\ge1$ there exist $n_0^*=n_0^*(K)\in\N$ and $k_2^*=k_2^*(K)\in\N$ such that if $G$ is a group with finite symmetric generating set $S$ containing the identity and
\[
|S^{2n}|\le K|S^n|
\]
for some $n\ge n_0^*$ then there exist normal subgroups $H,\Gamma\normal G$ with $H\le\Gamma$ such that $H\subseteq S^{O^*_K(n)}$, such that $\Gamma/H$ is nilpotent of rank at most $O_K^*(1)$ and class at most $6\log_2K$, and such that $[G:\Gamma]\le k_2^*$.
\end{theorem}
\begin{remark*}
One can obtain the improved bound of $H\subseteq S^{12n}$ at the expense of bounding the nilpotency class by $O(K^2\log K)$ instead of $6\log_2K$; see \cite[Theorem 2.12]{bgt}.
\end{remark*}
The same methods that allow us to optimise the bounds in \cref{thm:bgt.gromov} allow us to effectivise the bounds on the rank and index in \cref{thm:bgt.doubling}, as follows.
\begin{theorem}\label{thm:bgt.doubling.k.effective}
For every $K\ge1$ there exists $n_0^*=n_0^*(K)\in\N$ such that if $G$ is a group with finite symmetric generating set $S$ containing the identity and
\[
|S^{2n}|\le K|S^n|
\]
for some $n\ge n_0^*$ then there exist normal subgroups $H,\Gamma\normal G$ with $H\le\Gamma$ such that $H\subseteq S^{O_K^*(n)}$, such that $\Gamma/H$ is $\exp(\exp(O(K^2)))$-nilpotent of class at most $6\log_2K$, and such that $[G:\Gamma]\le g(\exp(\exp(O(K^2))))$. If
\begin{equation}\label{eq:bgt.tripling.hyp}
|S^{3n}|\le K|S^n|,
\end{equation}
then we can in fact obtain that $\Gamma/H$ is $\exp(O(\log^3K))$-nilpotent and $[G:\Gamma]\le g(\exp(O(\log^3K)))$, without changing any of the other conclusions.
\end{theorem}
See \cref{sec:poly.vol.doub} for the proof. In \cref{thm:doubling.multiscale} we also give a fully finitary version of \cref{thm:bgt.doubling.k.effective} with the conclusions stated in terms of Lie progressions, just as we did with \cref{thm:finitary.poly-vol} for the conclusion of \cref{thm:bgt.gromov}.

Despite the similarity between the conclusions of \cref{thm:bgt.gromov,thm:bgt.doubling}, an analogue of \cref{thm:tao} for balls satisfying a doubling condition has so far proved elusive. The closest thing to such a result in the existing literature seems to be a theorem of Breuillard and the second author \cite[Theorem 1.1]{bt}, which does imply a polynomial upper bound on $|S^{mn}|$ for all $m\in\N$, but only under a slightly stronger hypothesis than \eqref{eq:doubling}, namely $|S^{2n+1}|\le K|S^n|$.

In a short companion paper \cite{ttdoubling}, we show that the hypothesis of this result can in fact be weakened to $|S^{2n}|\le K|S^n|$, as follows.
\begin{theorem}[{\cite[Theorem 1.2]{ttdoubling}}]\label{thm:doubling=>tripling}
Let $K\ge1$. Suppose $S$ is a symmetric subset of a group containing the identity, and $|S^{2n}|\le K|S^n|$ for some integer $n\ge2K^2$. Then $|S^{3n}|\le\exp(\exp(O(K^2)))|S^n|$, and $S^{2n}$ is an $\exp(\exp(O(K^2)))$-approximate group.
\end{theorem}
In the same paper, we prove similar results for locally compact groups and vertex-transitive graphs. Combining the result on vertex-transitive graphs with the results of the present paper and earlier work of ours on vertex-transitive graphs \cite{ttTrof}, we can then also, at last, provide an analogue of \cref{thm:tao} under a doubling hypothesis, generalised to the setting of vertex-transitive graphs and with effective bounds on the number and degree of the monomial pieces.
\begin{corollary}[polynomial growth of balls with bounded doubling]\label{cor:Tao.doubling}
For every $K\ge1$ there exists $n_0^*=n_0^*(K)\in\N$ such that if $\G$ is a connected, locally finite vertex-transitive graph such that $\beta_\G(2n)\le K\beta_\G(n)$ for some $n\ge n_0$ then there exists a non-decreasing continuous piecewise-monomial function $f:[1,\infty)\to[1,\infty)$ with $f(1)=1$, integer boundaries, and at most $\exp(\exp(O(K^2)))$ distinct pieces, each with degree a non-negative integer at most $\exp(\exp(O(K^2)))$, such that
$\beta_\G(m)\asymp_d^*f(m/n)\beta_\G(n)$
for every $m\ge n$. If $\beta_\G(3n)\le K\beta_\G(n)$ then the number of pieces and their degrees can be bounded above by $\exp(O(\log^3K))$.
\end{corollary}

\paragraph{Scaling limits.} An early application of \cref{thm:bgt.gromov}, by Benjamini, Finucane and the first author~\cite{bft}, was to \emph{scaling limits} of sequences of vertex-transitive graphs. Briefly, a sequence of compact metric spaces $X_1,X_2,\ldots$ is said to \emph{Gromov--Hausdorff converge} or \emph{GH converge} to a compact metric space $X$ if there exist $(1+o(1),o(1))$-quasi-isometries $X_n\to X$.
The \emph{scaling limit} of the sequence $(X_n)$, if it exists, is the GH limit of the sequence $X_1',X_2',\ldots$, in which each $X_n'$ is the space $X_n$ with the metric scaled by $\diam(X_n)^{-1}$. Benjamini, Finucane and the first author's main result shows that if $(\G_n)$ is a sequence of finite connected vertex-transitive graphs with diameters tending to infinity satisfying $\diam(\G_n)\ge|\G_n|^\delta$ then $(\G_n)$ has a subsequence with a scaling limit that is a torus of finite dimension with an invariant Finsler metric \cite[Theorem 1]{bft}. An analogous result for possibly infinite vertex-transitive graphs satisfying a relative polynomial-volume condition is given by \cite[Theorem 3.2.2]{bft}, and generalised in an earlier paper of ours \cite[Remark 2.7]{ttTrof}. See \cite{bft} for precise statements and definitions.

In the preprint \cite{ttGH}, we obtain a number of refinements of these results, for example featuring the addition of sharp bounds on the dimensions of the limiting objects. That preprint builds crucially on the results of the present paper; in particular, the sharp bound on the dimension from that paper relies on having the optimal bounds on the dimensions of the Lie progressions in the present paper.  Furthermore, \cref{thm:trof.finite} can be seen as a finitary analogue of the result on scaling limits of finite vertex-transitive graphs, when rescaled by the diameter (although neither one immediately implies the other). Similarly, \cref{thm:inj/Z} can be seen as a finitary analogue of the more general result on scaling limits of vertex-transitive graphs, with respect to any growing sequence of scales (see \cref{rem:GH} for details).

\paragraph{Coarse topology.} In \cref{sec:sc}, we introduce coarse analogues of simple connectedness and homotopy. While several formulations appear in the literature (see, for example, \cite{CornDelaHarpe}), we adopt the framework developed by de la Salle and the first author in \cite{dlST}. We then prove, for these notions, coarse analogues of two classical facts from topology: that a topological covering map with simply connected range is always a homeomorphism (\cref{thm:k-sc.covering}); and the \emph{path-lifting property} that if $f:X\to Y$ is a covering map between CW-complexes, $x$ and $x'$ are two points of $X$, $p$ is a path joining $x$ to $x'$, and $q$ is a path joining $f(x)$ to $f(x')$ that is homotopic to $f(p)$, then there exists a unique lift of $q$ joining $x$ to $x'$ (\cref{prop:k-simply.conn}). These are ingredients in the proofs of our main theorems rather than applications of them, but we highlight them here as we believe them to be of independent interest, and that they could also be useful in other contexts, notably in geometric group theory and combinatorics.

\subsection*{Acknowledgements} We are grateful to Itai Benjamini for asking some of the questions that inspired this project.

\section{Precise statement of results}\label{sec:statements}
In this chapter we present detailed statements of our main results. We start by presenting a full definition of a \emph{Lie progression}, as well as of a related object called a \emph{Mal'cev--Losert progression} that will fill the role played by the vague notion of `local Mal'cev--Losert completion' we described in the introduction.

We defined an object called a Lie progression in our first paper \cite{proper.progs}, but it turns out that in order to obtain the precision required in the present paper we need to refine the definition somewhat. We start by defining a general notion of a \emph{progression} in a group. Given elements $u_1,\ldots,u_d$ of some group $G$ and positive reals $L_1,\ldots,L_d$, we define the \emph{progression} $P(u;L)=P(u_1,\ldots,u_d;L_1,\ldots,L_d)$ via
\[
P(u;L)=\{u_1^{\ell_1}\cdots u_d^{\ell_d}:\ell_i\in\Z,|\ell_i|\le L_i\}.
\]
Similarly, given elements $e_1,\ldots,e_d$ of a real vector space $V$ and positive reals $L_1,\ldots,L_d$, for a subring $A\subseteq\R$ we define the \emph{box} $B_A(e;L)=B_A(e_1,\ldots,e_d;L_1,\ldots,L_d)$ via
\[
B_A(e;L)=\{\ell_1e_1+\cdots+\ell_de_d:\ell_i\in A ,|\ell_i|\le L_i\}.
\]
We will mostly be interested in the cases where $G$ is a nilpotent Lie group, $V$ is its Lie algebra, $e_1,\ldots,e_d$ is a basis of a $V$, and $u_i=\exp e_i$.

In general, if $P$ is a progression then the sets $P^n$ can grow exponentially in $n$, so it is not clear that such objects should have anything to do with polynomial growth. However, if we impose a certain additional technical condition called \emph{upper-triangular form} on a progression then it turns out that it does exhibit polynomial growth (see \cref{lem:upper-tri.doubling}, below). Given $C>0$, we say that a tuple $(u;L)=(u_1,\ldots,u_d;L_1,\ldots,L_d)$ of group elements $u_i$ and positive reals $L_i$ is in \emph{$C$-upper-triangular form} if, whenever $1\le i<j\le d$, for all four choices of signs $\pm$ we have
\[
[u_i^{\pm1},u_j^{\pm1}]\in P\left(u_{j+1},\ldots,u_d;\textstyle{\frac{CL_{j+1}}{L_iL_j},\ldots,\frac{CL_d}{L_iL_j}}\right).
\]
In this case we also say that the progression $P(u;L)$ is in $C$-upper-triangular form. Similarly, we say that a tuple $(e;L)=(e_1,\ldots,e_d;L_1,\ldots,L_d)$ of elements $e_i$ of a real vector space and positive reals $L_i$ is in \emph{$C$-upper-triangular form} over $\R$, $\Q$ or $\Z$, respectively, if whenever $1\le i<j\le d$ we have
we have
\[
[e_i,e_j]\in B_A\left(e_{j+1},\ldots,e_d;\textstyle{\frac{CL_{j+1}}{L_iL_j},\ldots,\frac{CL_d}{L_iL_j}}\right).
\]
with $A=\R$, $\Q$ or $\Z$. Moreover, given $Q\in\N$, we say that $(e;L)$ is in \emph{$Q$-rational} $C$-upper triangular form to mean that it is in $C$-upper-triangular form over $\Q$ and the rationals $\ell_k$ needed to write each $[e_i,e_j]=\ell_{j+1}e_{j+1}+\cdots+\ell_de_d$ have denominator at most $Q$. We say that a tuple is simply in \emph{upper-triangular form} to mean that it is in $C$-upper-triangular form for some $C>0$.
\begin{definition*}[raw Lie progression]
Suppose that $N$ is a simply connected nilpotent Lie group of dimension $d$ and homogeneous dimension $D$ with Lie algebra $\n$. Suppose further that $e_1,\ldots,e_d$ is a basis for $\n$, that $L_1,\ldots,L_d\ge1$ are such that $(e;L)$ is in upper-triangular form over $\Q$, and that the group elements $u_i=\exp e_i\in N$ are such that $(u;L)$ is in upper-triangular form. Then the progression $P=P(u;L)$ is called a \emph{raw Lie progression} with \emph{dimension} $d$ and \emph{homogeneous dimension} $D$.

If $(e;L)$ is in $C$-upper-triangular form over $\Q$, and $(u;L)$ is in $C$-upper-triangular form for a given $C>0$, then we say $P$ is in \emph{$C$-upper-triangular form}. If $(e;L)$ is in $Q$-rational upper-triangular form then we say $P$ is \emph{$Q$-rational}. If $(e;L)$ is in upper-triangular form over $\Z$ (equivalently in $1$-rational upper-triangular form) and $\exp\la e_1,\ldots,e_d\ra=\la u_1,\ldots,u_d\ra$ then we call $P$ \emph{integral}. The elements $e_i$ are called the \emph{basis} of $P$, the elements $u_i$ are called the \emph{generators} of $P$, the subgroup $\la u_1,\ldots,u_d\ra$ is called the \emph{lattice} of $P$, and the $L_i$ are called the \emph{lengths} of $P$.
\end{definition*}

\begin{definition*}[Lie progression]
Suppose $\tilde P$ is a raw Lie progression in the simply connected nilpotent Lie group $N$. Let $\Gamma$ be the lattice of $\tilde P$. Suppose that $G$ is a group, that $H\normal G_0\le G$, and that $\pi:\Gamma\to G_0/H$ is a homomorphism, and let $P$ be the pullback to $G_0$ of $\pi(\tilde P)$. Then $P$ is called a \emph{Lie progression projected from $N$}. The subgroup $H$ is called the \emph{symmetry group} of $P$, the raw Lie progression $\tilde P$ is called the \emph{underlying raw progression} of $P$, the group $\Gamma$ is called the \emph{lattice} of $P$, and $\pi$ is called the \emph{projector} of $P$. The dimension, homogeneous dimension, basis and generators of $\tilde P$ are called, respectively, the \emph{dimension}, \emph{homogeneous dimension}, \emph{basis} and \emph{generators} of $P$. We say that $P$ is in \emph{$C$-upper-triangular form}, \emph{$Q$-rational}, or \emph{integral} if $\tilde P$ is in $C$-upper-triangular form, $Q$-rational, or integral, respectively. We define the \emph{injectivity radius} of $P$ to be the supremum of those $j\in\N$ such that $\pi$ is injective on $\tilde P^j$.

Given a Lie progression $P$, we write $\tilde P$ for its underlying raw progression, $\dim P$ for its dimension, $\hdim P$ for its homogeneous dimension, and $\inj P$ for its injectivity radius.
\end{definition*}

\begin{remarks*}\mbox{}
\begin{enumerate}[label=(\arabic*)]
\item A Lie progression of dimension $0$ is simply a finite subgroup.
\item The symmetry group of a Lie progression $P$ can depend on the parameterisation of $P$. For example, the set $\Z_n$ can be realised as a $0$-dimensional Lie progression with symmetry group $\Z_n$, or as a $1$-dimensional progression with trivial symmetry group. In particular, the symmetry group of $P$ does not have to equal $\{g\in G:gP=P\}$ (though see \cref{cor:sym.grp} below for a proof that this is the symmetry group if $\inj P\ge2$).
\item If $P$ is a Lie progression then $\deg(\la P\ra)\le\hdim P$, since a lattice in a finite dimensional simply connected nilpotent group $N$ has growth degree $\hdim N$.
\end{enumerate}
\end{remarks*}

We can now state the following precise version of \cref{thm:finitary.poly-vol}, which follows directly from \cref{thm:bgt.gromov,prop:XP.orig}.
\begin{theorem}[fully finitary Gromov-type theorem]\label{thm:finitary.poly-vol.detailed}
For every $d,R\in\N_0$ there exist $n_0^*=n_0^*(d,R)\in\N$ and $\eps=\eps(d)>0$ such that if $G$ is a group with finite symmetric generating set $S$ containing the identity, and
\begin{equation}\label{eq:poly.vol.statements}
|S^n|\le\eps n^{d+1}|S|
\end{equation}
for some $n\ge n_0^*$, then there exist a set $X\subseteq S^{O_d^*(1)}$ of cardinality at most $O_d^*(1)$ containing the identity, and an $O_d(1)$-rational Lie progression $P$ of dimension at most $O_d(1)$ in $O_d(1)$-upper-triangular form and with injectivity radius at least $R$ generating a normal subgroup of $G$, such that
\[
XP^m\subseteq S^{mn}\subseteq XP^{O_{d,R}(m)}
\]
for every $m\in\N$, and such that distinct elements of $X$ belong to distinct cosets of $\la P\ra$.
\end{theorem}

Before we can present our main theorem -- a precise version of \cref{thm:main.intro} -- we need to make precise the notion of `local Mal'cev--Losert completion' that we introduced in the introduction.

\begin{definition*}[Mal'cev--Losert progression]
Suppose $G$ is a finitely generated group with polynomial growth. Let $H$ be the maximal normal finite subgroup of $G$, let $\L=F\ltimes N$ be a Mal'cev--Losert completion of $G$, let $j:G\to\L$ be the corresponding Mal'cev--Losert factorisation, and write $\Gamma=j(G)\cap N$. Suppose $X\subseteq G$ is a set of size $|F|$ containing the identity such that $j(G)=j(X)\Gamma$. Suppose further that $\tilde P$ is a raw Lie progression in $N$ with lattice $\Gamma$, and set $P=j^{-1}(\tilde P)$, noting that $P$ is a Lie progression projected from $N$ with symmetry group $H$. Then $XP$ is called a \emph{Mal'cev--Losert progression} with \emph{Mal'cev--Losert completion $\L$, \emph{Mal'cev--Losert factorisation} $j$,} \emph{index} $|X|=|F|$, \emph{dimension} $\dim P$ and \emph{homogeneous dimension} $\hdim P$.

We say that $XP$ is in \emph{$C$-upper-triangular form}, \emph{$Q$-rational}, or \emph{integral} if $P$ is in $C$-upper-triangular form, $Q$-rational, or integral, respectively. 
\end{definition*}

\begin{remark}\label{rem:g(dim.ML)}
By \cref{thm:losert.intro}, the index $|X|$ of a Mal'cev--Losert progression $XP$ is always at most $d(\dim P)$.
\end{remark}

We are now almost ready to state our main theorem, which recall is a precise version of \cref{thm:main.intro}. Part of that rough statement was the existence of Mal'cev--Losert progressions in various truncated presentations $\la S\mid R_{G,S}(\alpha r_i)\ra$. In fact, we can say more in the upcoming detailed version of \cref{thm:main.intro}, and obtain these Mal'cev--Losert progressions as `lifts' of progressions in $G$. Let us now make this notion precise. 

Suppose, then, that $G$ is a group with symmetric generating set $S$ containing the identity, let $r\in\N$, and write $\hat G=\la S\mid R_{G,S}(r)\ra$. Write $\hat S$ to mean the natural copy of $S$ in $\hat G$, and note that there is a unique homomorphism $\psi:\hat G\to G$ that takes each element of $\hat S$ to the corresponding element of $S$, which we will call the \emph{projection} from $\hat G$ to $G$. Moreover, $\psi$ is injective on $\hat S^{\lfloor r/2\rfloor}$, so there exists a unique inverse $\ph:S^{\lfloor r/2\rfloor}\to\hat S^{\lfloor r/2\rfloor}$ to $\psi$ on $S^{\lfloor r/2\rfloor}$. Given an element $x\in S^{\lfloor r/2\rfloor}$ or subset $A\subseteq S^{\lfloor r/2\rfloor}$, abbreviate $\ph(x)$ by $\hat x$ or $\ph(A)$ by $\hat A$, noting that this is consistent with the $\hat S$ notation.
\begin{prop}\label{prop:lift.prog}
Given $C,d\in\N$, there exists an integer $M=M(C,d)\ge2$ such that the following holds for every integer $\eta\ge M$. Suppose $G$ is a group with finite symmetric generating set $S$ containing the identity, and $P\subseteq S^r$ is a Lie progression of dimension $d$ in $C$-upper-triangular form with symmetry group $H$, lattice $\Gamma$, generators $u_i$ and projector $\pi$. Let $g_1,\ldots,g_d\in G$ be such that $\pi(u_i)=g_iH$, and set $\hat G=\la S\mid R_{G,S}(\eta r)\ra$. Then $\hat H$ is a subgroup of $\hat G$ normalised by each $\hat g_i$, and there exists a homomorphism $\hat\pi:\Gamma\to\la \hat g_1,\ldots,\hat g_d,\hat H\ra/\hat H$ such that $\hat\pi(u_i)= \hat g_i\hat H$ for each $i$, and such that $\hat P$ is the pullback of $\hat\pi(\tilde P)$ to $\la \hat g_1,\ldots,\hat g_d,\hat H\ra$. In particular, $\hat P$ is a Lie progression of dimension $d$ in $C$-upper-triangular form with symmetry group $\hat H$, lattice $\Gamma$, generators $u_i$ and projector $\hat\pi$.
\end{prop}
We prove \cref{prop:lift.prog} in \cref{section:basicLieProg}; when it holds, we will call $\hat P$ the \emph{lift} of $P$ to $\hat G$.

We are now, at last, ready to state our main theorem. The list of conclusions we can draw is rather extensive, so we will initially present an abridged version of the theorem containing only the basic framework required for these conclusions; in \cref{sec:additional} below we will complete our list of conclusions. Let us emphasise that this is purely for ease of reading, and should not be interpreted to mean that the conclusions we present imminently are more important than those we defer until \cref{sec:additional}; indeed some of the latter are crucial for applications.

Here, then, is the core statement of our main theorem, which we prove in \cref{ch:proof}.
\begin{theorem}[local Mal'cev--Losert completions]\label{prop:chain.of.progs-index}
Given $C,d,k,Q,R,t,\eta\in\N$ there exist $\rho=\rho(d,k,t,\eta)\in\N$ and $n_0=n_0(d,k)\in\N$ such that the following holds.

Suppose that $G$ is a group with finite symmetric generating set $S$ containing the identity, that $Q_0$ is a $Q$-rational Lie progression of dimension $d$ and injectivity radius at least $\rho$ in $C$-upper-triangular form such that $\la Q_0\ra\normal G$, and that $X_0\subseteq S^t$ is a subset of size at most $k$ such that distinct elements of $X_0$ belong to distinct cosets of $\la Q_0\ra$, and such that for some $n,\eta\in\N$ with $n\ge n_0$ we have $X_0Q_0^m\subseteq S^{mn}\subseteq X_0Q_0^{\eta m}$ for all $m\in\N$.

Then there exist non-negative integers $d'\le d$, $\alpha\ll_{C,d,\eta}1$ and $r_0<r_1<\cdots<r_{d'}$ with $n\le r_0\ll_{C,d,Q,R,\eta}n$, and $O_{C,d,\eta}(1)$-rational Lie progressions $P_0\subseteq P_1\subseteq\cdots\subseteq P_{d'}\subseteq G$ in $O_{C,d,\eta}(1)$-upper-triangular form containing $Q_0$, such that 
\begin{enumerate}[label=(\roman*)]
\item the groups $G_i=\la S\mid R_{G,S}(\alpha r_i)\ra$ have polynomial growth;
\item $G_{d'}=G$, and $G_i=\la S\mid R_{G,S}(\Omega_{C,d,Q,R,\eta}(r_{i+1}))\ra$ for each $i=0,\ldots,d'-1$; \label{item:chain.of.progs-index:inj} 
\item for each $i$ we have $[G:\la P_i\ra]\le g(\dim P_i)$ and $G=S^{g(\dim P_i)}\la P_i\ra$, and for every set $X_i\subseteq S^{g(\dim P_i)}$ of coset representatives for $\la P_i\ra$ in $G$ containing the identity we have that
	\begin{itemize}
	\item the lift $\hat X_i\hat P_i$ of $X_iP_i$ to $G_i$ is well defined and a Mal'cev--Losert progression;
	\item writing $\hat S_i$ for the copy of $S$ in $G_i$ we have $\hat X_i\hat P_i^m\subseteq\hat S_i^{mr_i}\subseteq\hat X_i\hat P_i^{O_{C,d,\eta}(m)}$ for all $m\in\N$;
	\item $\inj P_{d'}=\infty$, and $\inj P_i\ll_{C,d,Q}\frac{r_{i+1}}{r_{i}}\ll_{C,d,Q,R,\eta}\inj P_{i}$ for each $i=0,\ldots,d'-1$;
	\item the projection $G_i\to G$ is injective on $\hat X_i\hat P_i^R$.
	\end{itemize}
\item $d\ge\dim P_0>\dim P_1>\cdots>\dim P_{d'}$.
		\setcounter{enumres}{\value{enumi}}
\end{enumerate}
\end{theorem}

\subsection{Balls of polynomial size}

As described in the introduction, an important first consequence of \cref{prop:chain.of.progs-index} is a version of \cref{thm:finitary.poly-vol.detailed} with an effective index bound and the optimal bound of $d$ on the dimension. On a very high level, applying \cref{prop:chain.of.progs-index} to the conclusion of \cref{thm:finitary.poly-vol.detailed} allows us to make the index bound in that conclusion effective (see \cref{thm:rel.hom.dim.k(d).effective}). Reducing $\eps$ to a quantity depending effectively on this new index bound, we can then make the dimension bound optimal (\cref{thm:multiscale}). The resulting statement is as follows; the full proof is given in \cref{ch:proof}.
\begin{theorem}[fully finitary Gromov-type theorem with optimal dimension bound and effective index bound]\label{thm:finitary.poly-vol.optimal}
For every $d,R\in\N_0$ there exist $n_0^*=n_0^*(d,R)\in\N$ and $\eps=\eps(d,R)>0$ such that if $G$ is a group with finite symmetric generating set $S$ containing the identity, and
\begin{equation}\label{eq:finitary.poly-vol.optimal}
|S^n|\le\eps n^{d+1}|S|
\end{equation}
for some $n\ge n_0^*$ then there exists a set $X\subseteq S^{O_d(1)}$ of cardinality at most $O_d(1)$ containing the identity, an $O_d(1)$-rational Lie progression $P$ of dimension at most $d$ in $O_d(1)$-upper-triangular form with injectivity radius at least $R$ generating a normal subgroup of $G$, and $r\in\N$ satisfying $n^{1/2}\le r\le n$ such that
\[
XP^{m}\subseteq S^{mr}\subseteq XP^{O_d(m)}
\]
for all $m\in\N$, and such that distinct elements of $X$ belong to distinct cosets of $\la P\ra$.
\end{theorem}

Applying \cref{prop:chain.of.progs-index} once again to this new improved version of \cref{thm:finitary.poly-vol.detailed} then gives the following precise version of \cref{cor:finitary.poly-vol}; see \cref{ch:proof} for a detailed proof. As with \cref{prop:chain.of.progs-index}, we present the basic framework of conclusions here, and then a number of additional conclusions in \cref{sec:additional}.

\begin{corollary}[sharp fine-scale polynomial-volume theorem]\label{cor:main.poly.vol}
For every $d,R\in\N_0$ there exist $\eps=\eps(d,R)>0$ and $n_0^*=n_0^*(d,R)\in\N$ such that if $G$ is a group with finite symmetric generating set $S$ containing the identity and
\[
|S^n|\le\eps n^{d+1}|S|
\]
for some integer $n\ge n_0^*$ then there exist non-negative integers $d'\le d$, $\alpha\ll_d1$ and $r_0<r_1<\cdots<r_{d'}$ with $n^{1/2}\ll_{d,R}r_0\le n<r_1$, and $O_d(1)$-rational Lie progressions $P_0\subseteq P_1\subseteq\cdots\subseteq P_{d'}\subseteq G$ in $O_d(1)$-upper-triangular form such that
\begin{enumerate}[label=(\roman*)]
\item the groups $G_i=\la S\mid R_{G,S}(\alpha r_i)\ra$ have polynomial growth;
\item $G_{d'}=G$, and $G_i=\la S\mid R_{G,S}(\Omega_{d,R}(r_{i+1}))\ra$ for each $i=0,\ldots,d'-1$;
\item for each $i$ we have $[G:\la P_i\ra]\le g(\dim P_i)$ and $G=S^{g(\dim P_i)}\la P_i\ra$, and for every set $X_i\subseteq S^{g(\dim P_i)}$ of coset representatives for $\la P_i\ra$ in $G$ containing the identity we have that
	\begin{itemize}
	\item the lift $\hat X_i\hat P_i$ of $X_iP_i$ to $G_i$ is well defined and a Mal'cev--Losert progression;
	\item writing $\hat S_i$ for the copy of $S$ in $G_i$ we have $\hat X_i\hat P_i^m\subseteq\hat S_i^{mr_i}\subseteq\hat X_i\hat P_i^{O_{d}(m)}$ for all $m\in\N$;
	\item $\inj P_{d'}=\infty$, and $\inj P_i\ll_{d}\frac{r_{i+1}}{r_{i}}\ll_{d,R}\inj P_{i}$ for each $i=0,\ldots,d'-1$;
	\item the projection $G_i\to G$ is injective on $\hat X_i\hat P_i^R$.
	\end{itemize}
\item $d\ge\dim P_0>\dim P_1>\cdots>\dim P_{d'}$.\label{item:main.cor.dim}
\end{enumerate}
\end{corollary}
\begin{remarks}\mbox{}
\begin{enumerate}[label=(\roman*)]
\item In \cref{thm:doubling.multiscale} we give an analogue of \cref{thm:finitary.poly-vol.optimal} under a small-doubling assumption $|S^{2n}|\le K|S^n|$ or a small-tripling assumption $|S^{3n}|\le K|S^n|$ in place of \eqref{eq:finitary.poly-vol.optimal}. Combining this with \cref{prop:chain.of.progs-index} immediately gives a version of \cref{cor:main.poly.vol} with a small-doubling or small-tripling assumption.
\item The ineffectiveness of $n_0^*$ in \cref{cor:main.poly.vol} arises only from its dependence on $d$; it depends effectively on $R$.
\end{enumerate}
\end{remarks}

\subsection{Local Lie-group completions at \emph{all} scales}
Immediately after \cref{cor:finitary.poly-vol} we discussed how on scales close to $r_i$ in \cref{prop:chain.of.progs-index,cor:main.poly.vol} there is a more appropriate `Lie-group completion' than the local Mal'cev--Losert completion. We now present the result to make that precise. To set the context, given $i<d'$ in \cref{prop:chain.of.progs-index}, write $H_i$ for the symmetry group of $P_i$, and write $\L_i$ for the Mal'cev--Losert completion of $G_i$. The quotient $G_i\to G$ induces a quotient $G_i/\hat H_i\to G/H_i$, and $G_i/\hat H_i$ embeds discretely and cocompactly into $\L_i$, as illustrated in the following commuting diagram.
\[
\begin{tikzcd}
	G_i    \arrow[r, two heads]  \arrow[d, two heads]    &   G \arrow[d, two heads]  \\
	G_i/\hat H_i  \arrow[r, two heads]  \arrow[d, hook, "{\substack{\text{discrete}\\\text{cocompact}}}"]  & G/H_i  \\
	\L_i &
\end{tikzcd}
\]
Let $\alpha$ be the quantity appearing in \cref{prop:chain.of.progs-index,cor:main.poly.vol}, so that $G_i=\la S\mid R_{G,S}(\alpha r_i)\ra$ by definition. Now suppose $q\ge\alpha r_i$ and set $G_i'=\la S\mid R_{G,S}(q)\ra$, so that the quotient $G_i\to G$ is composed of the quotients $\phi:G_i\to G_i'$ and $\psi:G_i'\to G$. Also, set $H_i'=\psi^{-1}(H_i)=\phi(\hat H_i)$, so that $\phi$ and $\psi$ induce quotients $\Phi:G_i/\hat H_i\to G_i'/H_i'$ and $\Psi:G_i'/H_i'\to G/H_i$, resulting in the following refinement of the above diagram.
\[
\begin{tikzcd}
	G_i    \arrow[r, two heads, "\phi"]  \arrow[d, two heads]   &   G_i'  \arrow[r, two heads, "\psi"]  \arrow[d, two heads]  &   G \arrow[d, two heads]  \\
	G_i/\hat H_i  \arrow[r, two heads, "\Phi"]  \arrow[d, hook, "{\substack{\text{discrete}\\\text{cocompact}}}"]  & G_i'/H_i'  \arrow[r, two heads, "\Psi"]   & G/H_i  \\
	\L_i &&
\end{tikzcd}
\]
Our next result shows that for a certain rather large value of $q$ we may embed $G_i'/H_i'$ discretely and cocompactly in a quotient Lie group of $\L_i$, as follows.
\begin{theorem}[local Lie-group completions at all scales]\label{thm:inj/Z}
For every $C,d,Q,R\in\N$ there exists $\zeta=\zeta(C,d,Q,R)>0$ such that for each $i<d'$ in \cref{prop:chain.of.progs-index} or \cref{cor:main.poly.vol} we may obtain the following additional conclusion. Write $c_i$ for the class of $P_i$ and $H_i$ for its symmetry group. Let $q=\max\{\zeta r_{i+1}^{c_i/(c_i-1)},\alpha r_i\}$, and set  $G_i'=\la S\mid R_{G,S}(q)\ra$. Write $\phi:G_i\to G_i'$ and $\psi:G_i'\to G$ for the quotient homomorphisms, set $H_i'=\psi^{-1}(H_i)=\phi(\hat H_i)$, and write $\Phi:G_i/\hat H_i\to G_i'/H_i'$ and $\Psi:G_i'/H_i'\to G/H_i$ for the quotients induced by $\phi$ and $\psi$. Then $\ker\Phi$ is normal in $\L_i$ and contained in the centre of $\la\hat P_i\ra/\hat H_i$; the quotient $\L_i/\ker\Phi$ is a virtually connected-nilpotent Lie group; and $G_i'/\hat H_i'$ embeds discretely and cocompactly in $\L_i/\ker\Phi$, giving the following commuting diagram.
\[
\begin{tikzcd}
	G_i    \arrow[r, two heads, "\phi"]  \arrow[d, two heads]   &   G_i'  \arrow[r, two heads, "\psi"]  \arrow[d, two heads]  &   G \arrow[d, two heads]  \\
	G_i/\hat H_i  \arrow[r, two heads, "\Phi"]  \arrow[d, hook, "{\substack{\text{discrete}\\\text{cocompact}}}"]  & G_i'/H_i'  \arrow[r, two heads, "\Psi"]  \arrow[d, hook, "{\substack{\text{discrete}\\\text{cocompact}}}"]  & G/H_i  \\
	\L_i \arrow[r, two heads] & \L_i/\ker\Phi &
\end{tikzcd}
\]
\end{theorem}

Note that $C$ and $Q$ play no role in \cref{cor:main.poly.vol}, so in that setting $\zeta$ depends on $d$ and $R$ only. We prove \cref{thm:inj/Z} in \cref{sec:all.scales}.

\begin{remark}\label{rem:GH}One interpretation/consequence of \cref{thm:inj/Z} is that at all scales above $n$ the group $G$ has a Gromov--Hausdorff approximation by a virtually connected nilpotent Lie group. We now describe this interpretation informally, using the notation \(a\lll b\) or $
a/b=o(1)$ to mean that two positive quantities \(a\) and \(b\) are such that \(a/b\) is ``much smaller'' than \(1\). Let $m$ be such that $
r_i\lll m\ll r_{i+1}$ in \cref{prop:chain.of.progs-index} or \cref{cor:main.poly.vol}; we will describe the Gromov--Hausdorff approximation at scale $m$.
The fact that \(\operatorname{diam}(H_i)\ll r_i\) and \(r_i\lll m\) implies that the projection $
G\longrightarrow G/H_i
$ is a \((1,o(m))\)-quasi-isometry. On the other hand,  \(\Psi\) from \cref{thm:inj/Z} is an isometry when restricted to the ball of radius \(q/2\). Since \(m\ll r_{i+1}\lll q\), it follows that $
G_i'/H_i'
$ is a Gromov--Hausdorff approximation of \(G/H_i\), and hence of \(G\), at scale \(m\). Now, \(G_i'/H_i'\) embeds as a cocompact lattice in $
\L_i/\ker\Phi$;
with some work based on \S \ref{ch:progs}, one can find a left-invariant metric on \(\L_i/\ker\Phi\) such that the embedding is a Gromov-Hausdorff approximation at scale $m$ (we will present full details in another paper and update the present paper with a precise reference when we have done so.\footnote{Let us briefly sketch how such an invariant metric $d$ can be constructed on $\L_i$ such that the embedding of  \(G_i/H_i\to \L_i\) is a Gromov-Hausdorff approximation at scale $m$ (one can easily sees that pushing this distance to $\L_i/\ker\Phi$ gives the desired left-invariant distance). Let $D$ a fundamental domain, which determines a projection $\pi:\L_i\to G_i/H_i$.  Start by considering the $G_i/H_i$-invariant pseudo-metric $d_0$ on $\L_i$  obtained by pulling back via $\pi$ the metric on $G_i/H_i$. For this pseudo metric, the embedding is actually an isometry. Exploiting \cref{prop:goodReducReal1} and \cref{lem:smallComm}, one can show that for a good choice of $D$, any translate $g\cdot d_0:=d_0(g\cdot ,g\cdot)$ of $d_0$ by an element $g$ of $D$ is such that $|d_0(x,y)-g\cdot d_0(x,y)|=o(d_0(x,y))+O(r_i)$. In particular, the embedding remains (uniformly) a GH-approximation at scale $m$ for these translated metrics. The desired metric is obtained as $d:=\sup_{g\in D}g\cdot d_0=\sup_{g\in \L_i}g\cdot d_0$, which defines a left-invariant metric.})
We therefore deduce that
$
\L_i/\ker\Phi
$
is (for such a metric) a Gromov--Hausdorff approximation of \(G\) at scale \(m\).
\end{remark} 

\subsection{Additional conclusions}\label{sec:additional}

We now present a number of additional conclusions in \cref{prop:chain.of.progs-index,cor:main.poly.vol}; we prove all of these in \cref{sec:proof.additional}. We first show how to bound the dimension and homogeneous dimension of the progressions $P_i$ appearing in \cref{prop:chain.of.progs-index,cor:main.poly.vol} in terms of the growth of $|S^m|$.

\begin{theorem}[dimension bounds]\label{thm:dimbounds}
In \cref{prop:chain.of.progs-index,cor:main.poly.vol} we may conclude additionally that
\begin{enumerate}[label=(\roman*)]
\item $\frac12d(d-1)+1\ge\hdim P_0>\hdim P_1>\cdots>\hdim P_{d'}$;\label{conc:hdimPi}
\item for each $i$ we have $m^{\dim P_i}\ll_{C,d,Q,R,\eta}|S^m|/|S|$ for every $m\in\N$ with $r_i\le m<r_{i+1}$;\label{conc:dim.bound}
\item for each $i$ we have $m^{\hdim P_i}\ll_{C,d,Q,R,\eta}|S^m|$ for every $m\in\N$ with $r_i\le m<r_{i+1}$.\label{conc:hdim.bound}
\end{enumerate}
\end{theorem}

We now show how the growth of $|S^m|$ is constrained by the dimensions and homogeneous dimensions of the progressions $P_i$. In the case of \cref{cor:main.poly.vol}, this can be seen as a more detailed version of \cref{cor:tao}.

\begin{theorem}[growth]\label{thm:growth}
In \cref{prop:chain.of.progs-index,cor:main.poly.vol} we may conclude additionally that for each $i$ there exist a natural number $k_i\le\hdim P_i-\dim P_i+1$, a sequence $r_i=r_{i,0}<r_{i,1}<\cdots<r_{i,k_i}$ with $r_{i,k_i}=\infty$ if $i=d'$ and $r_{i,k_i}=r_{i+1}$ otherwise, and all other $r_{i,j}\in\N$, and a sequence $\dim P_i\le d_{i,0}<d_{i,1}<\cdots<d_{i,k_i}\le\hdim P_i$ such that for every integer $m\in[r_{i,j},r_{i,j+1})$ we have
\[
\frac{|S^m|}{|S^{r_{i,j}}|}\asymp_{C,d,Q,R,\eta}\left(\frac{m}{r_{i,j}}\right)^{d_{i,j}}.
\]
\end{theorem}

\cref{cor:bt} says that, when $G$ is finite, it becomes `essentially virtually abelian' at scales exceeding the square root of its diameter. This manifests itself in the structure of the Lie progressions arising from \cref{prop:chain.of.progs-index,cor:main.poly.vol}, as follows.
\begin{theorem}[a scale beyond which progressions are abelian]\label{thm:sqrt.scale}
There exists $M=M(C,d,Q,R,\eta)>0$ such that if $|G|\le\infty$ in \cref{prop:chain.of.progs-index} or \cref{cor:main.poly.vol} then we may conclude additionally that $P_i$ is abelian for all $i$ with $r_{i+1}>M\diam_S(G)^{\frac{1}{2}}$.
\end{theorem}

\subsection{Uniqueness in our main theorem}
Conclusion \ref{item:chain.of.progs-index:inj} of \cref{prop:chain.of.progs-index} shows that the scales $r_i$ appearing in that theorem are unique up to multiplicative constants, and the uniqueness of the Mal'cev--Losert completion and the containments $\hat X_i\hat P_i^m\subseteq\hat S_i^{mr_i}\subseteq\hat X_i\hat P_i^{O_{C,d,\eta}(m)}$ mean that the $P_i$ are also somewhat canonical. In fact, even if we relax the statement so that the $P_i$ are simply Lie progressions generating subgroups of finite index in $G$, the $P_i$ and $r_i$ are still essentially unique, as follows.
\begin{prop}\label{prop:uniqueness}
Given $C,D,k,Q,t,\eta\in\N$ there exists $R=R(C,D,k,Q,t,\eta)\in\N$ such that the following holds. Suppose $G$ is a group with finite symmetric generating set $S$, and there are natural numbers $r_1<\cdots<r_d$ and $r_1'<\cdots<r_{d'}'$ such that $r_1'\leq r_1$, subsets $X_1,\ldots,X_d,X_1',\ldots,X_{d'}'\subseteq S^t$ of size at most $k$, and $Q$-rational Lie progressions $P_1,\ldots,P_d,P_1',\ldots,P_{d'}'\subseteq G$ of dimension at most $D$ and injectivity radius at least $R$ in $C$-upper-triangular form such that $X_iP_i^{\lfloor m/r_i\rfloor}\subseteq S^m\subseteq X_iP_i^{\eta m/r_i}$ for all $m\ge r_i$ and $X_i'(P_i')^{\lfloor m/r'_i\rfloor}\subseteq S^m\subseteq X_i'(P_i')^{\eta m/r'_i}$ for all $m\ge r'_i$, such that distinct elements of $X_i$ belong to distinct cosets of $\la P_i\ra$ and distinct elements of $X_i'$ belong to distinct cosets of $\la P_i'\ra$, such that $\dim P_1>\cdots\dim P_d$ and $\dim P_1'>\cdots>\dim P_{d'}'$, and such that $r_{i+1}/r_i\le A\inj P_i$ and $r_{i+1}'/r_i'\le A\inj P_i'$ for some $A\in\N$. Then
\begin{enumerate}[label=(\arabic*)]
\item \label{item:unique.Haus}there exists $j$ such that the set $\{\log r_1,\ldots,\log r_d\}$ is at Hausdorff distance at most $\log AR$ from $\{\log r_j',\ldots,\log r_{d'}'\}$; and
\item \label{item:unique.progs}if $i\in\{1,\ldots,d\}$ satisfies either $i=d$ or $r_{i+1}\ge(AR)^2r_i$, then letting $j$ be maximal such that $r_j'\le ARr_i$, and writing $N_i$ and $N_j'$ for the respective nilpotent Lie groups from which $P_i$ and $P'_j$ are projected, $\Gamma_i$ and $\Gamma'_j$ for their respective lattices, $H_i$ and $H'_j$ for their respective symmetry groups, and $\pi_i$ and $\pi'_j$ for their respective projectors, the following conditions are satisfied:
\begin{enumerate}[label=(\roman*)]
\item $[H_i:H_i\cap H_j']\le|X_j'|$ and $[H_j':H_j'\cap H_i]\le|X_i|$;
\item $H_i\cap\la P_j'\ra=H_i\cap H_j'=\la P_i\ra\cap H_j'$, and there exist a sublattice $\Lambda_i$ of index at most $|X_j'|$ in $\Gamma_i$ such that
\[
\pi_i(\Lambda_i)=\frac{(\la P_i\ra\cap\la P_j'\ra)H_i}{H_i},
\]
a sublattice $\Lambda_j'$ of index at most $|X_i|$ in $\Gamma_j'$ such that
\[
\pi_j'(\Lambda_j')=\frac{(\la P_i\ra\cap\la P_j'\ra)H_j'}{H_j'},
\]
and an isomorphism $\psi_{ij}:\Lambda_i\to\Lambda_j'$ such that the diagram
\[
\xymatrix{  
  \Lambda_i  \ar[d]_{\pi_i}\ar[rrrr]^{\psi_{ij}} &&&&  \Lambda_j' \ar[d]^{\pi_j'}\\
\displaystyle{\frac{(\la P_i\ra\cap\la P_j'\ra)H_i}{H_i}}&\cong&\displaystyle{\frac{\la P_i\ra\cap\la P'_j\ra}{H_i\cap H_j'}}&\cong&\displaystyle{\frac{(\la P_i\ra\cap\la P_j'\ra) H_j'}{H_j'}}
  }
\]
commutes;
\item $N_i\cong N_j'$.
\end{enumerate}
\end{enumerate}
\end{prop}

\section{Background and notation}\label{sec:background}

We use the convention that the natural numbers $\N$ do not include $0$, and write $\N_0=\N\cup\{0\}$. 

Recall that we write $g(d)$ for the maximum order of a finite subgroup of $\GL_d(\Z)$. 

Recall that, given a group $G$ with finite symmetric generating set $S$ containing the identity and $r\in\N$, we write $R_{G,S}(r)$ for the set of relations of length at most $r$ in the elements of $S$ that are satisfied in $G$, which is to say the set of words of length at most $r$ in the elements of $S$ that evaluate to the identity in $G$. We then refer to the group with presentation $\la S\mid R_{G,S}(r)\ra$ as \emph{$G$ with presentation truncated at scale $r$.}

In general, we use $G$ to denote a group, and $\G$ to denote a graph, although in each case we will explicitly introduce the given object. We use the notation $\G(G,S)$ to denote the \emph{Cayley graph} on a group $G$ with respect to a generating set $S$, i.e. the graph with vertex set $G$, and with elements $g$ and $h$ connected by an edge exactly when $g\ne h$ and $g\in h(S\cup S^{-1})$.

Given a graph $\G$, a vertex $x\in\G$ and non-negative real $r\ge0$, we write $B_\G(x,r)$ for the ball of radius $r$ centred at $x$ in the graph metric on $\G$. If $\G$ is transitive then we write $\beta_\G(r)$ for the number of vertices in a ball of radius $r$, and $\sigma_\G(n)$ for the number of vertices in a sphere of radius $r$.

We adopt a standard version of asymptotic notation, in which $O(X)$ means a quantity bounded above by a constant multiple of $X$; $\Omega(X)$ means a quantity bounded below by a positive constant multiple of $X$; $X\ll Y$ and $Y\gg X$ both mean that $X$ is bounded above by a constant multiple of $Y$; and $X\asymp Y$ means that $X\ll Y$ and $Y\ll X$ both hold. If the implied constant depends on some other parameters $\lambda_1,\ldots,\lambda_k$, we indicate this with subscripts, e.g. $O_{\lambda_1,\ldots,\lambda_k}(X)$, $X\ll_{\lambda_1,\ldots,\lambda_k}Y$.

Whenever we use this notation in this fashion, it means that the implied constant is \emph{effective}, i.e. could be computed explicitly from our arguments. If the implied constant is ineffective, we use the same notation but with the addition of an asterisk, e.g. $O^*(X)$, $X\ll^*Y$. We caution that, for example, the notation $O_\lambda(X)$ means a quantity that is at most a constant multiple of $X$, with the constant depending effectively on the parameter $\lambda$. However, despite the absence of an asterisk, the quantity $O_\lambda(X)$ could still be ineffective if $\lambda$ is itself an ineffective constant. This occurs on some occasions in our proofs, but never in the statements of our results.

Given elements $u_1,\ldots,u_d$ of a group or Lie algebra and positive reals $L_1,\ldots,L_d$, and given integers $i_1<\ldots<i_k\in[1,d]$, we often abbreviate $(u_{i_1},\ldots,u_{i_k};L_{i_1},\ldots,L_{i_k})$ to simply $(u_{i_1},\ldots,u_{i_k};L)$.

\subsection{Growth}
We have the following universal absolute lower bound on the growth of a connected vertex-transitive graph.
\begin{lemma}\label{lem:sphere.size.1}
Let $\G$ be a connected vertex-transitive graph. Then $\sigma_\G(n)\ge2$ for every positive integer $n<\diam(\G)$. In particular, $\beta_\G(n)\ge2n+1$ for all $n<\diam(\G)$, and $|\G|\ge2\diam(\G)$.
\end{lemma}
\begin{proof}
We may assume that $\sigma_\G(n)=1$ and show that $\diam(\G)=n$. We show first that $\diam(\G)<2n$. We prove the contrapositive: if there existed vertices $x,y$ with $d(x,y)=2n$ then there would be a vertex $z$ at distance $n$ from each of $x$ and $y$, so that $\sigma_\G(n)\ge2$.

Now let $o\in\G$. If $\diam(\G)>n$ then the unique vertex at distance $n$ from $o$ is a cut point for $\G$, and hence by transitivity every vertex is a cut point for $\G$. However, we have already shown that $\diam(\G)<\infty$, and no vertex at distance $\diam(\G)$ from $o$ can be a cut point.
\end{proof}
The bounds of \cref{lem:sphere.size.1} are all tight, and attained by cycles of even length and the bi-infinite path.

We also have universal relative lower bounds on the growth of connected vertex-transitive graphs, or more generally connected regular graphs.
Write $\rad_x(\G)=\min\{r\in\N:B_\G(x,r)=\G\}$ for the \emph{radius} of a conntected graph $\G$ centred at $x\in\G$. Note that if $\Gamma$ is vertex-transitive then $\diam(\G)=\rad_x(\G)$ for every $x\in\G$.
\begin{lemma}\label{lem:growth.lb.rel.lin}
Let $\G$ be a connected $k$-regular graph, let $x\in\G$, and fix $n\in\N$ with $n\le\rad_x(\G)$. Then $|B_\G(x,n)|\ge\frac{1}{3}(k+1)n$.
\end{lemma}
\begin{proof}
Since $n\le\rad_x(\G)$ there is a geodesic of length $n$ starting at $x$. Writing $x=x_0,x_1,\ldots,x_n$ for the vertices of this geodesic in increasing order of distance from $x$, the lemma follows from the fact that the balls $B_\G(x_{3m},1)$ with $m\in\N_0$ and $3m<n$ are disjoint subsets of $B_\G(x,n)$ of size $k+1$.
\end{proof}

We now record two standard results about finite-index subgroups of finitely generated groups that are relevant to growth.

\begin{lemma}[{\cite[Lemma 7.2.2]{hall}}]\label{lem:fin.ind.gen}
Let $k\in\N$. Suppose $G$ is a group with a finite symmetric generating set $S$ containing the identity, and $H$ is a subgroup of index $k$ in $G$. Then $H\cap S^{2k-1}$ generates $H$.
\end{lemma}

\begin{lemma}[{\cite[Lemma 11.2.1]{tointon.book}}]\label{lem:ball.cosets=>index}
Let $k\in\N$. Suppse $G$ is a group with a finite symmetric generating set $S$ containing the identity, and $H$ is a subgroup of index at least $k$ in $G$. Then $S^{k-1}$ has non-empty intersection with at least $k$ distinct left cosets of $H$.
\end{lemma}

Finally, the next lemma shows that, in certain circumstances, if $S^n$ is covered by a few translates of some set $Q$ then the further growth of $S$ is controlled by that of $Q$.
\begin{lemma}\label{lem:inclusions.local}
Let $k,n\in\N$. Suppose $G$ is a group with symmetric generating set $S$ containing the identity, and $X\subseteq S^k$ and $Q\subseteq G$ satisfy $S^{n+k}\subseteq XQ$. Then
\[
S^{mn+k}\subseteq XQ^m
\]
for all $m\in\N$. In particular, if $X\subseteq S^n$ and $S^{2n}\subseteq XQ$ then
\[
S^{mn}\subseteq XQ^{m-1}
\]
for every integer $m\ge2$.
\end{lemma}
\begin{proof}
The case $m=1$ is true by hypothesis, and for $m>1$ by induction we have $S^{mn+k}=S^nS^{(m-1)n+k}\subseteq S^nXQ^{m-1}\subseteq S^{n+k}Q^{m-1}\subseteq XQ^m$, as required.
\end{proof}

\subsection{Approximate groups}\label{sec:ag}

Given $K\ge1$, a \emph{$K$-approximate group} is a symmetric subset $A$ of a group containing the identity and satisfying $A^2\subseteq XA$ for some set $X$ of cardinality at most $K$. As we described in the introduction, approximate groups play a central role in the proof of \cref{thm:bgt.gromov}. They are also crucial in many of the arguments of the present paper.

In this section we record a number of specific results on approximate groups for use in later chapters. For more detailed background on approximate groups and information on some of their many other applications, see the second author's book \cite{tointon.book} or the surveys  surveys \cite{bgt.survey,app.grps,ben.icm,helf.survey,sand.survey}. 

At the most basic level, the relevance of approximate groups to the study of polynomial volume arises from the following results.
\begin{lemma}[{\cite[Lemma 8.4]{proper.progs}}]\label{lem:poly.pigeon}
Let $M,D>0$, let $\alpha,\beta\in(0,1)$, and let $q\in\N$. Then there exists $N=N_\alpha$ such that if $S$ is a finite subset of a group with $|S^n|\le Mn^D|S|$ for some $n\ge\max\{N,M(q\beta^{-1})^\frac{D+1}{1-\alpha}\}$ then there exists $k\in\N$ satisfying $n^\alpha<k<\beta n$ such that $|S^{qk}|\le q^{\frac{D+1}{1-\alpha}}|S^k|$.
\end{lemma}

\begin{lemma}[{\cite[Proposition 6.1]{ttTrof}}]\label{lem:tripling->AG}
Let $K\ge1$, and suppose $A$ is a precompact symmetric open set in a locally compact group with a Haar measure $\mu$. Suppose that $\mu(A^3)\le K\mu(A)$. Then $A^2$ is a precompact open $K^3$-approximate group.
\end{lemma}
\begin{prop}[{\cite[Proposition 2.1]{ttdoubling}}]\label{prop:doubling.covered.by.app.grp-lc}
Let $K\ge1$. Suppose $A$ is a precompact symmetric open set containing the identity in a locally compact group with a Haar measure $\mu$ satisfying $\mu(A^2)\le K\mu(A)$. Then there exists a precompact open $2^{12}K^{24}$-approximate group $U\subseteq A^4$ with measure $\mu(U)\le 4K^5\mu(A)$ and a finite set $X\subseteq A$ of cardinality at most $4K^4$ such that $A\subseteq XU$.
\end{prop}
When $A$ is finite, we have the following slightly more precise bounds in \cref{prop:doubling.covered.by.app.grp-lc}.
\begin{prop}[{\cite[Theorem 2.5.6]{tointon.book}}\footnote{The bounds are not stated this explicitly in the reference, but can easily be read out of the proof (including the exponent $18$ in the approximate-group parameter, which is erroneously stated as $24$ in the reference).}]\label{prop:doubling.covered.by.app.grp}
Let $K\ge1$, and suppose $A$ is a finite subset of a group satisfying $|A^2|\le K|A|$. Then there exists an $3^9K^{18}$-approximate group $U\subseteq A^2$ and a set $X\subseteq A$ of size at most $K^2$ such that $A\subseteq XU$.
\end{prop}

Finally, it will be useful to have the following standard result.
\begin{lemma}[{\cite[Proposition 2.6.5]{tointon.book}}]\label{lem:slicing}
Suppose that $A$ is a $K$-approximate group inside a group $G$ and that $H\le G$, and let $m\ge2$ be an integer. Then $A^m\cap H$ is a $K^{2m-1}$-approximate group.
\end{lemma}

\subsection{Nilpotent groups and commutators}\label{sec:nilp}

Given elements $x,y$ of a group $G$, we adopt the convention that the commutator $[x,y]=x^{-1}y^{-1}xy$. Given elements $x_1,\ldots,x_k$ of a group, we then define the \emph{simple commutator} $[x_1,\ldots,x_k]_k$ recursively via $[x_1]_1=x_1$ and
\[
[x_1,\ldots,x_k]_k=[[x_1,\ldots,x_{k-1}]_{k-1},x_k]
\]
for $k\ge2$; thus, in particular, $[x_1,x_2]_2=[x_1,x_2]$ by definition. We drop the subscript $k$ from the brackets when it is clear from the context what value it takes, such as is the case for $[x_1,\ldots,x_k]=[x_1,\ldots,x_k]_k$, for example.

We define the commutator $[H_1,H_2]$ of two subgroups $H_i\le G$ via $[H_1,H_2]=\langle[h_1,h_2]:h_i\in H_i\rangle$. Recall that the \emph{lower central series} of $G$ is the series
\[
G=\gamma_1(G)\ge\gamma_2(G)\ge\cdots
\]
defined by setting $\gamma_1(G)=G$ and $\gamma_k=[\gamma_{k-1}(G),G]$ for all $k\ge2$. The \emph{upper central series} of $G$ is the series
\[
\{1\}=Z_0(G)\le Z_1(G)\le Z_2(G)\le\cdots
\]
defined by setting $Z_0(G)=\{1\}$ and $Z_{k}(G)=\{z\in G:[z,g]\in Z_{k-1}G)\,\,\forall g\in G\}$ for all $k\ge2$; thus, each $Z_k(G)$ is defined so that $Z_k(G)/Z_{k-1}(G)$ is the centre of $G/Z_{k-1}(G)$.

Following \cite{mpty}, for each $k\in\N$ we also define the \emph{generalised commutator subgroup} $\overline G_k$ by
\begin{equation}\label{eq:gen.comms}
\overline\gamma_k(G)=\{g\in G:(\exists n\in\N)(g^n\in \gamma_k(G))\}.
\end{equation}
The generalised commutator subgroups are all characteristic in $G$, and if $G$ is finitely generated then each $\overline\gamma_k(G)$ contains $\gamma_k(G)$ as a finite-index subgroup \cite[Lemma 4.1]{mpty}. Moreover, we have
\begin{equation}\label{eq:gen.comm.filt}
[\overline\gamma_i(G),\overline\gamma_j(G)]\le\overline\gamma_{i+j}(G)
\end{equation}
for all $i,j\in\N$ \cite[Lemma 4.5]{mpty}.

The group $G$ is said to be \emph{nilpotent} if there exists $k$ such that $\gamma_k(G)=\{1\}$. The \emph{(nilpotency) class} of $G$ is then defined to be the (unique) $c\in\N$ such that $\gamma_c(G)\ne\gamma_{c+1}=\{1\}$. Equivalently (as it turns out), the class of $G$ is the (unique) $c$ such that $Z_{c-1}(G)\ne Z_c(G)=G$.

The \emph{Bass--Guivarc'h formula} \cite{bass,guiv} states that if $G$ is nilpotent of class $c$ then the growth degree of $G$ can be expressed as  \[
\deg(G)=\sum_{i=1}^{c} ir(i),
\]
where $r(i)$ is the torsion-free rank of the abelian quotient $\gamma_i(G)/\gamma_{i+1}(G)$, i.e., the number of infinite factors in the decomposition of this quotient as a direct sum of cyclic groups.

The set of finite-order elements of a nilpotent group $G$ form a subgroup $T(G)\le G$, called the \emph{torsion subgroup} of $G$. The torsion subgroup is trivially characteristic, and $G/T(G)$ is torsion-free. If $G$ is finitely generated then $T(G)$ is finite.

We now follow a set up in \cite[\S1]{bg} that was in turn based on \cite[\S11.1]{hall} to define some more general commutators than just the simple ones described above. We define \emph{(formal) commutators} in the letters $x_1,\ldots,x_r$ recursively by defining each $x_i$ and $x_i^{-1}$ to be a formal commutator, and for every pair $\alpha,\alpha'$ of commutators defining $[\alpha,\alpha']$ also to be a formal commutator. We also write $[\alpha',\alpha]=[\alpha,\alpha']^{-1}$. To each commutator $\alpha$ we assign a \emph{weight vector} $\chi(\alpha)=(\chi_1(\alpha),\ldots,\chi_r(\alpha))$, defined recursively by setting $\chi_i(x_j^{\pm1})=\delta_{ij}$ and, given two formal commutators $\alpha,\alpha'$ in the $x_j$, defining $\chi([\alpha,\alpha'])=\chi(\alpha)+\chi(\alpha')$. We define the \emph{total weight} $|\chi(\alpha)|$ of a commutator $\alpha$ to be $\|\chi(\alpha)\|_1$. We call $\chi_i(\alpha)$ the \emph{weight} of $x_i$ in $\alpha$, or the \emph{$x_i$-weight} of $\alpha$. We define a commutator $[\alpha,\alpha']$ to be a \emph{trivial commutator} if $\alpha=\alpha'$ or if either $\alpha$ or $\alpha'$ is trivial.

Of course, if the letters $x_i$ are elements that generate a group $G$ then we may interpret commutators recursively via $[\alpha,\beta]=\alpha^{-1}\beta^{-1}\alpha\beta$. It is easy to see that a trivial commutator always has the identity element as its interpretation. If $G$ is nilpotent of class $c$ then those commutators of total weight greater than $c$ also have trivial interpretations in $G$.

Following \cite[\S11.1]{hall}, we distinguish certain commutators, which we denote by $u_1,u_2,\ldots$, as \emph{basic commutators}. These are so called because in a free group $F$ with free generators $x_1,\ldots,x_r$ and lower central series $F=F_1>F_2>\ldots$ the basic commutators total weight $k$ in the $x_i$ form a free basis of the free abelian group $F_k/F_{k+1}$ (see \cite[\S11.1]{hall}).

We define the basic commutators recursively. For $i=1,\ldots,r$ we set $u_i=x_i$. Then, having defined the basic commutators $u_1,\ldots,u_m$ of total weight less than $k$, we define a commutator $\alpha$ of total weight $k$ to be basic if
\begin{enumerate}
\item $\alpha=[u_i,u_j]$ for some $u_i,u_j$ with $i>j$, and
\item if $u_i=[u_s,u_t]$ then $j\ge t$.
\end{enumerate}
We then label the basic commutators of total weight $k$ as $u_{m+1},\ldots,u_{m'}$, ordered arbitrarily subject to the constraint that basic commutators with the same weight vector are consecutive. Note that the arbitrariness of the order implies that the list of basic commutators is not uniquely defined. Note, however, that if $r\ge2$ the commutators $[[\cdots[[x_2,x_1],x_1]\cdots],x_1]$ are always basic, so there are always basic commutators of every total weight, whereas if $r=1$ then $x_1$ is the unique basic commutator.

\subsection{Nilpotent Lie groups}\label{sec:nilp.Lie}\label{sec:basic.comms}
Every finitely generated torsion-free nilpotent group $G$ embeds as a lattice in a simply connected nilpotent Lie group. This Lie group is unique up to isomorphism, and is called the \emph{Mal'cev completion} of $G$ \cite{malcev,raghu}. We sometimes denote the Mal'cev completion of $G$ by $G^\R$.

The generalised commutator subgroups of $G$ defined in \eqref{eq:gen.comms} interact nicely with the commutator subgroups of $G^\R$, as follows.
\begin{lemma}[{\cite[Lemma 4.22]{mpty}}]\label{lem:gen.comm.Mcompletion}
Let $G$ be a finitely generated torsion-free nilpotent group. Then $\overline\gamma_k(G)=G\cap\gamma_k(G^\R)$ for each $k\in\N$.
\end{lemma}
\begin{corollary}\label{lem:gen.comm.finite.index}
Let $G$ be a finitely generated torsion-free nilpotent group, and suppose $H$ is a finite-index subgroup. Then $\overline\gamma_k(H)=H\cap\overline\gamma_k(G)$ for each $k\in\N$.
\end{corollary}
\begin{proof}
Since $H$ has finite index in $G$ we have $H^\R=G^\R$, so \cref{lem:gen.comm.Mcompletion} gives $\overline\gamma_k(H)=H\cap\gamma_k(G^\R)=H\cap G\cap\gamma_k(G^\R)=H\cap\overline\gamma_k(G)$.
\end{proof}

By the Mal'cev rigidity theorem \cite[Theorem 2.11]{raghu}, every homomorphism from a finitely generated torsion-free nilpotent group $G$ to a simply connected nilpotent Lie group $N$ uniquely extends to a Lie group homomorphism $G^\R\to N$.
In particular, if $G$ and $H$ are finitely generated torsion-free nilpotent groups with Mal'cev completions $G^\R$ and $H^\R$ then every homomorphism $G\to H$ extends to a homomorphism $G^\R\to H^\R$. 

The Mal'cev completion of the free nilpotent group of class $c$ on generators $x_1,\ldots,x_d$ is the free nilpotent \emph{Lie} group of class $c$ on generators $x_1,\ldots,x_d$. Indeed, let $\varphi$ be the obvious homomorphism from the free nilpotent group of class $c$ on generators $x_1,\ldots,x_d$ to the free nilpotent Lie group of class $c$ on generators $x_1,\ldots,x_d$. By the Mal'cev rigidity theorem, $\varphi$ uniquely extends to a Lie group homomorphism $\varphi^\R$ defined on the Mal'cev completion. By construction, this extended morphism is surjective, and since both Lie groups are simply connected and have the same dimension, $\varphi^\R$ must be an isomorphism.

Given a Lie algeba $\g$ over $\R$, we define the \emph{lower central series} of $\g$ inductively by $\g_1=\g$ and $\g_{k+1}=[\g_k,\g]=\Span_\R\{[X,Y]:X\in\g_k,Y\in\g\}$. The Lie algebra $\g$ is said to be \emph{nilpotent} if there exists $k$ such that $\g_k=\{0\}$. The \emph{class} of $G$ is then defined to be the (unique) $c\in\N$ such that $\g_c\ne\g_{c+1}=\{0\}$.

If $G$ is a simply connected nilpotent Lie group with Lie algebra $\g$ then $\g$ is nilpotent of the same class \cite{malcev}. There are mutually inverse diffeomorphisms $\exp:\g\to G$ and $\log:G\to\g$, and one can describe the group operation in $G$ in terms of addition and the Lie bracket in $\g$ via the \emph{Baker--Campbell--Hausdorff formula}, which states that for elements $X,Y\in\g$ we have
\begin{equation}\label{eq:bch}
\textstyle\exp(X)\exp(Y)=\exp(X+Y+\frac{1}{2}[X,Y]+\frac{1}{12}[X,[X,Y]]+\cdots)
\end{equation}
\cite[Theorem 1.2.1]{cor-gre}. The terms in this expression are all iterated Lie brackets in $X$ and $Y$ multiplied by rational coefficients. The precise values of the rationals coefficients are not important for our arguments; all that matters is that in a nilpotent Lie group the series has only finitely many non-zero terms, and these depend only on the class of the group.

Simply connected nilpotent Lie groups are \emph{uniquely divisible}, which is to say that if $G$ is such a group and $g\in G$ then for every $n\in\N$ there exists a unique $h\in G$ such that $h^n=g$.

We define the \emph{dimension} of a simply connected nilpotent Lie group $G$, denoted $\dim G$, to be the dimension of its Lie algebra as a real vector space. The class of $G$ is at most $\dim G$, so any bound that in principle depends on both the class and dimension of $G$ can in fact be taken to depend only on the dimension.

We define the \emph{homogeneous dimension} of $G$, denoted $\hdim G$, by
\[
\hdim G=\sum_{k=1}^sk(\dim\g_k-\dim\g_{k+1})=\sum_{k=1}^s\dim\g_k.
\]
If $\Gamma$ is a lattice in $G$ then $\dim\g_k-\dim\g_{k+1}$ is exactly the torsion-free rank of the abelian quotient $\gamma_k(\Gamma)/\gamma_{k+1}(\Gamma)$; in particular, if $N$ is a finitely generated nilpotent group with torsion subgroup $T$ then $\deg(N)$ is exactly the homogeneous dimension of the Mal'cev completion of $N/T$.

A basis $x_1,\ldots,x_d$ for a nilpotent Lie algebra $\g$ is called a \emph{strong Mal'cev basis} if $\Span_\R(x_i,\ldots,x_d)$ is an ideal of $\g$ for each $i$. It follows from \cite[Proposition 1.2.7]{cor-gre} that if $x_1,\ldots,x_d$ is a strong Mal'cev basis for $\g$ and $\g$ is the Lie algebra of a simply connected nilpotent Lie group $G$ then, writing $u_i=\exp x_i$ for each $i$, every element of $G$ has a unique expression in the form $u_1^{\ell_1}\cdots u_d^{\ell_d}$ with $\ell_i\in\R$. Note in particular that the upper-triangular-form property of the basis of a Lie progression implies that it is a strong Mal'cev basis for the Lie algebra of the Lie group from which the progression is projected.

\begin{lemma}[{\cite[Lemma 4.3]{proper.progs}}]\label{lem:dilates.subgroup}
Let $c\in\N$. Then there exists $n=n(c)\in\N$ if $G$ is a connected, simply connected nilpotent Lie group of nilpotency class at most $c$, and $\Lambda$ is an additive subgroup of the Lie algebra of $G$ with $[\Lambda,\Lambda]\subseteq\Lambda$, then $\exp(n\cdot\Lambda)$ is a subgroup of $G$.
\end{lemma}

\begin{lemma}[{\cite[Lemma 5.1]{bg}}]\label{lem:bg.rational.power}
Let $c,t\in\N$. There exist rational polynomials $p_1,\ldots,p_m$ drawn from a finite list depending only on $c$, and indices $i_1,\ldots,i_m\in\{1,\ldots,t\}$, such that if $y_1,\ldots,y_t$ are elements in a nilpotent Lie group of class $c$ and $\eta\in\R$ then $(y_1\cdots y_t)^\eta=y_{i_1}^{p_1(\eta)}\cdots y_{i_m}^{p_m(\eta)}$.
\end{lemma}
\begin{proof}[Remarks on the proof]
This is not quite the statement of \cite[Lemma 5.1]{bg}, but it follows from the proof as presented there. Specifically, it is not clear from the statement of \cite[Lemma 5.1]{bg} that the same polynomials and indices work for arbitrary elements $y_j$ in an arbitrary nilpotent Lie group of class $c$, but this follows from the proof (or alternatively from applying the statement there to the free nilpotent Lie group of class $c$ on generators $y_1,\ldots,y_t$). Moreover, in the statement of \cite[Lemma 5.1]{bg} the exponent $\eta$ is assumed to be rational, but the proof mentions explicitly that $\eta$ may in fact be an arbitrary real. (One could also use a density argument to pass from the rationals to the reals.)
\end{proof}

We define a framework for describing iterted Lie brackets in the Lie algebra, analogous to higher-weight commutators in a group. Specifically, given elements $v_1,\ldots,v_r$ of a Lie algebra, we define every $v_j$ to be a Lie bracket of weight $1$ in the $v_i$, and for every pair $\alpha,\alpha'$ of Lie brackets of weights $\omega,\omega'$, respectively, we define $[\alpha,\alpha']$ to be a Lie bracket in the $v_i$ of weight $\omega+\omega'$. We again extend $v_1,\ldots,v_r$ to a list $\overline v_1,\ldots,\overline v_d$ of \emph{basic Lie brackets} in the $v_i$; these are defined in exactly the same way as the basic commutators in a group, except that the bracket now represents the Lie bracket instead of the commutator. We caution that if $v_i=\log x_i$ for each $i$ then the basic commutators in the $v_i$ are not in general equal to the logarithms of the basic commutators in the $x_i$.

The following is well known, but we are not aware of a reference.
\begin{prop}\label{prop:compact.central}
Let $N$ be a connected nilpotent Lie group, and suppose $H$ is a compact subgroup of $N$. Then $H$ is central in $N$.
\end{prop}
\begin{proof}
Let $\tilde{N}$ be the universal cover of $N$. Then $N$ is the quotient of $\tilde{N}$ by a discrete normal subgroup $\Gamma\le\tilde{N}$. Note that as a discrete normal subgroup of a connected group, $\Gamma$ is central: the map $\tilde{N}\times \Gamma\to \Gamma$ defined by $(g,\gamma)\mapsto [g,\gamma]$ is continuous, and hence constant with respect to $g$. It follows that $\Gamma$ is isomorphic to $\Z^d$. Its Mal'cev completion $\Gamma_\R$ is therefore a central subgroup of $\tilde{N}$ isomorphic to $\R^d$. We therefore have a central extension $1\to \Gamma_\R \to \tilde{N}\to\tilde{N} /\Gamma_\R\to1$, which induces a central extension 
$1\to \Gamma_\R/\Gamma \to N\to\tilde{N} /\Gamma_\R\to1$. Since  $\tilde{N}/\Gamma_\R$ is simply connected, any compact subgroup is contained in $\Gamma_\R/\Gamma$, and therefore is central.
\end{proof}

\subsection{Mal'cev--Losert completions}\label{sec:ML}
In this section we develop in more detail the theory of Mal'cev--Losert completions as defined after \cref{thm:losert.intro}, including in particular showing how to deduce that theorem from Mann's and Losert's work.

A classical theorem of Mal'cev states that a torsion-free finitely generated nilpotent group embeds as a cocompact lattice in a simply connected nilpotent Lie group, called its Mal'cev completion. The Mal'cev completion enjoys very strong functorial properties: any morphism between two torsion-free finitely generated nilpotent groups uniquely extends to a morphism between their Mal'cev completions. 
Losert \cite{Losert} generalises this notion to  compactly generated locally compact group of polynomial growth by proving the following theorem.
\begin{theorem}\label{thm:Losert}
Let $G$ be a compactly generated locally compact group of polynomial growth. Then $G$ has a maximal compact normal subgroup $K$, and there exists a simply connected nilpotent Lie group $N$, a compact group $Q$ acting faithfully on $N$, and a proper cocompact morphism $j:G\to \mathcal{L}=Q\ltimes N$ with kernel $K$ such that $j(G)N$ is dense in $\mathcal{L}$. Moreover, if $j:G\to \mathcal{L}$ and
 $j':G\to \mathcal{L}'$ are two such morphisms then there exists a unique morphism $\phi:\mathcal{L}\to \mathcal{L}'$ such that $j'=\phi \circ j$. 
 \end{theorem}

We will call the group $\L$ a \emph{Mal'cev--Losert completion} of $G$, and $j$ the corresponding \emph{Mal'cev--Losert factorisation}, or sometimes \emph{Mal'cev--Losert embedding} if $K=\{1\}$. Since $\L$ is unique up to isomorphism,
from now on we will often simply refer to ``the'' Mal'cev--Losert completion of $G$. We will also sometimes omit the $j$, and consider $G/K$ to be a subgroup of $\L$.
\begin{lemma}\label{lem:ML.no.normal.compact}
Suppose $\L=Q\ltimes N$ is the Mal'cev--Losert completion of some compactly generated locally compact group $G$ of polynomial growth. Then $\L$ has no normal compact subgroups.
\end{lemma}
\begin{proof}
Suppose $H\normal\L$ is compact, and write $\pi:\L\to\L/H$ for the quotient homomorphism. Let $K$ be the maximal normal compact subgroup of $G$, and write $j:G\to\L$ for the Mal'cev--Losert factorisation. By maximality of $K$, the composition $j'=\pi\circ j:G\to\L/H$ has kernel $K$. Moreover, the image $\pi(Q)$ is compact, and since $N$ has no non-trivial normal compact subgroups we have $\pi(N)\cong N$ and hence $\pi(Q)\cap\pi(N)=\{1\}$, so that $j':G\to\L/H\cong\pi(Q)\ltimes N$ is a Mal'cev--Losert factorisation. 
\cref{thm:Losert} therefore implies that $\pi$ is invertible and hence that $H=\{1\}$.
\end{proof}

We shall use these results in the case where $G$ is finitely generated. In this case $K$ and $Q$ are finite, density of $j(G)N$ means simply that $\mathcal{L}=j(G)N$, and properness of $j$ means that $j(G)$ is discrete in $\L$.
 Let $G$ and $G'$ be two finitely generated group of polynomial growth, and let $p:G\to G'$ be a surjective morphism. We want to show that $p$ ``induces'' a unique morphism between the Mal'cev--Losert completions. Note that $p(K)$ is a finite normal subgroup of $G'$, hence it is contained in $K'$. This means that $p$ induces a surjective morphism from $G/K$ to $G'/K'$, so that we lose nothing by restricting attention to groups without non-trivial finite normal subgroups.
 
\begin{lemma}\label{lem:Losert}
Suppose $G$ is a finitely generated virtually nilpotent group with no non-trivial normal finite subgroups. Let $\Gamma\le G$ be a normal torsion-free nilpotent subgroup of finite index, and let $\L=F\ltimes N$ be the Losert--Mal'cev completion of $G$. Then $\Gamma$ is a cocompact subgroup of $N$. In particular, $N$ is the Mal'cev completion of $\Gamma$.
\end{lemma}
\begin{proof}
Since $\Gamma N$ is a finite extension of a simply connected nilpotent Lie group, it is an algebraic group. The subgroup $\Gamma\cap N$ is Zariski-dense in $N$, so $\Gamma$ is Zariski-dense in $\Gamma N$. Being nilpotent is closed for the Zariski topology, so it follows that $\Gamma N$ is nilpotent.

The torsion subgroup of $\Gamma N$ is finite, and characteristic in $\Gamma N$ and so normal in $\L$, and hence trivial by \cref{lem:ML.no.normal.compact}. Thus $\Gamma N$ is torsion-free nilpotent, so its Losert--Mal'cev completion $M\le\L$ is simply connected \cite[\S 1.2.]{Losert}. However, $N$ has finite index in $M$ because $N\le M\le\L$, so since they are both simply connected it must be that $N=M$.
\end{proof}

 \begin{prop}\label{prop:Losert}
Let $G$ and $G'$ be two finitely generated groups of polynomial growth without normal finite subgroups, and let $p:G\to G'$ be a surjective morphism.  If $j:G\to \mathcal{L}=Q\ltimes N$ and
 $j':G'\to \mathcal{L}'=Q'\ltimes N'$ are the Mal'cev--Losert completions of $G$ and $G'$ respectively, then there exists a unique morphism $\pi:  \mathcal{L}\to \mathcal{L}'$, which restricts to $p$ on $j(G)$. Besides, $\pi$ is surjective and satisfies $\pi(N)=N'$. 
 \end{prop}
 \begin{proof}
In the sequel we will simply consider $G$ and $G'$ as subgroups of their Mal'cev--Losert completions and omit $j$ and $j'$. We let $J=\ker p$.
  Let $M=N\cap G$. Then $M$ is a torsion-free nilpotent normal finite index subgroup of $G$, whose Mal'cev completion is $N$.
  Note that $J_M:=J\cap M=J\cap N$ is a normal subgroup of $M$.  Let $J_N$ be the Mal'cev completion of $J_M$. By functioriality of the Mal'cev completion, $J_N$ identifies to a closed normal subgroup of $N$. Since $J_N$ is normalized by $G$ (again by functioriality), it is normalized by $GN$, hence is normal in $\mathcal{L}$.

 Note that $I_M=M\cap J_N$ is discrete and squeezed between $J_M$ and $J_N$, hence its Mal'cev completion is $J_N$. It follows that $J_M$ has finite index in $I_M$. This implies that $p(I_M)$ is a finite normal subgroup of $G'$, which by assumption must be trivial. Therefore, we have $I_M=J_M$. Since $M/I_M$ injects in the simply connected nilpotent group $N/J_N$, it is torsion-free. 

Note that the surjective morphism $G/J_M\to G'$ is injective in restriction to $M/J_M$, which has finite index in $G/J_M$. Hence $G/J_M\to G'$ has finite kernel $V$. 
Since $M/J_M$ is torsion-free, $V$ and $M/J_M$ have trivial intersection, and as they are both normal, they form a direct product. In particular $V$ centralizes $M/J_M$, and therefore by functoriality of the Mal'cev completion, it centralizes $N/J_N$ as well. Being normalized by $N/J_N$ and $G/J_N$, $V$ is normal in $\mathcal{L}/J_N$. We have 
\[\mathcal{L}/J_N\simeq  Q\ltimes N/J_N,\] and \[(\mathcal{L}/J_N)/V\simeq  \bar{Q}\ltimes N/J_N\] where $\bar{Q}$ is the image of $Q$. Let $E$ be kernel of the morphism $\bar{Q}\to \Aut(N/J_N)$.  We let \[\mathcal L''=((\mathcal{L}/J_N)/V)/E\simeq  (\bar{Q}/E)\ltimes N/J_N.\]

By construction, $G\to \mathcal L''$ factorises through $p$. Besides,  $\bar{Q}/E$ acts faithfully on $N/J_N$. Therefore, the induced morphism $G'\to \mathcal L''$ is a Mal'cev--Losert completion, 
with the density requirement following from the fact that
$G'N/J_N$ coincides with the image of $GN=\mathcal L$ through the natural projection $\mathcal{L}\to \mathcal{L}''$.  This proves the existence.

We now turn to the proof uniqueness. The proof goes in two steps: first we prove our last statement, namely that any such map $\pi$ must be surjective and satisfies $\pi(N)=N'$, and then we show that $I= \ker \pi$ is uniquely determined by $J=\ker p$. 

We first claim that  $\pi(N)=N'$. Indeed, since $\pi(N)$ is connected, it is contained in $N'$, and since it contains $\pi(N\cap G)=p(N\cap G)$, which has finite index in $G'$, we have $\pi(N)=N'$. Since the image of $\pi$ contains $N'$ and $p(G)=G'$, we deduce from the fact that $G'N= \mathcal{L}'$ that $\pi$ is duly surjective. We deduce that $\pi(Q)=Q'$. Recall that we have defined $J_N$ to be the Mal'cev completion of $J_M=J\cap M$, where $M=N\cap G$. Note that $p(M)$ is a finite index nilpotent subgroup of $G'$, hence by \cref{lem:Losert}, its Mal'cev completion is $N'$. Now, by functoriality of the Mal'cev completion, the short exact sequence 
\[1\to J_M\to M\to p(M)\to 1\]
induces a short exact sequence between the respective Mal'cev completions:
\[1\to J_N\to N\to N'\to 1.\]
This shows that $N\cap\ker\pi=J_N$.

Since $\pi$ restricts to $p$ on $G$, we have $J\subset I$, and therefore
\[
JJ_N\subset I.
\]
We claim that equality holds, which will yield the desired conclusion. Indeed,
\[
I\cap N = J_N = JJ_N\cap N,
\]
so it suffices to prove that
\[
IN = JJ_NN = JN.
\]
As noted above, we have
\[
\mathcal L' \simeq Q' \ltimes \pi(N),
\]
and hence
\[
\mathcal L'/(IN)\simeq Q'.
\]
Since \(N'j'(G')=\mathcal L'\), it follows that
\[
Q' \simeq G'/(N'\cap G')
      = p(G)/p(N\cap G).
\]

On the other hand, since \(GN=\mathcal L\), we obtain
\[
\mathcal L/(JN)
   \simeq G/(JN\cap G)
   \simeq (G/J)/(N\cap G)
   = p(G)/p(N\cap G)
   \simeq Q'.
\]
Thus $\mathcal L/(IN)\) and \(\mathcal L/(JN)$ have the same cardinality. Since $JN\subset IN$, we conclude that
\[
IN=JN.
\]
The claim follows. Consequently, $I$ is uniquely determined by $J$, and therefore $\pi$ is unique.
\end{proof}

The bound of $g(\dim N)$ on the size of $F$ in \cref{thm:losert.intro} essentially comes from the proof of Mann's \cref{thm:mann}, which we will now present. In fact, later on in the present paper we will need a much more detailed and quantitative version of \cref{thm:mann}, so for the avoidance of repitition we will state and prove that more detailed result now; let us emphasise, though, that the extra details all result from essentially the same proof that Mann gave.

Recall from \eqref{eq:gen.comms} that for each $k\in\N$ we define the generalised commutator subgroup $\overline\gamma_k(G)=\{g\in G:(\exists n\in\N)(g^n\in \gamma_k(G))\}$.
\begin{prop}\label{prop:mann}
Suppose $G$ is a group containing a normal torsion-free nilpotent subgroup $\Gamma$ of class $c$, Hirsch length $d$ and index $k$. Then there exists a normal subgroup $K\normal G$ containing $\Gamma$ such that $G/K$ is isomorphic to a finite group of automorphisms of $\Z^d$, such that $[K:Z_c(K)]\le k$, such that $|\gamma_{c+1}(K)|\ll_{c,k}1$, such that each element of $G\setminus K$ acts non-trivially by conjugation on $\overline\gamma_i(K)/\overline\gamma_{i+1}(K)$ for at least one $i$, and such that $\overline\gamma_{c+1}(K)$ is the maximum normal finite subgroup of~$G$.
\end{prop}
In the proof of \cref{prop:mann} we use the following lemma, which will also help us then apply \cref{prop:mann} to deduce \cref{thm:losert.intro}.
\begin{lemma}\label{lem:action->MLcomp}
Let $G$ be a finitely generated group. Suppose $\Gamma\normal G$ is a normal torsion-free nilpotent subgroup of finite index, and let $N$ be the Mal'cev completion of $\Gamma$. Suppose further that $X$ is a complete set of coset representatives of $\Gamma$ in $G$ containing the identity, and that each non-identity element of $X$ acts non-trivially by conjugation on at least one quotient $\overline\gamma_k(\Gamma)/\overline\gamma_{k+1}(\Gamma)$. Then $G$ has no non-trivial normal finite subgroups and there exists a finite group $F$ acting on $N$ such that $|F|=|X|$ and $F\ltimes N$ is the Mal'cev--Losert completion of $G$ (and the Mal'cev--Losert embedding of $G$ extends that of $\Gamma$).
\end{lemma}
\begin{proof}
We first show that $G$ has no non-trivial normal finite subgroups. Let $H\normal G$ be a normal finite subgroup, and suppose $h\in H\setminus\{1\}$. Certainly $h\notin\Gamma$. It follows from \eqref{eq:gen.comm.filt} that every element of $\Gamma$ acts trivially by conjugation on every $\overline\gamma_k(\Gamma)/\overline\gamma_{k+1}(\Gamma)$, so the assumption on $X$ implies that $h$ acts non-trivially by conjugation on at least one quotient $\overline\gamma_k(\Gamma)/\overline\gamma_{k+1}(\Gamma)$; in particular, there exists $g\in\Gamma$ such that $hgh^{-1}\ne g$. However, this implies that $hgh^{-1}g^{-1}\in(\Gamma\cap H)\setminus\{1\}$ is an element of infinite order in $H$, a contradiction. Thus $H=\{1\}$, as required.

\cref{lem:Losert} now implies that the Mal'cev--Losert completion is of the form $F\ltimes N$, with $F$ finite and $\Gamma$ embedded cocompactly in $N$. We have $|F|\le|X|$ because $\Gamma\le N$. Conversely, suppose $g_1,g_2\in G$ are in distinct cosets of $\Gamma$. As with $h$ in the previous paragraph, the assumption on $X$ implies that $g_1g_2^{-1}$ acts non-trivially by conjugation on at least one quotient $\overline\gamma_k(\Gamma)/\overline\gamma_{k+1}(\Gamma)$. By \cref{lem:gen.comm.Mcompletion}, this implies that $g_1g_2^{-1}$ acts non-trivially by conjugation on at least one quotient $\gamma_k(N)/\gamma_{k+1}(N)$, so that $g_1g_2^{-1}\notin N$ and $g_1,g_2$ are in distinct cosets of $N$. Thus $|F|\ge|X|$, as required.
\end{proof}

\begin{proof}[Proof of \cref{prop:mann}]
Consider the series of generalised commutator subgroups $\Gamma=\overline\gamma_1(\Gamma)\ge\cdots\ge\overline\gamma_{c+1}(\Gamma)$ of $\Gamma$ as defined in \eqref{eq:gen.comms}, noting that $\overline\gamma_{c+1}(\Gamma)=\{1\}$ since $\Gamma$ is torsion-free. Following Mann \cite[Theorem 9.8]{mann.book}, set
\[
K=\bigcap_{i=1}^cC_G(\overline\gamma_i(\Gamma)/\overline\gamma_{i+1}(\Gamma)),
\]
noting that $\Gamma\le K\normal G$ as required. Note that $G/K$ is the group of automorphisms that conjugation by $G$ induces on $\bigoplus_{i=1}^c\overline\gamma_i(\Gamma)/\overline\gamma_{i+1}(\Gamma)\cong\Z^d$. It is easy to check using \eqref{eq:gen.comm.filt} and induction on $i$ that $\overline\gamma_i(\Gamma)\le Z_{c-i+1}(K)$ for each $i$, and hence in particular that $\Gamma\le Z_c(K)$, and hence $[K:Z_c(K)]\le k$, as required; it is also known that if $[K:Z_c(K)]\le k$ then $|\gamma_{c+1}(K)|\ll_{c,k}1$ \cite[Theorem 1.3]{TKP.baer} (see \cite{TKP.baer} for a detailed history of this result). Next, let $g\in G\setminus K$. By definition there exists $i$ such that $g$ acts non-trivially on $\overline\gamma_i(\Gamma)/\overline\gamma_{i+1}(\Gamma)$. Since $\overline\gamma_i(\Gamma)\le\overline\gamma_i(K)$, and $\overline\gamma_{i+1}(\Gamma)=\Gamma\cap\overline\gamma_{i+1}(K)$ by \cref{lem:gen.comm.finite.index}, it follows that $g$ acts non-trivially on $\overline\gamma_i(K)/\overline\gamma_{i+1}(K)$ also. Finally, since $\overline\gamma_{c+1}(K)$ is characteristic in $K$ it is normal in $G$. \cref{lem:action->MLcomp} implies that $G/\overline\gamma_{c+1}(K)$ has no normal finite subgroups, so $\overline\gamma_{c+1}(K)$ is the maximal such of $G$.
\end{proof}

\begin{proof}[Proof of \cref{thm:losert.intro}]
By hypothesis $G$ contains a finite-index normal torsion-free nilpotent subgroup $\Gamma$. Writing $c$ for the class of $\Gamma$ and $h$ for its Hirsch length, it therefore follows from \cref{prop:mann} that there exists a normal subgroup $K\normal G$ with index at most $g(h)$ containing $\Gamma$ such that $|\gamma_{c+1}(K)|<\infty$, such that each element of $G\setminus K$ acts non-trivially by conjugation on $\overline\gamma_i(K)/\overline\gamma_{i+1}(K)$ for at least one $i$, and such that $H=\overline\gamma_{c+1}(K)$ is the maximum normal finite subgroup of $G$. Note that $K/H$ is a normal torsion-free nilpotent subgroup of finite index in $G/H$ with growth degree at most $d$. Let $N$ be the Mal'cev completion of $K/H$, noting that $\hdim N\le d$ and $\dim N=h$. Then \cref{lem:action->MLcomp} implies that there exists a finite group $F$ of size $[G:K]\le g(\dim N)$ acting faithfully on $N$ such that $\mathcal{L}=F\ltimes N$ is the Mal'cev--Losert completion of $G/H$.
\end{proof}

\subsection{The value of $g(d)$ and its optimality in our results}\label{sec:g(d)}
The exact value of $g(d)$ is apparently known for all $d$, but unfortunately not all of the details are published. According to Friedland \cite[p. 3519]{friedland} and Mann \cite[p. 88]{mann.book}, Feit computed $g(d)$ for all values of $d$ and characterised those finite subgroups of maximal order in $\GL_d(\Z)$, showing in particular that $g(d)=2^dd!$ for $d=1,3,5$ and $d\ge11$. These results appear in an unpublished preprint \cite{feit}, which we have not seen. For sufficiently large $d$, Friedland \cite{friedland} proves this bound and shows that the unique subgroup of $\GL_d(\Z)$ with maximum order is the orthogonal group $\mathrm{O}_d(\Z)$. The best general bound we are aware of with a published proof is $g(d)\le(2d)!$ \cite[p.~175, eq.~(16)]{newman}.

It is perhaps worth providing a slight update to Friedland's account of Feit's work. Friedland reports that for large $d$, Feit's work depends on a certain bound on the \emph{Jordan number}~$j(d)$. The name and definition of this quantity come from Jordan's theorem, which states that for each $d\in\N$ there exists $j(d)\in\N$ such that an arbitrary finite subgroup of $\GL_d(\C)$ contains a normal abelian subgroup with index at most $j(d)$. According to Friedland, for large $d$ Feit's argument relies on the bound $j(d)\le(d+2)!$ for $d\ge64$, which appears in an unpublished manuscript left by Weisfeiler \cite{weisfeiler} when he disappeared hiking in Chile in 1985. However, since then, Collins \cite{collins} has computed the optimal value of $j(d)$ for all $d$, rendering Weisfeiler's unpublished work unnecessary.

In any case, our results as stated do not rely on knowledge of $g(d)$. Moreover, irrespective of its precise value, the fact that $g(d)$ is the optimal index bound in \cref{thm:mann} and our main results follows from the following proposition.
\begin{prop}\label{prop:g(d)}
Let $K$ be a finite subgroup of maximum size in $\GL_d(\Z)$, and let $G=\Z^d\rtimes K$.
Then $G$ has no non-trivial normal finite subgroups, and every finite-index nilpotent subgroup of $G$ is contained in $\Z^d$, and so in particular has index at least $g(d)$.
\end{prop}
Note that the Mal'cev--Losert completion of $G$ is $K\ltimes\R^d$, so that \cref{prop:g(d)} also shows $g(\dim N)$ to be optimal in \cref{thm:losert.intro}.

To prove \cref{prop:g(d)} it will be convenient to have the following straightforward lemma (which is trivial if $K=\mathrm{O}_n(\Z)$).
\begin{lemma}
Suppose $K\le\GL_d(\Z)$ is a finite subgroup of maximum size. Then the action of $K$ on $\Z^d$ has no non-zero fixed points.
\end{lemma}
\begin{proof}
Suppose $K$ has a non-zero fixed point $z\in\Z^d$. We may assume without loss of generality that $z$ is unimodular, and hence extend it to a basis for $\Z^d$, with respect to which every element of $K$ will be of the form
\[
\left(\begin{array}{c|c}
1 & \ast \\
\hline
0 & A
\end{array}\right)
\]
for some $A\in\GL_{d-1}(\Z)$. We may then define a homomorphism $\ph:K\to\GL_{d-1}(\Z)$ by
\[
\ph\left(\begin{array}{c|c}
1 & \ast \\
\hline
0 & A
\end{array}\right)=A,
\]
the kernel of which consists entirely of matrices of the form
\[
\left(\begin{array}{c|c}
1 & \ast \\
\hline
0 & I
\end{array}\right).
\]
Such a matrix is either the identity or of infinite order, so since $K$ is finite we conclude that $\ker\ph=I$ and hence that $K$ is isomorphic to a subgroup $K'\le\GL_{d-1}(\Z)$. The matrices of the form
\[
\left(\begin{array}{c|c}
\pm1 & 0 \\
\hline
0 & A
\end{array}\right)
\]
with $A\in K'$ then form a finite subgroup of $\GL_d(\Z)$ with size larger than that of $K$.
\end{proof}

\begin{proof}[Proof of \cref{prop:g(d)}]
Suppose $H$ is a normal finite subgroup of $G$. To see that $H=\{1\}$, suppose on the contrary that $(0,1)\ne(w,k)\in H$. If $k=1$ then $w\ne0$, contradicting the finiteness of $H$. If $k\ne1$ then there exists $z\in\Z^d$ such that $k(z)\ne z$. We then have $(-z,1)(w,k)(z,1)(w,k)^{-1}=(k(z)-z,1)\in H$ by normality, again contradicting the finiteness of $H$.

Suppose now that $\Gamma\le G$ is a finite-index nilpotent subgroup. First, note that since $G$ is virtually abelian, every quotient of $\Gamma$ is virtually abelian, including in particular the quotient by the torsion subgroup $T(\Gamma)$. The quotient $\Gamma/T(\Gamma)$ is therefore both torsion-free nilpotent and virtually abelian, and hence torsion-free abelian. In particular, this shows that $[\Gamma,\Gamma]\le T(\Gamma)$. Since $\Gamma$ is a finitely generated nilpotent group, $T(\Gamma)$ is finite, so that $[\Gamma,\Gamma]$ is also finite.

Now suppose that, contrary to our claim, $\Gamma\not\le\Z^d$, so that there exists $(w,k)\in\Gamma$ with $k\ne1$. Since $\Gamma'=\Gamma\cap\Z^d$ has finite index in $\Z^d$, every non-trivial element of $K$ has a non-trivial action on $\Gamma'$; in particular, there exists $(z,1)\in\Gamma'$ such that $k(z)\ne z$. We then have $(-z,1)(w,k)(z,1)(w,k)^{-1}=(k(z)-z,1)\in[\Gamma,\Gamma]$, contradicting the finiteness of $[\Gamma,\Gamma]$.
\end{proof}

\section{Lie progressions}\label{ch:progs}

Recall from the introduction that we define the \emph{progression} $P(u;L)$ on \emph{generators} $u_1,\ldots,u_d\in G$ with \emph{lengths} $L_1,\ldots,L_d$ is to be
\[
P(u;L)=\{u_1^{\ell_1}\cdots u_d^{\ell_d}:|\ell_i|\le L_i\}.
\]
We also extend this notation by defining
\[
P(u;\infty)=\{u_1^{\ell_1}\cdots u_d^{\ell_d}:\ell_i\in\Z\}.
\]
We caution that in our previous paper \cite{proper.progs} we used the term \emph{ordered progression} and the notation $P_\ord(u;L)$, to distinguish this notion of progression from various other notions of progression that played a more prominent role at the time we wrote that paper. We also insisted that the lengths $L_i$ be integers in that paper.

Progressions are ``almost" symmetric in the sense that for group elements $u_1,\ldots,u_d$ and lengths $L_1,\ldots,L_d>0$  the progression $P(u,L)$ trivially satisfies 
\begin{equation}\label{eq:inverseProg}
P(u,L)^{-1}\subset P(u,L)^d.
\end{equation}
It is also not hard to see that for $m\in\N$ and $L_i\ge1$ we have
\begin{equation}\label{eq:dilateProg}
P(u;mL)\subset P(u,L)^{2dm}.
\end{equation}

Recall from the introduction that the tuple $(u;L)=(u_1,\ldots,u_d;L_1,\ldots,L_d)$ is in \emph{$C$-upper-triangular form} if, whenever $1\le i<j\le d$, for all four choices of signs $\pm$ we have
\begin{equation}\label{eq:C-upp-tri}
[u_i^{\pm1},u_j^{\pm1}]\in P\left(u_{j+1},\ldots,u_d;\textstyle{\frac{CL_{j+1}}{L_iL_j},\ldots,\frac{CL_d}{L_iL_j}}\right).
\end{equation}
In this case we also describe the progression $P(u;L)$ as being in \emph{$C$-upper-triangular form}. Note that if $(u;L)$ is in upper-triangular form then $\la u_1,\ldots,u_d\ra=P(u;\infty)$.

Given $\mu>0$, we say that the progression $P(u;L)$ is \emph{$\mu$-proper} if the elements $u_1^{\ell_1}\cdots u_d^{\ell_d}$ are all distinct as the $\ell_i$ range over those integers with $|\ell_i|\le\mu L_i$.

We noted in our previous paper \cite{proper.progs} that proper progressions in upper-triangular form have bounded doubling and polynomial growth. To state this precisely, we recall a definition from that paper. Given a progression $P=P(u_1,\ldots,u_d;L)$ in upper-triangular form, for every pair $i,j$ with $i<j$ and every one of the four possible choices of sign there is by definition some (not necessarily unique) expression $u_{j+1}^{\ell_{j+1}}\cdots u_d^{\ell_d}$ for $[u_i^{\pm1},u_j^{\pm1}]$. For every pair $i,j$ with $i<j$ and every one of the four possible choices of sign we fix arbitrarily one such expression, which we call the \emph{$P$-expression} for $[u_i^{\pm1},u_j^{\pm1}]$. We then define \emph{weights} $\zeta(k)$ of the $u_k$ by setting $\zeta(k)=1$ if $u_k$ does not appear in the $P$-expression for any $[u_i^{\pm1},u_j^{\pm1}]$, and
\[
\zeta(k)=\max\{\zeta(i)+\zeta(j):\text{$u_k$ appears in the $P$-expression for some $[u_i^{\pm1},u_j^{\pm1}]$}\}
\]
otherwise. Note that this is recursively well-defined, although the definition may depend on the choice of $P$-expression.
\begin{lemma}[{\cite[Lemma 2.1]{proper.progs}\footnote{In \cite[Lemma 2.1]{proper.progs} the $L_i$ are assumed to be positive integers, but one can reduce to that case by replacing each $L_i$ with $\lfloor L_i\rfloor$ and then disregarding those $u_i$ with $L_i=0$.}}]\label{lem:upper-tri.doubling.dilate}
Suppose that $P=P(u_1,\ldots,u_d;L_1,\ldots,L_d)$ is a progression in $C$-upper triangular form. Then for each $n\in\N$ we have $P^n\subseteq P(u;O_{C,d}(n^\zeta L))$, where we abbreviate $n^\zeta L=(n^{\zeta(1)}L_1,\ldots,n^{\zeta(d)}L_d)$.
\end{lemma}
\begin{corollary}[{\cite[Corollary 2.2]{proper.progs}}]\label{lem:upper-tri.doubling}
Let $d\in\N$ and $\mu,C>0$. Suppose that $P$ is a $\mu$-proper progression of rank $d$ in $C$-upper-triangular form. Then $|P^n|\le O_{C,\mu,d}(n^{O_d(1)})|P|$ for every $n\in\N$.
\end{corollary}

One of the other notions of progression we used in \cite{proper.progs}, and which plays a brief role in the present parer, is a special kind of progression called a \emph{nilpotent progression}, originally defined by Breuillard and Green \cite[Definition 1.4]{bg}. Given elements a nilpotent group $G$ of class $s$, we define the \emph{nilpotent progression} with \emph{generators} $x_1,\ldots,x_r\in G$ and \emph{lengths} $L_1,\ldots,L_r>0$ to be the progression $P=P(x_1,\ldots,x_d;L_1,\ldots,L_d)$ obtained by extending $x_1,\ldots,x_r$ to a complete list $x_1,\ldots,x_d$ of basic commutators in the $x_i$ of weight at most $s$, and setting $L_i=L^{\chi(u_i)}$ for every $i>r$. Here we use the notation $L^\chi$ to denote the quantity $L_1^{\chi_1}\cdots L_r^{\chi_r}$. We define $r$ to be the \emph{rank} and $s$ to be the \emph{step} of $P$. We showed in our first paper \cite[Proposition 3.4]{proper.progs} that a nilpotent progression of rank $r$ and step $s$ is in $O_{r,s}(1)$-upper-triangular form.

Nilpotent progressions enter this paper via the following result of the second author \cite{nilp.frei,polylog}. 
\begin{theorem}[{\cite[Corollary 1.9]{polylog}}]\label{thm:nilp.frei}
Let $s\in\N$ and $K\ge1$. Suppose $G$ is an $s$-step nilpotent group and $A\subseteq G$ is a finite $K$-approximate group. Then there exist a subgroup $H\subseteq A^{K^{e^{O(s)}}}$ normalised by $A$ and a nilpotent progression $P$ of rank at most $e^{O(s^2)}K\log^{O(s)}2K$ such that
\[
A\subseteq HP\subseteq HA^{e^{O(s^3)}K^{s+1}\log^{O(s^2)}2K}.
\]
\end{theorem}

The following result shows that the notion of Lie progression that we use in the present paper generalises the notion of nilpotent progression.
\begin{prop}\label{prop:nilp.prog=>Lie.prog}
Suppose $P$ is a nilpotent progression of rank $r$ and step $s$. Then $P$ is an $O_{r,s}(1)$-rational Lie progression of dimension at most $(4r)^s$ in $O_{r,s}(1)$-upper-triangular form with trivial symmetry group, projected from the free nilpotent Lie group of rank $r$ and class $s$.
\end{prop}
\begin{proof}
Let $u_1,\ldots,u_r$ be the standard generators of the free nilpotent group $N$ of rank $r$ and class $s$, and extend them to the list $u_1,\ldots,u_d$ of basic commutators. The number $d$ of basic commutators in this list is certainly at most the number of all commutators of weight at most $s$ in $r$ generators, of which there are at most $(4r)^s$ \cite[Proposition 5.3.3]{tointon.book}. Let $\pi:N\to\la P'\ra$ be the unique homomorphism such that $\pi(u_i)=x_i$ for each $i$. We showed in \cite[Proposition 3.4]{proper.progs} that $(u_1,\ldots,u_d;L)$ is in $O_{r,s}(1)$-upper-triangular form. Furthermore, embedding $N$ in the free nilpotent Lie group of rank $r$ and class $s$, it follows from \cite[Proposition 4.4]{proper.progs} that $(\log u_1,\ldots,\log u_d;L)$ is in $O_{r,s}(1)$-rational $O_{r,s}(1)$-upper-triangular form.
\end{proof}

\subsection{Basic properties of Lie progressions}\label{section:basicLieProg}
The following result shows that the symmetry group of a Lie progression behaves on scales within the injectivity radius like the torsion subgroup of a nilpotent group.
\begin{lemma}\label{lem:finite.subgroup.of.proper}
Suppose $P$ is a Lie progression with symmetry group $H$, and that $K\subseteq P^{\lfloor\inj P/2\rfloor}$ is a finite subgroup. Then $K\le H$.
\end{lemma}
To prove this it will be convenient to introduce a variant of the classical notion of a \emph{Freiman homomorphism}, which we will call a \emph{local homomorphism}. Given groups $G$ and $H$, a subset $A\subseteq G$ containing the identity, and a subset $B\subseteq G$ containing $A^2$, we define a map $\ph:B\to H$ to be a \emph{local homomorphism} on $A$ if $\ph(ab)=\ph(a)\ph(b)$ for all $a,b\in A$. Note that if $A$ is a subgroup of $G$ then $\ph$ is a local homomorphism if and only if it is a genuine group homomorphism.
\begin{lemma}\label{lem:local.hom.inverse}
Suppose that $G$ and $H$ are groups and that $A\subseteq G$ contains the identity. Suppose further that $\ph:A^2\to H$ is injective on $A^2$ and a local homomorphism on $A$, and define $\psi:\ph(A^2)\to A^2$ by $\psi(\ph(a))=a$. Then $\psi$ is an injective local homomorphism on $\ph(A)$.
\end{lemma}
\begin{remark*}
It is necessary in \cref{lem:local.hom.inverse} to assume that $\ph$ is injective on $A^2$ and define $\psi$ on $\ph(A^2)$. For example, if $\ph:\Z\to\Z/(2m+1)\Z$ is the quotient homomorphism then the restriction of $\ph$ to $A=\{-m,\ldots,m\}$ is a local homomorphism, injective on $A$, but the map $\psi:\Z/(2m+1)\Z\to A$ given by $\psi(n)\mapsto n$ is not a local homomorphism, since $\psi(m+1)=-m\ne m+1=\psi(m)+\psi(1)$. This in turn will be the reason why in \cref{lem:finite.subgroup.of.proper} we have to assume $K\subseteq P^{\lfloor\inj P/2\rfloor}$ rather than $K\subseteq P^{\inj P}$.
\end{remark*}
\begin{proof}
Given $a,b\in A$ we have $\psi(\ph(a)\ph(b))=\psi(\ph(ab))=ab=\psi(\ph(a))\psi(\ph(b))$. 
\end{proof}
\begin{lemma}\label{lem:local.hom.pullback}
Suppose that $G$ and $H$ are groups and that $A\subseteq G$. Suppose further that $\ph:A^2\to H$ is injective on $A^2$ and a local homomorphism on $A$, and that $K\le H$ is such that $K\subseteq\ph(A)$. Then the map $K\to G$ defined by $\ph(a)\mapsto a$ is an injective group homomorphism. In particular, $\ph^{-1}(K)$ is a subgroup of $G$ isomorphic to $K$.
\end{lemma}
\begin{proof}
Define $\psi:\ph(A^2)\to A^2$ by $\psi(\ph(a))=a$; this is injective by definition, and a local homomorphism on $\ph(A)$ by \cref{lem:local.hom.inverse}. In particular, $\psi$ is a local homomorphism on $K$, so that $\psi|_K:K\to G$ is in fact an injective group homomorphism.
\end{proof}
\begin{proof}[Proof of \cref{lem:finite.subgroup.of.proper}]
Let $\tilde P$ be the underlying raw Lie progression of $P$, write $\Gamma$ for its lattice, and write $\pi:\Gamma\to\la P\ra/H$ for its projector. Since $\pi$ is injective on $(\tilde P^{\lfloor\inj P/2\rfloor})^2$, \cref{lem:local.hom.pullback} implies that $\pi^{-1}(KH/H)$ is a subgroup of $\Gamma$ isomorphic to $KH/H$. In particular it is finite, and hence trivial because $\Gamma$ has no torsion.
\end{proof}
The following is a particularly useful consequence of \cref{lem:finite.subgroup.of.proper}.
\begin{lemma}\label{lem:sym.grp.normal}
Suppose that $G$ is a group with finite symmetric generating set $S$ containing the identity, that $P\subseteq G$ is a Lie progression generating a normal subgroup of $G$, and that $X\subseteq G$ is a finite subset containing the identity such that $XP\subseteq S$, such that $S^2\subseteq XP^\eta$ for some $\eta\in\N$, and such that $X\cap\la P\ra=\{1\}$. Suppose further that $\inj P\ge2\eta$. Then the symmetry group of $P$ is normal in $G$.
\end{lemma}
\begin{proof}
Write $H$ for the symmetry group of $P$. Given $x\in X$ we have $xHx^{-1}\subseteq S^2\subseteq XP^\eta$. Since $\la P\ra\normal G$ and $X\cap\la P\ra=\{1\}$, this implies that $xHx^{-1}\subseteq P^\eta$, and hence by \cref{lem:finite.subgroup.of.proper} that $xHx^{-1}\le H$. Since $P$ normalises $H$ and $G$ is generated by $X\cup P$, this proves the result.
\end{proof}
The following result is for interest only, and does not play a role in our arguments.
\begin{corollary}\label{cor:sym.grp}
If $P$ is a Lie progression in a group $G$ and $\inj P\ge2$ then the symmetry group of $P$ is exactly the set $\{g\in G:gP=P\}$.
\end{corollary}
\begin{proof}
The set $\{g\in G:gP=P\}$ is a subgroup contained in $P$ and containing the symmetry group of $P$, so this follows from \cref{lem:finite.subgroup.of.proper}.
\end{proof}

Finally, let us prove \cref{prop:lift.prog} about the existence of lifts of Lie progressions to truncated presentations.
\begin{proof}[Proof of \cref{prop:lift.prog}]
Write $\psi:\hat G\to G$ for the projection, and let $\ph:S^{\lfloor\eta r/2\rfloor}\to\hat S^{\lfloor\eta r/2\rfloor}$ be the unique inverse to $\psi$ on $S^{\lfloor\eta r/2\rfloor}$, noting that $\psi$ is a local homomorphism on $S^{\lfloor\eta r/4\rfloor}$ by \cref{lem:local.hom.inverse}. Set $M\ge4$, so that $\hat H$ is a subgroup by \cref{lem:local.hom.pullback}. By definition, for each $h\in H$ and each $i$ there exists $h'\in H$ such that $g_ih=h'g_i$. Since $\hat g_i\hat h=\hat h'\hat g_i$ for each $\hat h\in\hat H$ and each $i$, it follows that $\hat H$ is normalised by each $\hat g_i$ as required.

Since every element of $\Gamma$ can be written uniquely in the form $u_1^{\ell_1}\cdots u_d^{\ell_d}$ with $\ell_i\in\Z$, we may define a map $\hat\pi:\Gamma\to\la\hat g_1,\ldots,\hat g_d,\hat H\ra/\hat H$ by setting $\hat\pi(u_1^{\ell_1}\cdots u_d^{\ell_d})=\hat g_1^{\ell_1}\cdots \hat g_d^{\ell_d}\hat H$. We need to show that this map is a homomorphism. To see this, first note that given $m_1,\ldots,m_d,n_1,\ldots,n_d\in\Z$, in order to write
\[
u_1^{m_1}\cdots u_d^{m_d}u_1^{n_1}\cdots u_d^{n_d}=u_1^{\ell_1}\cdots u_d^{\ell_d}
\]
with $\ell_i\in\Z$ it suffices to use the upper-triangular-form relations
\[
[u_i^{\pm1},u_j^{\pm1}]=u_{j+1}^{k_{j+1}}\cdots u_d^{k_d},
\]
which for each $i<j$ hold with some $|k_n|\le CL_n/L_iL_j$, where $L_n$ are the lengths of $P$. Provided $\eta$ is large enough in terms of $C$ and $d$ only, for every such relation there is a corresponding relation
\[
[\hat g_i^{\pm1},\hat g_j^{\pm1}]\in \hat g_{j+1}^{k_{j+1}}\cdots \hat g_d^{k_d}\hat H
\]
in $\la\hat g_1,\ldots,\hat g_d,\hat H\ra/\hat H$, so that
\[
\hat g_1^{m_1}\cdots \hat g_d^{m_d}\hat g_1^{n_1}\cdots \hat g_d^{n_d}\hat H=\hat g_1^{\ell_1}\cdots \hat g_d^{\ell_d}\hat H.
\]
It follows that $\hat\pi$ is a homomorphism as required.
\end{proof}

\subsection{Rational and continuous progressions, and boxes in Lie algebras}\label{section:real/LieProg}
In our first paper \cite{proper.progs}, we proved the following result for passing between boxes in upper-triangular form over $\Z$ in a nilpotent Lie algebra and progressions in upper-triangular form over $\Z$ in the corresponding Lie group.
\begin{prop}[{\cite[Proposition 4.1]{proper.progs}}]\label{prop:pp4.1.orig}
Let $C>0$. Suppose $G$ is a connected, simply connected nilpotent Lie group with Lie algebra $\g$. Suppose $e_1,\ldots,e_d$ is a basis for $\g$ such that $\exp\langle e_1,\ldots,e_d\rangle$ is a subgroup of $G$, and $L_1,\ldots,L_d\ge1$ are such that $(e;L)$ is in $C$-upper-triangular form over $\Z$. Then, writing $u_i=\exp e_i$, we have 
\[
P(u;L)\subseteq \exp B_\Z(e;O_{C,d}(L))
\]
and
\[
\exp B_\Z(e;L)\subseteq P(u;O_{C,d}(L)).
\]
Moreover, $P(u;L)$ is in $O_{C,d}(1)$-upper-triangular form.
\end{prop}
A particular consequence of this result, which we use repeatedly in this paper, is the following.
\begin{corollary}\label{cor:pp4.1.group}
Suppose $G$ is a connected, simply connected nilpotent Lie group with Lie algebra $\g$. Suppose $e_1,\ldots,e_d$ is a basis for $\g$ in upper-triangular form over $\Z$ such that $\exp\langle e_1,\ldots,e_d\rangle$ is a subgroup of $G$. Then, writing $u_i=\exp e_i$ for each $i$, we have $\exp\langle e_1,\ldots,e_d\rangle=\la u_1,\ldots,u_d\ra=P(u;\infty)$.
\end{corollary}
\begin{proof}
Since $(e;1)$ is in $C$-upper-triangular form for some $C>0$, \cref{prop:pp4.1.orig} implies that $P(u;1)$ is in upper-triangular form, and hence that $\la u_1,\ldots,u_d\ra=P(u;\infty)$.

Clearly $\la u_1,\ldots,u_d\ra\le\exp\langle e_1,\ldots,e_d\rangle$. On the other hand, an arbitrary element of $\exp\langle e_1,\ldots,e_d\rangle$ belongs to $B_\Z(e;L)$ for some $L_i\in\N$, and since $(e;L)$ is in $C_L$-upper-triangular form for some $C_L>0$, \cref{prop:pp4.1.orig} implies that $B_\Z(e;L)\subseteq P(u;\infty)$.
\end{proof}

In the present paper we will also need analogues of \cref{prop:pp4.1.orig} over the rationals and reals. More precisely, given elements $u_1,\ldots,u_d$ in a Lie group and positive reals $L_1,\ldots,L_d$, we define the
\emph{rational progression} $P_\Q(u;L)$ via
\[
P_\Q(u;L)=\{u_1^{\ell_1}\cdots u_d^{\ell_d}:\ell_i\in\Q\cap[-L_i,L_i]\},
\]
and the \emph{continuous progression} $P_\R(u;L)$ via
\[
P_\R(u;L)=\{u_1^{\ell_1}\cdots u_d^{\ell_d}:\ell_i\in[-L_i,L_i]\}.
\]
We say that $(u;L)=(u_1,\ldots,u_d;L_1,\ldots,L_d)$ is in $C$-upper-triangular form \emph{over $\Q$} if for every $i<j\le d$ and for all four choices of signs $\pm$ we have
\[
[u_i^{\pm1},u_j^{\pm1}]\subseteq P_\Q\left(u_{j+1},\ldots,u_d;\textstyle{\frac{CL_{j+1}}{L_iL_j},\ldots,\frac{CL_d}{L_iL_j}}\right),
\]
and similarly over $\R$.
\begin{prop}\label{prop:pp4.1}
Let $C>0$. Suppose $G$ is a connected, simply connected nilpotent Lie group with Lie algebra $\g$. Suppose $e_1,\ldots,e_d$ is a basis for $\g$, and that $L_1,\ldots,L_d\ge1$ are such that $(e;L)$ is in $C$-upper-triangular form over $\Q$. Then, writing $u_i=\exp e_i$, we have 
\begin{equation}\label{eq:dilates.concl.1}
P_\Q(u;L)\subseteq \exp B_\Q(e;O_{C,d}(L))
\end{equation}
and
\begin{equation}\label{eq:dilates.concl.2}
\exp B_\Q(e;L)\subseteq P_\Q(u;O_{C,d}(L)).
\end{equation}
Moreover, $P_\ord(u;L)$ is in $O_{C,d}(1)$-upper-triangular form over $\Q$.
\end{prop}
\begin{remark}\label{rem:pp4.1}By density of $\Q$ in $\R$, the same result holds with $\Q$ replaced by $\R$, a fact that we use later without further explicit mention.
\end{remark}

The proof of \cref{prop:pp4.1} is essentially the same as that of \cref{prop:pp4.1.orig}, and uses certain functions mapping a set of letters to a commutator or Lie bracket in those letters. We call these functions \emph{bracket forms}. To define bracket forms, we momentarily treat brackets as formal objects that can be interpreted either as commutators or as Lie brackets depending on the context. Following \cite[Definition 3.2]{nilp.frei}, given letters $v_1,\ldots,v_r$, the function $\alpha_i$ defined by $\alpha_i(v_1,\ldots,v_r)=v_i$ is a bracket form of weight $1$, and then given two bracket forms $\alpha,\alpha'$ of weights $\omega,\omega'$, respectively, the function $[\alpha,\alpha']$ defined by $[\alpha,\alpha'](v_1,\ldots,v_r)=[\alpha(v_1,\ldots,v_r),\alpha'(v_1,\ldots,v_r)]$ is a bracket form of weight $\omega+\omega'$. Thus, for example, the function $\alpha:(w_1,w_2)\mapsto[w_1,[w_1,w_2]]$ is a bracket form of weight $3$, and if $x_1,x_2$ are elements of a group then $\alpha(x_1,x_2)$ is the commutator $[x_1,[x_1,x_2]]$, whilst if $v_1,v_2$ are elements of a Lie algebra then $\alpha(v_1,v_2)$ is the Lie bracket $[v_1,[v_1,v_2]]$.
\begin{lemma}[{\cite[Lemma 4.6]{proper.progs}}]\label{lem:comms.1}
Let $\alpha$ be a bracket form of weight $m$. Then there exists a sequence $\beta_1,\beta_2,\ldots$ of bracket forms of weight greater than $m$, of which at most finitely many have any given weight, and rationals $q_1,q_2,\ldots$ such that if $x_1,\ldots,x_m$ are elements of a connected, simply connected nilpotent Lie group, and $v_i=\log x_i$ are elements of the corresponding Lie algebra, then
\[
\log\alpha(x_1,\ldots,x_m)=\alpha(v_1,\ldots,v_m)+q_1\beta_1(v_1,\ldots,v_m)+q_2\beta_2(v_1,\ldots,v_m)+\cdots,
\]
with each $\beta_j$ featuring each $v_i$ at least once.
\end{lemma}

\begin{lemma}\label{lem:upper-tri}
Suppose $e_1,\ldots,e_d$ are elements of a Lie algebra and $L_1,\ldots,L_d>0$ are such that $(e;L)$ is in $C$-upper-triangular form over $\Q$. Let $\beta$ be a bracket form of weight $r$. Then for every $i_1\le\ldots\le i_r$ we have
\[
\textstyle\beta(e_{i_1},\ldots,e_{i_r})\in B_\Q\left(e_{i_r+1},\ldots,e_d;\textstyle{\frac{O_{C,d,r}(L_{i_r+1})}{L_{i_1}\cdots L_{i_r}},\ldots,\frac{O_{C,d,r}(L_d)}{L_{i_1}\cdots L_{i_r}}}\right)
\]
\end{lemma}
\begin{proof}
This is a routine induction on $r$.
\end{proof}

\begin{lemma}\label{lem:QL.group}
Suppose $e_1,\ldots,e_d$ is a basis of the Lie algebra $\g$ of a connected, simply connected nilpotent Lie group $G$, and suppose $L_1,\ldots,L_d>0$ are such that $(e;L)$ is in $C$-upper-triangular form over $\Q$. Then
\[
(\exp B_\Q(e;L))^2\subseteq\exp B_\Q(e;O_{C,d}(L)).
\]
\end{lemma}
\begin{proof}
This follows from the Baker--Campbell--Hausdorff formula \eqref{eq:bch} and Lemma \ref{lem:upper-tri}.
\end{proof}

\begin{proof}[Proof of \cref{prop:pp4.1}]
The inclusion \eqref{eq:dilates.concl.1} follows from writing
\[
P_\Q(u;L)=\exp B_\Q(e_1;L_1)\cdots\exp B_\Q(e_d;L_d)
\]
and applying \cref{lem:QL.group} repeatedly.

To prove \eqref{eq:dilates.concl.2}, observe using the Baker--Campbell--Hausdorff formula \eqref{eq:bch} that for $\ell_i\in\Q$ we have
\[
\exp(-\ell_1e_1)\exp(\ell_1e_1+\cdots+\ell_de_d)\subseteq\exp\Span_\Q(e_2,\ldots,e_d).
\]
Since $e_1,\ldots,e_d$ is a basis for $\g$ and $\exp:\g\to G$ is bijective, this combines with Lemma \ref{lem:QL.group} to imply that
\[
\exp B_\Q(e;L)\subseteq P_\Q(u_1;L_1)\exp B_\Q(e_2,\ldots,e_d;O_{C,d}(L_2),\ldots,O_{C,d}(L_d)),
\]
from which \eqref{eq:dilates.concl.2} follows by induction.

To see that $P(u;L)$ is in $O_{C,d}(1)$-upper-triangular form over $\Q$, note first that  \cref{lem:comms.1,lem:upper-tri} imply that for $i<j$ we have
\[
[u_i^{\pm1},u_j^{\pm1}]\in\exp B_\Q\left(e_{j+1},\ldots,e_d;\textstyle{\frac{O_{C,d}(L_{j+1})}{L_iL_j},\ldots,\frac{O_{C,d}(L_d)}{L_iL_j}}\right).
\]
It therefore follows from \eqref{eq:dilates.concl.2} applied to
\[
B_\Q\left(e_{j+1},\ldots,e_d;\textstyle{\frac{O_{C,d}(L_{j+1})}{L_iL_j},\ldots,\frac{O_{C,d}(L_d)}{L_iL_j}}\right)
\]
that
\[
[u_i^{\pm1},u_j^{\pm1}]\in P_\Q\left(u_{j+1},\ldots,u_d;\textstyle{\frac{O_{C,d}(L_{j+1})}{L_iL_j},\ldots,\frac{O_{C,d}(L_d)}{L_iL_j}}\right).\qedhere
\]
\end{proof}

\begin{lemma}\label{lem:pp.L2.1}
Suppose $G$ is a connected, simply connected nilpotent Lie group with Lie algebra $\g$. Suppose $e_1,\ldots,e_d\in\g$ and $L_1,\ldots,L_d\ge1$ are such that $(e;L)$ is in $C$-upper-triangular form over $\Q$, and let $u_i=\exp e_i$ for each $i$.
Then $P_\Q(u;L)^n\subseteq P_\Q(u;O_{C,d,n}(L))$ for each $n\in\N$.
\end{lemma}
\begin{remark}\label{rem:pp.L2.1}
As with \cref{prop:pp4.1}, by density of $\Q$ in $\R$ we automatically have the same result with $\Q$ replaced by $\R$.
\end{remark}
\begin{proof}
We have
\begin{align*}
P_\Q(u;L)^n&\subseteq(\exp B_\Q(e;O_{C,d}(L)))^n&&\text{(by \cref{prop:pp4.1})}\\
             &\subseteq\exp B_\Q(e;O_{C,d,n}(L))&&\text{(by \cref{lem:QL.group})}\\
             &\subseteq P_\Q(u;O_{C,d,n}(L))&&\text{(by \cref{prop:pp4.1}).}
\end{align*}
\end{proof}

In the next two results, just as in the definition of nilpotent progressions, given $L\in\R^r$ and $\chi\in\N^r$ we use the notation $L^\chi$ to denote the quantity $L_1^{\chi_1}\cdots L_r^{\chi_r}$. If $f_1,\ldots,f_r$ are elements of some nilpotent Lie algebra of class $c$, extended to a list $\overline f=f_1,\ldots,f_d$ of basic Lie brackets of weight at most $c$, and $L_1,\ldots,L_r>0$, then $B_\R(\overline f;L^\chi)$ is the set of linear combinations of $\sum_i\lambda_i f_i$ where  $|\lambda_i|\leq L^{\chi(i)}$. We also write $L\overline f=(L^{\chi(1)}f_1,\ldots,L^{\chi(d)}f_d)$. We remark that $B_\R(\overline f;L^\chi)$ is the continuous version of the \emph{nilbox} $\mathfrak{B}(f_1,\ldots,f_r;L)$ appearing in \cite{bg}.

\begin{lemma}\label{lem:real.B.upper-tri}
Let $x_1,\ldots,x_r$ be elements of a simply connected nilpotent Lie group of class $c$. Write $f_i=\log x_i$ for each $i$, and extend $f_1,\ldots,f_r$ to a list $\overline f=f_1,\ldots,f_d$ of basic Lie brackets of weight at most $c$. Then for every $L\in\N^d$ we have $[B_\R(\overline f;L^\chi),B_\R(\overline f;L^\chi)]\subseteq O_{r,c}(1)B_\R(\overline f;L^\chi)$.
\end{lemma}
\begin{proof}
We may assume that $x_1,\ldots,x_r$ freely generate the free nilpotent group of rank $r$ and class $c$, embedded in the free nilpotent Lie group of rank $r$ and class $c$. Since $B_\R(\overline f;L^\chi)=B_\R(L\overline{f};1)$, we may assume that $L=1$. The lemma then follows from the fact that $B_\R(\overline f;1)$ is a compact neighbourhood of the origin.
\end{proof}

\begin{lemma}\label{lem:freeequivalences}
Let $x_1,\ldots,x_r$ be elements of a simply connected nilpotent Lie group of class $c$. Write $f_i=\log x_i$ for each $i$, and extend $f_1,\ldots,f_r$ to a list $\overline f=f_1,\ldots,f_d$ of basic Lie brackets of weight at most $c$. Then for every $n\in\N$ and  for every $L\in\N^d$, we have
\begin{enumerate}[label=(\roman*)]
\item $\exp(B_\R(\overline f;(nL)^\chi))\approx_{r,c}\exp(B_\R(f;nL))$, \label{item:freeequiv.1}
\item $\exp(B_\R(\overline f;(nL)^\chi))\approx_{r,c,n}\exp(B_\R(\overline f;L^\chi))$, \label{item:freeequiv.2}
\item $\exp(B_\R(\overline f;L^\chi))^n\subseteq\exp(B_\R(\overline f;(O_{r,c,n}(L))^\chi))$, \label{item:freeequiv.iv}
\item $P_{\R}(x;L)^n\approx_{r,c}\exp(B_\R(\overline f;(nL)^\chi)$, \label{item:freeequiv.i}
\item $P_{\R}(x;L)^n\approx_{r,c}\exp(B_\R(f;nL)$. \label{item:freeequiv.3}
\end{enumerate}
\end{lemma}

\begin{proof}
We may assume that $x_1,\ldots,x_r$ freely generate the free nilpotent group of rank $r$ and class $c$, embedded in the free nilpotent Lie group of rank $r$ and class $c$. Observe that for all $n\in \N$ we have $P_{\R}(x;nL)=P_{\R}(x^L;n)$, $B_\R(\overline f;(nL)^\chi)=B_\R(L\overline{f};n^\chi)$, and $B_\R(f;nL)=B_\R(Lf;n)$, which reduces the the proof of these statements to the case where $L=1$. 

Moreover, for the proof of \ref{item:freeequiv.1}, we can similarly reduce to the case where both $L$ and $n$ equal $1$, in which case it simply follows from the fact that $\exp(B_\R(f;1))$ and $\exp(B_\R(\overline f;1))$ are both compact generating sets of the the free nilpotent Lie group of rank $r$ and class $c$. The same argument applies to \ref{item:freeequiv.2}.
Taking logarithms on both sides, \ref{item:freeequiv.iv} follows from the fact that every bounded set of the Lie algebra is contained in $B_\R(\overline f;C^\chi)$ for a sufficiently large $C$. 
We now turn to the proof of \ref{item:freeequiv.i}, which we have reduced to $P_{\R}(x;1)^n\approx_{r,c}\exp(B_\R(\overline f;n^\chi)$. 
For $n=1$, this follows from the fact that $P_{\R}(x;1)$ and $\exp(B_\R(\overline f;1))$ are compact generating sets. 
The proof of \cite[Theorem II.1]{Gui} gives the existence of constants $C=O_{r,c}(1)$ and $n_0=O_{r,c}(1)$ such that for all $n\geq n_0$.
\[ \exp(B_\R(\overline f;(n/C)^\chi))\subset P_{\R}(x;1)^n\subset \exp(B_\R(\overline f;(Cn)^\chi)).\]
We conclude by \ref{item:freeequiv.2} (applied with $n=C$).
Finally \ref{item:freeequiv.3} results from the combination of \ref{item:freeequiv.1} and \ref{item:freeequiv.i}.
\end{proof}

\begin{remark*}
Explicit constants can be obtained in \cref{lem:freeequivalences} using the Baker--Campbell--Hausdorff formula instead of the above ``compact generating set'' argument.
\end{remark*}

\begin{prop}\label{prop:goodReducReal1}
Let $G$ be a connected, simply connected nilpotent Lie group with Lie algebra $\g$ with strong Mal'cev basis $e_1,\ldots,e_d$. Write $u_i=\exp e_i$ for each $i$, write $\Gamma=\langle u_1,\ldots,u_d\rangle$, and let $L_1,\ldots,L_d\in\N$ be such that $(u;L)$ is in $C$-upper-triangular form. Then for every $k\in\N$ we have
\[
P(u,L)^k\subseteq P_{\R}(u;L)^k\cap\Gamma
\]
and
\[
P_{\R}(u;L)^k\cap\Gamma\subseteq P(u,L)^{O_{d,C}(k)}.
\]
\end{prop}

We first prove a version of \cref{prop:goodReducReal1} in a free nilpotent group.
\begin{prop}\label{prop:FreeReducCase}
Let $N_{d,s}$ be the free nilpotent group of class $s$ on $d$ generators $x_1,\ldots x_d$, embedded in the free nilpotent Lie group of class $s$ on $d$ generators, and let $L_1,\ldots, L_d>0$. Then for all $k\in \N$ we have
\begin{equation}\label{eq:FreeReducCase1}
P_{\R}(x;L)^k\subset  P(x^L;1)^{O_{d,s}(k)}P_{\R}(x;L)^{O_{d,s}(1)}\subset P(x;L)^{O_{d,s}(k)}P_{\R}(x;L)^{O_{d,s}(1)}.
\end{equation}
Moreover,
\begin{equation}\label{eq:FreeReducCase2}
P_{\R}(x;L)^k\cap N\subset  P(x;L)^{O_{d,s}(k)}.
\end{equation}
\end{prop}
\begin{proof}
The second inclusion of \eqref{eq:FreeReducCase1} is trivial. To prove the first inclusion, note first that applying the automorphism of $N_{\R}$ that maps $x_i^{L_i}$ to $x_i$ reduces the statement to the case where $L_i=1$, and so it suffices to prove that
\[
P_{\R}(x;1)^k\subset  P(x;1)^{O_{d,s}(k)}P_{\R}(x;1)^{O_{d,s}(1)}.
\]
Let $D$ be a compact subset such that $N D=N^{\R}$. Observe that the map $\phi:N_{\R} \to N$ that sends $g\in N_{\R}$ to the unique $\gamma\in N$ such that $g\in \gamma D$ is a left-inverse of the inclusion $N\to N_{\R}$. Since the latter is a quasi-isometry for the word metric associated to $P_{\R}(x;1)$ and $P(x;1)$, we deduce that  $\phi$ itself is a quasi-isometry. In particular, the ball $P_{\R}(x;1)^k$ of radius $k$ must me contained in the preimage  by $\phi$ of a ball of radius $O_{d,s}(k)$. This implies that 
\[
P_{\R}(x;1)^k\subset  P(x;1)^{O_{d,s}(k)}D,
\]
and so the first inclusion of \eqref{eq:FreeReducCase1} follows from the fact that $P_{\R}(x;1)$ generates $N^{\R}$. The inclusion \eqref{eq:FreeReducCase2} then follows from the fact that $P_{\R}(x;L)\cap N=P(x;L)$.
\end{proof}

\begin{proof}[Proof of Proposition \ref{prop:goodReducReal1}]
The first conclusion is trivial.

Write $s$ for the nilpotency class of $\Gamma$, noting that $s\ll_d1$, and let $N_{d,s}$ be the free nilpotent group of class $s$ on $d$ generators $x_1,\ldots x_d$. Write $\pi:N_{d,s}\to \Gamma$ for the homomorphism mapping $x_i$ to $u_i$ for each $i$. Then note that
\begin{align*}
P_{\R}(u;L)^k&\subseteq P(u;L)^{O_{d}(k)}P_{\R}(u;L)^{O_d(1)}
                                &\text{(by Proposition \ref{prop:FreeReducCase})}\\
         &\subseteq P(u;L)^{O_{d}(k)}P_{\R}(u;O_{d,C}(L))
                                &\text{(by Lemma \ref{lem:pp.L2.1})}.
\end{align*}
Since the $e_i$ form a strong Mal'cev basis and $\Gamma=P(u;\infty)$, it follows that
\begin{align*}
P_{\R}(u;L)^k\cap\Gamma&\subseteq P(u;L)^{O_{d}(k)}P(u;O_{d,C}(L))\\
           &\subseteq P(u;L)^{O_{d,C}(k)} &\text{(by \eqref{eq:dilateProg})},
\end{align*}
which gives the second conclusion.
\end{proof}

\subsection{Passing between arbitrary Lie progressions and integral Lie progressions}\label{section:integral/rational}
An important technical step implicit in our first paper was to show that a rational Lie progression can be covered by a few translates of an integral Lie progression. This allowed us to carry out the main arguments of that paper in that more specialised setting, where the geometry-of-numbers arguments in particular behave much better. This step can be captured explicitly in the following form; for completeness we provide a proof at the end of this section.
\begin{prop}\label{lem:make.integral}
Suppose $P=P(u;L)$ is a $Q$-rational raw Lie progression of dimension $d$ in $C$-upper-triangular form with basis $e_1,\ldots,e_d$. Then there exists a natural number $M\ll_dQ$ and a subset $X\subseteq P(u;M)$ such that $(Me;L)$ is in $CM$-upper-triangular form over $\Z$, such that $(u^M;L)$ is in $O_{C,d,Q}(1)$-upper-triangular form, such that $\exp\la Me_1,\ldots,Me_d\ra=\la u_1^M,\ldots,u_d^M\ra=P(u^M;\infty)$, such that $X\cap P(u^M;\infty)=\{1\}$, and such that
\begin{equation}\label{eq:make.integral}
P^r\subseteq XP(u^M;L)^{O_{C,d,Q}(r)}
\end{equation}
and
\begin{equation}\label{eq:make.integral.bdd}
XP(u^M;L)^r\subseteq P^{O_{d,Q}(r)}
\end{equation}
for all $r\in\N$.
\end{prop}
In the present paper, where we want to have optimal bounds on the indices in \cref{thm:rel.hom.dim,thm:abs.hom.dim}, and more generally  \cref{prop:chain.of.progs-index}, we cannot afford for the set $X$ arising in \cref{lem:make.integral} to appear in our conclusions. The second, and most important, aim of this section is therefore to develop tools that will allow us to `undo' \cref{lem:make.integral} once we have applied the results of our first paper, passing back to a rational Lie progression and absorbing the error set $X$ in the process.

The main tool we use for this is the following result, in which we have a raw Lie progression $P(u;L)$ (which for the sake of generality is assumed to be $Q$-rational in the proposition but in our applications will be integral) generating a lattice that has finite index in the lattice $\Gamma'$ of primary interest, and we show that $P(u;L)$ can be well approximated by a rational progression $P(u',L')$ that generates the whole of $\Gamma'$.

\begin{prop}\label{lem:gettingridofX}
Suppose that $P(u;L)$ is a $Q$-rational raw Lie progression of dimension $d$ in $C$-upper-triangular form in a simply connected nilpotent Lie group $N$. Let $k\in\N^d$, and suppose further that $P(u^{1/k},\infty)$ is a group, and that $\Gamma'$ is some other group satisfying $P(u;\infty)\le\Gamma'\le P(u^{1/k},\infty)$.
Then there exists an $O_{Q,d,k}(1)$-rational raw Lie progression $P(u',L')$ in $O_{C,d,k}(1)$-upper-triangular form such that $P(u';\infty)=\Gamma'$, and such that
\begin{enumerate}[label=(\roman*)]
\item $P(u;L)\subseteq P(u',O_{C,d,k}(L'))$,
\item $P(u';L')\subseteq P(u^{1/k},O_{C,d,k}(L))$.
\end{enumerate}
\end{prop}

\begin{proof}
We will prove by induction on $d$ that the desired conclusions hold as well as
\begin{itemize}
\item[(i')]  $P_\R(u,L)\subseteq P_\R(u',O_{C,d,k}(L'))$,
\item[(ii')] $P_\R(u',L') \subseteq P_\R(u^{1/k},O_{C,d,k}(L))$,
\end{itemize}
with the list $L'_1,\ldots,L'_d$ consisting of the same numbers as the list $L_1,\ldots,L_d$, possibly reordered, and, writing $e_i=\log u_i$ and $e_i'=\log u_i'$, that each $e'_i$ can be expressed in the form $\sum_{j=1}^da_{ij}e_j$ for some $a_{ij}\in\{0,\frac1{k_j},\ldots,\frac{k_j-1}{k_j}\}$.

We may assume that $C\ge1$. We claim that we may then also assume that
\begin{equation}\label{eq:L_d.almost.max}
L_d\ge C^{i-d}L_i
\end{equation}
for $i=1,\ldots,d-1$. Indeed, if this is not the case, then let $i$ be maximal such that this inequality fails. If $e_i$ is central then we may simply shift $e_i$ and $L_i$ to the far right (so that $e_d$ becomes $e_{d-1}$ and $e_i$ becomes $e_d$) without affecting the upper-triangular form, after which this inequality holds for all $i$. On the other hand, if $e_i$ is not central then there exists $j$ such that $[e_i,e_j]\ne0$, and then the upper triangular form implies that there exists $k>i$ such that $CL_k\ge L_iL_j$, and hence in particular $L_k\ge C^{-1}L_i>C^{d-i-1}L_d\ge C^{d-k}L_d$, contradicting the maximality of $i$.
 
Write $\pi:N\to N/u_d^\R$ for the quotient homomorphism, and $H=\Gamma'\cap u_d^\R$. Define $\bar u_i=\pi(u_i)$ and $\bar e_i=\log\bar u_i$ for $i=1,\ldots,d-1$, and note that $\pi(P)=P(\bar u;L)$ is a $Q$-rational raw Lie progression of dimension $d-1$ in $C$-upper-triangular form in the simply connected nilpotent Lie group $N/u_d^\R$. By induction, there therefore exists an $O_{Q,d,k}(1)$-rational raw Lie progression $P(\bar u',\bar L')$ in $O_{C,d,k}(1)$-upper-triangular form such that $P(\bar u',\infty)=\Gamma'/H$, such that
\begin{itemize}
\item[(iii)] $P(\bar{u},L)\subseteq P(\bar{u}',O_{C,d,k}(\bar{L}'))$,
\item[(iv)] $P(\bar{u}',\bar{L}')\subseteq P(\bar{u}^{1/k},O_{C,d,k}(L))$,
\end{itemize}
such that
\begin{itemize}
\item[(iii')] $P_\R(\bar{u},L)\subseteq P_\R(\bar{u}',O_{C,d,k}(\bar{L}'))$,
\item[(iv')] $P_\R(\bar{u}',\bar{L}')\subseteq P_\R(\bar{u}^{1/k},O_{C,d,k}(L))$,
\end{itemize}
such that $\bar L'$ is a reordering of $L_1,\ldots,L_{d-1}$, and such that, writing $\bar e_i=\log\bar u_i$ and $\bar e_i'=\log\bar u_i'$, we have $\bar e'_i=\sum_{j=1}^{d-1} a_{ij}\bar{e}_j$ with $a_{ij}\in\{0,\frac1{k_j},\ldots,\frac{k_j-1}{k_j}\}$.

We start by defining $e'_i$, and hence $u'_i=\exp e'_i$, for $i<d$. By definition, $\sum_{j=1}^{d-1} a_{ij}e_j\in\log\Gamma'+\langle e_d/k_d\rangle$, so there exists $q\in\N$ such that $\sum_{j=1}^{d-1} a_{ij}e_j+(q/k_d)e_d\in\log\Gamma'$. Since $e_d\in\log\Gamma'$, we may assume that $q\in \{0,\ldots,k_d-1\}$, and then define $a_{id}=q/k_d$ and set
\[
e'_i=\sum_{j=1}^{d} a_{ij}e_j.
\]
We next define $e'_d$, and hence $u'_d=\exp e'_d$. Specifically, we set $e'_d=e_d/n$, with $n$ the unique positive divisor of $k_d$ such that $u'_d$ is a generator of $H$. The set $e'_1,\ldots,e_d'$ is then a basis for the Lie algebra $\n$, the set $u'_1,\ldots,u'_d$ is a generating set for $\Gamma'$, and we have $a_{ij}\in\{0,\frac1{k_j},\ldots,\frac{k_j-1}{k_j}\}$ as required. Finally, we define  $L'=(\bar{L}'_1,\ldots, \bar{L}'_{d-1}, L_d)$, noting that this is a reordering of $L$ as required.

We now verify the inclusions (i), (ii), (i') and (ii'), and the rationality and upper-triangularity of $P(u',L')$.
First, note that by \cref{prop:pp4.1,lem:pp.L2.1} (see \cref{rem:pp4.1,rem:pp.L2.1}), (iii') and (iv') give
\begin{itemize}
\item[(iii'')] $B_\R(\bar{e};L)\subseteq B_{\R}(\bar{e}';O_{C,d,k}(L'))$,
\item[(iv'')] $B_{\R}(\bar{e}';L')\subseteq B_{\R}(\bar{e};O_{C,d,k}(L))$.
\end{itemize}
We claim moreover that
\begin{itemize}
\item[(i'')]  $B_\R(e;L)\subseteq B_\R(e';O_{C,d,k}(L'))$,
\item[(ii'')] $B_\R(e';L') \subseteq B_\R(e;O_{C,d,k}(L))$.
\end{itemize}
Indeed, (iii'') implies that $B_\R(e_1,\ldots,e_{d-1};L)\subseteq B_{\R}(e'_1,\ldots,e'_{d-1};O_{C,d,k}(L'))$ modulo $\R e_d$. Since $|a_{id}|\le1$ for each $i$, this in turn implies that
\begin{align*}
B_\R(e_1,\ldots,e_{d-1};L)&\subseteq B_{\R}(e'_1,\ldots,e'_{d-1};O_{C,d,k}(L'))+B_{\R}(e'_d;O_{C,d,k}(\textstyle\sum_{i=1}^{d-1}L'_i))\\
       &=B_{\R}(e'_1,\ldots,e'_{d-1};O_{C,d,k}(L'))+B_{\R}(e'_d;O_{C,d,k}(\textstyle\sum_{i=1}^{d-1}L_i)),
\end{align*}
and then \eqref{eq:L_d.almost.max} gives (i'') as claimed. The proof of (ii'') from (iv'') is similar.

We now claim that $(e',L')$ is in $O_{Q,d,k}(1)$-rational $O_{C,d,k}(1)$-upper-triangular form, as required by the proposition. The fact that $(\bar e';\bar L')$ is in $O_{Q,d,k}(1)$-rational $O_{C,d,k}(1)$-upper-triangular form implies that
\[
[e'_i,e'_j]=b_{j+1}e'_{j+1}+\cdots+b_{d-1}e'_{d-1}+qe'_d
\]
for some rationals
\[
b_\ell\in\left[-\frac{O_{C,d,k}(L'_\ell)}{L'_iL'_j},\frac{O_{C,d,k}(L'_\ell)}{L'_iL'_j}\right]
\]
with denominators at most $O_{Q,d,k}(1)$, and some $q\in\R$. The fact that $q$ is also rational with denominator at most $O_{Q,d,k}(1)$ follows from the fact that $[e'_i,e'_j]=\sum_{\ell,\ell'}a_{i\ell}a_{j\ell'}[e_\ell,e_{\ell'}]$, the $Q$-rational upper-triangular form of $(e;L)$, and the bounds on the denominators of the $a_{i\ell}$ and the $b_\ell$. The fact that
\[
q\in\left[-\frac{O_{C,d,k}(L'_d)}{L'_iL'_j},\frac{O_{C,d,k}(L'_d)}{L'_iL'_j}\right]
\]
follows because
\begin{align*}
L'_iL'_j[e'_i,e'_j]&\in[B_\R(e,O_{C,d,k}(L)),B_\R(e,O_{C,d,k}(L))]&&\text{(by (ii''))}\\
       &\subseteq B_\R(e,O_{C,d,k}(L))&&\text{(by the upper-triangular form of $(e,L)$)}\\
       &\subseteq B_\R(e',O_{C,d,k}(L'))&&\text{(by (i'')).}
\end{align*}

We now claim that $(u',L')$ is in $O_{C,d,k}(1)$-upper-triangular form over $\Z$. It is in $O_{C,d,k}(1)$-upper-triangular form over $\Q$ by \cref{prop:pp4.1}, so since $e_1',\ldots,e_d'$ is a strong Mal'cev basis it suffices to show that $[u_i',u_j']\in P(u',\infty)$ for each $i,j$. To see this, note first that $(\bar u',\bar L')$ is in upper-triangular form over $\Z$ by definition, which means that for an arbitrary given pair $i,j$ of indices such that $1\le i<j<d$, we have
\[
[u_i',u_j']=(u'_{j+1})^{\ell_{j+1}}\cdots(u'_{d-1})^{\ell_{d-1}}(u'_d)^{a/b}
\]
for some $\ell_i\in\Z$ and some reduced fraction $a/b$. It follows that $(u'_d)^{1/b}\in\Gamma'$. Since $u_d'$ generates $\Gamma'\cap u_d^\R$, this in turn implies that $b=1$.

Finally, \cref{prop:pp4.1} implies that (i'') and (ii'') give (i') and (ii'), as required, and the fact that for all $T=(T_1,\ldots,T_d)\in\N^d$ we have
\[
P_\R(u,T)\cap\Gamma'\supseteq P(u,T),
\]
\[
P_\R(u',T)\cap\Gamma'=P(u',T)
\]
and
\[
P_\R(u^{1/k},T)\cap\Gamma'\subseteq P(u^{1/k},T)
\]
means in particular that (i) and (ii) follow from (i') and (ii').
\end{proof}

In order to apply the above result, we need the lattice $\Gamma'$ to be contained in a lattice of the form $P(u^{1/k},\infty)$. Our next result is an auxiliary lemma showing that once we have $(u^{1/k},L)$ in upper-triangular form, on increasing $k$ we can assume that $P(u^{1/k},\infty)=\exp\la\frac1{k_1}e_1,\ldots,\frac1{k_d}e_d\ra$.
\begin{lemma}\label{lem:divide.basis}
Let $C>0$ and $m\in\N$. Suppose that $G$ is a simply connected nilpotent Lie group with Lie algebra $\g$, that $e_1,\ldots,e_d$ is a basis for $\g$, and that $L_1,\ldots,L_d>0$ are such that $(e;L)$ is in $C$-upper-triangular form over $\Z$. Then, writing $n=n(d)$ for the natural number given by \cref{lem:dilates.subgroup}, setting $k_i=m^{2^{i-1}}n^{2^{i-1}-1}$ for $i=1,\ldots,d$, and writing $u_i=\exp e_i$ for each $i$, the following conditions hold:
\begin{enumerate}[label=(\roman*)]
\item $(\frac1ke;L)$ is in $O_{d,m}(C)$-upper-triangular form over $\Z$;
\item $(u^{1/k},L)$ is in $O_{C,d,m}(1)$-upper-triangular form;
\item $\exp\left\la\frac1{k_1}e_1,\ldots,\frac1{k_d}e_d\right\ra=\la u_1^{1/k_1},\ldots,u_d^{1/k_d}\ra=P(u^{1/k};\infty)$.
\end{enumerate}
\end{lemma}
\begin{proof}
First, note that
\[
\left(\frac1{mn}e_1,\frac1{(mn)^2}e_2,\ldots,\frac1{(mn)^{2^{d'-1}}}e_d;L\right)
\]
is in $O_{m,d}(C)$-upper-triangular form over $\Z$. This implies that $(\frac1ke;L)$ is in $O_{m,d}(C)$-upper-triangular form over $\Z$; since the nilpotency class of $G$ is at most $d$, it also means that we may apply \cref{lem:dilates.subgroup} to conclude that $\exp\left\la\frac1{k_1}e_1,\ldots,\frac1{k_{d'}}e_d\right\ra$ is a group. It then follows from \cref{cor:pp4.1.group} that this group is equal to $\la u_1^{1/k_1},\ldots,u_d^{1/k_d}\ra=P(u^{1/k};\infty)$, and from \cref{prop:pp4.1.orig} that $(\hat u^{1/k},L)$ is in $O_{C,d,m}(1)$-upper-triangular form.
\end{proof}

Finally, following result will allow us, once we have applied \cref{lem:gettingridofX} to a progression coming from \cref{lem:make.integral}, to absorb the error set $X$ into a power of the resulting progression.

\begin{lemma}\label{lem:absorption}
Suppose $P=P(u;L)$ is an infinitely proper progression of dimension $d$ in $C$-upper-triangular form, and that $X\subseteq\langle P\rangle$ is such for all $x\in X$ we have $x^2 \in XP^r$ Then $X\subseteq P(u;O_{C,d,r}(L))$.
\end{lemma}

\begin{proof} Since $P$ is infinitely proper, every element $x$ of $X$ has a unique decomposition as $u_1^{n_1}\ldots u_d^{n_d}$. It is enough to show that $n_j=O_{C,d,r}(L_j)$ for each $j$. We proceed by induction on $d$. We do not technically need to separate out the case $d=1$, but it is instructive to do so. In that case, let $x\in X$ be such that $|n_1|$ is maximal. Then $x^2$ has coordinate $2n_1$, but belongs to $XP^r$, so that $2|n_1|\leq rL_1+|n_1|$, and hence $|n_1|\leq rL_1$, as required. In general, by projecting modulo $\langle u_d\rangle$ and applying the induction hypothesis we may assume that for each $j<d$ the $j$th coordinate of every element of $X$ is in $O_{C,r,d}(L_j)$. This means that if $x\in X$ is such that $|n_d|$ is maximal, then $x\in P(u;{O_{C,d,r}(L)})u_d^{n_d}$. By \cref{lem:upper-tri.doubling.dilate} and the centrality of $u_d$, this in turn implies that $x^2\in P(u;{O_{C,d,r}(L)})u_d^{2n_d}$. In particular, the $d$th coordinate $m_d$ of $x^2$ satisfies $m_d=2n_d+O_{C,d,r}(L_d)$. However, since $x^2\in XP^r$, \cref{lem:upper-tri.doubling.dilate} and the centrality of $u_d$ imply that $m_d=n_d'+O_{C,d,r}(L_d)$ for some $n_d'\in\Z$ with $|n'_d|\le|n_d|$, so that $n_d=O_{C,d,r}(L_d)$ as required.
\end{proof}

We close this section by proving \cref{lem:make.integral}.

\begin{proof}[Proof of \cref{lem:make.integral}]
First, note that $(Qe,L)$ is in upper-triangular form over $\Z$, so that by \cite[Lemma 4.3]{proper.progs} there exists $Q'=Q'(d)\in\N$ such that $\exp\langle QQ'e_1,\ldots,QQ'e_d\rangle$ is a subgroup of $N$. Set $M=QQ'$, noting that $M\le O_d(Q)$ and $(Me;L)$ is in $CM$-upper-triangular form over $\Z$ as required. \cref{prop:pp4.1.orig} then implies that $(u^M;L)$ in $O_{C,d,Q}(1)$-upper-triangular form, and \cref{cor:pp4.1.group} implies that $\exp\la Me_1,\ldots,Me_d\ra=\la u_1^M,\ldots,u_d^M\ra=P(u^M;\infty)$. Set
\[
S=P\cup P^{-1}\cup P(u;M)P(u^M;L)\cup(P(u;M)P(u^M;L))^{-1},
\]
so that $S$ is symmetric and contains both $P$ and $P(u;M)P(u^M;L)$. It follows from \cref{lem:upper-tri.doubling.dilate} and \cite[Lemma 4.5]{proper.progs} that
\[
S^2\subseteq P^{O_{d,Q}(1)}\subseteq P(u;O_{C,d,Q}(L))\subseteq P(u;M)P(u^M;O_{C,d,Q}(L))\subseteq P(u;M)P(u^M;L)^{O_{C,d,Q}(1)}.
\]
Applying \cref{lem:inclusions.local} with $X=P(u;M)$ therefore implies that
\begin{equation}\label{eq:make.integral.prelim}
P^r\subseteq P(u;M)P(u^M;L)^{O_{C,d,Q}(r)}
\end{equation}
for all $r\in\N$.

Now set $Y=P(u;M)\cap P(u^M;\infty)$, noting that $Y^2\subseteq P(u;M)P(u^M;L)^{O_{C,d,Q}(1)}$ by \eqref{eq:make.integral.prelim}, and hence that $Y^2\subseteq YP(u^M;L)^{O_{C,d,Q}(1)}$. \cref{lem:absorption} then implies that $Y\subseteq P(u^M;L)^{O_{C,d,Q}(1)}$. Setting $X=(P(u;M)\setminus Y)\cup\{1\}$, it then follows from \eqref{eq:make.integral.prelim} that
\[
P^r\subseteq XYP(u^M;L)^{O_{C,d,Q}(r)}\subseteq P(u^M;L)^{O_{C,d,Q}(r)}
\]
for all $r\in\N$, giving \eqref{eq:make.integral}.

Finally, \eqref{eq:make.integral.bdd} is trivial, even with $X$ replaced by its superset $P(u;M)$.
\end{proof}

\subsection{Construction of Lie progressions with large injectivity radius}
In our first paper \cite{proper.progs} we adapted an argument from Bilu \cite{bilu} from the abelian case to show that an integral Lie progression with small injectivity radius can be approximated by an integral Lie progression of lower dimension. In this section we extend that result to rational Lie progressions, as follows.

\begin{prop}\label{prop:properprog}
Let $R\in\N$. Suppose $P_0$ is a $Q$-rational Lie progression in $C$-upper-triangular form with dimension $d$, injectivity radius less than $R$ and symmetry group $H_0$, projected from the nilpotent Lie group $N_0$ with lattice $\Gamma_0$ by $\pi_0:\Gamma_0\to\la P_0\ra/H_0$. Then there exists a normal subgroup $H_1\normal\la P_0\ra$ containing $H_0$ and satisfying
\begin{equation}\label{eq:pi(Gamma_0)small.diam}
H_1\subseteq P_0^{O_{C,d,Q,R}(1)},
\end{equation}
and an $O_{d,Q}(1)$-rational Lie progression $P_1$ in $O_{C,d,Q}(1)$-upper-triangular form with dimension strictly less than $d$, injectivity radius at least $R$ and symmetry group $H_1$ such that, writing $N_1$ for the nilpotent Lie group from which $P_1$ is projected, $\Gamma_1\le N_1$ for its lattice, and $\pi_1:\Gamma_1\to\la P_1\ra/H_1$ for its projector, there exists a surjective Lie group homomorphism $\beta:N_0\to N_1$ such that $\beta(\Gamma_0)=\Gamma_1$, such that
\begin{equation}\label{eq:properprog.containment}
\beta(\tilde P_0)\subseteq\tilde P_1\subseteq\beta(\tilde P_0)^{O_{C,d,Q,R}(1)},
\end{equation}
such that the diagram
\begin{equation}\label{eq:quot.prog.diagram-bilu.revisited}
\begin{CD}
 \Gamma_0                      @>\pi_0>>           \langle P_0\rangle/H_0\\
@V\beta VV               @VVV\\
\Gamma_1     @>\pi_1>>    \langle P_0\rangle/H_1
\end{CD}
\end{equation}
commutes, and such that $H_1$ is the pullback to $\la P_0\ra$ of $\pi_0(\Gamma_0\cap\ker\beta)$. 
\end{prop}
Note that \eqref{eq:pi(Gamma_0)small.diam}, \eqref{eq:properprog.containment} and the commutativity of \eqref{eq:quot.prog.diagram-bilu.revisited} imply that $P_0\subseteq P_1\subseteq P_0^{O_{C,d,Q,R}(1)}$.

Our earlier result for integral progressions was similar to the following (we modify various details here to reflect what we need in the present paper).

\begin{prop}\label{prop:reduce.dim.when.not.proper}
Let $R\in\N$. Suppose $P_0$ is an integral Lie progression in $C$-upper-triangular form with dimension $d$, injectivity radius less than $R$ and symmetry group $H_0$, projected from the nilpotent Lie group $N_0$ with lattice $\Gamma_0$ by $\pi_0:\Gamma_0\to\la P_0\ra/H_0$. Then there exists a normal subgroup $H_1\normal\la P_0\ra$ containing $H_0$ and satisfying
\[
H_1\subseteq P_0^{O_{C,d,R}(1)},
\]
and an integral Lie progression $P_1$ in $O_{d}(1)$-upper-triangular form with dimension strictly less than $d$, injectivity radius at least $R$ and symmetry group $H_1$ such that, writing $N_1$ for the nilpotent Lie group from which $P_1$ is projected, writing $\Gamma_1\le N_1$ for its lattice, and writing $\pi_1:\Gamma_1\to\la P_1\ra/H_1$ for its projector, there exists a surjective Lie group homomorphism $\beta:N_0\to N_1$ such that $\beta(\Gamma_0)=\Gamma_1$, such that
\[
\beta(\tilde P_0)\subseteq\tilde P_1\subseteq\beta(\tilde P_0)^{O_{C,d,R}(1)},
\]
such that the diagram
\begin{equation}\label{eq:quot.prog.diagram}
\begin{CD}
 \Gamma_0                      @>\pi_0>>           \langle P_0\rangle/H_0\\
@V\beta VV               @VVV\\
\Gamma_1     @>\pi_1>>    \langle P_0\rangle/H_1
\end{CD}
\end{equation}
commutes, and such that $H_1$ is the pullback to $\la P_0\ra$ of $\pi_0(\Gamma_0\cap\ker\beta)$. There also exist elements $z_1,\ldots,z_m\in\ker\beta\cap\tilde P_0$ such that $[z_i,z_j]\in\Span_\R(z_{j+1}\ldots,z_m)$ whenever $i<j$, and such that $\ker\beta=P_\R(z,\infty)$.
\end{prop}

\begin{proof}
This is very close to the statement of \cite[Proposition 7.3]{proper.progs} (to which we refer hereafter in this proof as `the reference'), but it differs in several details. The first of these is just terminological: in the reference, the Lie progressions are not stated to be integral, but that is just because Lie progressions in that paper are all integral by definition. We now describe and address the remaining differences.
\begin{itemize}
\item Instead of injectivity radius, the reference uses a closely related notion called \emph{properness}: a Lie progression $P$ with generators $u_1,\ldots,u_d$, lengths $L_1,\ldots,L_d$ and projector $\pi$ is called \emph{$m$-proper} for $m\in\N$ if the elements $\pi(u_1^{\ell_1}\cdots u_d^{\ell_d})$ are all distinct as the $\ell_i$ range over those integers with $|\ell_i|\le mL_i$. However, it easily follows from \eqref{eq:dilateProg} and \cref{lem:upper-tri.doubling.dilate} that the truth of the result is unaffected by interchanging these two notions.
\item In the reference the symmetry group $H_0$ is assumed to be trivial. This is easily adressed by applying the result there in the quotient $\langle P_0\rangle/H_0$.
\item In the reference there is no mention of the homomorphism $\beta$. In particular, it is not therefore stated that $\beta(\Gamma_0)=\Gamma_1$, that $\beta(\tilde P_0)\subseteq\tilde P_1\subseteq\beta(\tilde P_0)^{O_{C,d,R}(1)}$, or that the diagram \eqref{eq:quot.prog.diagram} commutes. The homomorphism $\beta$ is in fact the homomorphism $\Phi$ appearing in the proof of \cite[Proposition 7.6]{proper.progs}, and \eqref{eq:quot.prog.diagram} is part of the commutative diagram \cite[(7.5)]{proper.progs}. The fact that $\beta(\Gamma_0)=\Gamma_1$ follows from the fact that $\Phi_j(\Gamma_j)=\Gamma_{j+1}$ in the proof of \cite[Proposition 7.6]{proper.progs} (or from $\beta(\tilde P_0)\subseteq\tilde P_1\subseteq\beta(\tilde P_0)^{O_{C,d,R}(1)}$, which we are about to prove). The fact that $\beta(\tilde P_0)\subseteq\tilde P_1\subseteq\beta(\tilde P_0)^{O_{C,d,R}(1)}$ follows from the proof of \cite[(7.1)]{proper.progs} but with \cite[(7.7)]{proper.progs} replaced by the expression $\Phi(\exp B_\Z(e;k_{C,d}rL))\subseteq\exp B_\Z(e';rL')\subseteq\Phi(\exp B_\Z(e;O_{C,d,m}(rL)))$; this in turn follows from the penultimate displayed equation in the proof of \cite[Proposition 7.6]{proper.progs} and the commutativity of \cite[(7.5)]{proper.progs}.
\item In the reference, $H_1$ is not stated to be the pullback to $\langle P_0\rangle$ of $\pi_0(\Gamma_0\cap\ker\beta)$, but this follows from the definition of $H_{j+1}$ in the proof of \cite[Proposition 7.6]{proper.progs} and the commutativity of \cite[(7.5)]{proper.progs}.
\item In the reference, there is no explicit mention of the elements $z_i$ in the statement of the proposition. However, they do appear explicitly in its proof, specifically in the proof of \cite[Proposition 7.6]{proper.progs}.\qedhere
\end{itemize}
\end{proof}

\begin{proof}[Proof of \cref{prop:properprog}]
Write $u_1,\ldots,u_d\in\Gamma_0$ for the generators, $e_1,\ldots,e_d$ for the basis, and $L_1,\ldots,L_d>0$ for the lengths of $P_0$, so that $\tilde P_0=P(u;L)$ and $e_i=\log u_i$. Let $M\ll_dQ$ be as given by \cref{lem:make.integral}, let $X_0\subseteq P(u;M)$ be the subset given by the same lemma, and write $\Gamma_0'=\la u_1^M,\ldots,u_d^M\ra$. Thus, $\tilde Q_0=P(u^M;L)$ is an integral raw Lie progression in $O_{C,d,Q}(1)$-upper-triangular form and $\la\tilde Q_0\ra=\Gamma_0'=\exp\la Me_1,\ldots,Me_d\ra=P(u^M,\infty)$. Furthermore, $X_0\cap\Gamma_0'=\{1\}$, and
\begin{equation}\label{eq:X.absorbable}
\tilde P_0^r\subseteq X_0\tilde Q_0^{O_{C,d,Q}(r)}
\end{equation}
and
\begin{equation}\label{eq:make.integral.bdd}
X_0\tilde Q_0^r\subseteq\tilde P_0^{O_{d,Q}(r)}
\end{equation}
for all $r\in\N$. Note that $\Gamma_0=X_0\Gamma_0'$ by \eqref{eq:X.absorbable}.

Let $Q_0$ be the pullback to $\la P_0\ra$ of $\pi_0(\tilde Q_0)$, so that $Q_0$ is an integral Lie progression in $O_{C,d,Q}(1)$-upper-triangular form with symmetry group $H_0$. We claim that $\inj Q_0\ll_{C,d,Q}R$. Indeed, since $\inj P_0<R$, there exist $\ell_1,\ldots,\ell_d\in\Z$, not all zero, such that $u_1^{\ell_1}\cdots u_d^{\ell_d}\in\ker\pi\cap\tilde P_0^R$. \cref{lem:bg.rational.power} implies that there exist rational polynomials $p_1,\ldots,p_t$ (with $t\ll_d1$), the denominators of whose coefficients are all at most $O_d(1)$, such that for all $\eta\in\R$ we have $(u_1^{\ell_1}\cdots u_d^{\ell_d})^\eta=u_{i_1}^{p_1(\eta)\ell_{i_1}}\cdots u_{i_t}^{p_t(\eta)\ell_{i_t}}$. In particular, there exists a positive integer $\eta\ll_dM$ such that $(u_1^{\ell_1}\cdots u_d^{\ell_d})^\eta\in\Gamma_0'$. However, $(u_1^{\ell_1}\cdots u_d^{\ell_d})^\eta$ also belongs to $\ker\pi\cap\tilde P_0^{O_d(RM)}$, and hence, by \eqref{eq:X.absorbable}, to $\ker\pi\cap X_0\tilde Q_0^{O_{C,d,Q}(R)}$. Since $X_0\cap\Gamma_0'=\{1\}$, we deduce that $(u_1^{\ell_1}\cdots u_d^{\ell_d})^\eta\in\ker\pi\cap\tilde Q_0^{O_{C,d,Q}(R)}$, so that $\inj Q_0\ll_{C,d,Q}R$ as claimed.

Letting $\alpha=\alpha(C,d,Q)>0$ be a quantity to be determined later but depending on $C$, $d$ and $Q$ only, \cref{prop:reduce.dim.when.not.proper} therefore implies that there exists an integral Lie progression $Q_1$  in $O_d(1)$-upper-triangular form with dimension $d'<d$, injectivity radius at least $\alpha R$, and symmetry group $H_1'\subseteq Q_0^{O_{C,d,Q,R}(1)}$ containing $H_0$ such that, writing $N_1$ for the nilpotent Lie group from which $Q_1$ is projected, $\Gamma_1'$ for its lattice, and $\pi'_1:\Gamma_1'\to\la Q_0\ra/H'_1$ for its projector, there exits a surjective Lie group homomorphism $\beta:N_0\to N_1$ such that $\beta(\Gamma_0')=\Gamma_1'$, such that the diagram
\[
\begin{CD}
 \Gamma_0'                      @>\pi_0>>           \langle Q_0\rangle/H_0\\
@V\beta VV               @VVV\\
\Gamma_1'     @>\pi'_1>>    \langle Q_0\rangle/H_1'
\end{CD}
\]
commutes, such that $H_1'$ is the pullback to $\la P_0\ra$ of $\pi_0(\Gamma_0'\cap\ker\beta)$, and such that
\[
\beta(\tilde Q_0)\subseteq\tilde Q_1\subseteq\beta(\tilde Q_0^{O_{C,d,Q,R}(1)}).
\]
\cref{prop:reduce.dim.when.not.proper} also implies that there exist elements $z_1,\ldots,z_m\in\ker\beta\cap\tilde Q_0$ such that
\begin{equation}\label{eq:z_j.central}
[z_i,z_j]\in\Span_\R(z_{j+1}\ldots,z_m)
\end{equation}
whenever $i<j$, and such that $\ker\beta=P_\R(z,\infty)$.

Set $\Gamma_1=\beta(\Gamma_0)$, and let $H_1$ be the pullback to $\la P_0\ra$ of $\pi_0(\Gamma_0\cap\ker\beta)$, noting that $H_1'\le H_1\normal\la P_0\ra$. Write $\ph:\la P_0\ra/H_0\to\la P_0\ra/H_1$ for the quotient homomorphism, and define a homomorphism $\pi_1:\Gamma_1\to\la P_0\ra/H_1$ via $\pi_1\circ\beta=\ph\circ\pi_0$, so that the diagram \eqref{eq:quot.prog.diagram-bilu.revisited} becomes
\[
\begin{CD}
 \Gamma_0                      @>\pi_0>>           \langle P_0\rangle/H_0\\
@V\beta VV               @VV\ph V\\
\Gamma_1     @>\pi_1>>    \langle P_0\rangle/H_1
\end{CD}
\]
and commutes as required. Note that $H_1\cap\la Q_0\ra=H_1'$, and hence $\ker\pi_1\cap\Gamma'=\ker\pi_1'\cap\Gamma'$.

To prove \eqref{eq:pi(Gamma_0)small.diam}, let $g\in\Gamma_0\cap\ker\beta$, and note that $g=xp$ for some elements $x\in X_0$ and $p\in \Gamma_0'$. If there is some other element $q\in \Gamma_0'$ such that $xq\in\ker\beta$, then $q^{-1}p\in\Gamma_0'\cap\ker\beta$, so that $\pi_0(q^{-1}p)\in H_1'/H_0\subseteq Q_0^{O_{C,d,Q,R}(1)}/H_0\subseteq P_0^{O_{C,d,Q,R}(1)}/H_0$. It therefore suffices to show that there exists $q\in \Gamma_0'$ such that $xq\in\ker\beta$ and $\pi_0(xq)\in P_0^{O_{C,d,Q,R}(1)}/H_0$. In fact, we will prove a more precise statement, namely that there exists $q\in \Gamma_0'$ such that $xq\in\ker\beta\cap\tilde P_0^{O_{C,d,Q}(1)}$. First, note that by definition of the $z_i$ we have $xp=z_1^{r_1}\cdots z_m^{r_m}$ for some $r_i\in\R$. Applying \eqref{eq:z_j.central} repeatedly, we can therefore choose integers $\ell_1,\ell_2\ldots,\ell_m$ in turn such that $xpz_1^{\ell_1}\cdots z_m^{\ell_m}=z_1^{c_1}\cdots z_m^{c_m}$ for some $c_i\in[0,1)$. Set $q=pz_1^{\ell_1}\cdots z_m^{\ell_m}$, noting that $q\in\Gamma_0'$ and $xq\in\ker\beta$ as required. Since each $z_i\in P(u;ML)$, \cref{lem:bg.rational.power} implies that $xq\in P_\R(u,O_d(ML))^{O_d(1)}$, and so \cref{lem:pp.L2.1} implies that $xq\in P_\R(u;O_{C,d,Q}(L))$. Since $e$ is a strong Mal'cev basis and $xq\in\Gamma_0=P(u;\infty)$, this in turn implies that $xq\in P(u;O_{C,d,Q}(L))\subseteq\tilde P_0^{O_{C,d,Q}(1)}$ as required, giving \eqref{eq:pi(Gamma_0)small.diam} as claimed.

Write $\hat u_1,\ldots,\hat u_{d'}$ for the generators of $Q_1$, write $\hat e_1,\ldots,\hat e_{d'}$ for its basis (so that $\hat e_i=\log\hat u_i$), and write $\hat L_1,\ldots,\hat L_{d'}$ for its lengths. Since $\Gamma_1'$ has index at most $|X_0|=O_{d,Q}(1)$ in $\Gamma_1$, there exists
\begin{equation}\label{eq:Gamma1.m}
m\ll_{d,Q}1
\end{equation}
such that $g^m\in\Gamma_1'$ for all $g\in\Gamma_1$. In particular, this implies that $m\log g\in\la\hat e_1,\ldots,\hat e_{d'}\ra$ for all $g\in\Gamma_1$, and hence that $\Gamma_1\subseteq\exp\left\la\frac1m\hat e_1,\ldots,\frac1m\hat e_{d'}\right\ra$. By \cref{lem:divide.basis}, there exist natural numbers $k_1,\ldots,k_{d'}\ll_{d,Q}1$, each a multiple of $m$, such that $\exp\left\la\frac1{k_1}\hat e_1,\ldots,\frac1{k_{d'}}\hat e_{d'}\right\ra=\left\la\hat u_1^{1/k_1},\ldots,\hat u_{d'}^{1/k_{d'}}\right\ra=P(\hat u^{1/k},\infty)$, and such that $(\hat u^{1/k},\hat L)$ is in $O_{d,Q}(1)$-upper-triangular form.
We may therefore apply \cref{lem:gettingridofX} to conclude that there exists an $O_{d,Q}(1)$-rational raw Lie progression $\tilde P_1=P(u^{(1)},L^{(1)})$ in $O_{C,d,Q}(1)$-upper-triangular form such that $P(u^{(1)},\infty)=\Gamma_1$, and such that
\begin{equation}\label{eq:getridofX.i}
\tilde Q_1\subseteq P(u^{(1)},O_{C,d,Q}(L^{(1)}))
\end{equation}
and
\begin{equation}\label{eq:getridofX.ii}
\tilde P_1 \subseteq P(\hat u^{1/k},O_{C,d,Q}(\hat L)).
\end{equation}

Set $X_1=\beta(X_0)$, and let $P_1$ be the pullback to $\la P_0\ra$ of $\pi_1(\tilde P_1)$. It follows from \eqref{eq:X.absorbable} that $X_0^2\subseteq X_0\tilde Q_0^{O_{C,d,Q}(1)}$, so that $X_1^2\subseteq X_1\tilde Q_1^{O_{C,d,Q}(1)}$, and hence $X_1^2\subseteq X_1\tilde P_1^{O_{C,d,Q}(1)}$ by \eqref{eq:getridofX.i}. Since $X_1\subseteq\Gamma_1=\la\tilde P_1\ra$ by definition, \cref{lem:absorption} therefore implies that
\begin{equation}\label{eq:X1.absorbed}
X_1\subseteq\tilde P_1^{O_{C,d,Q}(1)}.
\end{equation}
Combined with \eqref{eq:X.absorbable}, this implies that
\[
\beta(\tilde P_0)\subseteq\beta(X_0\tilde Q_0^{O_{C,d,Q}(1)}))\subseteq X_1\tilde Q_1^{O_{C,d,Q}(1)}\subseteq\tilde P_1^{O_{C,d,Q}(1)}\tilde Q_1^{O_{C,d,Q}(1)},
\]
and then another application of \eqref{eq:getridofX.i} implies that $\beta(\tilde P_0)\subseteq\tilde P_1^{O_{C,d,Q}(1)}$. \cref{lem:pp.L2.1} then implies that $\beta(\tilde P_0)\subseteq P_\Q(u^{(1)},O_{C,d,Q}(L^{(1)}))$, and since $\beta(\tilde P_0)\subseteq\Gamma_1=\la\tilde P_1\ra$ and $e^{(1)}_1,\ldots,e^{(1)}_{d'}$ is a strong Mal'cev basis this in fact implies that $\beta(\tilde P_0)\subseteq P(u^{(1)},O_{C,d,Q}(L^{(1)}))$. Upon increasing the lengths $L^{(1)}_i$ by factors depending only on $C,d,Q$, we may therefore conclude that $\beta(\tilde P_0)\subseteq\tilde P_1$, which is the first inclusion of \eqref{eq:properprog.containment} (note that by \cref{lem:upper-tri.doubling.dilate} this does not affect the truth of \eqref{eq:getridofX.ii}).

To prove the second inclusion of \eqref{eq:properprog.containment}, first note that
\begin{equation}\label{eq:P1.in.X1.P}
\tilde P_1\subseteq X_1\Gamma_1'=X_1P(\hat u;\infty).
\end{equation}
Next, note that $\tilde P_1\subseteq P_\Q(\hat u;O_{C,d,Q}(\hat L))$ by \eqref{eq:getridofX.ii}, and that $X_1\subseteq P_\Q(\hat u;O_{C,d,Q}(\hat L))$ by \eqref{eq:X1.absorbed}, \eqref{eq:getridofX.ii} and \cref{lem:upper-tri.doubling.dilate}. Since $e^{(1)}_1,\ldots,e^{(1)}_{d'}$ is a strong Mal'cev basis, \eqref{eq:P1.in.X1.P} and \cref{lem:pp.L2.1} therefore combine to imply that
\[
\tilde P_1\subseteq X_1P(\hat u;O_{C,d,Q}(\hat L))\subseteq X_1\tilde Q_1^{O_{C,d,Q}(1)}\subseteq\beta(X_0\tilde Q_0^{O_{C,d,Q,R}(1)}),
\]
which by \eqref{eq:make.integral.bdd} completes the proof of \eqref{eq:properprog.containment}.

It remains to show that $\inj P_1\ge R$. To see this, note that if $p\in\tilde P_1^r\cap\ker\pi_1\setminus\{1\}$ for some $r\in\N$ then by \eqref{eq:Gamma1.m} and definition of $m$ we have $p^m\in\tilde P_1^{O_{d,Q}(r)}\cap\Gamma_1'\cap\ker\pi_1=\tilde P_1^{O_{d,Q}(r)}\cap\Gamma_1'\cap\ker\pi'_1$. By \eqref{eq:getridofX.ii} and \cref{lem:upper-tri.doubling.dilate}, this implies that $p^m\in P(\hat u^{1/k},\hat L)^{O_{C,d,Q}(r)}\cap\Gamma_1'\cap\ker\pi'_1$, and hence by \cref{prop:goodReducReal1} that
\[
p^m\in P_\R(\hat u,\hat L)^{O_{C,d,Q}(r)}\cap\Gamma_1'\cap\ker\pi'_1\subseteq P(\hat u,\hat L)^{O_{C,d,Q}(r)}\cap\ker\pi'_1=\tilde Q_1^{O_{C,d,Q}(r)}\cap\ker\pi'_1.
\]
Setting $\alpha$ to be the implied constant in the final term of this last expression, we deduce that $r>R$, which proves that $\inj P_1\ge R$ as required.
\end{proof}

\subsection{Powers of Lie progressions}\label{section:powersLieProg}
A crucial part of the proof of \cref{prop:chain.of.progs-index} is to show that, for each $i$, the power $P_i^{\inj P_i}$ of the progression $P_i$ can be approximated by a Lie progression of lower dimension -- this lower-dimensional progression then becomes $P_{i+1}$. However, even the fact that $P_i^{\inj P_i}$ can be approximated by a progression of the \emph{same} dimension as itself is far from trivial, and turns out to be an important stepping stone on the way to producing $P_{i+1}$. We achieve this via the following result, which also plays an important role in the proof of results such as \cref{thm:finitary.poly-vol.detailed}.
\begin{prop}\label{prop:powergood.rational}
Suppose $P_0=P(u^{(0)},L^{(0)})$ is a $Q$-rational raw Lie progression of dimension $d$ in $C$-upper-triangular form. Let $m\in\N$. Then there exists an $O_{d,Q}(1)$-rational raw Lie progression $P=P(u,L)$ in $O_{C,d,Q}(1)$-upper-triangular form in the same Lie group as $P_0$ such that
\begin{equation}\label{eq:powergood.rational}
P_0^m\subseteq P\subseteq P_0^{O_{C,d,Q}(m)}.
\end{equation}
\end{prop}

It is convenient to prove \cref{prop:powergood.rational} for integral Lie progressions first and then generalise to arbitrary Lie progressions using the material of \cref{section:integral/rational}.

\begin{prop}\label{prop:Powergood}
Suppose $P(u;L)$ is an integral raw Lie progression with dimension $d$ in $C$-upper-triangular form. Let $m\in\N$. Then there exists an integral raw Lie progression $P(u',L')$ in $O_d(1)$-upper-triangular form such that $P(u;L)^m\subseteq P(u',L')\subseteq P(u;L)^{O_{C,d}(m)}$.
\end{prop}

In proving \cref{prop:Powergood} we use the following slight generalisation of a result from our first paper. We say a convex body $B$ in a real vector space is \emph{strictly thick} with respect to a lattice $\Lambda$ if there exists some $\lambda<1$ such that $\lambda B\cap\Lambda$ generates $\Lambda$.

\begin{prop}\label{prop:pp.5.1}
Let $\alpha\in\N$ and $d\in\N$. Suppose $\n$ is a nilpotent Lie algebra of dimension $d$ and $\Lambda$ is a lattice in $\n$ satisfying $[\Lambda,\Lambda]\subseteq\Lambda$. Suppose further that $B$ is a symmetric convex body in $\n$ such that $[B,B]\subset\alpha B$ and $\alpha^{-1}B$ is strictly thick with respect to $\Lambda$. Then there exists a basis $e_1,\ldots,e_d$ of $\Lambda$ and $L_1,\ldots,L_d>0$ such that
\[
B\subseteq B_\R(e;L)\subseteq O_d(1)B,
\]
and such that $(e;L)$ is in $\alpha$-upper-triangular form. If $\alpha\in\N$ then we may also insist that $L_i\in\N$ for each $i$.
\end{prop}
\begin{proof}
The case $\alpha=1$ is \cite[Proposition 5.1]{proper.progs} from our first paper, and the general statement follows from applying this to the strictly thick convex body $\alpha^{-1}B$. Indeed, the fact that $[B,B]\subseteq\alpha B$ implies that $[\alpha^{-1}B,\alpha^{-1}B]=\alpha^{-2}[B,B]\subseteq\alpha^{-1}B$, so \cite[Proposition 5.1]{proper.progs} implies that there exists a basis $e_1,\ldots,e_d$ of $\Lambda$ and integers $L_1,\ldots,L_d$ such that $\alpha^{-1}B\subseteq B_\R(e;L)\subseteq O_d(\alpha^{-1})B$ and such that $(e;L)$ is in $1$-upper-triangular form. This easily implies that $B\subseteq B_\R(e;\alpha L)\subseteq O_d(1)B$ and that $(e;\alpha L)$ is in $\alpha$-upper-triangular form.
\end{proof}
\begin{remark}
By adapting the proof of \cite[Proposition 5.1]{proper.progs} it is not hard to obtain $L_i\in\N$ in \cref{prop:pp.5.1} even without the assumption that $\alpha\in\N$.
\end{remark}

\begin{proof}[Proof of Proposition \ref{prop:Powergood}]
Write $e_i=\log u_i$ for each $i$, and extend $e$ to a list $\overline e$ of basic Lie brackets as in \cref{sec:basic.comms}. It follows from \cref{lem:real.B.upper-tri} and \cref{lem:freeequivalences} \ref{item:freeequiv.i} that the strictly thick convex body $B=B_\R(\overline e;(mL)^\chi)$ satisfies
\[
P_{\R}(u;L)^m\approx_d \exp(B)
\]
and $[B,B]\subseteq\alpha B$ for some $\alpha=\alpha_d\in\N$. If $m\le\alpha$ then \cref{lem:upper-tri.doubling.dilate} implies that the proposition is satisfied by taking $u'_i=u_i$ and each $L_i'$ some bounded multiple of $L_i$, so we may assume that $m>\alpha$. This implies that $\alpha^{-1}B$ is strictly thick, so by \cref{prop:pp.5.1} there exists a basis $e_1',\ldots,e_d'$ of the lattice $\la e_1,\ldots,e_d\ra$ and integers $L_1'',\ldots,L_d''$ such that $(e';L'')$ is in $\alpha$-upper-triangular form and such that
\begin{equation}\label{eq:B.approx.box}
B\subseteq B_\R(e';L'')\subseteq kB
\end{equation}
for some integer $k=k_d$. Since $B$ is a convex body, we have $kB=\{kb:b\in B\}$, and hence $\exp(kB)=\exp\{kb:b\in B\}\subseteq(\exp B)^k$, and so \eqref{eq:B.approx.box} implies that
$\exp(B)\approx_d\exp(B_{\R}(e';L''))$. On the other hand, Lemma \ref{lem:freeequivalences} \ref{item:freeequiv.3}   implies that writing $u_i'=\exp e_i'$ for each $i$ we have
$\exp(B_{\R}(e';L''))\approx_d P_{\R}(u',L'')$, and hence
\begin{equation}\label{eq:Powergood.real}
P_{\R}(u;L)^m\approx_d P_{\R}(u',L''),
\end{equation}
whilst \cref{prop:pp4.1.orig} implies that $(u';L'')$ is in $O_d(1)$-upper-triangular form.

Note that $\la u_1,\ldots,u_d\ra=\la u_1',\ldots,u_d'\ra$ by \cref{cor:pp4.1.group}. Writing $\Gamma$ for this group, \cref{prop:goodReducReal1} implies that
\[
P(u;L)^m\subseteq P_{\R}(u;L)^m\cap\Gamma\subseteq P_{\R}(u',L'')^{O_d(1)}\cap\Gamma\subseteq P(u',L'')^{O_d(1)}.
\]
\cref{lem:upper-tri.doubling.dilate} then implies that there exist $L'_i\ll_d L''_i$ for each $i$ such that
\[
P(u;L)^m\subseteq P(u',L').
\]
On the other hand, \cref{prop:goodReducReal1} and \eqref{eq:Powergood.real} imply that
\[
P(u',L')\subseteq P(u',L'')^{O_d(1)}\subseteq P_\R(u',L'')^{O_d(1)}\cap\Gamma\subseteq P_{\R}(u;L)^{O_d(m)}\cap\Gamma\subseteq P(u;L)^{O_{C,d}(m)}.
\]
Finally, both $(e',L')$ and $(u',L')$ are in $O_d(1)$-upper-triangular form because $(e',L'')$ and $(u',L'')$ are.
\end{proof}

\begin{proof}[Proof of \cref{prop:powergood.rational}]
Write $e_i^{(0)}=\log u_i^{(0)}$ for each $i$. Let $M\ll_dQ$ and $X\subseteq P(u^{(0)};M)$ be as given by \cref{lem:make.integral}, so that $(Me^{(0)};L^{(0)})$ is in $CM$-upper-triangular form over $\Z$, so that $((u^{(0)})^M;L)$ is in $O_{C,d,Q}(1)$-upper-triangular form, so that \[\exp\la Me_1^{(0)},\ldots,Me_d^{(0)}\ra=\la(u_1^{(0)})^M,\ldots,(u_d^{(0)})^M\ra=P((u^{(0)})^M;\infty),\] and so that
\begin{equation}\label{eq:translate.collisions.rational.fin-ind}
P_0^r\subseteq XP((u^{(0)})^M;L^{(0)})^{O_{C,d,Q}(r)}
\end{equation}
and
\begin{equation}\label{eq:XP0^M^r.in.P0^r}
XP((u^{(0)})^M;L^{(0)})^r\subseteq P_0^{O_{d,Q}(r)}
\end{equation}
for every $r\in\N$.

Let $n=n(C,d,Q)\in\N$ be an integer to be specified shortly. By \cref{prop:Powergood}, there exist a basis $e_1^{(1)},\ldots,e_d^{(1)}$ for $\la Me_1^{(0)},\ldots,Me_d^{(0)}\ra$ and $L^{(1)}_1,\ldots,L^{(1)}_d\in\N$ such that $(e^{(1)},L^{(1)})$ is in $O_d(1)$-upper triangular form over $\Z$ and such that, writing $u_i^{(1)}=\exp e_i^{(1)}$ for each $i$, we have
\begin{equation}\label{eq:collisions.rational.powergood}
P((u^{(0)})^M;L^{(0)})^{nm}\subseteq P(u^{(1)},L^{(1)})\subseteq P((u^{(0)})^M;L^{(0)})^{O_d(nm)}.
\end{equation}

By \cref{lem:divide.basis}, there exist natural numbers $a_1,\ldots,a_d\ll_{d,Q}1$, each one an integer multiple of $M$, such that $(\frac1{a_1} e^{(1)}_1,\ldots,\frac1{a_d} e^{(1)}_d;L^{(1)})$ is in $O_{d,Q}(1)$-upper-triangular form over $\Z$, such that $((u^{(1)})^{1/a},L^{(1)})$ is in $O_{d,Q}(1)$-upper-triangular form, and such that $\exp\left\la\frac1{a_1} e^{(1)}_1,\ldots,\frac1{a_d} e^{(1)}_d\right\ra=P((u^{(1)})^{1/a},\infty)$. We may therefore apply \cref{lem:gettingridofX} to conclude that there exists an $O_{d,Q}(1)$-rational raw Lie progression $P(u,L^{(2)})$ in $O_{C,d,Q}(1)$-upper-triangular form such that $P(u;\infty)=\la P_0\ra$, and such that
\begin{equation}\label{eq:P1.in.P2}
P(u^{(1)},L^{(1)})\subseteq P(u,O_{C,d,Q}(L^{(2)})),
\end{equation}
and
\begin{equation}\label{eq:P2.in.P1/a}
P(u,L^{(2)})\subseteq P((u^{(1)})^{1/a},O_{C,d,Q}(L^{(1)})).
\end{equation}

By \eqref{eq:translate.collisions.rational.fin-ind} and \eqref{eq:XP0^M^r.in.P0^r} we have $X^2\subseteq P_0^{O_{d,Q}(1)}\subseteq XP((u^{(0)})^M,L^{(0)})^{O_{C,d,Q}(1)}$. By choosing $n$ large enough in terms of $C$, $d$ and $Q$, we may therefore ensure that $X^2\subseteq XP(u^{(1)},L^{(1)})\subseteq XP(u,O_{C,d,Q}(L^{(2)}))\subseteq XP(u,L^{(2)})^{O_{C,d,Q}(1)}$. \cref{lem:absorption} then implies that $X\subseteq P(u,O_{C,d,Q}(L^{(2)}))$, which combines with \eqref{eq:P1.in.P2} and \cref{lem:upper-tri.doubling.dilate} to show that
\begin{equation}\label{eq:collisions.rational.1->2}
XP(u^{(1)},L^{(1)})\subseteq P(u;L)
\end{equation}
for some $L\le O_{C,d,Q}(L^{(2)})$.

Provided $n$ is chosen large enough in terms of $C$, $d$ and $Q$, \eqref{eq:translate.collisions.rational.fin-ind}, \eqref{eq:collisions.rational.powergood} and \eqref{eq:collisions.rational.1->2} imply that $P_0^m\subseteq P(u;L)$, giving the first required inclusion. Moreover \eqref{eq:collisions.rational.powergood} implies in particular that $(u_i^{(0)})^M\in P(u^{(1)},L^{(1)})$ for each $i$, so that for a given $i$ there exist $\ell_1,\ldots\ell_d\in\N$ with $|\ell_j|\le L_j^{(1)}$ such that $(u_i^{(0)})^M=(u^{(1)}_1)^{\ell_1}\cdots(u^{(1)}_d)^{\ell_d}$. \cref{lem:bg.rational.power} therefore implies that $u_i^{(0)}=(u^{(1)}_{j_1})^{\ell_{j_1}p_1(1/M)}\cdots(u^{(1)}_{j_t})^{\ell_{j_t}p_t(1/M)}$ for some indices $j_1,\ldots,j_t$ and rational polynomials $p_1,\ldots,p_t$ depending only on $d$. In particular, this implies that $u_i^{(0)}\in P_\Q(u^{(1)};L^{(1)})^{O_{d}(1)}$, and then \cref{lem:pp.L2.1} implies that $X\subseteq P_\Q(u^{(1)};O_{C,d,Q}(L^{(1)}))$. Moreover, \eqref{eq:P2.in.P1/a} and \cref{lem:pp.L2.1} imply that $P(u;L)\subseteq P((u^{(1)})^{1/a};O_{C,d,Q}(L^{(1)}))\subseteq P_\Q(u^{(1)};O_{C,d,Q}(L^{(1)}))$. Since $P(u;L)\subseteq\Gamma=XP(u^{(1)};\infty)$, and since $e^{(1)}_1,\ldots,e^{(1)}_d$ is a strong Mal'cev basis, by \cref{lem:pp.L2.1} it follows that $P(u,L)\subseteq XP(u^{(1)};O_{C,d,Q}(L^{(1)}))$. By \eqref{eq:collisions.rational.powergood} and \eqref{eq:XP0^M^r.in.P0^r}, this in turn implies that
\[
P(u;L)\subseteq P_0^{O_{C,d,Q}(nm)}.
\]
Since $n$ depends only on $C,d,Q$, this gives the second required inclusion.
\end{proof}

\subsection{Fully finitary Gromov-type theorems}

In this section we prove \cref{thm:finitary.poly-vol.detailed}. Many of the bounds in that theorem are inherited from \cref{thm:bgt.gromov}; since we will gradually improve the bounds of \cref{thm:bgt.gromov} as we proceed through the rest of the paper, it will be useful to record exactly how the bounds of \cref{thm:finitary.poly-vol.detailed} depend on the bounds we have in \cref{thm:bgt.gromov}. In order to to keep track of this, we define the following `statement', of which \cref{thm:bgt.gromov} is one instance.
\begin{statement}\label{thm:rel.hom.dim.k(d)}
If $G$ is a group with finite symmetric generating set $S$ containing the identity and
\[
|S^n|\le\eps_0(d)n^{d+1}|S|
\]
for some integer $n\ge n_0(d)$, then there exist normal subgroups $H,\Gamma\normal G$ with $H\le\Gamma$ such that $H\subseteq S^n$, such that $\Gamma/H$ is nilpotent with class at most $O_d(1)$, and such that $[G:\Gamma]\le k(d)$.
\end{statement}
We will then often write sentences such as `Let $\eps_0(d)$, $n_0(d)$ and $k(d)$ be such that \cref{thm:rel.hom.dim.k(d)} holds'. Note, for example, that \cref{thm:bgt.gromov} shows that \cref{thm:rel.hom.dim.k(d)} holds with $k(d)=k^*(d+1)$, $\eps_0(d)=1$ and $n_0(d)=n_0^*(d+1)$, whilst \cref{thm:rel.hom.dim} will show that it holds with $k(d)=g(d)$ and some potentially much smaller value of $\eps_0$.

\cref{thm:finitary.poly-vol.detailed} then follows from \cref{thm:bgt.gromov} and the following result.
\begin{prop}\label{prop:XP.orig}
Let $d,R\in\N_0$, and let $\eps_0>0$ and $n_0,k\in\N$ be such that \cref{thm:rel.hom.dim.k(d)} holds for this value of $d$. Then there exist $n_1=n_1(d,R,n_0,k)\in\N$ and $\eps=\eps(d,\eps_0)>0$ such that if $G$ is a group with finite symmetric generating set $S$ containing the identity and
\[
|S^n|\le\eps n^{d+1}|S|
\]
for some integer $n\ge n_1$ then there exist a set $X\subseteq S^k$ of cardinality at most $k$ containing the identity, and an $O_d(1)$-rational Lie progression $P$ of dimension at most $O_d(1)$ in $O_d(1)$-upper-triangular form and with injectivity radius at least $R$ generating a normal subgroup of $G$ such that
\[
XP^m\subseteq S^{mn}\subseteq XP^{O_{d,R}(m)}
\]
for every $m\in\N$, and such that distinct elements of $X$ belong to distinct cosets of $\la P\ra$.
\end{prop}
\begin{proof}
We may certainly insist that $n_1\ge2$, so that $|S^{\lfloor n/2\rfloor}|\le3^{d+1}\eps\left\lfloor\frac{n}2\right\rfloor^{d+1}|S|$. Provided $n_1$ is large enough and $\eps$ is small enough, \cref{thm:rel.hom.dim.k(d)} therefore implies that there exist normal subgroups $H,\Gamma\normal G$ with $H\le\Gamma$ such that $H\subseteq S^{\lfloor n/2\rfloor}$, such that $\Gamma/H$ is nilpotent of class at most $O_d(1)$, and such that $[G:\Gamma]\le k$.

Let $\beta=\beta_{d,R}>0$ be a quantity to be determined shortly but depending only on $d$ and $R$. Provided $n_1$ is large enough in terms of $d$ and $R$, it then follows from \cref{lem:poly.pigeon} that there exists $j\in\N$ satisfying $n^{1/2}<j<\beta n$ such that $|S^{3j}|\ll_d|S^j|$, then from \cref{lem:tripling->AG} that $S^{2j}$ is an $O_d(1)$-approximate group, and then finally from \cref{lem:slicing} that $S^{4j}\cap\Gamma$ is an $O_d(1)$-approximate group. \cref{thm:nilp.frei} then implies that there is a nilpotent progression $Q_0$ of rank and step at most $O_d(1)$ in the quotient $\Gamma/H$, and a subgroup $H_0\le\Gamma/H$ normalised by $Q_0$, such that $(S^{4j}\cap\Gamma)H/H\subseteq Q_0H_0\subseteq S^{O_d(j)}H/H$. By \cref{prop:nilp.prog=>Lie.prog}, there therefore exists an $O_d(1)$-rational Lie progression $P_0\subseteq\Gamma$ of dimension $O_d(1)$ in $O_d(1)$-upper-triangular form, the symmetry group of which is the pullback to $\Gamma$ of $H_0$, such that $(S^{4j}\cap\Gamma)H\subseteq P_0\subseteq S^{O_d(j)}H$. This shows $P_0$ generates $\Gamma$, which is normal.

By \cref{lem:ball.cosets=>index}, we may pick a complete set $X\subseteq S^k$ of coset representatives for $\Gamma$ in $G$, including the identity. In particular, $|X|\le k$ as required, and $S^{2j}\subseteq X\Gamma$. By setting $n_1\ge k^2$ we may ensure that $j\ge k$, so that in fact $S^{2j}\subseteq X(S^{3j}\cap\Gamma)\subseteq XP_0$, and hence $S^{mj}\subseteq XP_0^m$ for all $m\in\N$ by \cref{lem:inclusions.local}. Since $X\subseteq S^j$, we may therefore fix $A=A_d\in\N$ such that
\[
S^{mj}\subseteq XP_0^m\subseteq S^{Amj}H
\]
for all $m\in\N$.

Let $B=B_{d,R}\in\N$ be a quantity to be determined shortly but depending only on $d$ and $R$, set $\beta=1/2AB$, and set $r=\lfloor n/2ABj\rfloor$, noting that $r\ge1$ and $r\ge n/4ABj$ by the choice of $\beta$. \cref{prop:powergood.rational} then implies that there exists an $O_d(1)$-rational Lie progression $P_1$ of dimension $O_d(1)$ in $O_d(1)$-upper-triangular form such that $P_0^r\subseteq P_1\subseteq P_0^{O_d(r)}$. If $\inj P_1\ge R$ then set $P=P_1$; if not, then by \cref{prop:properprog} there exists an $O_d(1)$-rational Lie progression $P$ in $O_d(1)$-upper-triangular form with dimension at most $O_d(1)$ and injectivity radius at least $R$ such that $P_1\subseteq P\subseteq P_1^{O_{d,R}(1)}$. In either case, there exists a choice of $B$ such that $P_0^r\subseteq P\subseteq P_0^{Br}$, and hence
\[
\begin{split}
XP^m\subseteq XP_0^{Brm} \subseteq S^{ABrmj}H \subseteq S^{\lfloor mn/2\rfloor}H\qquad\qquad\qquad\qquad\qquad\\
\subseteq S^{mn}\subseteq S^{\lceil mn/j\rceil j}\subseteq XP_0^{\lceil mn/j\rceil}\subseteq XP_0^{4ABmr+1}\subseteq XP^{O_{d,R}(m)},
\end{split}
\]
for all $m\in\N$, as required.
\end{proof}

It will also be useful to have versions of the above arguments under small-doubling or -tripling hypotheses, as follows.
\begin{statement}\label{st:bgt.doubling}
Let $G$ be a group with finite symmetric generating set $S$ containing the identity. If
\[
|S^{2n}|\le K|S^n|
\]
for some $n\ge n_0(K)$ then there exist normal subgroups $H,\Gamma\normal G$ with $H\le\Gamma$ such that $H\subseteq S^{O^*_K(n)}$, such that $\Gamma/H$ is nilpotent of class at most $6\log_2K$, and such that $[G:\Gamma]\le k_2(K)$. If
\[
|S^{3n}|\le K|S^n|
\]
for some $n\ge n_0(K)$ then there exist normal subgroups $H,\Gamma\normal G$ with $H\le\Gamma$ such that $H\subseteq S^{O^*_K(n)}$, such that $\Gamma/H$ is nilpotent of class at most $6\log_2K$, and such that $[G:\Gamma]\le k_3(K)$.
\end{statement}

\begin{prop}\label{prop:XP.doubling}
Let $K\ge1$ and $R\in\N_0$, and let $n_0,k_2,k_3\in\N$ with $k_3\le k_2$ be such that \cref{st:bgt.doubling} holds for this value of $K$. Suppose $G$ is a group with finite symmetric generating set $S$ containing the identity satisfying
\[
|S^{2n}|\le K|S^n|
\]
for some integer $n\ge\max\{n_0,k_2,2K^2\}$. Then there exist a set $X\subseteq S^{k_2}$ of cardinality at most $k_2$ containing the identity, and an $O_K(1)$-rational Lie progression $P$ of dimension at most $\exp(\exp(O(K^2)))$ and class at most $6\log_2K$ in $O_K(1)$-upper-triangular form and with injectivity radius at least $R$ generating a normal subgroup of $G$ such that
\[
S^{mn}\subseteq XP^m\subseteq S^{O_{K,R}(mn)+O^*_{K,R}(n)}
\]
for every $m\in\N$, and such that distinct elements of $X$ belong to distinct cosets of $\la P\ra$. If
\begin{equation}\label{eq:XP.tripling}
|S^{3n}|\le K|S^n|,
\end{equation}
then we may in fact obtain the quasipolynomial bound $\dim P\le\exp(O(\log^3K))$ on the dimension of $P$, and also that $X\subseteq S^{k_3}$ and $|X|\le k_3$, whilst leaving all the other conclusions unchanged.
\end{prop}
\begin{proof}
Set $k=k_3$ if \eqref{eq:XP.tripling} holds, and $k=k_2$ otherwise. \cref{st:bgt.doubling} implies that there exist normal subgroups $H,\Gamma\normal G$ with $H\le\Gamma$ such that $H\subseteq S^{O_K^*(n)}$, such that $\Gamma/H$ is nilpotent of class at most $6\log_2K$, and such that $[G:\Gamma]\le k$. It follows from \cref{thm:doubling=>tripling} that $S^{2n}$ is an $\exp(\exp(O(K^2)))$-approximate group, and then from \cref{lem:slicing} that $S^{4n}\cap\Gamma$ is an $\exp(\exp(O(K^2)))$-approximate group. Under the stronger tripling bound \eqref{eq:XP.tripling}, it follows from \cref{lem:tripling->AG} that $S^{2n}$ is a $K^3$-approximate group, and then from \cref{lem:slicing} that $S^{4n}\cap\Gamma$ is a $K^9$-approximate group. \cref{thm:nilp.frei} then implies that there is a nilpotent progression $Q_0$ in the quotient $\Gamma/H$ with step at most $6\log_2K$ and rank at most $\exp(\exp(O(K^2)))$, or at most $\exp(O(\log^3K))$ if \eqref{eq:XP.tripling} holds, and a subgroup $H_0\le\Gamma/H$ normalised by $Q_0$ such that $(S^{4n}\cap\Gamma)H/H\subseteq Q_0H_0\subseteq S^{O_K(n)}H/H$.  By \cref{prop:nilp.prog=>Lie.prog}, there therefore exists an $O_K(1)$-rational Lie progression $P_0\subseteq\Gamma$ of dimension at most $\exp(\exp(O(K^2)))$, or at most $\exp(O(\log^3K))$ if \eqref{eq:XP.tripling} holds, and class at most $6\log_2K$ in $O_K(1)$-upper-triangular form, the symmetry group of which is the pullback to $\Gamma$ of $H_0$, such that $(S^{4n}\cap\Gamma)H\subseteq P_0\subseteq S^{O_K(n)}H$. This shows $P_0$ generates $\Gamma$, which is normal.

By \cref{lem:ball.cosets=>index}, we may pick a complete set $X\subseteq S^k$ of coset representatives for $\Gamma$ in $G$, including the identity. In particular, $|X|\le k$ as required, and $S^{2n}\subseteq X\Gamma$. Since $n\ge k$, we in fact have $S^{2n}\subseteq X(S^{3n}\cap\Gamma)\subseteq XP_0$, and hence $S^{mn}\subseteq XP_0^m$ for all $m\in\N$ by \cref{lem:inclusions.local}. Since $X,H\subseteq S^n$, we therefore have
\[
S^{mn}\subseteq XP_0^m\subseteq S^{O_K(mn)+O_K^*(n)}
\]
for all $m\in\N$. If $\inj P_0\ge R$ then set $P=P_0$; if not, then by \cref{prop:properprog} there exists an $O_K(1)$-rational Lie progression $P$ in $O_K(1)$-upper-triangular form with dimension at most $\exp(\exp(O(K^2)))$, or at most $\exp(O(\log^3K))$ if \eqref{eq:XP.tripling} holds, class at most $6\log_2K$ and injectivity radius at least $R$ such that $P_0\subseteq P\subseteq P_0^{O_{K,R}(1)}$.
\end{proof}

\subsection{Lower bounds on growth in Lie progressions}\label{section:LowerBounds}

As we mentioned right near the beginning of the introduction to this paper, the upper bound on the growth degree of $G$ in Gromov's theorem as we stated it in \cref{thm:gromov} follows trivially from the polynomial-growth hypothesis \eqref{eq:o(n^(d+1)} and the existence of the growth lower bound $|S^n|\gg n^{\deg G}$ given by \eqref{eq:BG}. We prove the upper bounds on dimension and homogeneous dimension given by \cref{thm:dimbounds} \ref{conc:dim.bound} and \ref{conc:hdim.bound} in a similar way, but with the asymptotic growth lower bound of \eqref{eq:BG} replaced by the following quantitative volume lower bounds.
\begin{prop}\label{prop:dimension.bound}
Let $d,k,t,\eta\in\N$. Suppose $G$ is a group with finite symmetric generating set $S$ containing the identity, $X\subseteq S^t$ contains the identity and satisfies $|X|\le k$, and $P$ is a Lie progression of dimension at most $d$ such that $XP^r\subseteq S^{rn}\subseteq XP^{\eta r}$ for all $r\in\N$ for some $n\in\N$. Then
\[
|S^m|\gg_{\eta,d,k,t}m^{\hdim P}\qquad\text{and}\qquad|S^m|\gg_{\eta,d,k,t}m^{\dim P}|S|
\]
for all $m\in\N$ satisfying $n\le m\le n\inj P$.
\end{prop}

We start by proving universal lower bounds on the growth of a generating set for the lattice of~$P$.
\begin{prop}\label{prop:growthLowerBound}\label{prop:growth.lower.bound.dim}
Let $d\in\N$. Suppose that $\Gamma$ is a finitely generated torsion-free nilpotent group with Mal'cev completion of dimension $d$, and that $S$ is a finite symmetric generating subset of $\Gamma$ containing the identity. Then
\[
|S^n|\gg_dn^d|S|
\]
for all $n\in\N$.
\end{prop}
We use the following standard result from additive combinatorics.
\begin{lemma}\label{lem:sumset.lower.bound}
Let $A$ and $B$ be finite subsets of $\Z$. Then $|A+B|\ge |A|+|B|-1$. In particular, by induction on $n$ we have $|nA|\ge n(|A|-1)$.
\end{lemma}
\begin{proof}
Label the elements of $A$ as $a_1<\ldots<a_p$ and the elements of $B$ as $b_1<\ldots<b_q$, and note that $a_1+b_1<a_2+b_1<\cdots<a_p+b_1<a_p+b_2<\cdots<a_p+b_q$, so that all of the elements in this last list are distinct.
\end{proof}
We also use the following combinatorial lemma.
\begin{lemma}[{\cite[Lemmas 2.6.2 \& 2.6.3]{tointon.book}}]\label{lem:lb.quotient.kernel}
Suppose $G$ is a group, $H\le G$ is a subgroup, and $A\subseteq G$ is a finite symmetric subset. Then $|A|\le|AH/H||A^2\cap H|$ and $|A^{m+n}|\ge|A^mH/H||A^n\cap H|$ for every $m,n\in\N$.
\end{lemma}

\begin{proof}[Proof of \cref{prop:growthLowerBound}]
We proceed by induction on $d$, the base case $d=1$ following from \cref{lem:sumset.lower.bound}. Let $c$ be the class of $\Gamma$, and let $N$ be the Mal'cev completion of $\Gamma$. Let $x_1,\ldots,x_c\in S$ be such that $[x_1,\ldots,x_c]\ne1$, and let $H$ be the central one-parameter subgroup of $N$ generated by $[x_1,\ldots,x_c]$. Let $r\ll_d1$ be the word length of $[x_1,\ldots,x_c]$. Since $|S^n|\ge|S|$ for all $n\in\N$, it suffices to consider only $n\ge2r$. Given $n\ge2r$, we claim that $|S^{\lfloor n/2\rfloor}\cap H|\gg_dn|S^2\cap H|$. To see this, first note that $S^r\cap H$ contains both the identity and $[x_1,\ldots,x_c]^{\pm1}$, and hence that $|S^r\cap H|\ge3$. \cref{lem:sumset.lower.bound} and the fact that $n\ge2r$ therefore imply that $|S^{\lceil n/2\rceil}\cap H|\ge|(S^r\cap H)^{\lfloor n/2r\rfloor}|\gg\lfloor n/2r\rfloor|S^r\cap H|\gg_dn|S^2\cap H|$ as claimed. Since $\dim(N/H)=d-1$, by induction on $d$ we have that $|S^{\lfloor n/2\rfloor}H/H|\gg_dn^{d-1}|SH/H|$, and the proposition then follows from \cref{lem:lb.quotient.kernel}.
\end{proof}

We now reduce \cref{prop:dimension.bound} to the case $X=\{1\}$ via the following lemma.
  
\begin{lemma}\label{prop:fromStoT}
Let $\eta,t,n\in\N$, and let $G$ be a group with a finite symmetric generating set $S$ containing the identity. Suppose $X\subseteq S^t$ and $Q\subseteq G$ both contain the identity and satisfy
\[
XQ^r\subseteq S^{rn}\subseteq XQ^{\eta r}
\]
for $r=1$ and $r=2t+1$. Then writing $\hat Q=Q\cup Q^{-1}$ and $T=S^{2t+1}\cap\hat Q^{2\eta}$ we have $S\subseteq XT$;
\begin{equation}\label{eq:S^j.cap.Q}
S^j\cap\hat Q^\eta\subseteq T^j,\qquad j=1,\ldots,n;
\end{equation}
and
\[
\hat Q\subseteq T^n\subseteq\hat Q^{4^{|X|}(2t+1)\eta}.
\]
\end{lemma}

\begin{proof}
First, the fact that $S\subseteq XQ^\eta$ and $X\subseteq S^t$ implies that $S\subseteq X(S^{t+1}\cap Q^\eta)\subseteq XT$, as required.

Next, let $j\in\{1,\ldots,n\}$, and let $g\in S^j\cap\hat Q^\eta$. By definition there exist elements $s_1,\ldots,s_j\in S$ such that $g=s_j\cdots s_1$. Set $g_0=1$ and $g_i=s_is_{i-1}\cdots s_1$ for $i=1,\ldots,j$, noting that $g=g_j$. Since $S^j\subseteq S^n\subseteq XQ^\eta$, it follows that for each $i=0,1,\ldots,j$ there exist $x_i\in X$ and $q_i\in\hat Q^\eta$ such that $g_i=x_iq_i$. Moreover, since $1,g_j\in\hat Q^\eta$ we may assume that $q_0=x_0=x_j=1$, and hence $g=q_j$. We then have $q_{i}q_{i-1}^{-1}=x_is_ix_{i-1}^{-1}\in S^{2t+1}\cap\hat Q^{2\eta}=T$ for each $i$, and hence $g=q_j\in T^j$, giving \eqref{eq:S^j.cap.Q} as required.

Since $\hat Q\subseteq S^n$, the case $j=n$ of \eqref{eq:S^j.cap.Q} implies in particular that $\hat Q\subseteq T^n$. It therefore remains only to show that $T^n\subseteq\hat Q^{4^{|X|}(2t+1)\eta}$. To prove this, we first claim that there exist $X'\subseteq X$ and $m\in\N$ satisfying $(2t+1)\eta\le m\le4^{|X|}(2t+1)\eta$ such that
\begin{equation}\label{eq:T<P.2}
X\hat Q^{(2t+1)\eta}\subseteq X'\hat Q^m
\end{equation}
and
\begin{equation}\label{eq:T<P.3}
x\hat Q^m\cap\hat Q^{2m}=\varnothing
\end{equation}
for every $x\in X'\setminus\{1\}$. To prove this claim, we first check whether it holds with $X'=X$ and $m=(2t+1)\eta$. These choices certainly satisfy \eqref{eq:T<P.2}. If they also satisfy \eqref{eq:T<P.3} then the claim holds. If not, there exists $x\in X'\backslash\{1\}$ such that $x\hat Q^m\cap\hat Q^{2m}\ne\varnothing$, and hence $x\in\hat Q^{3m}$ and so
\[
x\hat Q^m\subseteq\hat Q^{4m}.
\]
If we replace $X'$ by $X'\backslash\{x\}$ and $m$ by $4m$ then \eqref{eq:T<P.2} is therefore still satisfied, and so we check again whether \eqref{eq:T<P.3} is satisfied, and repeat if necessary. This process terminates after at most $|X|$ steps, and so the claim is proved.

We now claim that $T^j\subseteq\hat Q^m$ for $j=1,\ldots,n$, which of course implies in particular that $T^n\subseteq\hat Q^{4^{|X|}(2t+1)\eta}$ as required. The case $j=1$ follows from the definition of $T$ and the fact that $m\ge2\eta$, so by induction we may assume that $T^{j-1}\subseteq\hat Q^m$, and hence that $T^j\subseteq\hat Q^{m+2\eta}$. Since $T^n\subseteq S^{(2t+1)n}$ we also have $T^n\subseteq XQ^{(2t+1)\eta}$ by hypothesis, and hence $T^j\subseteq X'\hat Q^m$. It follows that $T^j\subseteq\hat Q^{m+2\eta}\cap X'\hat Q^m$, and hence, by definition of $X'$ and $m$, that $T^j\subseteq\hat Q^m$, as claimed.
\end{proof}

\begin{proof}[Proof of \cref{prop:dimension.bound}]
Applying \cref{prop:fromStoT} and \eqref{eq:inverseProg} we obtain a positive integer $\rho=\rho(d,k,t,\eta)$ and a symmetric subset $T\subseteq G$ containing the identity such that $P\subseteq T^n\subseteq P^{\rho}$, $T\subseteq S^{2t+1}$, and $S\subseteq XT$. Note that this last containment implies that $|T|\ge|S|/k$.

We will prove that the desired statement holds provided $n\le m\le n\inj P/2\rho$; this easily implies the same statement when $n\le m\le n\inj P$, possibly with worse implied constants. Given such $m$, we have $T^m\subseteq P^{\lceil m/n\rceil\rho}\subseteq P^{\inj P}$. Therefore, writing $H$ for the symmetry group of $P$ and $\pi$ for its projector, there exists a unique symmetric subset $\tilde T\subseteq\tilde P^{\rho}$ containing the identity such that $\pi(\tilde T)=TH/H$, and that this set satisfies $|T^mH|=|\tilde T^m||H|$ and  $|\tilde T||H|=|TH|\ge|S|/k$. \cref{prop:growth.lower.bound.dim} implies that $|\tilde T^m|\gg_dm^{\dim P}|\tilde T|$, whilst \cite[Proposition~3.1]{lmtt} implies that $|\tilde T^m|\gg_{d}m^{\hdim P}$. Since we also have $S^m\supseteq T^{\lfloor m/(2t+2)\rfloor}H$ for all $m$, the desired result follows.
\end{proof}

\subsection{Growth of Lie progressions}\label{section:growth}
In his proof of \cref{thm:tao}, Tao \cite[\S4]{tao} implicitly proved a partial analogue of \cref{cor:main.poly.vol} valid in a certain ultralimit. In his argument, objects called \emph{ultra coset nilprogressions} played a role analogous to that of the Lie progressions $P_i$ appearing in our \cref{cor:main.poly.vol}, and a key part of his proof consisted of estimating the growth of these ultra coset progressions.

In order to prove \cref{cor:tao}, we adapt Tao's argument to our finitary setting, the more refined nature of which allows us to make the quantitative aspects of his computation much more precise, to an extent that allows us to obtain the improvements of \cref{cor:tao} compared to \cref{thm:tao}.

The main technical content of this section lies in the following proposition.
\begin{prop}[fine-scale growth of Lie progressions]\label{prop:growthfunction}
Suppose $P$ is a $Q$-rational Lie progression of dimension $d$ in $C$-upper-triangular form. Then there exists a continuous, increasing, piecewise-monomial function $f$, with degree increasing, bounded below by $d$, and bounded above by $\hdim P$, such that $|P^n|\ll_{C,d,Q}f(n)$ for all $n\in\N$ and $|P^n|\asymp_{C,d,Q}f(n)$ for all $n\le\inj P$.
\end{prop}

We start by noting that it is sufficient to prove \cref{prop:growthfunction} with a polynomial function in place of a piecewise-monomial function, as follows.

\begin{lemma}\label{lem:poly=piecewise.mono}
Let $k,r\in\Z$ with $0\le k\le r$, let $\alpha_k,\alpha_{k+1},\ldots,\alpha_r\ge0$, and define a polynomial $f$ by $f(x)=\alpha_kx^k+\alpha_{k+1}x^{k+1}+\cdots+\alpha_rx^r$. Then there exist $x_k=1\le x_{k+1}\le\cdots\le x_r<x_{r+1}=\infty$ such that the piecewise-monomial function $h:[1,\infty)\to[0,\infty)$ defined by $h(x)=\alpha_ix^i$ for $x\in[x_i,x_{i+1})$ is continuous and satisfies $f(x)\asymp_rh(x)$ for $x\ge1$.
\end{lemma}
\begin{proof}
We claim that setting $h(x)=\max_i\alpha_ix^i$ satisfies the lemma. This function $h$ is certainly continuous. Moreover, if $0<x<y$ and if $i$ and $j$ are such that $h(x)=\alpha_ix^i$ and $h(y)=\alpha_jy^j$ then we have $\alpha_ix^i\geq \alpha_jx^j$ and $\alpha_iy^i\leq\alpha_jy^j$. This in turn implies that $(y/x)^i\le(y/x)^j$, and hence that $i\le j$, so that the $i$ for which $\alpha_ix^i$ is maximal is non-decreasing in $x$. This easily implies the existence of $x_k=1\le x_{k+1}\le\cdots\le x_r<x_{r+1}=\infty$ such that $h(x)=\alpha_ix^i$ for $x\in[x_i,x_{i+1})$. Finally, following Tao \cite[\S4]{tao}, note that for every $a,b>0$ we have $a+b\asymp\max\{a,b\}$, so that $f(x)\asymp_rh(x)$ as required.
\end{proof}

The next step in the proof of \cref{prop:growthfunction} is the following estimate. Recall that we defined basic Lie brackets in Section \ref{sec:basic.comms}. 
\begin{prop}\label{prop:prog.as.nilbox}
Suppose that $G$ is a connected, simply connected nilpotent Lie group with Lie algebra $\g$ of dimension $d$, and $P(x;L)\subseteq G$ is an integral raw Lie progression with basis $e_1,\ldots,e_d$ in $C$-upper-triangular form. Then extending $e=(e_1,\ldots,e_d)$ to a list $\overline e=(e_1,\ldots,e_r)$ of basic Lie brackets, and writing $\vol$ for Lebesgue measure on the Lie algebra $\g$ of $G$, normalised so that the lattice $\Lambda=\la e_1,\ldots,e_d\ra$ has determinant $1$, we have
\[
\vol(B_\R(\overline e;(nL)^\chi))\ll_{C,d}|P(x;L)^n|\ll_d\vol(B_\R(\overline e;(nL)^\chi))
\]
for all $n\in\N$.
\end{prop}

\begin{lemma}\label{lem:box.dilate.size}
Let $v_1,\ldots,v_r\in\R^d$, let $L_1,\ldots,L_r>0$, and let $\alpha_1,\ldots,\alpha_r\ge1$. Then
\[
\vol(B_\R(v,\alpha L))\le\lceil\alpha_1\rceil\cdots\lceil\alpha_r\rceil\vol(B_\R(v,L)).
\]
\end{lemma}
\begin{proof}
It suffices to consider the case in which the $\alpha_i$ are integers, and in that case $B_\R(v,\alpha L)$ can be covered by $\alpha_1\cdots\alpha_r$ translates of $B_\R(v,L)$.
\end{proof}

\begin{proof}[Proof of \cref{prop:prog.as.nilbox}]
To prove the upper bound, let $n\in\N$, let $Q=\{\ell_1e_1+\cdots+\ell_de_d:\ell_i\in[0,1]\}$, and note that
\begin{align*}
P^n&\subseteq\exp\Lambda\cap P_\R(x;L)^n\\
   &\subseteq\exp\Lambda\cap\exp\big(B_\R(\overline e;(nL)^\chi)\big)^{O_d(1)}&&\text{(by Lemma \ref{lem:freeequivalences} \ref{item:freeequiv.i})}\\
   &\subseteq\exp\big(\Lambda\cap B_\R(\overline e;(O_d(nL))^\chi)\big)&&\text{(by Lemma \ref{lem:freeequivalences} \ref{item:freeequiv.iv})},
\end{align*}
and hence
\begin{align*}
|P^n|&\le|\Lambda\cap B_\R(\overline e;(O_d(nL))^\chi)|\\
    &=\vol\big(\big(\Lambda\cap B_\R(\overline e;(O_d(nL))^\chi)\big)+Q\big)\\
    &\le\vol B_\R(\overline e;(O_d(nL))^\chi)\\
    &\ll_d\vol B_\R(\overline e;(nL)^\chi)&\text{(by \cref{lem:box.dilate.size}),}
\end{align*}
as required.

We now come to the lower bound. A result of van der Corput \cite{vdC} states that for any convex body $K\subset\R^d$ we have $|\Z^d\cap K|\ge\frac{1}{2^d}\vol(K)$, which in this case implies that
\[
\vol(B_\R(\overline e;(nL)^\chi))\le2^d|\Lambda\cap B_\R(\overline e;(nL)^\chi)|
\]
for every $n\in\N$. Since
\begin{align*}
\exp\big(\Lambda\cap B_\R(\overline e;(nL)^\chi)\big)
    &\subset\la x_1,\ldots,x_d\ra\cap P_\R(x;L)^{O_d(n)}&&\text{(by Lemma \ref{lem:freeequivalences} \ref{item:freeequiv.i})}\\
    &\subset P(x,L)^{O_{C,d}(n)}&&\text{(by \cref{prop:goodReducReal1}),}
\end{align*}
this implies that for some constant $k=k_d>0$ and every $n\in\N$ we have
\begin{equation}\label{eq:vol<}
\vol(B_\R(\overline e;(nL)^\chi))\ll_d|P(x,L)^{kn}|.
\end{equation}
Given $n\ge k$ we therefore have
\begin{align*}
\vol(B_\R(\overline e;(nL)^\chi))
     &\ll_d\vol(B_\R(\overline e;(\lfloor n/k\rfloor L)^\chi))&&\text{(by \cref{lem:box.dilate.size})}\\
     &\ll_d|P(x,L)^n|&&\text{(by \eqref{eq:vol<}),}
\end{align*}
as required. For $n\le k$, note that by \cref{lem:upper-tri.doubling} we have $|P(x,L)^{k^2}|\ll_{C,d}|P(x,L)|$, and hence
\begin{align*}
\vol(B_\R(\overline e;(nL)^\chi))
     &\le\vol(B_\R(\overline e;(kL)^\chi))\\
     &\ll_d|P(x,L)^{k^2}|&&\text{(by \eqref{eq:vol<})}\\
     &\ll_{C,d}|P(x,L)^n|,
\end{align*}
as required.
\end{proof}

To estimate the expression $\vol(B_\R(\overline e;(nL)^\chi))$ appearing in \cref{prop:prog.as.nilbox} we use the following linear-algebraic lemma, which was implicit in \cite[\S4]{tao}.
\begin{lemma}\label{lem:cramer}
Let $d,r\in\N$ with $r\ge d$, and let $M_1,\ldots,M_r>0$. Suppose $\R^d$ is spanned by elements $v_1,\ldots,v_r$. Then there exist $i_1<\ldots<i_d$ such that
\[
B_\R(v;M)\subset r\cdot B_\R(v_{i_1},\ldots,v_{i_d};M_{i_1},\ldots,M_{i_d}).
\]
\end{lemma}
\begin{proof}
Pick the $i_1<\ldots<i_d$ that maximise $\vol(B_\R(v_{i_1},\ldots,v_{i_d};M_{i_1},\ldots,M_{i_d}))$. 
On reordering the $v_i$, we can assume that $i_j=j$ for all $1\leq j\leq d$. 
It suffices to show that for a given $v_k$ we have
\begin{equation}\label{eq:one.basis.vector}
M_kv_k\in  B_\R(v_{1},\ldots,v_{d};M_{1},\ldots,M_{d}).
\end{equation}
View each $v_i$ as a column vector, and write $A$ for the $d\times d$ matrix with columns $M_{1}v_{1},\ldots,M_{d}v_{d}$. Cramer's rule implies that the solution $y\in\R^d$ to the equation $Ay=M_kv_k$ satisfies
\[
|y_j|=\frac{\vol(B_\R(v_{1},\ldots,v_{j-1},v_k,v_{j+1},\ldots,v_{d};M_{1},\ldots,M_{j-1},M_k,M_{j+1},\ldots,M_{d}))}{\vol(B_\R(v_{1},\ldots,v_{d};M_{1},\ldots,M_{d}))}
\]
(i.e.\ with $v_{j}$ replaced by $v_k$ in the numerator). By maximality, this implies in particular that $|y_j|<1$, which gives \eqref{eq:one.basis.vector}, as required.
\end{proof}

\begin{lemma}\label{lem:box=poly}
In the setting of \cref{prop:prog.as.nilbox} there exists a polynomial $f$ in which every non-trivial term has a positive coefficient and degree between $d$ and $\hdim N$ such that $\vol(B_\R(\overline e;(nL)^\chi))\asymp_df(n)$ for every $n\in\N$.
\end{lemma}
\begin{proof}
Following Tao \cite[\S4]{tao}, define the polynomial $f$ by
\[
f(x)=\sum_{1\le i_1<\ldots<i_d\le r}\vol(B_\R(e_{i_1},\ldots,e_{i_d};L^\chi))x^{\sum_{j=1}^d|\chi(i_j)|}.
\]
This polynomial certainly has no negative coefficients, and each non-zero term in the sum defining $f$ has degree at least $d$. To see that the degree of $f$ is at most $\hdim N$, note that the coefficient $\vol(B_\R(e_{i_1},\ldots,e_{i_d};L^\chi))$ is non-zero only when the vectors $e_{i_1},\ldots,e_{i_d}$ are linearly independent. This means that in every such coefficient, and for every $k\in\N$, the number of vectors $e_{i_j}$ with $|\chi(i_j)|\ge k$ must be at most $\dim\n_k$. In particular, for every set $e_{i_1},\ldots,e_{i_d}$ of vectors satisfying $\vol(B_\R(e_{i_1},\ldots,e_{i_d};L^\chi))\ne0$ we have $\sum_{j=1}^d|\chi(i_j)|\le\sum_{k=1}^d\dim\n_k=\hdim N$, as required.

It remains to prove that $\vol(B_\R(\overline e;(nL)^\chi))\asymp_df(n)$ for $n\in\N$. First, note that
\[
f(n)=\sum_{1\le i_1<\ldots<i_d\le r}\vol(B_\R(e_{i_1},\ldots,e_{i_d};(nL)^\chi))
\]
for $n\in\N$. It then follows from Lemma \ref{lem:cramer} that $\vol(B_\R(\overline e;(nL)^\chi))\le r^df(n)$. On the other hand, the fact that $B_\R(e_{i_1},\ldots,e_{i_d};(nL)^\chi)\subseteq B_\R(\overline e;(nL)^\chi)$ for every $i_1,\ldots,i_d$ implies that $f(n)\le\textstyle{r\choose d}\vol(B_\R(\overline e;(nL)^\chi))$.
\end{proof}

\begin{proof}[Proof of \cref{prop:growthfunction}]
Write $x_1,\ldots,x_d$ for the generators of $P$ and $L_1,\ldots,L_d$ for its lengths. It follows from \cref{lem:make.integral} that there exists a natural number $M\ll_dQ$ and a subset $X\subseteq G$ of size at most $O_{d,Q}(1)$ such that $P(x^M;L)$ is an integral raw Lie progression in $O_{C,d,Q}(1)$-upper-triangular form, and such that
\begin{equation}\label{eq:growth-make.integral}
\tilde P^n\subseteq XP(x^M;L)^{O_{C,d,Q}(n)}
\end{equation}
and
\begin{equation}\label{eq:growth-integral.bdd}
P(x^M;L)^n\subseteq\tilde P^{kn}
\end{equation}
for all $n\in\N$ and some $k=k(d,Q)$. It then follows from \cref{prop:prog.as.nilbox,lem:box=poly} that there exists a continuous, increasing, piecewise-monomial function $f$ with degree increasing, bounded below by $d$, and bounded above by $\hdim P$, such that $|P(x^M;L)^n|\asymp_{C,d,Q}f(n)$ for all $n\in\N$.

We claim that $|\tilde P^n|\asymp_{C,d,Q}f(n)$ for all $n\in\N$. We certainly have $|\tilde P^n|\ll_{C,d,Q}f(n)$ for all $n\in\N$, thanks to \eqref{eq:growth-make.integral} and the upper bound on the degree of $f$. For $n\ge k$, \eqref{eq:growth-integral.bdd} and the upper bound on the degree of $f$ imply that $f(n)\ll_df(\lfloor n/k\rfloor)\ll_{C,d,Q}|\tilde P^n|$. This implies in particular that for $n<k$ we have $f(n)\le f(k)\ll \ll_{C,d,Q}|\tilde P^k|\ll_{C,d}|\tilde P^n|$, the last bound coming from \cref{lem:upper-tri.doubling}, and so the claim is proved. Since $|P^n|\le|\tilde P^n||H|$ for all $n\in\N$ and $|P^n|=|\tilde P^n||H|$ for all $n\le\inj P$, where $H$ is the symmetry group of $P$, the proposition follows.
\end{proof}

\subsection{Injectivity modulo the centre}\label{section:injectModCenter}
It is a standard fact that a discrete normal subgroup of a connected group is central; to see why, note that if $\Gamma$ is a discrete normal subgroup of a connected topological group $G$ then the map $G\times \Gamma\to \Gamma$ defined by $(g,\gamma)\mapsto [g,\gamma]$ is continuous, and hence constant with respect to $g$. In this section we prove a finitary analogue of this fact for Lie progressions; this in turn will underpin \cref{thm:inj/Z}.

Given a Lie progression $P$ with lattice $\Gamma$ and projector $\pi$, and a normal subgroup $U\normal\Gamma$, one may define the injectivity radius of $P$ `modulo $U$' as the supremum over $j$ such that two elements of $P^j$ which have same image under $\pi$ must differ by an element of $U$. In \cref{thm:inj/Z} this notion naturally arises with \(U\) as the centre of 
$\Gamma$.
Precisely, we define the \emph{injectivity radius of $P$ modulo the centre}, which we denote by $\injz P$, via $\injz P=\sup\{j\in\N_0:\ker\pi\cap(\tilde P^j\tilde P^{-j})\subseteq Z(\Gamma)\}$.

The following general property of Lie progressions is then a finitary analogue of the fact that a discrete normal subgroup of a connected group is central.
\begin{prop}\label{prop:ProperCenter}
Given $C\ge1$ and $d\in\N$ there exists $R=R(C,d)$ such that if $P$ is a Lie progression of class $c$ and dimension $d$ in $C$-upper-triangular form with injectivity radius at least $R$, the lengths of which are all at least $m\in\N$, then $\injz P\ge m^{\frac{1}{c-1}}$.
\end{prop}
Note that \cref{prop:ProperCenter} holds by definition in the case $c=1$ under the convention that $m^{1/0}=\infty$.

The proof of \cref{prop:ProperCenter} rests on a series of lemmas.
\begin{lemma}\label{lem:smallComm}
Let $C,\eps>0$ and $n\in\N$. Suppose $G$ is a simply connected nilpotent Lie group of class $c\ge2$ with Lie algebra $\g$, and $\Omega\subseteq\g$ is a symmetric convex body satisfying $[\Omega,\Omega]\subseteq C\Omega$. Then 
\[
\log[g,h]\in O_{c,C}(n^{c-1}\eps)\Omega
\]
for every $g\in\exp(\Omega)^n$ and $h\in\exp(\eps\Omega)$.
\end{lemma}
\begin{proof}
Let $g\in\exp(\Omega)^n$ and $h\in\exp(\eps\Omega)$. It follows from the Baker--Campbell--Hausdorff formula that
\begin{equation}\label{eq:bch.app}
\log g\in O_c(n)\Omega+O_c(n^2)[\Omega,\Omega]+\cdots+O_c(n^c)[\Omega,\ldots,\Omega]_c,
\end{equation}
and hence that there exists $z\in\gamma_c(G)$ such that $gz$ satisfies
\[
\log(gz)\in O_c(n)\Omega+O_c(n^2)[\Omega,\Omega]+\cdots+O_c(n^{c-1})[\Omega,\ldots,\Omega]_{c-1}.
\]
Note that since $z$ is central in $G$ we have $[gz,h]=[g,h]$. We showed in \cite[Lemma 4.6]{proper.progs} that there are bracket forms $\beta_1,\ldots,\beta_{k_c}$ of weight at least $3$ and at most $c$ and rationals $q_1,\ldots,q_{k_c}$, all of which depend only on $c$, such that
\[
\log[(gz),h]=[\log(gz),\log h]+q_1\beta_1(\log(gz),\log h)+\cdots+q_{k_c}\beta_{k_c}(\log(gz),\log h).
\]
The result then follows from the bilinearity of the Lie bracket.
\end{proof}

\begin{lemma}\label{lem:RealCommutators}
Let $C>0$ and $m\in\N$. Suppose $G$ is a simply connected nilpotent Lie group of class $c\ge2$ and dimension $d$ with Lie algebra $\g$ with basis $e_1,\ldots,e_d$, and write $x_i=\exp e_i$ for each $i$. Suppose $L_1,\ldots,L_d\ge1$ are such that $(e;L)$ is in $C$-upper-triangular form over $\R$. Then
\[
[g,h]\in P_{\R}(x,L)^{O_{C,d}(1)}
\]
for every $g\in P_{\R}(x,L)^{m^{\frac{1}{c-1}}}$ and $h\in P_{\R}(x,L/m)$.
\end{lemma}
\begin{proof}
Let $g\in P_{\R}(x,L)^{m^{\frac{1}{c-1}}}$ and $h\in P_{\R}(x,L/m)$. \cref{lem:freeequivalences} \ref{item:freeequiv.3} implies that
\[
g\in\exp(B_\R(e,L))^{O_d(m^{\frac{1}{c-1}})},\qquad h\in\exp(B_\R(e,L/m))^{O_d(1)},
\]
and then the Baker--Campbell--Hausdorff formula implies that
\[
h\in\exp(B_\R(e,O_{C,d}(L/m)))
\]
(cf. \eqref{eq:bch.app}). We therefore have
\begin{align*}
[g,h]&\in\exp(B_\R(e,O_{C,d}(L)))&&\text{(by \cref{lem:smallComm})}\\
   &\subseteq P_{\R}(x,L)^{O_{C,d}(1)}&&\text{(by \cref{lem:freeequivalences} \ref{item:freeequiv.3}),}
\end{align*}
as required.
\end{proof}

\begin{lemma}\label{lem:LieCommutators}
Let $m\in\N$ and $C>0$. Suppose $G$ is a simply connected nilpotent Lie group of class $c\ge2$ and dimension $d$ with Lie algebra $\g$ with basis $e_1,\ldots,e_d$, and write $x_i=\exp e_i$ for each $i$. Suppose $L_1,\ldots,L_d\ge1$ are such that $(x;L)$ is in $C$-upper-triangular form and $(e;L)$ is in $C$-upper-triangular form over $\R$. Then
\[
[g,h]\in P(x;L)^{O_{C,d}(1)}
\]
for every $g\in P(x,L)^{m^{\frac{1}{c-1}}}$ and $h\in P(x,L/m)$.
\end{lemma}
\begin{proof}
Given $g\in P(x,L)^{m^{\frac{1}{c-1}}}$ and $h\in P(x,L/m)$ we have
\begin{align*}
[g,h] & \in P_{\R}(x,L)^{O_{C,d}(1)}\cap\langle x_1,\ldots,x_d\rangle && \text{(by \cref{lem:RealCommutators})}\\
        & \subseteq P(x,L)^{O_{C,d}(1)}  &&  \text{(by \cref{prop:goodReducReal1}),} 
\end{align*}        
as required.
\end{proof}

\begin{proof}[Proof of Proposition \ref{prop:ProperCenter}]
Write $H$ for the symmetry group of $P$, $\pi$ for its projector, and $x_1,\ldots,x_d$ for its generators. If the injectivity radius of $P$ modulo the centre is less than $m^{\frac{1}{c-1}}$ then there exits a non-central $g\in P(x,L)^{m^{\frac{1}{c-1}}}\cap\ker\pi$. The fact that $g$ is not central implies that there exists $i\in \{1,\ldots,d\}$ such that $[g,x_i]\ne 1$, and since $L_i\ge m$ we have $x_i\in P(x,L/m)$, so \cref{lem:LieCommutators} implies that $[g,x_i]\in P(x,L)^{O_{C,d}(1)}\cap\ker\pi$. We may therefore take $R$ to be the constant implied by the $O_{C,d}(1)$ notation in this expression.
\end{proof}

The following result shows that \cref{prop:ProperCenter} can be applied to any progression containing a large power of a generating set (recall the definition of the weights $\zeta(i)$ from just before \cref{lem:upper-tri.doubling.dilate}).

\begin{lemma}\label{prop:LowerBoundL_i}
Let $G$ be a group with finite symmetric generating set $S$ containing the identity, let $C>0$ and $d,m\in\N$, and let $x_1,\ldots,x_d\in G$ and $L_1,\ldots,L_d\in\N$. Suppose $P(x;L)$ is infinitely proper and in $C$-upper-triangular form and contains $S^m$. Then $L_i\ge C^{1-\zeta(i)}m^{\zeta(i)}$ for each $i$.
\end{lemma}

\begin{proof}
This is similar to part of the proof of \cite[Theorem 1.11]{proper.progs}. Enumerate $\{i:\zeta(i)=1\}$ as $i_1,\ldots,i_r$. It is shown in \cite[Lemma 9.1]{proper.progs} that the map $\pi:G\to\Z^r$ defined by $x_1^{\ell_1}\cdots x_d^{\ell_d}\mapsto(\ell_{i_1},\ldots,\ell_{i_r})$ is a homomorphism. Since $S$ generates $G$, this implies that for each $j=1,\ldots,r$ there exists $s_j\in S$ with positive $x_{i_j}$-coordinate, and then that $s_j^m$ has $x_{i_j}$-coordinate at least $m$. Since $s_j^m\in P(x;L)$, this forces $L_{i_j}\ge m$, which proves the lemma for all $i$ with $\zeta(i)=1$. The bound for general $i$ then follows by induction, using the upper-triangular form.
\end{proof}

Finally, we have the following analogue of \cref{lem:finite.subgroup.of.proper} for $\injz P$, which shows that finite subgroups of Lie progressions in some sense behave on scales within the injectivity radius modulo the centre like compact subgroups of connected nilpotent Lie groups, in that they are central (see \cref{prop:compact.central}).

\begin{lemma}\label{lem:finite<injz}
Suppose $P$ is a Lie progression with symmetry group $H$, and that $K\subseteq P^{\lfloor\injz P/2\rfloor}$ is a finite subgroup. Then $KH/H\le Z(\la P\ra/H)$.
\end{lemma}
Note that with the convention $A^\infty=\la A\ra$ the statement of \cref{lem:finite<injz} still makes sense for $\injz P=\infty$, and indeed it still holds in this case by the same proof.
\begin{proof}Let $N$ be the nilpotent Lie group from which $P$ is projected, and write $\Gamma<N$ for its lattice and $\pi:\Gamma\to\la P\ra/H$ for its projector. Abbreviate $R=\injz P$, and define $Z_R=\la\ker\pi\cap(\tilde P^R\tilde P^{-R})\ra$, noting that $Z_R$ is central in $\Gamma$ by definition, and hence central in $N$ as $\log \Gamma$ generates the Lie algebra of $N$. Moreover, $\pi$ factors through the homomorphism $\pi_R:\Gamma/Z_R\to\la P\ra/H$ defined by $\pi_R(gZ_R)=\pi(g)$.

We claim that $\pi_R$ is injective on $\tilde P^RZ_R/Z_R$. Indeed, if $\pi_R(pZ_R)=\pi_R(qZ_R)$ for some $p,q\in\tilde P^R$ then by definition of $\pi_R$ we have $\pi(p)=\pi(q)$, and then by definition of $R$ we have $pq^{-1}\in Z_R$. \cref{lem:local.hom.pullback} therefore implies that $\pi_R^{-1}(K/H)$ is a subgroup of $\Gamma/Z_R$ isomorphic to $K/H$. In particular, $\pi_R^{-1}(K/H)$ is a compact subgroup of $N/Z_R$. \cref{prop:compact.central} therefore implies that $\pi_R^{-1}(K/H)$ is central in $\Gamma/Z_R$, which implies the desired result.
\end{proof}

\section{A finitary Mann-type theorem}
In this section we prove the following finitary version of Mann's \cref{thm:mann}.
\begin{theorem}\label{thm:index.bound}
Given $C,d,k,t,\eta\in\N$, there exist positive integers $r=r(d,k,t,\eta)$ and $j=j(C,d,k,t,\eta)$ such that, provided $\eta$ is at least the quantity $M(C,d)$ appearing in \cref{prop:lift.prog}, the following holds. Let $G$ be a group with finite symmetric generating set $S$ containing the identity, and suppose $P$ is a Lie progression of class $c$ and dimension $d$ with injectivity radius at least $r$ in $C$-upper-triangular form, and that $X\subseteq S^t$ is a set of size at most $k$ containing the identity such that for some $n\ge 1$ we have
\[
XP^m\subseteq S^{mn}\subseteq XP^{\eta m}
\]
for all $m\in\N$, and such that $xP^j\cap yP^j=\varnothing$ for all distinct $x,y\in X$. Let $\hat G=\la S\mid R_{G,S}(4\eta n)\ra$. Then there exists a normal subgroup $\hat K\normal\hat G$ such that $[\hat G:\hat K]\le g(d)$ and $\gamma_{c+1}(\hat K)\subseteq\hat S^{n+O_{c,k}(1)}$, such that $\hat K/\gamma_{c+1}(\hat K)$ has Hirsch length $d$, such that each element of $\hat G\setminus\hat K$ acts non-trivially by conjugation on $\overline\gamma_i(\hat K)/\overline\gamma_{i+1}(\hat K)$ for at least one $i$, and such that $\overline\gamma_{c+1}(\hat K)$ is the maximum normal finite subgroup of $\hat G$. Moreover, if $xPx^{-1}\subseteq P^{2\eta}$ for every $x\in X$ then we may conclude in addition that $\hat K\supseteq\hat P$.
\end{theorem}

The condition $xP^j\cap yP^j=\varnothing$ in the hypothesis of \cref{thm:index.bound} can be thought of as saying that the translates of $P$ (or rather, the subgroup generated by $P$) by elements of $X$ are `locally disjoint', and is a finitary analogue of the fact that distinct cosets of a subgroup are genuinely disjoint. The condition $xPx^{-1}\subseteq P^{2\eta}$ in the final sentence of the statement of \cref{thm:index.bound} can be thought of as a local version of normality of $\la P\ra$; in the case where $X\cap\la P\ra=\{1\}$, it is in fact exactly equivalent to normality of $\la P\ra$ in $G$, as follows.
\begin{lemma}\label{eq:locally.normal}
Suppose $G$ is a group with finite symmetric generating set $S$ containing the identity, and $X,Q\subseteq G$ are subsets of $G$ containing the identity such that $XQ\subseteq S$ and $S^2\subseteq XQ^\eta$, such that $\la Q\ra\normal G$, and such that $X\cap\la Q\ra=\{1\}$. Then $xQx^{-1}\subseteq Q^\eta$ for every $x\in X$.
\end{lemma}
\begin{proof}
Given $x\in X$, we have $xQx^{-1}\subseteq S^2\cap\la Q\ra\subseteq XQ^\eta\cap\la Q\ra$. Since $X\cap\la Q\ra=\{1\}$, this implies that $xQx^{-1}\subseteq Q^\eta$, as required.
\end{proof}

Our approach to \cref{thm:index.bound} is inspired by Mann's proof of \cref{thm:mann}, which recall actually gives the more detailed \cref{prop:mann}. However, one obvious way in which \cref{prop:mann} falls short of what we need for \cref{thm:index.bound} is that that theorem requires not just a bound on the \emph{size} of $\gamma_{c+1}(K)$, but on its \emph{diameter}. We obtain this using the following lemma.
\begin{lemma}\label{lem:comms.bdd.diam}
Let $m,j\in\N$ with $j\ge2$, let $G$ be a group with symmetric generating set $S$ containing the identity, and suppose $|\gamma_j(G)|\le m$. Then $\gamma_j(G)\subseteq S^{3\cdot 2^{j-2}m+m^2}$.
\end{lemma}
\begin{proof}
We will define a sequence $S_1\subseteq S_2\subseteq\ldots$ of symmetric subsets of $G$ recursively as follows. First, set $S_1=\{[s_1,\ldots,s_j]:s_i\in S\}\cup\{1\}$. Then, having defined $S_1,\ldots,S_k$, set $H_i=\langle S_i\rangle$ for each $i$. If $H_k\normal G$ then stop. If $H_k\not\normal G$, on the other hand, then there exist $x\in S_k$ and $s\in S$ such that $sxs^{-1}\notin H_k$. Fix one such pair of elements $x,s$, and let $S_{k+1}=S_k\cup\{(sxs^{-1})^{\pm1}\}$. Since all elements of each $S_i$ are simple commutators of weight $j$, we have $S_i\subseteq \gamma_j(G)$ for each $i$, so this process must therefore terminate at some $S_k$ with $k\le m$. By definition, this means that $H_k\normal G$, and so since $\gamma_j(G)$ is generated by $S_1$ as a normal subgroup of $G$, it must in fact be that $H_k=\gamma_j(G)$. By construction, $S_i\subseteq S^{\lambda(j)+2(i-1)}\subseteq S^{3\cdot2^{j-1}+2m}$. Finally, $S_k$ generates $\gamma_j(G)$, and by \cref{lem:sphere.size.1}, it must do so with diameter at most $m/2$. In other words, $\gamma_j(G)=S_k^{\lfloor m/2\rfloor}$, and the result follows.
\end{proof}
A more serious issue with applying \cref{prop:mann} to prove \cref{thm:index.bound} is that the group $G$ appearing in \cref{thm:index.bound} does not obviously satisfy the hypotheses of \cref{prop:mann}. We spend most of the rest of this chapter refining the structure of $G$ in order to overcome this.

\subsection{Large-scale simple connectedness and rough covering maps}\label{sec:sc}
Given a graph $X$, and some $k\in \N$, we define a $2$-dimensional CW-complex $P_k(X)$ whose 1-skeleton is $X$, and whose  $2$-cells are  $m$-gons for $0\leq m\leq k$, defined by simple loops $(x_0,\ldots,x_m=x_0)$ of length $m$ in $X$, up to cyclic permutations. As a topological object, every $2$-cell is a disc attached along its boundary to a simple loop, so that the intersection of $2$ different $2$-cells belongs to the $1$-skeleton. 
\begin{definition}[\cite{dlST}]
We say that a graph $X$ is \emph{simply connected at scale $k$}, or \emph{$k$-simply connected}, if $P_k(X)$ is simply connected. If there exists such a $k$, then we shall say that $X$ is \emph{large-scale simply connected}.
\end{definition}

A Cayley graph $\G(G,S)$ is $k$-simply connected if and only if $G$ has a presentation $\langle S\mid R\rangle$ with $R$ a set of relations of length at most $k$ (see for instance the proof of \cite[Lemma 7.91]{DrutuKapovich}).

It is shown in \cite[Theorem 2.2]{dlST} that if $X$ is $k$-simply connected and $f:X\to Y$ is a $(C,K)$-quasi-isometry then $Y$ is $O_{C,K,k}(1)$-simply connected. One can think of this as being analogous to the fact that topological simple connectedness is preserved under homeomorphisms. In this section we prove the following analogue of the fact that a topological covering map with simply connected range is always a homeomorphism.

\begin{theorem}\label{thm:k-sc.covering}
Let $C,C',K,k\in\N$. Then there exist $R_0=R_0(C,C',k)\in\N$ and $R'=R'(C,K,k)\in\N$ such that the following holds. Let $X$ and $Y$ be non-empty, connected graphs, and write $d_X$ and $d_Y$ for their respective graph metrics. Suppose that $Y$ is $k$-simply connected and that $f:X\to Y$ satisfies the following conditions for some integer $R\ge R_0$:
\begin{enumerate}[label=(\roman*)]
\item\label{enum:k-sc.lift} for every pair of vertices $x\in X$ and $y\in Y$ with $d_Y(f(x),y)\le R'$, there exists $x'\in X$ such that $d_X(x,x')\le R/2$ and $d_Y(f(x'),y)\le K$;
\item\label{enum:k-sc.Lip} every pair of neighbours $x,x'$ in $X$ satisfies $d_Y(f(x),f(x'))\le C$;
\item\label{enum:k-sc.biLip} for every $x,x'\in X$ with $d_X(x,x')\le R$, if $d_Y(f(x),f(x'))\le C+2K$ then $d_X(x,x')\le C'$.
\end{enumerate}
Then $f$ is a $(C\vee(C'/C),C'\vee K)$-quasi-isometry and $X$ is $O_{C,C',k}(1)$-simply connected.
\end{theorem}
To see how the map $f$ appearing in \cref{thm:k-sc.covering} can be thought of as a rough covering map, note that conditions \ref{enum:k-sc.lift}--\ref{enum:k-sc.biLip} say that the $R'$-neighbourhood of any point $f(x)$ in the image of $f$ is quasi-isometric via $f$ to the $R/2$-neighbourhood of $x$. One can think of as being analogous to the fact that for a topological covering map $\ph:U\to V$, every point $\ph(u)$ in the image of $\ph$ has a neighbourhood that is homeomorphic via $\ph$ to a neighbourhood of $u$.

In proving that large-scale simple connectedness is stable under quasi-isometries, de la Salle and the first author obtained a characterisation of large-scale simple connectedness in terms of a certain large-scale analogue of homotopy. First, given a constant $C\ge1$, define a \emph{$C$-path} between vertices $x,x'$ in a graph $\G$ to be a sequence $x=v_0,\ldots,v_n=x'$ of vertices of $\G$ such that $d(v_i,v_{i+1})\le C$ for all $i$; we call $n$ the \emph{length} of this $C$-path. Given another constant $L>0$, we then define a relation $\approx_{C,L,x,x'}$ on the set of $C$-paths between $x$ and $x'$ by saying $p\approx_{C,L,x,x'}q$ if, writing $p=(u_0,\ldots,n_m)$ and $q=(v_1,\ldots,v_n)$, there exist non-negative integers $j_1,j_2,j_2',j_3$ such that
\begin{itemize}
\item $m=j_1+j_2+j_3$ and $n=j_1+j_2'+j_3$;
\item $u_i=v_i$ for all $i\le j_1$;
\item $u_{j_1+j_2+i}=v_{j_1+j_2'+i}$ for all $i\le j_3$;
\item $j_2+j_2'\le L$.
\end{itemize}
Finally, we define the equivalence relation $\sim_{C,L,x,x'}$ on the set of $C$-paths between $x$ and $x'$ to be the equivalence relation generated by $\approx_{C,L,x,x'}$. One can think of $p\sim_{C,L,x,x'}q$ as meaning that $p$ and $q$ are `large-scale homotopic'.

De la Salle and the first author observe in the proof of \cite[Theorem 2.2]{dlST} that if $\G$ is $k$-simply connected then for every $C\ge1$ there exists $L=L(C,k)$ such that for all $x,x'\in\G$ the equivalence relation $\sim_{C,L,x,x'}$ has a unique equivalence class, and moreover that if there exists some $C\ge1$ and $L$ for which the equivalence relation $\sim_{C,L,x,x'}$ has a unique equivalence class for all $x,x'\in\G$ then $\G$ is $O_{C,L}(1)$-simply connected. This is analogous to the definition of topogical simple connectedness as meaning that any two paths with the same endpoints are homotopic.

Recall that if $f:X\to Y$ is a covering map between CW-complexes, $x$ and $x'$ are two points of $X$, $p$ is a path joining $x$ to $x'$, and $q$ is a path joining $f(x)$ to $f(x')$, homotopic to $f(p)$, then there exists a unique lift of $q$ joining $x$ to $x'$. Our next result is a large-scale analogue of this. Here, and throughout the rest of this section, if $p=(p_0,\ldots,p_n)$ is a $C$-path in a graph $X$, and $f$ is a map from $X$ to some other graph $Y$, then we define $f(p)=(f(p_0),\ldots,f(p_n))$.
\begin{prop}\label{prop:k-simply.conn}
Let $C,C',K,L,R\in\N$ with $R\ge\max\{2C'+1,2LC'\}$. Let $X$ and $Y$ be graphs, and write $d_X$ and $d_Y$ for their respective graph metrics. Suppose that $f:X\to Y$ is such that
\begin{enumerate}[label=(\roman*)]
\item\label{enum:k-sc.lift'} for every pair of vertices $x\in X$ and $y\in Y$ with $d_Y(f(x),y)\le CL+K$, there exists $x'\in X$ such that $d_X(x,x')\le R/2$ and $d_Y(f(x'),y)\le K$;
\item\label{enum:k-sc.Lip'} every pair of neighbours $x,x'$ in $X$ satisfies $d_Y(f(x),f(x'))\le C$;
\item\label{enum:k-sc.biLip'} for every $x,x'\in X$ with $d_X(x,x')\le R$, if $d_Y(f(x),f(x'))\le C+2K$ then $d_X(x,x')\le C'$.
\end{enumerate}
Suppose further that $x,x'\in X$, that $p$ is a path from $x$ to $x'$, that $q$ is a $C$-path from $f(x)$ to $f(x')$, and that the $C$-path $f(p)$ satisfies $f(p)\sim_{C,L,f(x),f(x')}q$. Then there exists a $C'$-path $p'$ from $x$ to $x'$ of the same length as $q$ such that $p\sim_{C',L,x,x'}p'$, and such that $d_Y(f(p'_i),q_i)\le K$ for each $i$. Moreover, if $q=f(r)$ for some other path $r$ from $x$ to $x'$ then $d_X(p'_i,r_i)\le C'$ for each $i$ and $p'\sim_{C',L\vee4C',x,x'}r$.
\end{prop}

We isolate the following easy lemma.
\begin{lemma}\label{lem:k-conn.fellow.travel}
Let $C>0$. Let $X$ be a graph with graph metric $d_X$. Let $x,x'\in X$, and suppose $p,p'$ are $C$-paths of length $n$ from $x$ to $x'$ such that $d_X(p_i,p'_i)\le C$ for each $i$. Then $p\sim_{C,4C,x,x'}p'$.
\end{lemma}
\begin{proof}
For each $j=1,\ldots,n-1$, let $p^{(j)}$ agree with path $p$ for the first $j$ steps, then move to $p'_j$, and then continue along the path $p'$ until $x'$. Then $p\approx_{C,4C,x,x'}p^{(2)}\approx_{C,4C,x,x'}p^{(3)}\approx_{C,4C,x,x'}\cdots\approx_{C,4C,x,x'}p^{(n-1)}\approx_{C,4C,x,x'}p'$.
\end{proof}

\begin{proof}[Proof of \cref{prop:k-simply.conn}]
Since $f(p)\sim_{C,L,f(x),f(x')}q$, there exist $C$-paths $q^{(1)},q^{(2)},\ldots,q^{(t)}$, say of lengths $n_1,\ldots,n_t$, respectively, with $q^{(1)}=f(p)$, $q^{(t)}=q$, and $q^{(j)}\approx_{C,L,f(x),f(x')}q^{(j+1)}$ for each $j$. We will construct $C'$-paths $p^{(1)},p^{(2)},\ldots,p^{(t)}$ between $x$ and $x'$ of lenghts $n_1,\ldots,n_t$, respectively, such that $d_Y(f(p^{(j)}_i),q^{(j)}_i)\le K$ for each $i$ and $j$, such that $p^{(j)}\approx_{C',L,f(x),f(x')}p^{(j+1)}$ for each $j$, and such that $p^{(1)}=p$. Setting $p'=p^{(t)}$ then satisfies the first requirement of the proposition.

We may simply define $p^{(1)}=p$, noting that then $f(p^{(1)})=q^{(1)}$ by definition. By induction, we may then assume that we have $C'$-paths $p^{(1)},p^{(2)},\ldots,p^{(j-1)}$ between $x$ and $x'$ satisfying the required properties, including in particular the property that $d_Y(f(p^{(j-1)}_i),q^{(j-1)}_i)\le K$ for each $i$. Since $q^{(j-1)}\approx_{C,L,f(x),f(x')}q^{(j)}$, there exist non-negative integers $i_1,i_2,i_2',i_3$ such that $n_{j-1}=i_1+i_2+i_3$ and $n_j=i_1+i_2'+i_3$, such that $q^{(j-1)}_i=q^{(j)}_i$ for all $i=0,\ldots,i_1$, such that $q^{(j-1)}_{i_1+i_2+i}=q^{(j)}_{i_1+i_2'+i}$ for all $i=0,\ldots,i_3$, and such that $i_2+i_2'\le L$. Define $p^{(j)}_i=p^{(j-1)}_i$ for all $i=0,\ldots,i_1$ and $p^{(j-1)}_{i_1+i_2+i}=p^{(j)}_{i_1+i_2'+i}$ for all $i=0,\ldots,i_3$. For each $i=1,\ldots,i_2'-1$, since $d_Y(q_{i_1}^{(j)},q_{i_1+i}^{(j)})\le CL$ and $d_Y(f(p_{i_1}^{(j)}),q_{i_1}^{(j)})\le K$, we have $d_Y(f(p_{i_1}^{(j)}),q_{i_1+i}^{(j)})\le CL+K$. Invoking property \ref{enum:k-sc.lift'} of $f$, for each such $i$ we may then find a point $p^{(j)}_{i_1+i}\in X$ such that $d_X(p^{(j)}_{i_1+i},p^{(j)}_{i_1})\le R/2$ and $d_Y(f(p^{(j)}_{i_1+i}),q^{(j)}_{i_1+i})\le K$. Note also that $d_X(p^{(j)}_{i_1+i_2'},p^{(j)}_{i_1})\le LC'\le R/2$. By the triangle inequality, we therefore have $d_X(p^{(j)}_{i_1+i-1},p^{(j)}_{i_1+i})\le R$ and $d_Y(f(p^{(j)}_{i_1+i-1}),f(p^{(j)}_{i_1+i}))\le C+2K$ for each $i=1,\ldots,i'_2$, so that $d_X(p^{(j)}_{i_1+i-1},p^{(j)}_{i_1+i})\le C'$ for each such $i$ by property \ref{enum:k-sc.biLip'} of $f$ and hence $p^{(j)}$ is a $C'$-path satisfying $p^{(j-1)}\approx_{C',L,f(x),f(x')}p^{(j)}$, as required.

In the event that $q=f(r)$ for some other path $r$ from $x$ to $x'$, we first show that $d_X(p^{(t)}_i,r_i)\le C'$ for each $i$. We have $p^{(t)}_0=x=r_0$ by definition, so by induction it suffices to prove that $d_X(p^{(t)}_i,r_i)\le C'$ for an arbitrary $i\ge1$ assuming $d_X(p^{(t)}_{i-1},r_{i-1})\le C'$. Since $d_Y(f(p^{(t)}_i),f(r_i))=d_Y(f(p^{(t)}_i),q^{(t)}_i)\le K$, and $d_X(p^{(t)}_i,r_i)\le2C'+1$ by the induction hypothesis and the triangle inequality, property \ref{enum:k-sc.biLip'} of $f$ implies that $d_X(p^{(t)}_i,r_i)\le C'$, as required. Finally, we then have $p^{(t)}\sim_{C',4C',x,x'}r$ by \cref{lem:k-conn.fellow.travel}.
\end{proof}

\begin{proof}[Proof of \cref{thm:k-sc.covering}]
Since $Y$ is $k$-simply connected, there exists $L=L(C,k)$ such that $\sim_{C,L,y,y'}$ has a single equivalence class for all $y,y'\in Y$. Defining $R_0=\max\{2C'+1,2LC'\}$ and $R'=CL+K$, \cref{prop:k-simply.conn} therefore immediately implies that $\sim_{C',L\vee4C',x,x'}$ has a single equivalence class for all $x,x'\in X$, and hence that $X$ is $O_{C,C',k}(1)$-simply connected.

We now prove that $f$ is a $(C\vee(C'/C),C'\vee K)$-quasi-isometry. Property \ref{enum:k-sc.Lip} of $f$ implies that $d_Y(f(x),f(x'))\le Cd_X(x,x')$ for all $x\in X$. On the other hand, given $x,x'\in X$, let $p$ be a path from $x$ to $x'$, and let $q$ be a $C$-path of minumum length from $f(x)$ to $f(x')$. Write $n$ for the length of $q$, and note that $n\le\lceil d_Y(f(x),f(x'))/C\rceil$. Since $f(p)$ is also a $C$-path from $f(x)$ to $f(x')$, we have $f(p)\sim_{C,L,f(x),f(x')}q$, and so \cref{prop:k-simply.conn} implies that there exists a $C'$-path $p'$ from $x$ to $x'$ of length $n$. In particular, $d(x,x')\le C'd_Y(f(x),f(x'))/C+C'$.

Finally, we claim that, given $y\in Y$, there exists $x\in X$ such that $d_Y(f(x),y)\le K$. Since $X$ is not empty, we may fix $x_0\in X$. The claim holds trivially for $y=f(x_0)$, so we may assume that $y\ne f(x_0)$ and, by induction on $d_Y(f(x_0),y)$ (using that $Y$ is connected), that $y'\in Y$ is a neighbour of $Y$ and that there exists $x'\in X$ with $d_Y(f(x'),y')\le K<R'$. The claim then follows from property \ref{enum:k-sc.lift} of $f$.
\end{proof}

\subsection{Reduction to infinite injectivity radius}\label{section:reducInftyRadius}
In this section we reduce \cref{thm:index.bound} to the setting in which the nilpotent quotient $\langle P\rangle/H$ is torsion-free. The basic idea is to lift $P$ and $H$ to a truncated presentation of $G$, and then use \cref{thm:k-sc.covering} to show that $\langle P\rangle/H$ is torsion-free in this lift; this is the content of the following proposition. We can then prove \cref{thm:index.bound} in the lift, and project the resulting group $K$ back down to $G$.

\begin{prop}\label{prop:index.reduc.tf}
Given $d,k,t,\eta\in\N$ and $C\ge1$, there exist $r=r(d,k,t,\eta)\in\N$ and $j=j(C,k,t,d,\eta)\in\N$ such that, provided $\eta$ is at least the quantity $M(C,d)$ appearing in \cref{prop:lift.prog}, the following holds. Suppose that $G$ is a group with finite symmetric generating set $S$ containing the identity, that $P$ is a Lie progression of class $c$ and dimension at most $d$ in $C$-upper-triangular form with injectivity radius at least $r$, and that $X\subseteq S^t$ is a set of size at most $k$ containing the identity such that for some $n\in\N$ we have
\[
XP^m\subseteq S^{mn}\subseteq XP^{\eta m}
\]
for all $m\in\N$, and such that
\[
xP^j\cap yP^j=\varnothing
\]
for all distinct $x,y\in X$. Write $H$ for the symmetry group of $P$, and let $\hat P$ be the lift of $P$ to $\hat G=\la S\mid R_{G,S}(4\eta n)\ra$. Then $\hat P$ has infinite injectivity radius, distinct elements of $\hat X$ belong to distinct cosets of $\la\hat P\ra$, we have
\begin{equation}\label{eq:inclusions'}
\hat X\hat P^m\subseteq\hat S^{mn}\subseteq\hat X\hat P^{2\eta m}
\end{equation}
for all $m\in\N$, and $\hat H=\overline\gamma_{c+1}(\la\hat P\ra)$, which is in particular characteristic in $\la\hat P\ra$. Moreover, if $xPx^{-1}\subseteq P^{2\eta}$ for every $x\in X$ then $\la\hat P\ra\normal\hat G$.
\end{prop}
\begin{proof}
Since $S^{2n}\subseteq XP^{2\eta}\subseteq S^{2\eta n}$, we have $\hat S^{2n}\subseteq\hat X\hat P^{2\eta}$, which in turn gives \eqref{eq:inclusions'} by \cref{lem:inclusions.local}.

We will now show that $\hat P$ has infinite injectivity radius. By \cref{prop:fromStoT} and \eqref{eq:inclusions'}, the set $T=\hat S^{2t+1}\cap(\hat P\cap\hat P^{-1})^{4\eta}$ satisfies
\begin{equation}\label{eq:index.TP}
\hat P\subseteq T^n\subseteq(\hat P\cap\hat P^{-1})^{4^k(4t+2)\eta}\subseteq\hat P^{4^k(4t+2)d\eta}.
\end{equation}
In particular, $T$ generates $\langle\hat P\rangle$. Now suppose $R,R'\in\N$ are such that $R$ is even and $R\ge4\eta R'$.
We claim that
\begin{enumerate}[label=(\roman*)]
\item\label{enum:k-sc.lift.apply} for every pair of elements $h\in\langle\hat P\rangle$ and $g\in G$ with $d_{\hat S^n}(h,g)\le R'$, there exists $h'\in\langle\hat P\rangle$ such that $d_{T^n}(h,h')\le R/2$ and $d_{\hat S^n}(h',g)\le1$; and
\item\label{enum:k-sc.Lip.apply} every pair of elements $h,h'\in\langle\hat P\rangle$ with $h\in h'T^n$ satisfies $d_{\hat S^n}(h,h')\le2t+1$.
\end{enumerate}
We also claim that, provided $j\ge4^k(4t+2)dR\eta$,
\begin{enumerate}[resume, label=(\roman*)]
\item\label{enum:k-sc.biLip.apply} for every $h,h'\in X$ with $d_{T^n}(h,h')\le R$, if $d_{\hat S^n}(h,h')\le2t+3$ then $d_{T^n}(h,h')\le2\eta(2t+3)$.
\end{enumerate}
To prove that condition \ref{enum:k-sc.lift.apply} holds, we note first that $\hat P\subseteq T^n$ by \cref{prop:fromStoT}, and hence that $\hat S^{mn}\subseteq\hat X\hat T^{2\eta mn}\subseteq\hat S^nT^{2\eta mn}$ for all $m$. Since $\hat S$ and $T$ are symmetric, this implies that $\hat S^{mn}\subseteq T^{2\eta mn}\hat S^n$ for all $m$. In particular, if $d_{\hat S^n}(h,g)\le R'$ then $h^{-1}g\in T^{2\eta R'n}\hat S^n$. This means that there exists $h_0\in T^{2\eta R'n}\subseteq T^{Rn/2}$ such that $d_{\hat S^n}(hh_0,g)\le1$, so that we may take $h'=hh_0$. Condition \ref{enum:k-sc.Lip.apply} holds because $T^n\subseteq(\hat S^n)^{2t+1}$ by definition. Finally, to prove that condition \ref{enum:k-sc.biLip.apply} holds provided $j\ge4^k(4t+2)dR\eta$, note first that $T^{Rn}\cap\hat S^{(2t+3)n}\subseteq T^{Rn}\cap\hat X\hat P^{2\eta(2t+3)}$. If $j\ge4^k(4t+2)dR\eta$ then, since $T^{Rn}\subseteq\hat P^{4^k(4t+2)dR\eta}$ by \eqref{eq:index.TP}, this in fact implies that $T^{Rn}\cap\hat S^{(2t+3)n}\subseteq\hat P^{2\eta(2t+3)}$; by \eqref{eq:index.TP}, this in turn gives $T^{Rn}\cap\hat S^{(2t+3)n}\subseteq T^{2\eta(2t+3)n}$, and hence \ref{enum:k-sc.biLip.apply} as claimed.

Let $R_0=R_0(2t+1,2\eta(2t+3),4\eta)\in\N$ and $R'=R'(2t+1,1,4\eta)\in\N$ be the integers appearing in \cref{thm:k-sc.covering}, set $R$ to be the smallest even integer that is at least $\max\{R_0,4\eta R'\}$, and note that the Cayley graph $\G(\hat G,\hat S^n)$ is $4\eta$-simply connected. Provided $j\ge4^k(4t+2)dR\eta$, we may therefore apply \cref{thm:k-sc.covering} to the map
\[
\begin{array}{ccccc}
f&:&\G(\langle\hat P\rangle,T^n)&\to&\G(\hat G,\hat S^n)\\
  &&x&\mapsto&x,
\end{array}
\]
to conclude that $\G(\langle\hat P\rangle,T^n)$ is $O_{t,\eta}(1)$-simply connected. In particular, this implies that there exists $\ell\ll_{t,\eta}1$ such that $\langle\hat P\rangle$ admits a presentation $\langle T^n\mid B\rangle$, with $B$ a set of relations of length at most $\ell$. The first containment of \eqref{eq:index.TP} implies in particular that $\hat H\subseteq T^n$, and so $\langle\hat P\rangle/\hat H$ also admits a presentation $\langle T^n\mid B'\rangle$, with $B'$ a set of relations of length at most $\ell$. Provided $r\ge4^k(4t+2)d\ell\eta$, each element of $T\hat H/\hat H$ has a unique preimage in $\tilde P$ under $\hat\pi$, and, by \eqref{eq:index.TP}, these preimages satisfy every relation in $B'$. In particular, $\hat\pi$ is invertible, and so $\hat P$ has infinite injectivity radius.

We will now prove that distinct elements of $\hat X$ belong to distinct cosets of $\la\hat P\ra$. Recall that the projection $\psi:\hat G\to G$ is injective on $\hat S^{2\eta n}$, so an arbitrary element of $\hat S^{2\eta n}$ is of the form $\hat x$ for some $x\in S$, with $\hat x=\hat y$ if and only if $x=y$. Since $\hat X\subseteq S^n$, we may therefore consider two arbitrary elements $\hat x,\hat y\in\hat X$. By \eqref{eq:inclusions'}, we have $(\hat x^{-1}\hat y)^i\in S^{2nk}\subseteq\hat X\hat P^{4\eta k}$ for each $i=0,1,\ldots,k$, which implies by the pigeonhole principle that there exist integers $i<i'$ between $0$ and $k$ and $\hat z\in\hat X$ such that $(\hat x^{-1}\hat y)^i,(\hat x^{-1}\hat y)^{i'}\in\hat z\hat P^{4\eta k}$, and hence $(\hat x^{-1}\hat y)^{i'-i}\in\hat P^{4\eta(d+1)k}$. By \cref{lem:upper-tri.doubling.dilate}, and writing $u_1,\ldots,u_d$ for the basis of $\hat P$, writing $L_1,\ldots,L_d$ for its lengths, and writing $\hat\pi$ for its projector, it follows that there exist $\ell_1,\ldots,\ell_d\in\Z$ with $|\ell_i|\ll_{C,d,k,\eta}L_i$ such that $(\hat x^{-1}\hat y)^{i'-i}\in\hat\pi(u_1^{\ell_1}\cdots u_d^{\ell_d})$.

Now suppose in addition that $\hat x^{-1}\hat y\in\la\hat P\ra$. This means that $\hat x^{-1}\hat y\in\hat\pi(v)$ for some $v$ belonging to the lattice $\Gamma$ of $\hat P$, and hence that $(\hat x^{-1}\hat y)^{i'-i}\in\hat\pi(v^{i'-i})$. Since $\hat P$ has infinite injectivity radius, this implies that $v^{i'-i}=u_1^{\ell_1}\cdots u_d^{\ell_d}$. Applying \cref{lem:bg.rational.power} then gives $v=u_{j_1}^{p_{1}(1/(i'-i))\ell_{j_1}}\cdots u_{j_m}^{p_{m}(1/(i'-i))\ell_{j_m}}$ for some $m\in\N$, some indices $j_1,\ldots,j_m$, and some rational polynomials $p_1,\ldots,p_m$, each depending only on $d$, and then \cref{lem:pp.L2.1} implies that $v\in P_\Q(u;O_{C,d,k,\eta}(L))^{O_d(1)}\subseteq P_\Q(u;O_{C,d,k,\eta}(L))$. Since $\log u_1,\ldots,\log u_d$ is a Mal'cev basis and $v\in\Gamma$, it follows that in fact $v\in P(u;O_{C,d,k,\eta}(L))$, hence $v\in P(u;L)^q$ for some $q=q(C,d,k,\eta)$ by \cref{lem:upper-tri.doubling.dilate}, and hence $\hat x^{-1}\hat y\in\hat P^q$. Setting $j\ge q$, this implies that $x^{-1}y\in P^j$, hence $x=y$, and hence $\hat x=\hat y$, as required.

Since $\langle\hat P\rangle/\hat H$ is torsion-free nilpotent of class $c$, we certainly have $\overline\gamma_{c+1}(\la\hat P\ra)\le\hat H$. Moreover, since $\langle\hat P\rangle/U$ is torsion-free, every subgroup of $\langle\hat P\rangle$ properly containing $\overline\gamma_{c+1}(\la\hat P\ra)$ is infinite; since $\hat H$ is finite, it follows that $\hat H=\overline\gamma_{c+1}(\la\hat P\ra)$ as required.

Finally, if $xPx^{-1}\subseteq P^{2\eta}$ for every $x\in X$ then, since $xPx^{-1}\subseteq S^{2n}$ and $P^{2\eta}\subseteq S^{2\eta n}$, we have $\hat x\hat P\hat x^{-1}\subseteq\hat P^{2\eta}$ for every $\hat x\in\hat X$, and hence $\la\hat P\ra\normal\hat G$ as required.
\end{proof}

\subsection{Reduction to the case of a normal torsion-free nilpotent subgroup}\label{section:H=1}
In the previous section we reduced \cref{thm:index.bound} to the case in which $P$ had infinite injectivity radius. In order to apply \cref{prop:mann}, in this section we further reduce to the case in which $H$ is trivial and $\langle P\rangle$ is normal.
\begin{prop}\label{prop:index.reduc.mann}
Suppose that $G$ is a group with finite symmetric generating set $S$ containing the identity, that $P$ is a Lie progression of class $c$ and dimension $d$ with infinite injectivity radius, and that for some $n,\eta\in\N$ there exists a subset $X\subseteq S^n$ of size at most $k$ containing the identity such that
\[
XP^m\subseteq S^{mn}\subseteq XP^{\eta m}
\]
for all $m\in\N$. Let $H$ be the symmetry group of $P$. Then the subgroups
\begin{align*}
H_0&=\bigcap_{x\in X}xHx^{-1},\\
U&=\bigcap_{x\in X}x\langle P\rangle x^{-1}
\end{align*}
are normal in $G$, and satisfy $[G:U]\le k^k$ and $U\cap H=H_0$. Moreover, the quotient $U/H_0$ is torsion-free nilpotent of class $c$ and Hirsch length $d$.
\end{prop}

We start with a general lemma showing that if a subgroup is contained in a few translates of a Lie progression with large injectivity radius, then it is in fact contained in a few cosets of that Lie progression's symmetry group.
\begin{lemma}\label{lem:sym.grps.commensurable}
Let $d,k,\eta\in\N$. Suppose that $P$
is a Lie progression of dimension at most $d$ with symmetry group $H$ and injectivity radius at least $3^k(d+1)\eta$ in a group $G$, that $K\le G$, and that $X\subseteq G$ is a set of size at most $k$ such that $K\subseteq XP^\eta$. Then $[K:K\cap H]\le k$.
\end{lemma}
\begin{proof}
Recall that $P^{-1}\subseteq P^d$, so that, given $j\in\N$ and $x\in X$, if $P^j$ contains some element of $xP^\eta$, then $P^{j+(d+1)\eta}$ contains all of $xP^\eta$. Combined with the hypothesis that $K\subseteq XP^\eta$, this implies that for each integer $j\ge(d+1)\eta$, either $K\cap P^{2j}=K\cap P^j$, or there exists $x\in X$ such that $K\cap P^{2j}$ contains elements of $xP^\eta$ but $K\cap P^{j-(d+1)\eta}$ does not. In particular, for each $i\in\N$, either $K\cap P^{2\cdot 3^i(d+1)\eta}=K\cap P^{3^i(d+1)\eta}$, or there exists $x\in X$ such that $K\cap P^{2\cdot 3^i(d+1)\eta}$ contains elements of $xP^\eta$ but $K\cap P^{2\cdot3^{i-1}(d+1)\eta}$ does not. By the pigeonhole principle, it follows that there exists $j\le3^{k-1}(d+1)\eta$ such that
\begin{equation}\label{eq:H.cap.P^j.group}
K\cap P^j=K\cap P^{2j}.
\end{equation}
We claim that
\begin{equation}\label{H.cap.P=H.cap.H}
K\cap P^j=K\cap H,
\end{equation}
or equivalently that if $g\in K\cap P^j$ then $g\in H$. Indeed, given $g\in K\cap P^j$, pick $\tilde g\in\tilde P^j$ such that $\pi(\tilde g)\in gH$. If $\tilde g\ne1$ then we would be able to define $q$ to be the maximal positive integer such that $\tilde g^q\in\tilde P^j$. However, by definition of $g$ and $q$, we would then have $\tilde g^{q+1}\in\tilde P^{2j}\setminus\tilde P^j$; since $\inj P\ge2j$, this would in turn imply that $g^{q+1}\in P^{2j}\setminus P^j$, and hence in particular that $g^{q+1}\in(K\cap P^{2j})\setminus(K\cap P^j)$, contradicting \eqref{eq:H.cap.P^j.group}. It must therefore in fact be that case that $\tilde g=1$, and hence that $g\in H$ as claimed. This proves \eqref{H.cap.P=H.cap.H}.

Now note that if $g,g'\in(xP^\eta)\cap K$ for a given $x\in X$ then $g^{-1}g'\in P^{(d+1)\eta}\cap K\subseteq P^j\cap K=K\cap H$, so that the set $(xP^\eta)\cap K$ is contained in a single left coset of $K\cap H$. Since
\[
K=\bigcup_{x\in X}\left((xP^\eta)\cap K\right)
\]
by hypothesis, we conclude that $K$ is contained in a union of at most $k$ left cosets of $K\cap H$, as required.
\end{proof}

\begin{proof}[Proof of \cref{prop:index.reduc.mann}]
Since $H$ is normalised by $P$ and $G=X\langle P\rangle$, we have
\[
H_0=\bigcap_{g\in G}gHg^{-1},\qquad\qquad U=\bigcap_{g\in G}g\langle P\rangle g^{-1},
\]
and hence $H_0,U\normal G$ as claimed. Since $[G:\langle P\rangle]\le k$, we also trivially have $[G:U]\le k^k$.

Let $u\in U\cap H$. Enumerate the elements of $X$ as $x_1=1,x_2,\ldots,x_k$, and set $X_i=x_iXx_i^{-1}$, $H_i=x_iHx_i^{-1}$ and $P_i=x_iPx_i^{-1}$ for each $i$, noting that for each $i$ we have $x_i^{-1}Px_i\subseteq S^{3n}\subseteq XP^{3\eta}$, and hence $H\subseteq X_iP_i^{3\eta}$. Since $P$ has infinite injectivity radius, \cref{lem:sym.grps.commensurable} therefore implies that $[H:H\cap H_i]\le k$ for each $i$. Since $u$ normalises each $H\cap H_i$, for each $i$ there therefore exists an integer $j_i\in[0,k]$ such that $u^{j_i}\in H_i$. By definition of $U$, for each $i$ there exists $p_i\in\langle P\rangle$ such that $u=x_ip_ix_i^{-1}$, meaning that $p_i^{j_i}\in H$. Since $P$ has infinite injectivity radius, this in turn implies that $p_i\in H$, and hence that $u\in H_i$. Since this holds for all $i$, we therefore have $u\in H_0$, as required.

Finally, $U/H_0=U/(U\cap H)$ is isomorphic to a finite-index subgroup of $\langle P\rangle/H$, so is torsion-free nilpotent of class $c$ and Hirsch length $d$, as required.
\end{proof}

\subsection{Proof of \cref{thm:index.bound}}\label{section:indexbounds}
Take $r$ and $j$ be as given by \cref{prop:index.reduc.tf}. Let $\hat P$ be the lift of $P$ to $\hat G$ 
as in \cref{prop:index.reduc.tf}, noting in particular that $\dim\hat P=d$, and that if $xPx^{-1}\subseteq P^{2\eta}$ for every $x\in X$ then $\la\hat P\ra\normal\hat G$. \cref{prop:index.reduc.tf,prop:index.reduc.mann} then imply that $\hat G$ has normal subgroups $H_0\subseteq\hat S^n$ and \change{$U\supseteq H_0$}, with $U=\la\hat P\ra$ if $xPx^{-1}\subseteq P^{2\eta}$, such that $U/H_0$ is torsion-free nilpotent of class $c$, Hirsch length $d$, and index at most $k^k$ in $\hat G/H_0$. \cref{prop:mann} therefore implies that there exists a normal subgroup $\hat K\normal\hat G$ containing $U$ such that $[\hat G:\hat K]\le g(d)$ and $[\gamma_{c+1}(\hat K)H_0:H_0]\ll_{c,k}1$, such that each element of $\hat G\setminus\hat K$ acts non-trivially by conjugation on $\overline\gamma_i(\hat K)/\overline\gamma_{i+1}(\hat K)$ for at least one $i$, and such that $\overline\gamma_{c+1}(\hat K)$ is the maximum normal finite subgroup of $\hat G$. \cref{lem:comms.bdd.diam} then implies that $\gamma_{c+1}(\hat K)\subseteq\gamma_{c+1}(\hat K)H_0\subseteq\hat S^{O_{c,k}(1)}H_0\subseteq\hat S^{n+O_{c,k}(1)}$.
Since $U$ has finite index in $\hat K$, the Hirsch length of $\hat K/\gamma_{c+1}(\hat K)$ is $d$ as required.\qed

\section{Proof of the main results}\label{ch:proof}
In this section we complete the proofs of our main results. We start by proving a preliminary version of \cref{prop:chain.of.progs-index} in which the progressions are not necessarily Mal'cev--Losert progressions, and with no particular consideration of the index.
\begin{prop}\label{prop:chain.of.progs}
Suppose that $P_0$ is a $Q$-rational Lie progression of dimension $d$ in $C$-upper-triangular form, and let $R\in\N$. Then there exist non-negative integers $d'\le d$ and $r_0=1<r_1<\cdots<r_{d'}$, and $O_{C,d,Q}(1)$-rational Lie progressions $P_1,\ldots,P_{d'}$ in $O_{C,d,Q}(1)$-upper-triangular form with injectivity radius at least $R$, each generating the same group as $P_0$, such that $P_{d'}$ has infinite injectivity radius and such that the following conditions are satisfied for $i=1,\ldots,d'$:
\begin{enumerate}[label=(\roman*)]
\item $\dim P_{i}<\dim P_{i-1}$;\label{cond:dim.drop}
\item $\inj P_{i-1}\ll_{C,d,Q}\frac{r_{i}}{r_{i-1}}\ll_{C,d,Q,R}\inj P_{i-1}$; and\label{cond:inj.rad}
\item for every integer $m\ge r_i$ we have $P_i^{\lceil m/r_i\rceil}\subseteq P_0^m\subseteq P_i^{O_{C,d,Q}(m/r_i)}$.\label{cond:approx}
\end{enumerate}
\end{prop}
\begin{proof}
It suffices to prove the proposition in the special case $R=1$. Indeed, suppose we have a sequence $P_1,\ldots,P_{d'}$ of Lie progressions satisfying all the required properties except $\inj P_i\ge R$, with $\inj P_i\asymp_{C,d,Q}r_{i+1}/r_i$ for all $i=0,1,\ldots,d'-1$. Let $P_{i_1},\ldots,P_{i_\ell}$ be the subsequence of those progressions with injectivity radius at least $R$, noting that this subsequence is not empty because $\inj P_{d'}=\infty$. If $P_{i_n}$ and $P_{i_{n+1}}$ are two consecutive such progressions, we have
\[
\frac{r_{i_{n+1}}}{r_{i_n+1}}\ll_{C,d,Q}R^{i_{n+1}-r_{i_n+1}}\le R^{d'},
\]
and hence
\[
\frac{r_{i_{n+1}}}{r_{i_n}}\ll_{C,d,Q,R}\frac{r_{i_n+1}}{r_{i_n}}\asymp_{C,d,Q}\inj P_{i_n}.
\]
The proposition is then satisfied by taking these progressions.

We now begin the proof of the case $R=1$. If $\inj P_0=\infty$, which is in particular the case when $d=0$, the proposition is satisfied by taking $d'=0$. We may therefore assume that $P_0$ has finite injectivity radius $j$, that $d\ge1$, and, by induction, that the proposition has been proven for all smaller values of $d$. By \cref{prop:powergood.rational} there exists an $O_{C,d,Q}(1)$-rational Lie progression $P'$ in $O_{C,d,Q}(1)$-upper-triangular form projected from the same Lie group as $P_0$ and satisfying $\tilde P_0^j\subseteq \tilde P'\subseteq\tilde P_0^{O_{C,d,Q}(j)}$. This implies in particular that the injectivity radius of $P'$ is at most $1$, and so \cref{prop:properprog} implies that there exists
 an $O_{C,d,Q}(1)$-rational Lie progression $P_1$ in $O_{C,d,Q}(1)$-upper-triangular form with dimension strictly less than $d$ and injectivity radius at least $1$ such that $P'\subseteq P_1\subseteq(P')^{O_{C,d,Q}(1)}$. Fixing $r_1\asymp_{C,d,Q}j$ such that $\tilde P_1^2\subseteq\beta_1(\tilde P_0)^{r_1}$, the required conditions all hold for $i=1$. The proposition then follows from applying the induction hypothesis to $P_1$.
\end{proof}

\subsection{Fine-scale structure of powers of Lie progressions}\label{sec:fine-scale}
In this section we prove our main result, \cref{prop:chain.of.progs-index}. We start by constructing a progression at a scale comparable to $n$, as follows. Note that this statement also contains \cref{thm:finite.mann}.

\begin{prop}\label{prop:fine-scale.optimise.index}
Given $C,d,k,R,t,\eta\in\N$ there exists $\rho=\rho(d,k,t,\eta)\in\N$ and $n_0=n_0(d,k)\in\N$ such that the following holds.

Suppose $G$ is a group with finite symmetric generating set $S$ containing the identity, $P_0\subseteq G$ is a rational Lie progression of dimension $d$ and injectivity radius at least $\rho$ in $C$-upper-triangular form 
such that $\la P_0\ra\normal G$, and $X_0\subseteq S^t$ is a subset of size at most $k$ containing the identity such that distinct elements of $X_0$ belong to distinct cosets of $\la P_0\ra$, and such that for some integer $n\ge n_0$ we have
\[
X_0P_0^m\subseteq S^{mn}\subseteq X_0P_0^{\eta m}
\]
for all $m\in\N$.

Then there exist natural numbers $q\ll_{C,d,R,\eta}1$ and $\alpha\ll_{C,d,\eta}1$, and an $O_{C,d,\eta}(1)$-rational Lie progression $P\subseteq G$ of dimension at most $d$ in $O_{C,d,\eta}(1)$-upper-triangular form containing $P_0$, such that
\begin{itemize}
\item the group $\hat G=\la S\mid R_{G,S}(\alpha qn)\ra$ has polynomial growth;
\item $G=S^{g(\dim P)}\la P\ra$, and for every set $X\subseteq S^{g(\dim P)}$ of coset representatives for $\la P\ra$ in $G$ containing the identity we have that
	\begin{itemize}
	\item the lift $\hat X\hat P$ of $XP$ to $\hat G$ is well defined and a Mal'cev--Losert progression;
	\item writing $\hat S$ for the copy of $S$ in $\hat G$ we have $\hat X\hat P^m\subseteq\hat S^{mqn}\subseteq\hat X\hat P^{O_{C,d,\eta}(m)}$ for all $m\in\N$;
	\item the projection $\hat G\to G$ is injective on $\hat X\hat P^R$.
	\end{itemize}
\end{itemize}
\end{prop}
\begin{proof}
All bounds depending on $\eta$ in this proposition also depend on $C$ and $d$, so on increasing $\eta$ if necessary we may assume without loss of generality that $\eta$ is at least the quantity $M(C,d)$ appearing in \cref{thm:index.bound,prop:lift.prog,prop:index.reduc.tf}. In particular, we may lift $P_0$ to $G_{4\eta n}=\la S\mid R_{G,S}(4\eta n)\ra$. Write $\hat P_0$ for this lift, and write $S_{4\eta n}$ for the lift of $S$ to $G_{4\eta n}$.

Write $c$ for the class of $P_0$, noting that $c\ll_d1$. \cref{eq:locally.normal} implies that $xP_0x^{-1}\subseteq P_0^{2\eta}$ for every $x\in X_0$. Provided $\rho$ is bigger than the quantity $r(d,k,t,\eta)$ appearing in \cref{thm:index.bound}, that theorem therefore implies that there exists a normal subgroup $\hat K_0\normal G_{4\eta n}$ containing $\hat P_0$ such that $[G_{4\eta n}:\hat K_0]\le g(d)$ and $\gamma_{c+1}(\hat K_0)\subseteq S_{4\eta n}^{n+O_{d,k}(1)}$, such that $\hat K_0/\gamma_{c+1}(\hat K_0)$ has Hirsch length $d$, such that each element of $G_{4\eta n}\setminus\hat K_0$ acts non-trivially by conjugation on $\overline\gamma_i(\hat K_0)/\overline\gamma_{i+1}(\hat K_0)$ for at least one $i$, and such that $\overline\gamma_{c+1}(\hat K_0)$ is the maximum normal finite subgroup of $G_{4\eta n}$. On setting $n_0$ large enough, we may conclude that $\gamma_{c+1}(\hat K_0)\subseteq S_{4\eta n}^{2n}$. Let $K_0$ be the image of $\hat K_0$ under the quotient $G_{4\eta n}\to G$.

Set $A=K_0\cap S^{4n}$, noting that $P_0\subseteq A$. We may require that $n\ge k$, so that \cref{lem:fin.ind.gen} implies that $K_0=\la A\ra$. Moreover, the hypothesis implies that $|S^{3n}|/|S^n|\le|P_0^{3\eta}|/|P_0|$, so that $|S^{3n}|\ll_{C,d,\eta}|S^n|$ by \cref{lem:upper-tri.doubling}. It therefore follows from \cref{lem:tripling->AG} that $S^{2n}$ is an $O_{C,d,\eta}(1)$-approximate group, and then from \cref{lem:slicing} that $A$ is an $O_{C,d,\eta}(1)$-approximate group. Applying \cref{thm:nilp.frei} in the quotient $K_0/\gamma_{c+1}(K_0)$ and \cref{prop:nilp.prog=>Lie.prog} therefore implies that there exists an $O_{C,d,\eta}(1)$-rational Lie progression $Q_0$ of dimension at most $O_{C,d,\eta}(1)$ in $O_{C,d,\eta}(1)$-upper-triangular form such that $A\subseteq Q_0\subseteq A^{O_{C,d,\eta}(1)}\gamma_{c+1}(K_0)\subseteq S^{\ell n}$ for some natural number $\ell\ll_{C,d,\eta}1$. Note that $Q_0$ generates $K_0$.

Let $R'=R'(C,d,R,\eta)\ge R$ be a positive integer to be chosen later but depending only on $C,d,R,\eta$. From now on, whenever any other parameters depend on the choice of $R'$ we state this explicitly, even if they otherwise depend only on some subset of $C,d,R,\eta$; this is to make it clear that there is no circularity in the definition of our parameters. \cref{prop:chain.of.progs} implies that there exist a positive integer $r\ll_{C,d,R',\eta}1$, and an $O_{C,d,\eta}(1)$-rational Lie progression $P$ in $O_{C,d,\eta}(1)$-upper-triangular form with injectivity radius at least $R'$, dimension at most $O_{C,d,\eta}(1)$ such that $\la P\ra=K_0$ and such that $P\subseteq Q_0^{r}\subseteq P^{O_{C,d,\eta}(1)}$. Furthermore, it follows from \cref{lem:upper-tri.doubling.dilate} that on increasing $r$ and the lengths of $P$ by multiples bounded in terms of $C$, $d$ and $\eta$ only we may assume in addition that $Q_0\subseteq P$, and hence $P_0\subseteq P$ as required. Write $H$ for the symmetry group of $P$,
and set $q=r\ell+1$.

For any set $X\subseteq S^n$ of coset representatives for $\la P\ra=K_0$ in $G$ containing the identity, note that $S^{2n}\subseteq XK_0$ and hence $S^{2n}\subseteq XA$. \cref{lem:inclusions.local} therefore implies that $S^{mn}\subseteq XA^{m-1}\subseteq XQ_0^{m-1}$ for all $m\ge2$, and hence that there exists $\xi=\xi(C,d,\eta)$ such that 
\[
XP^m\subseteq S^{mqn}\subseteq XP^{\xi m}
\]
for all $m\in\N$. We may require that $n\ge g(d)$, so that this applies in particular to any set $X\subseteq S^{g(d)}$ of coset representatives for $K_0$; fix an arbitrary such set $X$ for the remainder of this proof. We may also choose $\xi$ large enough for \cref{prop:index.reduc.tf} to apply to $P$.

Observe that the theorem is trivial when $d=0$. If $\dim P<d$, then the result follows by induction on $d$, after replacing $P_0$ with $P$, provided that $R'$ is chosen sufficiently large as a function of $C$, $d$, $\eta$, and $R$ so as to ensure that the corresponding value of $\rho$ is large enough. This is possible since $R'$ may be taken arbitrarily large in terms of $C$, $d$, $\eta$, and $R$ without affecting any of the other parameters associated with $P$.
We may therefore assume that $\dim P\ge d$. 
Set $\hat G=\la S\mid R_{G,S}(4\xi qn)\ra$, and let $\hat S$, $\hat X$, $\hat P$ and $\hat H$ be the lifts of $S$, $X$, $P$ and $H$ respectively to $\hat G$. Choose 
$R'$ large enough that $\inj P$ is at least the quantity $r(\dim P,[G:K_0],g(d),\xi)$ 
appearing in \cref{prop:index.reduc.tf}, so that $\hat P$ has infinite injectivity radius, $\hat H=\overline\gamma_{c+1}(\la\hat P\ra)$, distinct elements of $\hat X$ belong to distinct elements of $\la\hat P\ra$, and $\hat X\hat P^m\subseteq\hat S^{mqn}\subseteq\hat X\hat P^{2\xi m}$ for all $m\in\N$. Note that $\la\hat P\ra/\hat H$ has Hirsch length $\dim P\ge d$.

Write $\pi:G_{4\eta n}\to\hat G$ for the quotient homomorphism, noting that $\pi(\hat K_0)$ has finite index in $\la\hat P\ra$ and $\pi(\overline\gamma_{c+1}(\hat K_0))\subseteq\overline\gamma_{c+1}(\la\hat P\ra)=\hat H$. The kernel of $\pi$ is finite; indeed, if this were not the case then $\ker\pi\cap\hat K_0$ would be infinite, forcing $\la\hat P\ra/\hat H$ to have Hirsch length strictly less than $d$, a contradiction. Since $\overline\gamma_{c+1}(\hat K_0)$ is the maximal normal finite subgroup of $G_{4\eta n}$, this implies in particular that $\ker\pi\le\overline\gamma_{c+1}(\hat K_0)$. It follows that $\pi$ induces an isomorphism $G_{4\eta n}/\overline\gamma_{c+1}(\hat K_0)\to\hat G/\pi(\overline\gamma_{c+1}(\hat K_0))$, so that $\hat G/\pi(\overline\gamma_{c+1}(\hat K_0))$ is torsion-free and hence $\pi(\overline\gamma_{c+1}(\hat K_0))=\hat H$ and $\hat H$ is the maximal normal finite subgroup of $\hat G$. It then follows from \cref{lem:action->MLcomp} that $\hat X\hat P$ is a Mal'cev--Losert progression in $\hat G$ as required.

Finally, the fact that $\hat P$ has infinite injectivity radius means that the projection $\hat G\to G$ is injective on $\hat P^R$. Since $X$ is a set of coset representatives of $\la\hat P\ra$, this projection is also injective on $\hat X\hat P^R$ as required.
\end{proof}

\begin{proof}[Proof of \cref{prop:chain.of.progs-index}]
We will first prove the desired statement without conclusion \ref{item:chain.of.progs-index:inj}. We may assume by induction that this version of the statement holds for all smaller values of $d$, and in particular that the values $\rho(d_0,k_0,t_0,\eta_0)$ are well defined for all $d_0<d$ and all values of $k_0,t_0,\eta_0$.

Let $P_0$ be the progression obtained by applying \cref{prop:fine-scale.optimise.index} to $Q_0$, and set $r_0=qn$. If $\inj P_0=\infty$ then we have our modified statement with $d'=0$, so we may assume that $\inj P_0<\infty$. In that case, let $\eta_0=\eta_0(C,d,Q)\in\N$ be a quantity to be specified shortly and let $R'=R'(C,d,Q,R,\eta)$ be the maximum of $R$ and all possible values of $\rho(d_0,g(\dim P_0),g(\dim P_0),\eta_0)$ with $d_0<d$.

\cref{prop:chain.of.progs} gives $n'\in\N$ and an $O_{C,d,Q}(1)$-rational Lie progression $Q_1$ in $O_{C,d,Q}(1)$-upper-triangular form, with injectivity radius at least $R'$ and dimension strictly less than that of $P_0$, such that $\inj P_0\ll_{C,d,Q}n'\ll_{C,d,Q,R'}\inj P_0$, and such that 
$Q_1^m\subseteq P_0^{mn'}\subseteq Q_1^{O_{C,d,Q}(m)}$ for all $m\in\N$. Fixing a set $Y\subseteq S^{g(\dim P_0)}$ of coset representatives for $\la P_0\ra$ in $G$, it then follows from the definition of $P_0$ that $|Y|\le g(\dim P_0)$ and that $YQ_1^m\subseteq S^{mn'r_0}\subseteq YQ_1^{O_{C,d,Q}(m)}$ for all $m\in\N$. Setting $\eta_0$ to be the implied constant in this last containment, the statement without \ref{item:chain.of.progs-index:inj} then follows by induction on $d$. 

We now obtain conclusion \ref{item:chain.of.progs-index:inj}. For each $i$, write $\psi_i:G_i\to G$ for the quotient homomorphism. Let $\zeta\in(0,1]$ be a constant to be specified shortly depending only on $C,d,Q,R,\eta$. We claim that, for a suitable choice of $\zeta$, we have
$\ker\psi_i\cap\hat S_i^{\lfloor\zeta r_{i+1}\rfloor}=\{1\}$, where we interpret $\zeta r_{d'+1}$ to mean $\infty$. To see this, suppose $a\ne1$ belongs to $\hat S_i^{\lfloor\zeta r_{i+1}\rfloor}\subseteq\hat S^{\lceil\zeta(r_{i+1}/r_i)\rceil r_i}$. Fixing some set $X_i\subseteq S^{g(\dim P_i)}$ of coset representatives for $\la P_i\ra$ in $G$ containing the identity, this implies in particular that $a=\hat xp$ for some $x\in X_i$ and $p\in\hat P_i^{O_{C,d,\eta}(\zeta(r_{i+1}/r_i))}$.
Setting $\zeta$ small enough we may therefore ensure that $p\in\hat P_i^{\inj P_i}$. If $x\ne1$ then $\psi_i(a)\in x\la P_i\ra$ and so $\psi_i(a)\ne1$ as claimed. On the other hand, if $x=1$ then $a=p\ne1$, and so $\psi_i(a)=\psi_i(p)\ne1$ as claimed.

If $\zeta r_{i+1}<\alpha r_i$ then \ref{item:chain.of.progs-index:inj} is satisfied trivially for this $i$. If not, the previous claim implies that if a word of length at most $\zeta r_{i+1}$ in the elements of $S$ evaluates to the identity in $G$ then it also evaluates to the identity in $G_i=\la S\mid R_{G,S}(\alpha r_i)\ra$, so that $\la S\mid R_{G,S}(\alpha r_i)\ra=\la S\mid R_{G,S}(\zeta r_{i+1})\ra$ and \ref{item:chain.of.progs-index:inj} is again satisfied.
\end{proof}

\subsection{Polynomial-volume and doubling results}\label{sec:poly.vol.doub}
In this section we show how to use our main result to obtain the sharp, fully finitary Gromov-type theorem \cref{thm:finitary.poly-vol.optimal}, as well as its multi-scale extension \cref{cor:main.poly.vol}. After that, we obtain analogous results under a small-doubling or small-tripling hypothesis.

We start by showing how, starting with \cref{thm:rel.hom.dim.k(d)}, we can obtain a Lie progression with the correct dimension bound.
\begin{prop}\label{thm:multiscale}
Let $d,R\in\N_0$, and let $\eps_0=\eps_0(2d+1)>0$, $n_0=n_0(2d+1)\in\N$ and $k=k(2d+1)\in\N$ be such that \cref{thm:rel.hom.dim.k(d)} holds. Then there exist $\eps=\eps(d,\eps_0)>0$ and $n_1=n_1(d,R,n_0,k)\in\N$ such that if $G$ is a group with finite symmetric generating set $S$ containing the identity and
\[
|S^n|\le\eps n^{d+1}|S|
\]
for some integer $n\ge n_1$, then there exist an integer $r\in[n^{1/2},n)$, an $O_d(1)$-rational Lie progression $P$ of dimension at most $O_d(1)$ in $O_d(1)$-upper-triangular form with injectivity radius at least $R$ generating a normal subgroup of $G$, and a set $X\subseteq S^k$ of cardinality at most $k$ containing the identity such that distinct elements of $X$ belong to distinct cosets of $\la P\ra$, and such that \[XP^m\subseteq S^{mr}\subseteq XP^{O_d(m)},\] for all $m\in\N$. Moreover, if $\eps$ is allowed to depend in addition on $k$ and $R$ then we may conclude further that $\dim P\le d$.
\end{prop}

It is worth remarking at this point that although $n_1$ depends effectively on $d,n_0,k$, and as such is not adorned with an asterisk, \cref{thm:rel.hom.dim.k(d)} is not currently known to hold for any effective value of $n_0$, so \cref{thm:multiscale} cannot currently give an effective bound on $n_1$.

\begin{proof}
Note that $|S^{\lceil n^{1/2}\rceil}|\le\eps n^{d+1}|S|\le\eps\lceil n^{1/2}\rceil^{2(d+1)}|S|$. Setting $r_0=\lceil n^{1/2}\rceil$, and provided $n_1\ge n_0$, \cref{prop:XP.orig} then gives a set $X\subseteq S^k$ of cardinality at most $k$ containing the identity, and an $O_d(1)$-rational Lie progression $P_0$ of dimension at most $O_d(1)$ in $O_d(1)$-upper-triangular form and with injectivity radius at least $1$ generating a normal subgroup of $G$ such that $XP_0^m\subseteq S^{mr_0}\subseteq XP_0^{O_d(m)}$ for every $m\in\N$, and such that distinct elements of $X$ belong to distinct cosets of $\la P_0\ra$. \cref{prop:chain.of.progs} then gives non-negative integers $d'\ll_d1$ and $q_0=1<q_1<\cdots<q_{d'}$ and $O_d(1)$-rational Lie progressions $P_1,\ldots,P_{d'}$ in $O_d(1)$-upper-triangular form with injectivity radius at least $R$, each with dimension at most $\dim P_0$ and generating the same group as $P_0$, such that \[XP_i^m\subseteq S^{mr_0q_i}\subseteq XP_i^{O_d(m)},\] for all $m\in\N$ for each $i$, and such that $P_{d'}$ has infinite injectivity radius and $\inj P_i\gg_{d,R}q_{i+1}/q_i$ for all other $i$. Setting $r_i=q_ir_0$ for each $i$, this implies in particular that \[XP_i^m\subseteq S^{mr_i}\subseteq XP_i^{O_d(m)},\] for all $m\in\N$ for each $i$.

Define $i_0$ so that $r_{i_0}\le n<r_{i_0+1}$; note that, provided $n_1$ is large enough in terms of $d$ and $R$, if $i_0=0$ then $\inj P_0\ge R$ because $\inj P_0\gg_{d,R}r_1/r_0$. Setting $P=P_{i_0}$ and $r=r_{i_0}$ then satisfies all conclusions of the proposition except possibly for the final bound $\dim P\le d$, which is required for small enough $\eps$. To prove this, note that by \cref{prop:dimension.bound} we have $|S^m|\gg_{d,k}m^{\dim P}|S|$ for all $m\in\N$ satisfying $r\le m\le r\inj P$, hence $|S^m|\gg_{d,k,R}m^{\dim P}|S|$ for all $m\in\N$ satisfying $r\le m\le n$, and in particular $n^{\dim P_0}\ll_{d,k,R}|S^n|/|S|\le\eps n^{d+1}$. Since $\dim P_0$ is an integer, provided $\eps$ is chosen small enough in terms of $d$, $k$ and $R$, this forces $\dim P\le d$ as required.
\end{proof}

We now combine \cref{thm:multiscale} with \cref{thm:index.bound} to bootstrap the conclusions of \cref{thm:bgt.gromov} and obtain an \emph{effective} bound on the index, as follows.

\begin{prop}\label{thm:rel.hom.dim.k(d).effective}
For every non-negative integer $d$ there exist $n_0^*=n_0^*(d)\in\N$, and $k=k(d)\in\N$ such that if $G$ is a group with finite symmetric generating set $S$ containing the identity and
\[
|S^n|\le n^{d+1}|S|
\]
for some integer $n\ge n_0$ then there exist normal subgroups $H,\Gamma\normal G$ with $H\le\Gamma$ such that $H\subseteq S^n$, such that $\Gamma/H$ is nilpotent with class at most $O_d(1)$, and such that $[G:\Gamma]\le k$.
\end{prop}
\begin{proof}
Let $R^*=R^*(d)$ be a quantity to be specified shortly but depending only on~$d$. Provided $n_0^*$ is large enough in terms of $d$, we have $|S^{\lceil n/2\rceil}|\le n^{d+1}|S|\le\lceil n/2\rceil^{d+2}|S|$. 
Recall that we denote by $k^*(d)$ the upper bound on the index given by \cref{thm:bgt.gromov}. 
Let $k_0^*=k_0^*(d)$ be the quantity $k^*(d+2)$.

Provided $n_0^*$ is large enough in terms of~$d$ and~$R^*$, \cref{thm:multiscale} therefore give a subset $X\subseteq S^{k_0^*}$ of cardinality at most $k_0^*$ containing the identity; a Lie progression $P_0$ of dimension at most $O_d(1)$, and hence class $c\ll_d1$, in $C$-upper-triangular form for some $C\ll_d1$ and with injectivity radius at least $R^*$; and a positive integer $r_0\le\lceil n/2\rceil$ such that $XP_0^m\subseteq S^{mr_0}\subseteq XP_0^{\eta m}$ for all $m\in\N$ and some $\eta=\eta(d)$ at least the quantity $M(C,d)$ coming from \cref{thm:index.bound}, and such that distinct elements of $X$ belong to distinct cosets of $\la P_0\ra$. Let $R^*$ be the maximum possible value of $r(\dim P_0,k_0^*,k_0^*,\eta)$ in \cref{thm:index.bound}, noting that this is indeed bounded in terms of $d$ because $\dim P_0\ll_d1$. Then \cref{thm:index.bound} gives a natural number $q^*=q^*(d)$ (coming from the $O_{c,k}(1)$ notation in the conclusion of that proposition, the ineffectiveness arising via the ineffectiveness of $k_0^*$) and a normal subgroup $K\normal G$ such that $[G:K]\le g(\dim P_0)\ll_d1$ and $\gamma_{c+1}(K)\subseteq S^{r_0+q^*}\subseteq S^{\lceil n/2\rceil+q^*}$. Provided $n^*_0>2q^*$, this in turn implies that $\gamma_{c+1}(K)\subseteq S^n$, so that if we take $\Gamma=K$ and $H=\gamma_{c+1}(K)$ then we have $H,\Gamma\normal G$ and $H\subseteq S^n\cap\Gamma$, with $[G:\Gamma]\ll_d1$ and $\Gamma/H$ nilpotent of class $O_d(1)$ as required.
\end{proof}

\begin{proof}[Proof of \cref{thm:finitary.poly-vol.optimal}]
This follows immediately from \cref{thm:rel.hom.dim.k(d).effective,thm:multiscale}.
\end{proof}

\begin{proof}[Proof of \cref{cor:main.poly.vol}]
Let $\xi=\xi(d,R)\in(0,1)$ and $R_0=R_0(d)\in\N$ be quantities to be specified shortly. Set $n_0^*$ large enough that $\lceil\xi n\rceil\le n$, noting that this introduces some dependence of $n_0^*$ on $\xi$, so that $|S^{\lceil\xi n\rceil}|\le|S^n|\le\eps\xi^{-(d+1)}\lceil\xi n\rceil^{d+1}|S|$. Provided $\eps$ is small enough in terms of $d$, $R_0$ and $\xi$, and $n_0^*$ is large enough in terms of $d$ and $R_0$, \cref{thm:finitary.poly-vol.optimal} then implies that there exist an $O_d(1)$-rational Lie progression $Q_0$ of dimension at most $d$ in $O_d(1)$-upper-triangular form with injectivity radius at least $R_0$ generating a normal subgroup of $G$, a set $Z_0\subseteq S^{O_d(1)}$ of coset representatives of $\la Q_0\ra$ in $G$ with cardinality at most $O_d(1)$ containing the identity, and a natural number $p_0$ satisfying $(\xi n)^{1/2}\le p_0\le\xi n$ such that $Z_0Q_0^{m}\subseteq S^{mp_0}\subseteq Z_0Q_0^{O_d(m)}$ for all $m\in\N$. Provided $n_0^*$ and $R_0$ are set large enough, \cref{prop:chain.of.progs-index} then gives the desired result, except possibly for the bounds $r_0\le n<r_1$. Nonetheless, it does give $r_0\ll_{d,R}p_0\le\xi n$, which provided $\xi$ is small enough gives $r_0\le n$ as required. Moreover, if $r_1\le n$ then we can fix $j$ maximal such that $r_j\le n$ then relabel each index $i$ as $i-j$ and delete all the resulting negative indices to obtain $r_0\le n<r_1$.
\end{proof}

We now briefly look at finitary Gromov-type theorems with doubling or tripling hypotheses in place of polynomial-volume hypotheses. We start by showing that \cref{st:bgt.doubling} leads to a similar statement with a fully finitary conclusion expressed in terms of Lie progressions, simlarly to how \cref{thm:finitary.poly-vol.detailed} compares to \cref{thm:bgt.gromov}. Note that, combined with \cref{prop:chain.of.progs-index}, the following theorem can also be extended to a multi-scale result along the lines of \cref{cor:main.poly.vol}.

\begin{theorem}[fully finitary Gromov-type theorem under small-doubling]\label{thm:doubling.multiscale}
Let $K\ge1$ and $R\in\N_0$, and let $n_0,k_2,k_3\in\N$ with $k_3\le k_2$ be such that \cref{st:bgt.doubling} holds for this value of $K$. Suppose $G$ is a group with finite symmetric generating set $S$ containing the identity satisfying
\[
|S^{2n}|\le K|S^n|
\]
for some integer $n\ge\max\{n_0,k_2,2K^2\}$. Then there exist $r\in\N$ satisfying $n\le r_0\ll^*_{K,R}n$, an $O_K(1)$-rational Lie progression $P$ of dimension at most $\exp(\exp(O(K^2)))$ -- or, if
\begin{equation}\label{eq:tripling.multiscale}
|S^{3n}|\le K|S^n|,
\end{equation}
dimension at most $\exp(O(\log^3K))$ -- and class at most $6\log_2K$ in $O_K(1)$-upper-triangular form with injectivity radius at least $R$ generating a normal subgroup of $G$, and a set $X\subseteq S^{k_2}$ of cardinality at most $k_2$ containing the identity --  or, if \eqref{eq:tripling.multiscale} holds, a set $X\subseteq S^{k_3}$ of cardinality at most $k_3$ containing the identity -- such that distinct elements of $X$ belong to distinct cosets of $\la P\ra$, and such that $XP^m\subseteq S^{mr}\subseteq XP^{O^*_K(m)}$ for all $m\in\N$.
\end{theorem}
\begin{proof}
Set $k=k_3$ if \eqref{eq:tripling.multiscale} holds, and $k=k_2$ otherwise. \cref{prop:XP.doubling} gives a set $X\subseteq S^k$ of cardinality at most $k$ containing the identity, and an $O_K(1)$-rational Lie progression $P_0$ of dimension at most $\exp(\exp(O(K^2)))$ under the doubling assumption, or $\exp(O(\log^3K))$ under the tripling assumption, and class at most $6\log_2K$ in $O_K(1)$-upper-triangular form and with injectivity radius at least $1$ generating a normal subgroup of $G$ such that $S^{mn}\subseteq XP_0^m\subseteq S^{O_K^*(mn)}$ for every $m\in\N$, and such that distinct elements of $X$ belong to distinct cosets of $\la P_0\ra$. \cref{prop:chain.of.progs} then gives non-negative integers $d\ll_K1$ and $q_0=1<q_1<\cdots<q_{d}$ and $O_K(1)$-rational Lie progressions $P_1,\ldots,P_{d}$ in $O_K(1)$-upper-triangular form with injectivity radius at least $R$, each with dimension at most $\dim P_0$ and generating the same group as $P_0$, such that for each $i$ we have $XP_i^{m}\subseteq S^{\lambda^* mq_in}\subseteq XP_i^{O_K(\lambda^* m)}$ for all $m$ and some integer $\lambda^*=\lambda^*(K)$, and such that $P_{d}$ has infinite injectivity radius and $\inj P_i\gg_{K,R}q_{i+1}/q_i$ for all other $i$. Setting $r_i=\lambda^*q_in$ for each $i$, this implies in particular that $XP_i^m\subseteq S^{mr_i}\subseteq XP_i^{O_K^*(m)}$ for all $m\in\N$ for each $i$. If $\inj P_0\ge R$ then the theorem is satisfied by taking $P=P_0$ and $r=r_0$. If not then $r_1\ll^*_{K,R}n$, so the theorem is satisfied by taking $P=P_1$ and $r=r_1$.
\end{proof}

Using the above result to bootstrap the conclusion of \cref{thm:bgt.doubling}, we can now prove the more effective \cref{thm:bgt.doubling.k.effective}.

\begin{proof}[Proof of \cref{thm:bgt.doubling.k.effective}]
We first claim that \cref{st:bgt.doubling} holds with $k_2=g(\exp(\exp(O(K^2)))$ and $k_3=g(\exp(O(\log^3K)))$ (and some ineffective value of $n_0(K)$). Let $R^*=R^*(K)>0$ be a quantity to be specified shortly but depending only on~$K$. Let $k^*=k^*(K)$ be the quantity from \cref{thm:bgt.doubling}. Provided $n$ is large enough in terms of~$K$, \cref{thm:bgt.doubling,thm:doubling.multiscale} therefore give a subset $X\subseteq S^{k^*}$ of cardinality at most $k^*$ containing the identity; a Lie progression $P_0$ of class $c\le6\log_2K$ and dimension at most $\exp(\exp(O(K^2)))$ under the doubling assumption or $\exp(O(\log^3K))$ under the tripling assumption, in $C$-upper-triangular form for some $C\ll_K1$ and with injectivity radius at least $R^*$;  and a positive integer $r_0\ll^*_{K,R^*}n$ such that $XP_0^m\subseteq S^{mr_0}\subseteq XP_0^{\eta^*m}$ for all $m\in\N$ and some $\eta^*=\eta^*(K)$ at least the quantity $M(C,\exp(\exp(O(K^2))))$ coming from \cref{thm:index.bound}, and such that distinct elements of $X$ belong to distinct cosets of $\la P_0\ra$. Let $R^*$ be the maximum possible value of $r(\dim P_0,k^*,k^*,\eta^*)$ in \cref{thm:index.bound}. Then \cref{thm:index.bound} gives a normal subgroup $\Gamma\normal G$ such that $[G:\Gamma]\le g(\dim P_0)$ and $\gamma_{c+1}(\Gamma)\subseteq S^{r_0+O_{c,k^*}(1)}\subseteq S^{O_K^*(n)}$. The claim is therefore satisfied by taking $H=\gamma_{c+1}(K)$.

Let $i=2$ under the doubling assumption and $i=3$ under the tripling assumption. Now that we know \cref{st:bgt.doubling} holds with the desired value of $k_i$, provided $n$ is large enough in terms of~$K$, \cref{thm:doubling.multiscale} gives a subset $Y\subseteq S^{k_i}$ of cardinality at most $k_i$ containing the identity, a Lie progression $Q_0$ of class at most $6\log_2K$ and dimension at most $\exp(\exp(O(K^2)))$ under the doubling assumption or $\exp(O(\log^3K))$ under the tripling assumption, with injectivity radius at least $4\eta^*$, and a positive integer $r_0'\ll^*_{K,\eta^*}n$ such that $YQ_0^m\subseteq S^{mr'_0}\subseteq YQ_0^{\eta^*m}$ for all $m\in\N$. In particular, $YQ_0\subseteq S^{r'_0}$ and $S^{2r'_0}\subseteq YQ_0^{2\eta^*}$, so \cref{lem:sym.grp.normal} implies that the symmetry group of $Q_0$ is normal in $G$. We may therefore take $\Gamma=\la Q_0\ra$, and $H$ the symmetry group of $Q_0$.
\end{proof}

\subsection{Local Lie-group completions at all scales}\label{sec:all.scales}
In this section we prove \cref{thm:inj/Z}. Recall from \cref{section:injectModCenter} the definition of the injectivity radius of a Lie progression \emph{modulo the centre}.
\begin{prop}\label{prop:inj/Z}
In \cref{prop:chain.of.progs-index} or \cref{cor:main.poly.vol} we may conclude in addition that for each $i<d'$ we have $\injz P_i\gg_{C,d,Q,R,\eta}r_{i+1}^{c_i/(c_i-1)}/r_i$, where $c_i$ is the class of $P_i$.
\end{prop}
Here we adopt the convention that $r_{i+1}^{1/0}=\infty$; this arises only when $c_i=1$, in which case the resulting statement is trivial.
\begin{proof}
Let $R_0=R_0(C,d,\eta)$ and $A=A(C,d,Q,R,\eta)>1$ be parameters to be specified shortly; we may assume $R\ge R_0$ in \cref{prop:chain.of.progs-index} or \ref{cor:main.poly.vol}. Fix $i<d'$. The required statement holds if $c_i=1$, so we may assume that $c_i\ge2$. Let $X_i\subseteq S^{g(\dim P_i)}$ be a set of coset representatives for $\la P_i\ra$ in $G$.

We start by defining an auxiliary $O_{C,d,\eta}(1)$-rational Lie progression $Q_i$ in $O_{C,d,\eta}(1)$-upper-triangular form with the same lattice, symmetry group and projector as $P_i$, and $k\in\N$ such that $P_i^k\subseteq Q_i\subseteq P_i^{O_{C,d,\eta}(k)}$. If $r_{i+1}<2Ar_i$ then set $k=1$ and $Q_i=P_i$, noting that $\inj Q_i\ge R_0$. Otherwise, let $k\in\N$ be maximal such that $kAr_i\le r_{i+1}$, and let $Q_i$ be given by \cref{prop:powergood.rational}. In this case, $\inj Q_i\gg_{C,d,Q,R,\eta}r_{i+1}/kr_i\ge A$, which again gives $\inj Q_i\ge R_0$ provided we set $A$ large enough. Furthermore, it follows from the definition of $Q_i$ that there exist positive integers $q_i\ge kr_i\gg_{C,d,Q,R,\eta}r_{i+1}$ and $\xi\ll_{C,d,\eta}1$ such that $X_iQ_i^m\subseteq S^{mq_i}\subseteq X_iQ_i^{\xi m}$ for all $m\in\N$. Letting $T=S^{2g(\dim P_i)+1}\cap(Q_i\cup Q_i^{-1})^{2\xi}$, \cref{prop:fromStoT} therefore implies that $Q_i\subseteq T^{q_i}\subseteq Q_i^{O_{C,d,\eta}(1)}$. Setting $R_0$ large enough, it follows that we may lift $T$ to a unique subset $\tilde T$ in the lattice of $Q_i$ such that $\tilde Q_i\subseteq\tilde T^{q_i}\subseteq\tilde Q_i^{O_{C,d,\eta}(1)}$. Writing $\tilde Q_i=P(u;L)$, it then  follows from \cref{lem:upper-tri.doubling.dilate} that $\tilde T^{q_i}\subseteq P(u;O_{C,d,\eta}(L))$; since $P(u;O_{C,d,\eta}(L))$ is in $O_{C,d,\eta}(1)$-upper-triangular form, \cref{prop:LowerBoundL_i} therefore implies that $L_j\gg_{C,d,\eta} q_i\gg_{C,d,Q,R,\eta}r_{i+1}$ for each length $L_j$ of $Q_i$. Finally, provided $R_0$ is large enough, it follows from \cref{prop:ProperCenter} that $\injz Q_i\gg_{C,d,Q,R,\eta}r_{i+1}^{1/(c_i-1)}$, and hence by definition of $Q_i$ that $\injz P_i\gg_{C,d,Q,R,\eta}kr_{i+1}^{1/(c_i-1)}\gg_{C,d,Q,R,\eta}r_{i+1}^{c_i/(c_i-1)}/r_i$, as claimed.
\end{proof}

\begin{proof}[Proof of \cref{thm:inj/Z}]
Let $i<d'$. Write $\Gamma_i$ for the lattice of $P_i$, write $\Pi:G_i/\hat H_i\to G/H_i$ for the quotient homomorphism. Embedding $\Gamma_i<G_i/\hat H_i$, the restriction of $\Pi$ to $\Gamma_i$ is then the projector of $G_i$. We first claim that, for a suitable choice of $\zeta$, we have $\ker\Pi\cap\hat S_i^{\lfloor\zeta r_{i+1}^{c_i/(c_i-1)}\rfloor}\hat H_i\subseteq Z(\Gamma)$. To see this, suppose $a\in\hat S_i^{\lfloor\zeta r_{i+1}^{c_i/(c_i-1)}\rfloor}\subseteq\hat S^{\lceil\zeta(r_{i+1}^{c_i/(c_i-1)}/r_i)\rceil r_i}$. Fixing some set $X_i\subseteq S^{g(\dim P_i)}$ of coset representatives for $\la P_i\ra$ in $G$ containing the identity, this implies in particular that $a=\hat xp$ for some $x\in X_i$ and $p\in\hat P_i^{O_{C,d,\eta}(\zeta(r_{i+1}^{c_i/(c_i-1)}/r_i))}$. Thanks to \cref{prop:inj/Z}, on setting $\zeta$ small enough we may therefore ensure that $p\in\hat P_i^{\injz P_i}$. If $x\ne1$ then $\Pi(a\hat H_i)\in x\la P_i\ra/H_i$, so $\Pi(a\hat H_i)\ne H_i$ and $a\hat H_i\notin\ker\Pi$. On the other hand, if $x=1$ then $a\in\hat P_i^{\injz P_i}$, hence $a\hat H_i\in\tilde P_i^{\injz P_i}$. In particular, if $a\hat H_i\in\ker\Pi$ this forces $a\hat H_i\in Z(\Gamma)$ as claimed.

If $\zeta r_{i+1}^{c_i/(c_i-1)}\le\alpha r_i$ then the theorem is trivial, so we may restrict attention to the case where $q=\zeta r_{i+1}^{c_i/(c_i-1)}$. By definition, $\ker\Phi$ is then generated as a normal subgroup of $G_i/\hat H_i$ by $\ker\Pi\cap\hat S_i^{\lfloor\zeta r_{i+1}^{c_i/(c_i-1)}\rfloor}$, and is hence contained in $Z(\Gamma)=Z(\la\hat P_i\ra/\hat H_i)$ as required. To see that it is also normal in $\L_i$, write $N_i$ for the Mal'cev completion of $\la\hat P_i\ra/\hat H_i$, so that $\L_i=(G_i/\hat H_i)\cdot N_i$. Now $Z(\la\hat P_i\ra/\hat H_i)\le Z(N_i)$, so that $\ker\Phi$ is normalised by $N_i$. As the kernel of a homomorphism of $G_i/\hat H_i$, it is also normalised by $G_i/\hat H_i$, and hence normal in $\L_i$ as claimed.

To end the demonstration, we now prove that the quotient $\L_i/\ker\Phi$ is a virtually connected-nilpotent Lie group, and that $G_i'/\hat H_i'$ embeds discretely and cocompactly in $\L_i/\ker\Phi$. 
Since $\ker \Phi$ is discrete, $\L_i/\ker\Phi$ is a Lie group, whose connected component contains the image of the connected component of $\L_i$. It follows that the connected component of $\L_i/\ker\Phi$ is a connected nilpotent Lie group, and has index bounded by the index of the connected component of $\L_i$, which is finite. If $K$ is a compact subset of $\L_i$ such that $\L_i=(G_i/\hat H_i) K$, then by commutativity of the diagram, the image of $K'$ in $\L_i/\ker\Phi$ satisfies $\L_i/\ker\Phi=(G_i'/\hat H_i') K'$. In other words, $G_i'/\hat H_i'$ is a cocompact subgroup of $\L_i/\ker\Phi$. Finally, if $U$ is neighborhood of $1$ in $\L_i$ whose intersection with $G_i/\hat H_i$ is reduced to $1$, then $(\ker \Phi) U\cap G_i/\hat H_i= \ker \Phi$, and so the image $U'$ of $U$ in $\L_i/\ker\Phi$ only intersects the image of $G_i/\hat H_i$ (which coincides with $G_i'/\hat H_i'$) at $1$.  Since $U'$ is a neighborhood of $1$ in $\L_i/\ker\Phi$, this proves that $G_i'/\hat H_i'$ is a discrete subgroup of $\L_i/\ker\Phi$.
\end{proof}

\subsection{Proof of the additional conclusions}\label{sec:proof.additional}
In this section we prove the additional conclusions presented in \cref{sec:additional}.

\begin{proof}[Proof of \cref{thm:dimbounds}] The first inequality of \ref{conc:hdimPi} follows because $\dim P_0\le d$ and $\frac12d(d-1)+1$ is the largest possible homogeneous dimension of a simply connected nilpotent Lie group of dimension $d$. For the subsequent inequalities, write $N_i$ for the simply connected nilpotent Lie group appearing in the Mal'cev--Losert completion of $G_i$, so that $\hdim P_i=\hdim N_i$ and $\dim P_i=\dim N_i$. \cref{prop:Losert} then implies that each $N_{i+1}$ is a quotient of $N_i$ with strictly lower dimension, and hence strictly lower homogeneous dimension, so that $\hdim P_{i+1}<\hdim P_i$ as required.

Since $\inj P_i\gg_{C,d,Q,R,\eta}r_{i+1}/r_i$, the remaining bounds follow from \cref{prop:dimension.bound}.
\end{proof}

\begin{proof}[Proof of \cref{thm:growth}]
This follows from \cref{prop:growthfunction} and the bounds on the injectivity radius of the $P_i$, combined with the fact that if $f:[1,\infty)\to[1,\infty)$ is a non-decreasing continuous piecewise-monomial function with finitely many pieces, each of which has degree at most $d$, then $f(cx)\le c^df(x)$ for every $x,c\ge1$.
\end{proof}

\begin{proof}[Proof of \cref{thm:sqrt.scale}]
Abbreviate $\gamma=\diam_S(G)$, and fix $M$ with a value to be determined shortly. For each $i$, let $H_i$ be the symmetry group of $P_i$ and let $X_i\subseteq S^{g(\dim P_i)}$ be a set of coset representatives for $\la P_i\ra$ in $G$. The fact that $|G|<\infty$ implies that $\dim P_{d'}=0$, and in particular that $P_{d'}$ is abelian. It therefore suffices to show that if $i+1$ is such that $r_{i+1}>M\gamma^{1/2}$ and $P_{i+1}$ is abelian then $P_i$ is also abelian. Given such an $i$, it follows from \cref{prop:inj/Z} that
\begin{equation}\label{eq:injz.first.lace}
\injz P_{i}\gg_{C,d,Q,R,\eta}r_{i+1}^{c_{i}/(c_{i}-1)}/r_{i}\ge(M\gamma^{1/2})^{1/(c_{i}-1)}(r_{i+1}/r_{i}).
\end{equation}
We claim first that $P_i$ has class at most $2$. To see this, first note that $H_{i+1}\subseteq S^{r_{i+1}}\subseteq X_iP_i^{O_{C,d,Q,\eta}(r_{i+1}/r_i)}$, and hence $H_{i+1}\cap\la P_i\ra\subseteq P_i^{O_{C,d,Q,\eta}(r_{i+1}/r_i)}$. Provided $M$ is large enough, this implies by \eqref{eq:injz.first.lace} that $H_{i+1}\cap\la P_i\ra\subseteq P_i^{\lfloor\injz P_{i}/2\rfloor}$, and hence by \cref{lem:finite<injz} that $(H_{i+1}\cap\la P_i\ra)/H_i\le Z(\la P_i\ra/H_i)$. In particular, $\la P_i\ra/(H_{i+1}\cap\la P_i\ra)$ is a central quotient of $\la P_i\ra/H_i$. However, $\la P_i\ra/(H_{i+1}\cap\la P_i\ra)$ is isomorphic to a subgroup of the abelian group $\la P_{i+1}\ra/H_{i+1}$, and hence abelian, so that $\la P_i\ra/H_i$ is nilpotent of class at most $2$. Setting $R_0\ge10$, this implies that whenever $u,v,w$ are generators of $\tilde P_i$ we have $[[u,v],w]=1$, and hence that $P_i$ has class at most $2$ as required.

We may now use the bound $c_i\ge2$ to bootstrap the estimate \eqref{eq:injz.first.lace} and obtain
 \[
 \injz P_{i}\gg_{C,d,Q,R,\eta}M\gamma^{1/2}\frac{r_{i+1}}{r_{i}}.
 \]
On the other hand, we have
\[
G=S^\gamma=S^{\lceil\gamma^{1/2}/M\rceil r_{i+1}}\subseteq X_iP_i^{O_{C,d,Q,\eta}(\lceil\gamma^{1/2}/M\rceil r_{i+1}/r_i)}.
\]
Provided $M$ is large enough, it follows that $G\subseteq X_iP_i^{\lfloor\injz P_i/2\rfloor}$, and hence that $\la P_i\ra\subseteq P_i^{\lfloor\injz P_i/2\rfloor}$. \cref{lem:finite<injz} then implies that $\la P_i\ra/H_i$ is abelian, and hence that $P_i$ is abelian.
\end{proof}

\subsection{Uniqueness in the main theorem}\label{section:uniqueness}
In this section we prove \cref{prop:uniqueness}. The main content lies in the following lemma.

\begin{lemma}\label{lem:uniqueness}
Given $C,D,k,Q,t,\eta\in\N$ there exists $R=R(C,D,k,Q,t,\eta)\in\N$ such that the following holds. Suppose $G$ is a group with finite symmetric generating set $S$, and there are natural numbers $r\ge r'$, subsets $X,X'\subseteq S^t$ of size at most $k$, and $Q$-rational Lie progressions $P,P'\subseteq G$ of dimension at most $D$ in $C$-upper-triangular form such that $XP^m\subseteq S^{mr}\subseteq XP^{\eta m}$ and $X'(P')^m\subseteq S^{mr'}\subseteq X'(P')^{\eta m}$ for all $m\in\N$, such that distinct elements of $X$ belong to distinct cosets of $\la P\ra$ and distinct elements of $X'$ belong to distinct cosets of $\la P'\ra$, and such that $\inj P\ge R$ and $\inj P'\ge Rr/r'$. Then writing $N$ and $N'$ for the respective nilpotent Lie groups from which $P$ and $P'$ are projected, $\Gamma$ and $\Gamma'$ for their respective lattices, $H$ and $H'$ for their respective symmetry groups, and $\pi$ and $\pi'$ for their respective projectors, the following conditions are satisfied:
\begin{enumerate}[label=(\roman*)]
\item $[H:H\cap H']\le|X'|$ and $[H':H'\cap H]\le|X|$;
\item $H\cap\la P'\ra=H\cap H'=\la P\ra\cap H'$;
\item there exist a sublattice $\Lambda$ of index at most $|X'|$ in $\Gamma$ such that
\[
\pi(\Lambda)=\frac{(\la P\ra\cap\la P'\ra)H}{H},
\]
a sublattice $\Lambda'$ of index at most $|X|$ in $\Gamma'$ such that
\[
\pi'(\Lambda')=\frac{(\la P\ra\cap\la P'\ra)H'}{H'},
\]
and an isomorphism $\psi:\Lambda\to\Lambda'$ such that the diagram
\[
\xymatrix{  
  \Lambda  \ar[d]_{\pi}\ar[rrrr]^{\psi} &&&&  \Lambda' \ar[d]^{\pi'}\\
\displaystyle{\frac{(\la P\ra\cap\la P'\ra)H}{H}}&\cong&\displaystyle{\frac{\la P\ra\cap\la P'\ra}{H\cap H'}}&\cong&\displaystyle{\frac{(\la P\ra\cap\la P'\ra) H'}{H'}}
  }
\]
commutes;
\item $N\cong N'$.
\end{enumerate}
\end{lemma}
\begin{proof}
\cref{prop:powergood.rational} implies that there exists an $O_{D,Q}(1)$-rational Lie progression $P''$ in $O_{C,D,Q}(1)$-upper-triangular form with symmetry group $H'$, lattice $\Gamma'$ and projector $\pi'$ such that
\[
\tilde P'\subseteq(\tilde P')^{\Omega_{C,D,Q}(r/r')}\subseteq\tilde P''\subseteq(\tilde P')^{\lfloor r/r'\rfloor},
\]
and hence $\inj P''\ge R$ and
\[
X'(P'')^m\subseteq S^{mr}\subseteq X'(P'')^{\alpha m}.
\]
for some $\alpha\in\N$ satisfying $\eta\le\alpha\ll_{C,D,Q}\eta$ and at least the quantity $M(C,d)$ appearing in \cref{prop:index.reduc.tf}.

Note that $H\subseteq S^{r}\subseteq X'(P'')^\alpha$ and $H'\subseteq S^{r}\subseteq XP^\eta$. By \cref{lem:sym.grps.commensurable}, and choosing $R\ge3^k(D+1)\alpha$, this implies that
\[
[H:H\cap H']\le|X'|\qquad\text{and}\qquad[H':H'\cap H]\le|X|.
\]
It also implies that $H\cap\la P'\ra\subseteq S^{r}\cap\la P''\ra\subseteq(P'')^\alpha$ and $\la P\ra\cap H'\subseteq S^{r}\cap\la P\ra\subseteq P^\eta$, and hence by \cref{lem:finite.subgroup.of.proper} that
\[
H\cap\la P'\ra=H\cap H'=\la P\ra\cap H',
\]
and hence in particular that
\[
\frac{(\la P\ra\cap\la P'\ra)H}{H}\cong\frac{\la P\ra\cap\la P'\ra}{H\cap H'}\cong\frac{(\la P\ra\cap\la P'\ra) H'}{H'}.
\]

Choosing $R$ to be at least the quantity $r(D,k,t,\alpha)$ appearing in \cref{prop:index.reduc.tf}, that proposition implies that the lifts $\hat P,\hat P',\hat P'',\hat H,\hat H'$ to $\hat G=\la S\mid R_{G,S}(4\alpha r)\ra$ of $P,P',P'',H,H'$, respectively, satisfy $\inj\hat P''=\inj\hat P=\infty$, hence $\Gamma'\cong\la\hat P''\ra/H'=\la\hat P'\ra/H'$ and $\Gamma\cong\la\hat P\ra/H$, and also $[\la\hat P\ra:\la\hat P\ra\cap\la\hat P'\ra]\le|X'|$ and $[\la\hat P'\ra:\la\hat P\ra\cap\la\hat P'\ra]\le|X|$. \cref{lem:finite.subgroup.of.proper} implies that $\hat H'\cap\la\hat P\ra\le\hat H$ and $\hat H\cap\la\hat P'\ra\le\hat H'$, hence
\[
\frac{(\la\hat P\ra\cap\la\hat P'\ra)\hat H}{\hat H}\cong\frac{\la\hat P\ra\cap\la\hat P'\ra}{\hat H\cap\hat H'}\cong\frac{(\la\hat P\ra\cap\la\hat P'\ra)\hat H'}{\hat H'}.
\]
Setting $\Lambda=\hat\pi^{-1}((\la\hat P\ra\cap\la\hat P'\ra)\hat H/\hat H)$ and $\Lambda'=(\hat\pi')^{-1}((\la\hat P\ra\cap\la\hat P'\ra)\hat H'/\hat H')$, where $\hat\pi$ and $\hat\pi'$ are the projectors of $\hat P$ and $\hat P'$.  We therefore have $[\Gamma:\Lambda]\le|X'|$ and $[\Gamma':\Lambda']\le|X|$. Since $\inj\hat P'=\inj\hat P=\infty$, $\hat\pi$ and $\hat\pi'$ are injective, hence we have an isomorphism $\psi:\Lambda\to\Lambda'$ such that the diagram
\[
\xymatrix{  
  \Lambda  \ar[d]_{\hat\pi}\ar[rrrr]^{\psi} &&&&  \Lambda' \ar[d]^{\hat\pi'}\\
\displaystyle{\frac{(\la\hat P\ra\cap\la\hat P'\ra)\hat H}{\hat H}}\ar[d]&\cong&\displaystyle{\frac{\la\hat P\ra\cap\la\hat P'\ra}{\hat H\cap\hat H'}}&\cong&\displaystyle{\frac{(\la\hat P\ra\cap\la\hat P'\ra)\hat H'}{\hat H'}}\ar[d]\\
\displaystyle{\frac{(\la P\ra\cap\la P'\ra)H}{H}}&\cong&\displaystyle{\frac{\la P\ra\cap\la P'\ra}{H\cap H'}}&\cong&\displaystyle{\frac{(\la P\ra\cap\la P'\ra) H'}{H'}}
  }
\]
commutes. This also implies that $\Gamma$ and $\Gamma'$ are commensurable, and hence that $N\cong N'$.
\end{proof}

\begin{proof}[Proof of \cref{prop:uniqueness}]
Take $R$ as in \cref{lem:uniqueness}. We first claim that for each $i\in\{1,\ldots,d\}$ there exists $j\in\{1,\ldots,d'\}$ such that $r_i/r_j'\in[1/AR,AR]$. Let $j$ be maximal such that $r_j'\le r_i$, noting that there exists such a $j$ because $r_1'<r_i$. If $r_i\ge r_{j+1}'/AR$ then the claim is satisfied by $j+1$. If not then either $r_i\le r_{j+1}'/AR$ or $j=d'$, in which case $\inj P_j'\ge Rr_i/r_j'$ and \cref{lem:uniqueness} implies that $\dim P_i=\dim P_j'$. If $r_{i-1}\ge r_j'$ then the same argument would imply that $\dim P_{i-1}=\dim P_j'$, contrary to our hypotheses, so we must have $r_{i-1}<r_j'\le r_i$. If $r_j'\le r_i/AR$ then the same argument again implies that $\dim P_j'=\dim P_{i-1}$, so it must be that $r_j'\ge r_i/AR$, and the claim is satisfied.

Now let $\ell$ be maximal such that $r_\ell\le ARr_1$. The previous claim implies that $r_\ell'/r_1\in [1/AR,AR]$, and if $\ell\ne d'$ then by definition it satisfies $r_1<r_{\ell+1}'$. The same argument as in the previous claim then implies that for all $j\in\{\ell+1,\ldots,d'\}$ there exists $i\in\{1,\ldots,d\}$ such that $r_i/r_j'\in[1/AR,AR]$, proving \ref{item:unique.Haus}.

Now suppose that $i\in\{1,\ldots,d\}$ is such that either $i=d$ or $r_{i+1}\ge(AR)^2r_i$, and let $j$ be maximal such that $r_j'\le ARr_i$. If $r_i\ge r_j'$ then \ref{item:unique.Haus} implies that either $r_{j+1}'\ge r_{i+1}/AR\ge ARr_i$ or $j=d'$, whilst if $r_i<r_j'$ then we have by definition of $i$ and $j$ that either $i=d$ or $r_{i+1}\ge(AR)^2r_i\ge ARr_j'$. In either case, \ref{item:unique.progs} follows from \cref{lem:uniqueness}.
\end{proof}

\section{Applications}\label{chap:apps}

In this section we prove the corollaries stated in \cref{sec:1st.cors}, and provide details of some of the applications described in \cref{sec:apps}. We start by proving most of the corollaries of \cref{sec:1st.cors}, the one exception being \cref{cor:tao}, which we defer until \cref{sec:tao}, after we have developed the theory of vertex-transitive graphs a bit further.

\begin{proof}[Proof of \cref{thm:rel.hom.dim,thm:abs.hom.dim}]
\cref{thm:rel.hom.dim} follows from applying \cref{cor:main.poly.vol} with $R=1$. To see that this also proves \cref{thm:abs.hom.dim}, note that under the hypothesis of that corollary, \cref{thm:dimbounds} \ref{conc:hdim.bound} implies that $n^{\hdim P_0}\ll_{d}\eps n^{d+1}$, which as long as $\eps$ is small enough gives $\hdim P_0\le d$, and hence $\deg\la P_0\ra\le d$.
\end{proof}

We have already noted that \cref{cor:fin.pres.poly} is immediate from \cref{cor:finitary.poly-vol}. As advertised then, let us also record the following immediate consequence of \cref{thm:doubling.multiscale,prop:chain.of.progs-index}, which recovers \cref{thm:fin.pres} with bounds that do not depend on the size of the generating set.
\begin{corollary}\label{cor:fin.pres}
For each $K>0$ there exists $n_0^*=n_0^*(K)\in\N$ such that if $G$ is a group and $S$ is a finite symmetric generating set for $G$ containing the identity and satisfying $|S^{2n}|\le K|S^n|$ for some integer $n\ge n_0^*$ then
\[
\#\Bigl\{m\in \mathbb{N} : m\ge \log_2n \text{ and $(G,S)$ has a new relation on scale $m$} \Bigr\}\ll^*_K1.
\]
\end{corollary}

We now move on to the applications described in \cref{sec:apps}.

\subsection{Finite groups}
\begin{proof}[Proof of \cref{cor:bt}]
Let $\eps=\eps(d)$ and $n_0^*=n_0^*(d,1)$ be as in \cref{cor:main.poly.vol}, and let $M=M(d,1)$ be as in \cref{thm:sqrt.scale}. Abbreviate $\gamma=\diam_S(G)$. Given $r\le\gamma$ we have $\diam_{S^r}(G)=\lceil\gamma/r\rceil$, and hence $|G|\ge\lceil\gamma/r\rceil|S^r|/3$ by \cref{lem:growth.lb.rel.lin}. Taking $r=\lceil\gamma^{1/2}\rceil$, this implies that
\[
|G|\ge\frac16\gamma^{1/2}|S^{\lceil\gamma^{1/2}\rceil}|,
\]
and hence by \eqref{eq:bt} that
\[
|S^{\lceil\gamma^{1/2}\rceil}|\le6A^{-\frac{d+2}2}(\gamma^{1/2})^{d+1}|S|.
\]
Set $A$ so that $6A^{-(d+2)/2}=\eps$ and $D^*$ so that $\gamma^{1/2}>n_0^*$. Apply \cref{cor:main.poly.vol} with $n=\lceil\gamma^{1/2}\rceil$, and let $i$ be maximal such that the resulting $r_i$ is at most $M\gamma^{1/2}$, noting that $\dim P_i\le d$, that $[G:\la P_i\ra]\le g(d)$, and that the symmetry group $H_i$ of $P_i$ satisfies $H_i\subseteq S^{r_i}\subseteq S^{M\gamma^{1/2}}$. It follows from \cref{thm:sqrt.scale} that $P_i$ is abelian. Since $H_i\normal G$ and $\la P_i\ra/H_i$ is abelian of rank at most $\dim P_i$, this proves the result.
\end{proof}

\subsection{Vertex-transitive graphs}\label{sec:VTGs}
As we mentioned in the introduction, in a recent paper \cite{ttTrof} we proved a finitary version of Trofimov's result that a vertex-transitive graph of polynomial growth is quasi-isometric to a Cayley graph. More precisely, one of the main results of that paper showed that if we have a polynomial upper bound on the volume of a ball of sufficiently large radius in a vertex-transitive graph then there is a virtually nilpotent group that admits a certain `algebraic' quasi-isometry to that graph \cite[Corollary 2.4]{ttTrof}.

In the first result of this section, we refine this result by obtaining the optimal bounds on the index and complexity of a nilpotent subgroup of the group, and more generally showing that it satisfies the conclusions of \cref{cor:main.poly.vol}.

Before we state this result, let us recall some notation from our previous paper \cite{ttTrof}. Given a vertex-transitive graph $\G$ and a group $G\le\Aut(\G)$, we write $G_x$ for the stabiliser of a vertex $x\in\G$. Moreover, if $\G$ is a vertex-transitive graph and $H\le\Aut(\G)$ is a subgroup then we define $\G/H$ to be the quotient graph with vertices $\{H(x):x\in\G\}$, and $H(x)\sim H(y)$ in $\G/H$ if and only if there exists $x_0\in H(x)$ and $y_0\in H(y)$ such that $x_0\sim y_0$ in $\G$. Note that $\G/H$ is trivial if and only if $H$ is transitive. We call the sets $H(x)\subseteq\G$ with $x\in\G$ the \emph{fibres} of the projection $\G\to\G/H$. If $G$ is another subgroup of $\Aut(\G)$, we say that the quotient graph $\G/H$ is \emph{invariant under the action of $G$ on $\G$} if for every $g\in G$ and $x\in\G$ there exists $y\in\G$ such that $gH(x)=H(y)$. If $H$ is normalised by $G$ then $\G/H$ is invariant under the action of $G$ \cite[Lemma 3.1]{ttTrof}, and the action of $G$ on $\G$ descends to an action of $G$ on $\G/H$ \cite[Lemma 3.2]{ttTrof}. We write $G_{\G/H}$ for the image of $G$ in $\Aut(\G/H)$ induced by this action; thus $G_{\G/H}$ is the quotient of $G$ by the normal subgroup $\{g\in G:gH(x)=H(x)\text{ for every }x\in\G\}$.
\begin{corollary}\label{cor:trof}
For every integer $d\ge0$, every $\lambda\in(0,1)$ and every $R\in\N$ there exist $\eps=\eps(d)>0$ and $n_0^*=n_0^*(d)\in\N$ such that the following holds. Suppose that $\G$ is a connected, locally finite vertex-transitive graph such that
\begin{equation}\label{eq:trof.one.scale}
\beta_\G(n)\le\eps n^{d+1}\beta_\G(1)
\end{equation}
for some $n\ge n_0^*$. Let $o\in\G$, and suppose $G\le\Aut(\G)$ be a transitive subgroup. Then there is a normal subgroup $H\normal G$ such that
\begin{enumerate}[label=(\roman*)]
\item\label{item:c.i} every fibre of the projection $\G\to\G/H$ has diameter at most $n$;
\item\label{item:c.ii} $G_{\G/H}=G/H$;
\item\label{item:c.iii} the group $G_{\G/H}$ has a normal $d$-nilpotent subgroup with index at most $g(d)$ acting freely on $\G/H$;
\item\label{item:c.iv} the set $S=\{g\in G_{\G/H}:d_{\G/H}(g(H(o)),H(o))\le1\}$ is a finite symmetric generating set for $G_{\G/H}$;
\item\label{item:c.v} every vertex stabiliser of the action of $G_{\G/H}$ on $\G/H$ has cardinality at most $g(d)$;
\item\label{item:c.vi} the quotient map $\G\to\G/H$ is a Lipschitz $(1,n)$-quasi-isometry, and the map $\Cay(G_{\G/H},S)\to\G/H$, $gH\mapsto gH(o)$ is a Lipschitz $(1,1)$-quasi-isometry; and
\item\label{item:c.vii} if $h$ denotes the cardinality\footnote{Since $H$ is a normal subgroup of $G$, all its orbits have same cardinality. Moreover, this cardinality here is finite.} of an orbit of $H$ then for all $m\geq 0$ we have
\[\beta_\G(m)/h\leq \beta_{\G/H}(m)\leq \beta_\G(m+n)/h\] and
 \[\beta_{\G/H}(m)\leq |S^m|\leq g(d)\beta_{\G/H}(m).\]
\end{enumerate}
\end{corollary}
The fact that the nilpotent subgroup in \ref{item:c.iii} can be required to act freely does not feature in the statement of \cite[Corollary 2.4]{ttTrof}, but was stated with a sketch of a proof in \cite[Remark 7.3]{spanos-tointon}. The proof here is somewhat simpler than that sketch, although both rest on similar applications of \cite[Proposition 6.5]{ttTrof}. The explicit bound on the orders of the stabilisers in \ref{item:c.v} is also new, and follows immediately from the additional detail in \ref{item:c.iii}.
\begin{remark}
In the proof of \cref{cor:trof}, in order to obtain conclusion \ref{item:c.iii} we apply \cref{cor:main.poly.vol} to the ball $S^{\lceil n/2\rceil}$. Note, therefore, that one can also deduce that the group $G_{\G/H}$ and the generating set $S$ satisfy the other conclusions of \cref{cor:main.poly.vol}  and Theorems \ref{thm:dimbounds}--\ref{thm:sqrt.scale}. Moreover, using \cref{lem:growth.VT->Cay}, below, one can essentially replace $|S^m|$ in any of those conclusions with $\beta_\G(m)$. We omit the (routine) details for brevity, but see e.g. the proof of \cref{cor:tao}, below.
\end{remark}
The second result of this section is the following more detailed version of \cref{thm:trof.finite}.

\begin{corollary}\label{cor:trof.finite}
For every $d\in\N_0$ there exist $A=A(d)>0$ and $n_0^*=n_0^*(d)\in\N$ such that the following holds. Suppose $\G$ is a connected, finite vertex-transitive graph such that $\diam(\G)\ge n_0$ and
\begin{equation}\label{eq:trof.finite.hyp}
\diam(\G)\ge A\left(\frac{|\G|}{\beta_\G(1)}\right)^\frac2{d+2}.
\end{equation}
Let $o\in\G$, and suppose $G\le\Aut(\G)$ is a transitive subgroup. Then there is a normal subgroup $H\normal G$ such that
\begin{enumerate}[label=(\roman*)]
\item\label{item:c.i.fin} every fibre of the projection $\G\to\G/H$ has diameter at most $O_d(\diam(\G)^\frac{1}{2})$;
\item\label{item:c.ii.fin} $G_{\G/H}=G/H$;
\item\label{item:c.iii.fin} $G_{\G/H}$ has an abelian subgroup of rank at most $d$ and index at most $g(d)$ acting freely on $\G/H$;
\item\label{item:c.iv.fin} the set $S=\{g\in G_{\G/H}:d_{\G/H}(g(H(o)),H(o))\le1\}$ is a symmetric generating set for $G_{\G/H}$;
\item\label{item:c.v.fin} every vertex stabiliser of the action of $G_{\G/H}$ on $\G/H$ has cardinality at most $g(d)$;
\item\label{item:c.vi.fin} the quotient map $\G\to\G/H$ is a Lipschitz $(1,O_d(\diam(\G)^{\frac{1}{2}}))$-quasi-isometry, and the map $\Cay(G_{\G/H},S)\to\G/H$, $gH\mapsto gH(o)$ is a Lipschitz $(1,1)$-quasi-isometry; and
\item\label{item:c.vii.fin} we have
\[\diam(\G)-O_d(\diam(\G)^\frac{1}{2})\leq \diam(\G/H)\leq \diam(\G);\]
and
\[\diam(\G/H)\leq \diam(\Cay(G_{\G/H},S))\leq \diam(\G/H)+1.\]
\end{enumerate}
\end{corollary}
The fact that the quasi-isometries defined in \ref{item:c.vi.fin} combine to show that $\G$ is quasi-isometric to $G_{\G/H}$ as required for \cref{thm:trof.finite} is a straightforward exercise; see e.g. \cite[Lemma 5.1]{ttTrof}.

\medskip

A straightforward but technically very useful fact that we use often in the proofs of \cref{cor:trof,cor:trof.finite} is that if $\G$ is a vertex-transitive graph, $o\in\G$ is a vertex and $G\le\Aut(\G)$ is a transitive subgroup then the set $S=\{g\in G:g(o)\in B_\G(o,1)\}$ is a symmetric generating set for $G$, and in fact $S^n=\{g\in G:g(o)\in B_\G(o,n)\}$ for all $n\in\N$ \cite[Lemma 3.4]{ttTrof}. In particular, this implies that if we write $\rad_o(A)=\inf\{n\in\N_0:A\subseteq B_\G(o,n)\}$ for a subset $A\subseteq\G$ then we have
\begin{equation}\label{eq:orbit.diameter}
\rad_o(H(o))=\inf\{k\in\N:H\subseteq S^k\}
\end{equation}
for an arbitrary subgroup $H\le G$, a fact that we will often use without explicit mention.

We also have the following straightforward result.
\begin{lemma}[{\cite[Lemma 3.5]{ttTrof}}, though see {\cite[Lemma 7.1]{spanos-tointon}} for a correction to the proof]\label{lem:pi(S)}
Suppose $\Gamma$ is a vertex-transitive graph, $G\le\Aut(\G)$ is a transitive subgroup and $H\normal G$, and write $\pi:G\to G_{\G/H}$ for the quotient homomorphism. Let $o\in\G$, let $S=\{g\in G:g(o)\in B_\G(o,1)\}$ and let $T=\{g\in G_{\G/H}:gH(o)\in B_\G(H(o),1)\}$. Then $\pi(G_o)=(G_{\G/H})_{H(o)}$ and $\pi(S^n)=T^n$ for all $n\in\N$.
\end{lemma}

We obtain the improved conclusions in \cref{cor:trof,cor:trof.finite} by using the following lemma to compare the balls in the graphs $\G$ to those in the groups $G_{\G/H}$.
\begin{lemma}\label{lem:growth.VT->Cay}
Suppose $\G$ is a vertex-transitive graph, $o\in\G$ is a vertex, $G\le\Aut(G)$ is a transitive subgroup, and $H\normal G$ is a normal finite subgroup, and set $S=\{g\in G_{\G/H}:d_{\G/H}(g(H(o)),H(o))\le1\}$. Then
\[
\frac{|(G_{\G/H})_{H(o)}|}{|H(o)|}\beta_\G(m)\le|S^m|\le\frac{|(G_{\G/H})_{H(o)}|}{|H(o)|}\beta_\G(m+\rad_o(H(o)))
\]
for every $m\in\N$.
\end{lemma}
\begin{proof}
This is essentially given by the proof of \cite[Proposition 9.1]{ttTrof}; the only required modification is to replace the $O_K(n)$ appearing in (9.3) of that proof with $\rad_o(H(o))$ (indeed, that argument takes place in the setting of specific $\G$, $o$, $G$ and $H$, and the $O_K(n)$ arises precisely because that is the bound we have on $\rad_o(H(o))$ in that setting).
\end{proof}

\begin{proof}[Proof of \cref{cor:trof}]
We will set $\eps$ first and then $n_0^*$, with $n_0^*$ depending on $\eps$.

Set $S_0=\{g\in G:d_\G(g(o),o)\le1\}$. Provided $n_0^*$ is large enough in terms of $d$, \cite[Corollary 2.4]{ttTrof} implies the existence of a subgroup $H_1\normal G$ contained in $S_0^{\lfloor n^{1/2}\rfloor}$. Set $S_1=\{g\in G_{\G/H_1}:d_{\G/H_1}(g(H_1(o)),H_1(o))\le1\}$. Setting $n_0^*$ sufficiently large to force $\lceil n/2\rceil+n^{1/2}\le n$, we have by \cref{lem:growth.VT->Cay} that
\[
\frac{|S_1^{\lceil n/2\rceil}|}{|S_1|}\le\frac{\beta_\G(n)}{\beta_\G(1)}\le2^{(d+1)}\eps\lceil n/2\rceil^{d+1}.
\]
Provided $n^*_0$ is large enough in terms of $d$, and $\eps$ is small enough in terms of $d$, \cref{thm:rel.hom.dim,lem:pi(S)} therefore imply that there exist subgroups $H_2\normal G$ such that $H_1\subseteq H_2\subseteq S_0^{\lceil n/2\rceil}H_1$, and $\Gamma\normal G$ containing $H_2$ such that $\Gamma/H_2$ is $d$-nilpotent and $[G:\Gamma]\le g(d)$.

Set $S_2=\{g\in G_{\G/H_2}:d_{\G/H_2}(g(H_2(o)),H_2(o))\le1\}$. \cref{lem:poly.pigeon,lem:tripling->AG} imply that there exists a natural number $m\le n^{1/2}$ such that $S_1^m$ is an $O_{d,\eps}(1)$-approximate group, and hence by \cref{lem:pi(S)} that $S_2^m$ is an $O_{d,\eps}(1)$-approximate group. Setting $K=((\Gamma_{\G/H_2})_{H_2(o)})^{G_{\G/H_2}}$ and applying \cite[Proposition 6.5]{ttTrof}, it follows that $K\subseteq S_2^{O_{d,\eps}(n^{1/2})}$. Note that $K\le\Gamma_{\G/H_2}$ by normality of $\Gamma$. Setting $H_3\le\Gamma$ to be the pullback of $K$ to $G$, it follows that $H_2\subseteq H_3\subseteq S_0^{O_{d,\eps}(n^{1/2})}H_2\subseteq S_0^{\lceil n/2\rceil+O_{d,\eps}(n^{1/2})}$. Setting $n_0^*$ is large enough in terms of $d$ and $\eps$, we may conclude that $H_3\subseteq S_0^n$. Now let $H$ be the kernel of the induced action of $G$ on $\G/H_3$.

It follows from \cite[Lemma 3.6]{ttTrof} that $H_3\subseteq H\subseteq S_0^n$ and $G_{\G/H}=G/H$, giving \ref{item:c.i} and \ref{item:c.ii} respectively. It also follows from \cite[Lemma 3.6]{ttTrof} that $H(x)=H_3(x)$ for all $x\in\G$. Since $H_2\le H$, we have that $\Gamma H/H$ is normal in $G/H$ and $d$-nilpotent with index at most $g(d)$. We claim moreover that $\Gamma H/H$ acts freely on $\G/H$. To prove this it suffices to prove that the only $g\in\Gamma$ with $gH(o)=H(o)$ are the elements $g\in\Gamma\cap H$. To see this, note that if $g\in\Gamma$ with $gH(o)=H(o)$ then there exists $h\in H_3$ such that $g(o)=h(o)$, and hence $h^{-1}g\in\Gamma_o$. Writing $\pi:G\to G_{\G/H_2}$ for the quotient homomorphism, it follows in particular that $\pi(h^{-1}g)\in K$, hence that $h^{-1}g\in H$, and hence that $g\in H$ as required. This proves \ref{item:c.iii}. Conclusion \ref{item:c.iv} is immediate; conclusion \ref{item:c.v} follows from conclusion \ref{item:c.iii}; and conclusion \ref{item:c.vi} follows from \cite[Lemmas 5.2 \& 5.3]{ttTrof}. The conclusions of \ref{item:c.vii} follow from \ref{item:c.vi}, combined with  \ref{item:c.v} for the second statement.
\end{proof}

\begin{proof}[Proof of \cref{cor:trof.finite}]
Writing $\gamma=\diam(\G)$, the hypothesis \eqref{eq:trof.finite.hyp} implies in particular that
\[
\beta_\G(\lceil\gamma^{1/2}\rceil)\le A^{-(d+2)/2}(\gamma^{1/2})^{(d+2)}\beta_\G(1).
\]
Provided $A$ is small enough and $\gamma$ is large enough in terms of $d$ only, \cref{cor:trof} therefore gives a normal subgroup $H_1\normal G$ such that every fibre of the projection $\G\to\G/H_1$ has diameter at most $\gamma^{1/2}$, such that $G_{\G/H_1}=G/H_1$, and such that every vertex stabiliser of the action of $G/H_1$ on $\G/H_1$ has cardinality $O_d(1)$.
Set $S_0=\{g\in G:d(g(e),e)\le1\}$ and $S_1=\{gH_1\in G/H_1:d_{\G/H_1}(gH_1(o),H_1(o))\le1\}$, and note that
\begin{equation}\label{eq:trof.fin.H0}
H_1\subseteq S_0^{\lfloor\gamma^{1/2}\rfloor}.
\end{equation}
Write $\pi:G\to G/H_1$ and $\psi:\G\to\G/H_1$ for the quotient maps, and note that \cite[Lemma 3.7]{ttTrof} combines with \eqref{eq:trof.fin.H0} to imply that $\psi^{-1}(B_{\G/H_1}(H_1(o),\gamma-\lfloor\gamma^{1/2}\rfloor-1))\subseteq B_\G(o,\gamma-1)\subsetneqq\G$,
and hence that $\diam(\G/H_1)\ge\gamma-\gamma^{1/2}$.  It follows from \eqref{eq:orbit.diameter} that $\diam_{S_1}(G/H_1)=\diam(\G/H_1)$, and so as long as $\gamma\ge4$ we conclude that
\begin{equation}\label{eq:diam.quot.large}
\gamma\ge\diam_{S_1}(G/H_1)\ge\gamma-\gamma^{1/2}\ge\frac\gamma2.
\end{equation}
On the other hand, \cite[Lemma 4.8]{ttTrof} implies that
\[
\frac{|G|}{|S_1|}=\frac{|\G|}{\beta_{\G}(1)},
\]
whilst \cite[Lemma 3.5]{ttTrof} implies that $|S_1|\ge|S_0|/|H_1|$, so \eqref{eq:trof.finite.hyp} and \eqref{eq:diam.quot.large} combine to imply that
\[
\diam_{S_1}(G/H_1)\ge\frac{A}2\left(\frac{|G/H_1|}{|S_1|}\right)^\frac2{d+2}.
\]
Provided $A$ is small enough and $\gamma$ is large enough in terms of $d$, it therefore follows from \cref{cor:bt}, \eqref{eq:diam.quot.large} and \cite[Lemma 3.5]{ttTrof} that there exists $H_2\normal G$ satisfying $H_1\subseteq H_2\subseteq S_0^{O_d(\gamma^{1/2})}H_1\subseteq S_0^{O_d(\gamma^{1/2})}$ such that $G/H_2$ has an abelian subgroup of rank at most $d$ and index at most $g(d)$. We then conclude as in \cref{cor:trof}. The inequalities \ref{item:c.vii.fin} are a straightforward consequence of the two statements of \ref{item:c.vi.fin}.
\end{proof}

\subsection{Growth of balls with polynomial volume}\label{sec:tao}
\begin{proof}[Proof of \cref{cor:tao}]
We first prove the corollary without the conclusion that the boundaries are integers; we first prove this in the special case of a Cayley graph. More precisely, we show that there exist $\eps=\eps(d)>0$ and $n_0^*=n_0^*(d)\in\N$ such that if $S$ is a finite symmetric generating set of a group containing the identity such that $|S^n|\le \eps n^{d+1}|S|$ for some $n\ge n_0$ then there exists a non-decreasing continuous piecewise-monomial function $f:[1,\infty)\to[1,\infty)$ with $f(1)=1$ satisfying conditions \ref{item:degrees-rel}--\ref{item:increases}, and condition (i$'$) if $|S^n|\le \eps n^{d+1}$, such that
\begin{equation}\label{eq:tao.grp}
|S^m|\asymp_df(m/n)|S^n|\qquad\text{for all $m\ge n$}.
\end{equation}
For suitable choices of $n_0^*$ and $\eps$ we can apply \cref{cor:main.poly.vol} with $R=1$. Let $r_0,r_1,\ldots,r_{d'}$ and $P_0,P_1,\ldots,P_{d'}$ be the resulting integers and progressions, noting in particular that $d'\le d$. Note also that if we choose $\eps$ small enough then conclusion \ref{conc:hdim.bound} of \cref{thm:dimbounds} implies that if $|S^n|\le \eps n^{d+1}$ then $\hdim P_0\le d$. \cref{prop:growthfunction} implies that for each $i$ there exists a continuous, increasing, piecewise-monomial function $f_i$, with degree increasing, bounded below by $\dim P_i$, and bounded above by $\hdim P_i$, such that $|P_i^m|\ll_df_i(m)$ for all $m\in\N$ and $|P_i^m|\asymp_df_i(m)$ for all $m\le\inj P_i$. Since $f_i(\lambda x)\le\lambda^{\deg f_i}f_i(x)$ for all $\lambda>1$ and $x\ge1$, and since $r_{i+1}/r_i\ll_d\inj P_i$, we in fact have $|P_i^m|\asymp_df_i(m)$ for all $m\le r_{i+1}/r_i$, and then \cref{prop:growthfunction} implies that $|S^m|\asymp_d|P_i^{\lfloor m/r_i\rfloor}|\asymp_df_i(m/r_i)$ for all integers $m$ satisfying $r_i\le m\le r_{i+1}$.

We now define $f$ inductively on each interval $[r_i/n,r_{i+1}/n]$, $i=0,1,\ldots,d'$, with the convention that $r_{d'+1}/n=\infty$. First, define $f(1)=1$, and define. Next, assuming $f$ has been defined up to $r_i/n$, for $x\in[r_i/n,r_{i+1}/n]$ define
\[
f(x)=\frac{f_i\left(\frac{xn}{r_i}\right)}{f_i(1)}f\left(\frac{r_i}{n}\right).
\]
This $f$ is then a non-decreasing continuous piecewise-monomial function with $f(1)=1$ satisfying \eqref{eq:tao.grp}. The pieces of $f$ have degree at most $\hdim P_0$, which is at most $\frac12d(d-1)+1$ as required, and at most $d$ if $|S^n|\le \eps n^{d+1}$. The only places where the degree of $f$ can decrease are at $r_i/n$ with $i=1,\ldots,d'$, and there are at most $d$ of these. Finally, since the degree of $f_i$ is increasing from something at least $\dim P_i$ to something at most $\hdim P_i$, the number of boundaries of $f$ across which the degree increases is at most $\sum_{i=0}^{d'}(\hdim P_i-\dim P_i)$. This in turn is at most $\sum_{i=1}^d(\frac12i(i-1)+1-i)=\sum_{i=1}^d(\frac12i(i-3)+1)$, and it is easy to verify by induction that this last sum is equal to $\frac16d^3-\frac12d^2+\frac13d$. This completes the proof in the Cayley case.

We now prove the general case. First, note that $\beta_\G(\lceil n^{1/2}\rceil)\le\eps\lceil n^{1/2}\rceil^{2(d+1)}\beta_G(1)$. Setting $n_0^*$ large enough and $\eps$ small enough, we may therefore apply \cref{cor:trof} to this ball; let $H$ be the resulting subgroup and $S$ the resulting set. Setting $n_0^*$ large enough that $n/2>\lceil n^{1/2}\rceil$, \cref{lem:growth.VT->Cay} then implies that
\[
\frac{|S^{\lceil n/2\rceil}|}{|S|}\le\frac{\beta_\G(n)}{\beta_\G(1)}\le\eps n^{d+1};
\]
by conclusion \ref{item:c.v} of \cref{cor:trof}, it also implies that
\[
|S^{\lceil n/2\rceil}|\ll_d\beta_\G(n).
\]
Provided $\eps$ is small enough in terms of $d$, we may therefore apply the Cayley case of the corollary to the ball $S^{\lceil n/2\rceil}$ and obtain a non-decreasing continuous piecewise-monomial function $f:[1,\infty)\to[1,\infty)$ with $f(1)=1$ satisfying conditions \ref{item:degrees-rel}--\ref{item:increases}, and condition (i$'$) if $\beta_\G(n)\le \eps n^{d+1}$, such that
\[
|S^m|\asymp_df(2m/n)|S^{\lceil n/2\rceil}|
\]
for all $m\ge n/2$. Since \cref{lem:growth.VT->Cay} implies that
\[
\frac{|S^{m-\lfloor n/2\rfloor}|}{|S^n|}\le\frac{\beta_\G(m)}{\beta_\G(n)}\le\frac{|S^m|}{|S^{\lceil n/2\rceil}|}
\]
for all $m\ge n$, the desired result follows.

Finally, we adjust the function $f$ so that its boundaries are integers. Suppose $x$ is the minimal non-integer boundary. Let $f_-$ and $f_+$ be the pieces of $f$ either side of $x$, let $y_0$ be the next boundary, and set $y=\min\{\lceil x\rceil,y_0\}$, noting that $f(x)=f_-(x)=f_+(x)$ and $f(y)=f_+(y)$. Note that $f_-(y)\asymp_d f(y)$ by the bound on the degree. Now replace $f$ with a new function $f'$ that is equal to $f$ on $[1,x)$, equal to $f_-$ on $[x,y)$, and equal to $(f_-(y)/f(y))\cdot f$ on $[y,\infty)$. This function still satisfies all the properties we obtained before, but has either one non-integer boundary fewer than $f$ (if $\lceil x\rceil\le y_0$) or one piece fewer than $f$ (if $y_0\le\lceil x\rceil$). After repeating this process bounedly many times we arrive at a function with integer boundaries as required.
\end{proof}

\begin{proof}[Proof of \cref{cor:Tao.doubling}]
This results from using \cite[Theorem 2.3]{ttTrof} in place of \cref{cor:trof}, and using \cref{thm:doubling.multiscale,prop:chain.of.progs-index} in place of \cref{cor:main.poly.vol}, in the above argument, combined with an application of \cite[Corollary 1.3]{ttdoubling} to show that \cite[Theorem 2.3]{ttTrof} holds under the hypothesis $\beta_\G(2n)\le K\beta_\G(n)$.
\end{proof}

 \footnotesize{
  \bibliographystyle{abbrv}
  \bibliography{biblio.bib}
  }
  
\end{document}